\documentclass[11pt,twoside]{amsart}
\usepackage{amsmath, amsthm, amscd, amsfonts, amssymb, graphicx, color}
\usepackage[bookmarksnumbered, plainpages]{hyperref}
\def\Dj{\rlap{--}D}
\newtheorem{thm}{Theorem}[section]

\newtheorem{conj}[thm]{Conjecture}
\newtheorem{cor}[thm]{Corollary}
\newtheorem{lem}[thm]{Lemma}

\newtheorem{defn}[thm]{Definition}
\newtheorem{rem}[thm]{\bf{Remark}}

\numberwithin{equation}{section}
\def\pn{\par\noindent}

\newcommand{\ds}{\displaystyle}
\newcommand{\bc}{\begin{center}}
\newcommand{\ec}{\end{center}}

\makeatletter
\DeclareRobustCommand*\cal{\@fontswitch\relax\mathcal}
\makeatother

\begin{document}
%------------------------------------------------------------------------------------%

%------------------------------------------------------------------------------------%
\title{A spanning union of cycles  in  rectangular grid graphs,
 thick grid cylinders  and   Moebius strips }
\author{Jelena \Dj oki\' c, Olga Bodro\v{z}a-Panti\'{c}$^*$ and Ksenija Doroslova\v cki}

\thanks{{\scriptsize
\hskip -0.4 true cm MSC(2010): Primary: 05C38; Secondary: 05C50, 05A15, 05C30, 05C85.
%Received: 23 November 2021, Accepted: dd mmmm yyyy.\\
$*$Corresponding author}}
\maketitle

\begin{abstract} Motivated to find the  answers to some of the questions that have occurred in  recent papers dealing with Hamiltonian cycles (abbreviated HCs) in some special classes of grid graphs we started the  investigation of  spanning unions of cycles, the  so-called 2-factors,  in these graphs (as a generalizations of HCs).
For all the  three types of  graphs from the title and  for any integer $m \geq 2$ we  propose an algorithm for obtaining a specially designed (transfer)  digraph
${\cal D}^*_m$. The problem of enumeration  of 2-factors is reduced to the problem of enumerating oriented walks in this digraph.
Computational results we gathered for $m \leq 17$ reveal some interesting properties both for the digraphs ${\cal D}^*_m$ and for the sequences of  numbers of  2-factors.
We prove some of them  for arbitrary $m \geq 2$.
\end{abstract}

\vskip 0.2 true cm

%------------------------------------------------------------------------------------%

\pagestyle{myheadings}
\markboth{\rightline {\sl \hskip 8.5 cm J. \Dj oki\' c,  O.  Bodro\v{z}a-Panti\'{c} and K. Doroslova\v cki }}
         {\leftline{\sl \hskip 8.5 cm  J. \Dj oki\' c,  O.  Bodro\v{z}a-Panti\'{c} and K. Doroslova\v cki}}

\bigskip
\bigskip

%------------------------------------------------------------------------------------%
%------------------------------------------------------------------------------------%

\section{\bf Introduction}
\vskip 0.4 true cm
\label{sec:intro}

We consider the following (labeled) graphs:
Rectangular grid graph $RG_m(n) = \ P_{m} \times P_{n}$,
   Thick grid cylinder  $TkC_m(n) = P_{m} \times  C_n$ and   Moebius strip (of fixed width)  $MS_m(n)$ (see Figures~ \ref{Pravougaonimrezni}- \ref{Mebijusovatraka}) where
 $P_n$ and $C_n$ denote the path and  cycle with $n$ vertices, respectively.
 {\em    Thick grid cylinder} $TkC_m(n)$  ({\em Moebius strip} $MS_m(n)$) can be obtained from the
 rectangular grid graph $RG_m(n+1)$ by contraction of vertices $A\equiv B_1, B_2, \ldots , B_{m-1}, B_m \equiv B$
  with vertices   $D \equiv D_1, \ldots , D_{m-1}, D_m \equiv C $  ($C \equiv D_m, D_{m-1}, \ldots ,D_2 , D_1 \equiv D $), respectively,
which does not produce  multiple  overlapping vertical edges.
Note that all observed graphs have  $m \cdot n$ vertices.

\begin{figure}[htb]
\begin{center}
\includegraphics[width=4in]{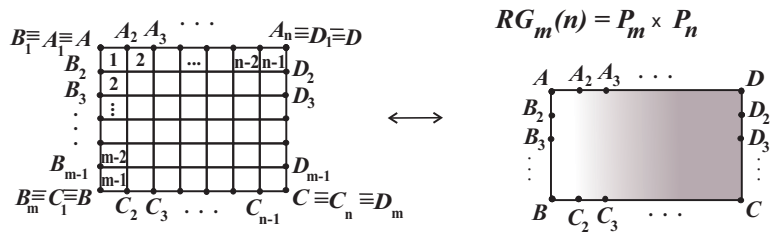}
\\ \ \vspace*{-18pt}
\end{center}
\caption{Rectangular Grid Graph  $P_{m} \times P_{n}$.}
\label{Pravougaonimrezni}
\end{figure}

\begin{figure}[htb]
\begin{center}
\includegraphics[width=4in]{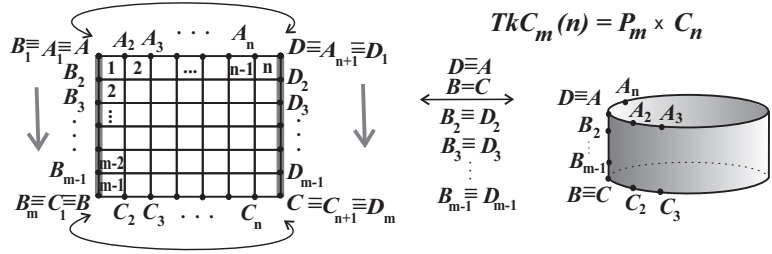}
\\ \ \vspace*{-18pt}
\end{center}
\caption{Identification of vertices $B_i$ with vertices $D_i$  in constructing the tick grid cylinder $TkC_m(n) = P_{m} \times C_{n}$.}
\label{Sirokicilindar}
\end{figure}

\begin{figure}[htb]
\begin{center}
\includegraphics[width=4in]{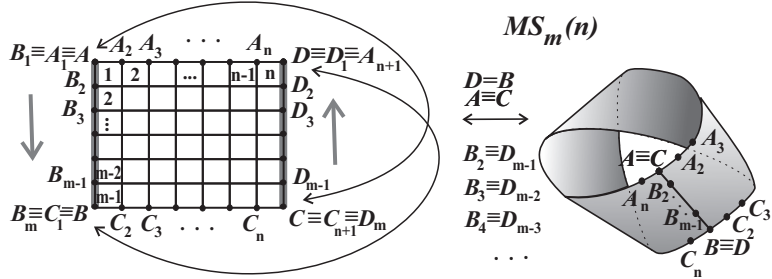}
\\ \ \vspace*{-18pt}
\end{center}
\caption{Identification of vertices $B_i$ with vertices $D_{m-i+1}$  in constructing the Moebius strip $MS_m(n) $.}
\label{Mebijusovatraka}
\end{figure}

A spanning r-regular subgraph of a graph is called  an {\em $r$-factor}.  For  $r=2$, it represents a spanning union of (disjoint) cycles.
Hence, Hamiltonian cycles are   connected 2-factors. All possible 2-factors of
 $RG_{4}(3) = P_4 \times P_3$ and  $TkC_{2}(2) = P_2 \times C_2$ are shown in  Figure~\ref{maliCilindar}.
  With the exception of  the last case, for both of the aforementioned graphs, all 2-factors  are Hamiltonian cycles.

\begin{figure}[htb]
\begin{center}
\includegraphics[width=4in]{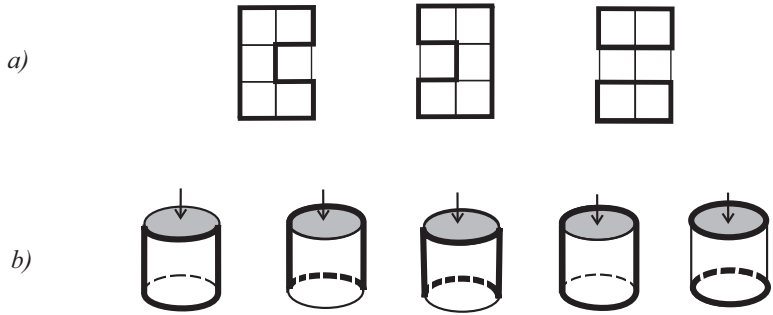}
\\ \ \vspace*{-18pt}
\end{center}
\caption{2-factors in a) \  $P_4 \times P_3$ b) \ $P_2 \times C_2$.}
\label{maliCilindar}
\end{figure}

The problems of enumerating and generating Hamiltonian paths in different classes of graphs arise in chemistry, biophysics (polymer melting  and protein folding), theoretical physics (study of magnetic systems with O(n) symmetry)\cite{J},  engineering (path planning problems for robots and machine tools)\cite{NW}
 and bioinformatics (security and intellectual property  protection by using the microelectrode dot array (MEDA) biochips) \cite{LCK},
  as well as in the theory of algorithms \cite{M}.
They  might be useful for the development of statistical algorithms that provide unbiased sampling of such paths \cite{KLO}.

A brief overview of the chronology of research on counting Hamiltonian cycles in different graph families can be found in \cite{VZB}.
 The enumeration of HCs on   specific  grid graphs has been studied extensively in
\cite{BKDP1}-\cite{BT94}, \cite{Kar}, \cite{KaP} and \cite{P}.
 The intrinsic properties of these grids naturally impose the transfer matrix approach on the problem of enumeration of HCs  and related topics\cite{EJ,KJ1}.
From the computational data obtained from some  recent papers a few  interesting phenomena (concerning  the rectangular grid graphs,
thin and  thick grid cylinders and their triangular variants) have arisen and were formulated as conjectures.
More precisely, the numbers of the  so-called contractible and non-contractible HCs (see Figure~\ref{primerRC}a-b) for thin cylinder graph ($C_{m} \times P_{n}$)  are  asymptotically equal (when $n \rightarrow \infty$) \cite{BKP} and the same is valid  for its triangular variant \cite{BKDP2}.

For the thick grid cylinder   $P_{m} \times C_{n}$ the  contractible  HCs are more numerous than the  non-contractible ones iff  $m$ is even, i.e. the total number of HCs
 \[ h_m(n)
  \sim \left\{\begin{array}{ll}
  a_{m,c} n \theta_{m,c}^n, & \mbox{if $m$ is even}, \\
   a_{m,nc}\theta_{m,nc}^n, & \mbox{if $m$ is odd },
  \end{array}\right. \]
where $\theta_{m,c}$, $\theta_{m,nc}$,  $a_{m,c}$ and $a_{m,nc}$ are the positive dominant
characteristic roots and their coefficients  for two types of HCs, respectively \cite{BKDjDP}.
 Additionally, the coefficient $a_{m,nc} $ for non-contractible HCs is equal to $1$ (computational data for $m \leq 10$) \cite{BKDP1}.
Also, positive dominant characteristic root  $\theta_{m,c}$ for contractible HCs in a thick grid cylinder is equal to the same one associated with rectangular grid graph $P_{m} \times P_{n}$  (computational data for $m \leq 10$) \cite{BKDP1,BPPB}.

The aim of this paper is to find the generating functions for the number of 2-factors in    the considered  graphs.  We were  wondering
if the same or similar properties   related to HCs   would  remain valid for 2-factors or not. We wanted to see if  some conclusions for 2-factors could
help proving the mentioned conjectures  for HCs.
 Additionally, we expand our research to the new class of grid  graphs -  Moebius strips  $MS_m(n)$.

We distinguish two types of cycles of a 2-factor   on the cylindrical
surface of $TkC_m(n)$ (viewed as closed Jordan curves): the \emph{ contractible} (abbr. \emph{c-type}) and the \emph{non-contractible} (abbr. \emph{nc-type}) ones  (see
Figure~\ref{primerRC} a-b).
The first ones divide the surface into one finite (called the \emph{ interior})
and one infinite region (called the \emph{ exterior}). One could imagine them being pasted onto the
cylindrical surface of $TkC_m(n)$.
The latter ones  divide the cylindrical surface into two infinite regions  resembling  a bracelet around an arm.
In Figure~\ref{primerRC})d), the  shown  2-factor consists of two contractible cycles and one  non-contractible cycle.

\begin{figure}[htb]
\begin{center}
\includegraphics[width=5in]{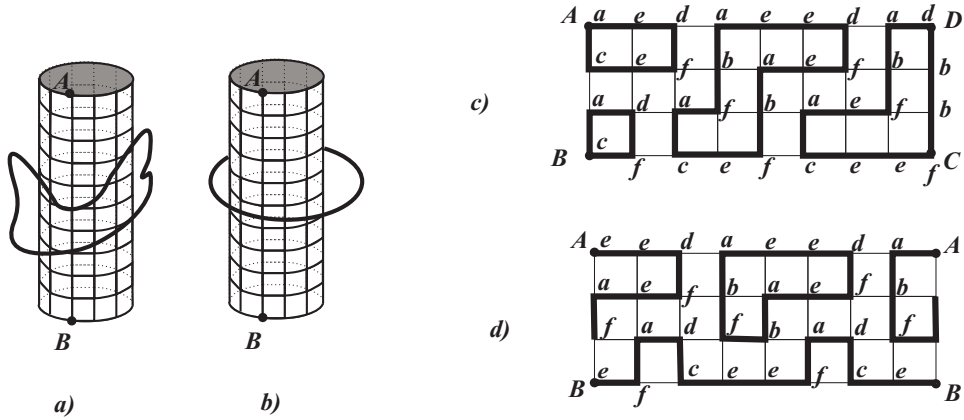}
\\ \ \vspace*{-18pt}
\end{center}
\caption{Types  of cycles in  $TkC_{m}(n)$ : a) contractible; b) non-contractible;
c) four (contractible) cycles in $RG_{4}(9)$; d)  one non-contractible and two contractible cycles in  $TkC_{4}(8)$}
\label{primerRC}
\end{figure}

For graphs $TkC_m(n)$ and  $MS_m(n)$ it is useful to observe the so-called  \emph{Rolling imprints} (RI) (introduced in \cite{BKDP1}):
Imagine that we at first  ``cut'' the surface of the  observed graph (with a given 2-factor) along the line AB (producing on this way the vertices $D_i$, i=1,2,  \ldots ,$m$ on the right side, again). Afterwards, we unroll (unwind) and flatten it.
Then, we produce infinitely many  copies $R_k \equiv A^{(k)}B^{(k)}C^{(k)}D^{(k)}(k \in Z)$   of this rectangle picture (with adding the superscript ``(k)'' on all vertex labels), and line them up to the left and to the right of the first one
 ($ R_0  \equiv A^{{0}}B^{{0}}C^{{0}}D^{{0}}) $ using translation and  glide-reflection, taking care that corresponding vertices ($B^{(k)}_i$ and $D^{(k-1)}_i$ for $TkC_m(n)$, or $B^{(k)}_i$ and $D^{(k-1)}_{m-i+1}$ for $MS_m(n)$) of adjacent copies are  contracted (see  Figure~\ref{primerTkC+MS}c)-d)).

\begin{figure}[htb]
\begin{center}
\includegraphics[width=4in]{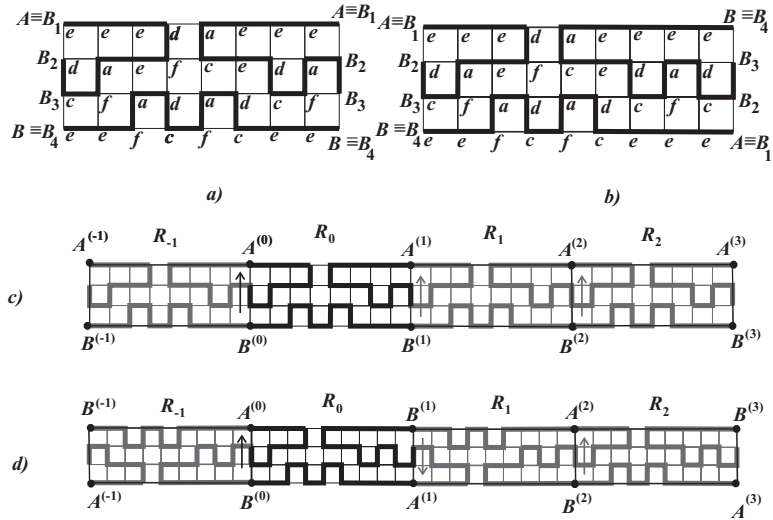}
\\ \ \vspace*{-18pt}
\end{center}
\caption{a) The spanning union of one c-type and one nc-type cycle in  $TkC_4(8)$; \  b) Hamiltonian  (short) cycle in   $ MS_4(9)$;
\ c) Rolling imprints for the example a) related to  $TkC_4(8)$;
\ d)  Rolling imprints for the example b) related to $ MS_4(9)$}
\label{primerTkC+MS}
\end{figure}

Note that, for each contractible cycle for  any type of observed  graphs,  there exist an  infinite number of  congruent cycles (polygons) in the infinite grid graph (RI).
For a non-contractible cycle in $TkC_m(n)$ there exists a  unique  infinite path which crosses lines $A^{(k)}B^{(k)}$ an  odd number of times.

In case of  $MS_m(n)$, due to topological reasons, there are three possible types of cycles: c-type and two nc-types shown in
Figure~\ref{seckanje}a) and b). The first of these  nc-cycles is called the  \emph{ long } nc-cycle. Its image in RI is the union of two infinite paths that cross the  line $A^{(k)}B^{(k)}$ an even number of times.
 The second one is called the  \emph{  short non-contractible} cycle (abbr. \emph{ short nc-type}).  Its image in Rolling imprints is the unique infinite path which crosses line $A^{(k)}B^{(k)}$ an odd number of times. A 2-factor of $MS_m(n)$  can have  at most one short nc-type cycle
(due to topological reasons, too). Additionally, the long  nc-cycle divides the surface of  $MS_m(n)$ into two parts, while the short one cannot divide that surface.
For example, the 2-factor shown in  Figure~\ref{seckanje}d) is the union of one c-cycle, two long nc-cycles and one short nc-cycle.

\begin{figure}[htb]
\begin{center}
\includegraphics[width=4in]{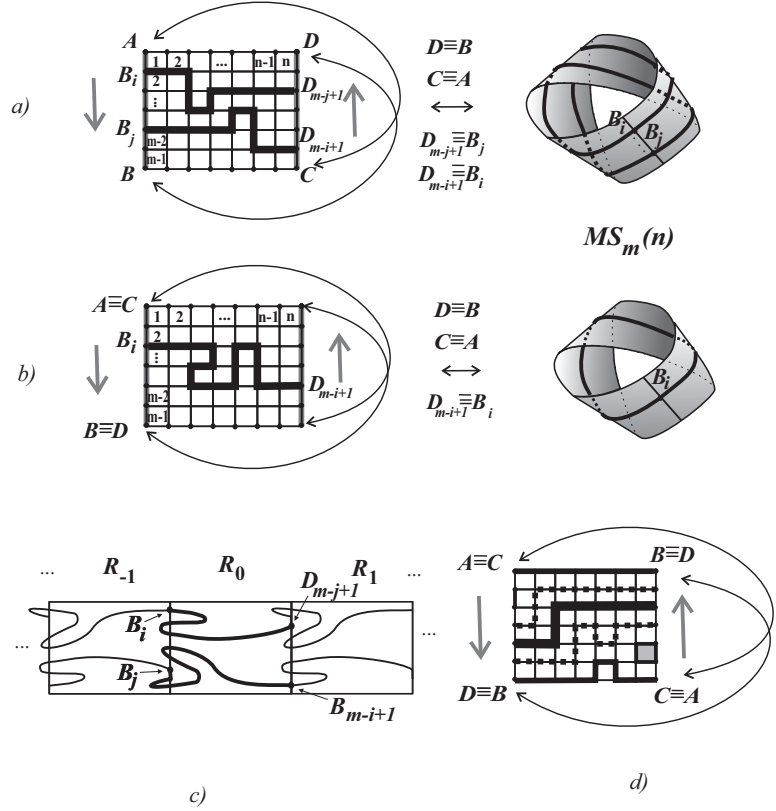}
\\ \ \vspace*{-18pt}
\end{center}
\caption{Type of cycles in   Moebius strip $MS_{m}(n)$: a) long nc-type cycle; b) short nc-type cycle; c) a long nc-cycle can cross the vertical side [AB]  more than 2 times;
d) an example of 2-factor with one short and two long nc-type cycles and one contractible cycle.}
\label{seckanje}
\end{figure}

\begin{defn} \hspace*{1cm}
\begin{itemize}
\item
We orient a (contractible) cycle $C$ in $ RG_m(n)$  clockwise and call it the  {\bf base figure}.
\item
If  a contractible cycle $C$ in  $TkC_m(n)$ or $MS_m(n)$ is disjoint with segment $[AB]$, then we orient
the  corresponding  cycle in RI that   lies entirely in the rectangle $R_0  \backslash [A^{(1)}B^{(1)}]$ clockwise and call it  the  {\bf base figure}.
Otherwise, let $B_i$ ($1 \leq i \leq m$) be the vertex from  the intersection $C \cap [AB]$ with minimal index $i$.  The  {\bf base figure} is the  corresponding  cycle in rolling imprints that contains the vertex $B^{(0)}_i$ oriented   clockwise, too.
\item
For  a non-contractible cycle $C$ in  $ TkC_m(n)$  (or short nc-cycle $C$ in $MS_m(n)$), let $B_i$ ($1 \leq i \leq m$) be the vertex from the  intersection $C \cap [AB]$ with minimal index $i$ ($1 \leq i \leq m$). We orient the part  of  image of $C$ in RI  from $B^{(0)}_i$  to $D^{(0)}_i \equiv B^{(1)}_{i} $ ($D^{(0)}_{m-i+1} \equiv B^{(1)}_i $) in this direction and call it the  {\bf base figure}.
\item
 Finally,  let $B_i$  ($1 \leq i \leq m$) denote  the vertex of segment  $[AB] $  with minimal index $i$ ($1 \leq i \leq m$) which belong  a long  nc-cycle $C$ in   $MS_m(n)$. The  {\bf base figure} is the part of the  infinite path (in RI) containing the vertex $B^{(0)}_i$ which is
   determined by vertices $B^{(0)}_i$ and  $D^{(1)}_{m-i+1} \equiv B^{(2)}_i$ and oriented from $B^{(0)}_i$ to  $D^{(1)}_{m-i+1}$.
\end{itemize}
\end{defn}

In this way, we establish a bijection between the  set of  all edges of the union of base figures for the considered 2-factor and  the set of all of  its edges.
The edges of   base figures, considered as oriented segments in rolling imprints,
 can be treated as unit vectors of 4 possible directions ($\uparrow$, $\downarrow$, $\rightarrow$ and $\leftarrow$).
Let  $\sharp_C(\uparrow)$,  $\sharp_C(\downarrow)$, $\sharp_C(\rightarrow)$ and  $\sharp_C(\leftarrow)$ denote the number of edges of corresponding direction which we pass  when walking through  the base figure of a cycle $C$.

 \begin{rem} \label{Primedba1}
  If $C$ is a c-cycle  (for all three graphs  $RG_m(n)$, $TkC_m(n)$ and $MS_m(n)$),  then
  \begin{eqnarray}  \label{r1} \sharp_C(\uparrow)= \sharp_C(\downarrow) \mbox{ and }\sharp_C(\rightarrow)= \sharp_C(\leftarrow). \end{eqnarray}
      If $C$ is an nc - cycle  for  $TkC_m(n)$, then
  \begin{eqnarray}  \label{r2} \sharp_C(\uparrow)= \sharp_C(\downarrow) \mbox{ and }  \sharp_C(\rightarrow)- \sharp_C(\leftarrow) = n.   \end{eqnarray}   \noindent
   If $C$ is a short nc - cycle for    $MS_m(n)$, then
  \begin{eqnarray}  \label{r3} \sharp_C(\downarrow)  -  \sharp_C(\uparrow) =  m+1 -2i  \mbox{ and } \sharp_C(\rightarrow)- \sharp_C(\leftarrow) = n.   \end{eqnarray}
    If $C$ is a long nc - cycle  for  $MS_m(n)$, then
         \begin{eqnarray}  \label{r4} \sharp_C(\uparrow)= \sharp_C(\downarrow) \mbox{ and } \sharp_C(\rightarrow)- \sharp_C(\leftarrow)= 2n. \end{eqnarray}  \end{rem}
Let us  denote the total numbers of 2-factors in $RG_m(n)$, $TkC_m(n)$ and  $MS_m(n)$ by $f^{RG}_m(n)$, $f^{TkC}_m(n)$ and $f^{MS}_m(n)$, respectively.
The total number of 2-factors in $TkC_m(n)$  which contain odd (even) number of non-contractible cycles is labeled by $f^{TkC}_{1,m}(n)$ ($f^{TkC}_{0,m}(n)$).
Similarly, label $f^{MS}_{1,m}(n)$ ($f^{MS}_{0,m}(n)$) represents  the total numbers of 2-factors in $MS_m(n)$ which contain (do not contain) a short cycle.
Thus, we have
 $$f^{TkC}_m(n) = f^{TkC}_{1,m}(n) + f^{TkC}_{0,m}(n) \mbox{ \ \   and  \ \ } f^{MS}_m(n) = f^{MS}_{1,m}(n) + f^{MS}_{0,m}(n).$$
\begin{thm}  \label{thm:prva} \hspace*{1cm}
\\  \hspace*{0.5cm}
 a) \  $ f^{RG}_m(n)=0  $ \ \ if  and only if \ \ both  $m$    and   $ n$ ($m, n \geq 2$) are odd; \\
\hspace*{0.5cm} b) \ $f^{TkC}_{0,m}(n)=0 $ \ \ if  and only if \ \ both  $m$    and   $ n$  ($m, n \geq 1$) are odd; \\
\hspace*{0.5cm} c) \  $  f^{TkC}_{1,m}(n)=0 $  \ \ if  and only if \ \  $ m$ is even   and $ n$ is  odd ($m, n \geq 1$); \\
\hspace*{0.5cm} d) \  $ f^{MS}_{0,m}(n)=0 $  \ \ if  and only if \ \ both  $m$    and   $ n$ ($m, n \geq 1$) are odd; \\
\hspace*{0.5cm} e) \  $ f^{MS}_{1,m}(n)=0 $  \ \ if  and only if \ \ both  $m$    and   $ n$ ($m, n \geq 1$) are even.
\end{thm}
\begin{proof}

 We derive sufficiency of the corresponding  condition (relative to parity of  $m$ and $n$)  by contraposition. \\
a) If   there exists  a 2-factor in $RG_m(n)$, then by using \eqref{r1} we conclude that the number of its  edges must be even. Consequently, it is not possible for both  $m$    and   $ n$ to be  odd. \\
b)  If   there exists  a 2-factor in $TkC_m(n)$ with an even number of nc-cycles, then by using  \eqref{r1} and \eqref{r2} we obtain the same conclusion  as in case a). \\
c)  If   there exists  a 2-factor in $TkC_m(n)$ with an odd number of nc-cycles, then by using  \eqref{r1} and \eqref{r2} we conclude that the number of its  edges $m \cdot n$ must be of the same parity as $n$. It further implies that   $ m$ is odd or   $ n$ is even, which is the negation of the  given condition. \\
d)  Suppose that there exists a 2-factor in  $MS_m(n)$ without a  short nc-cycle. By using \eqref{r1} and \eqref{r4} we arrive at the same conclusion as in cases a) and b).\\
e)   If   there exists  a 2-factor in $MS_m(n)$ with a  short nc-cycle, then from \eqref{r1}, \eqref{r3} and \eqref{r4} the number of edges ($m \cdot n$) must be of the same parity as $m+n+1$. This implies that $m$ and $n$ can only not  both  be even.

The proofs of necessity of these conditions go by construction of 2-factors for each of the  three remaining  combinations of $m$ and $n$ for each item separately.
For example, one of the  possible 2-factors for each of the three admitted  combinations  in  case e) (2-factors with a short nc-cycle  on $MS_m(n)$) are show in  Figures~\ref{egzistencija}. The rest of the proof   is left  to readers as an exercise.

\begin{figure}[htb]
\begin{center}
\includegraphics{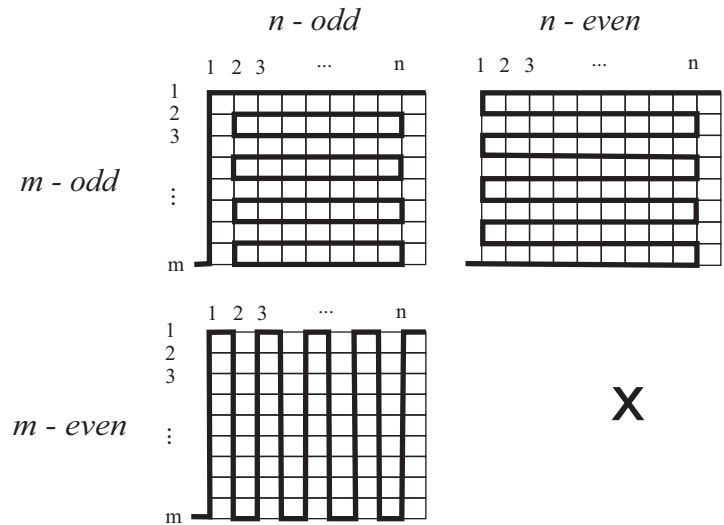}
\\ \ \vspace*{-18pt}
\end{center}
\caption{The existence of 2-factors in $MS_m(n)$  with a short nc-cycle.}
\label{egzistencija}
\end{figure}
\end{proof}

In Section~2,   we present a characterization of a 2-factor for each of the  considered graphs $G = G_m(n) \in   {\cal G} \stackrel{\rm def}{=} \{  RG_{m}(n),  TkC_m(n),  MS_{m}(n)\}$ obtained by the  vertex-coding approach.
In Section~3, using this  we propose applying the  transfer matrix method in order  to obtain  the numbers of 2-factors, labeled by $f^{G}_m(n)$,
 for the considered graph $G_m(n) \in   {\cal G} $. Actually, this enumeration problem  is reduced to the problem of enumerating oriented walks in
  a specially designed   digraph  ${\cal D}^*_m$ (so-called \emph{transfer digraph}).
Computational results we gathered for $m \leq 17$ (partially given in Section 4, the rest of them in Appendix) reveal some interesting properties of  the digraphs ${\cal D}^*_m$. We prove some of them  for arbitrary $m \geq 2$ in Section 3.
The properties referring to the asymptotic behaviour  of the  numbers of 2-factors $f^{G}_m(n)$  were observed to be similar to the ones that  appeared  while studying  Hamiltonian cycles.
In Section~4, we propose a few conjectures. The obtained generating functions for the sequences $f^{G}_m(n)$, for $m \leq 10$ are given in Appendix.

% =========
% Section 2
% =========
\section{\bf \ \ Code matrix}
\label{sec:ACV}

Let us  consider   all graphs: $RG_{m}(n)$, $TkC_m(n)$ and   $MS_{m}(n)$ simultaneously.
 First, ``cut and develop  in the plane" the  surfaces of
given graphs  as shown in  Figures~\ref{Pravougaonimrezni} - \ref{Mebijusovatraka}. Observe the corresponding rectangular grid graphs, whereby the described
identifications of vertices  have been performed.
  In this way, we can use the words: \emph{left, right, upper } and
  \emph{lower}    to mark the positions of  adjacent vertices in the  considered graph $G$, with respect to each other.

We label  each vertex of $G \equiv G_{m}(n)$ by an ordered pair \ $(i,j) \in  \{1,2, \ldots, m \} \times \{1,2, \ldots, n \}$ \  where \ $i$ \ represents
the ordinal number of the row viewed from  up to down, while \ $j$
\ represents the ordinal number of the column, viewed from left to
right. The vertices labeled by $A_1, A_2, \ldots A_n$ in  Figures~\ref{Pravougaonimrezni} - \ref{Mebijusovatraka} belong to the first row, and the vertices labeled by $B_1, B_2, \ldots B_m$ belong to the first column.

Let us observe an arbitrary 2-factor of $G$.
 One of six possible labels shown in Figure~\ref{cvornikod1} is
assigned to  each vertex \ $(i,j) \in   \{1,2, \ldots, m \} \times \{1,2, \ldots, n \}$.
 We call   this label an
  \emph{alpha-letter} of the
 vertex and denote it  by $\alpha_{i,j}$.
By reading alpha-letters  for vertices from the   same column in considered grid graph, from up to down,
we obtain an \emph{alpha word}.
For instance,  in the first 2-factor shown in  Figures~\ref{egzistencija}, for the first three  and for the last ($n$th) column the  corresponding alpha words are
$ab^{m-2}f$, $\ds e (ac)^{(m-1)/2}$, $e^m$ and $e (df)^{(m-1)/2}$, respectively.

For each alpha-letter $\alpha_{i,j}$, denote by $\overline{\alpha}_{i,j}$  ($\alpha'_{i,j}$) the adequate label shown in Figure~\ref{cvornikod1}  obtained
 by applying reflection  symmetry with the horizontal (vertical) axis as its line of symmetry onto  $\alpha_{i,j}$. Thus, $\overline{a} = c, \overline{b} = b, \overline{c} = a$,
 $ \overline{d} = f, \overline{e} = e$ and  $\overline{f} = d$ and $a' = d, b' = b, c' =f, d ' = a, e ' = e$ and $f ' = c$.

\begin{figure}[htb]
\begin{center}
\includegraphics{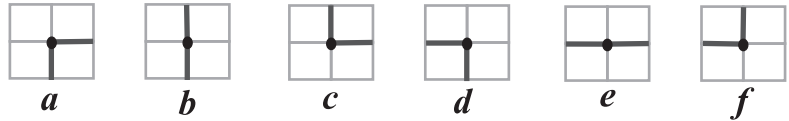}
\\ \ \vspace*{-18pt}
\end{center}
\caption{Six possible
situations for given 2-factor in any vertex}
\label{cvornikod1}
\end{figure}

\begin{defn}
For any alpha-word $\alpha \equiv \alpha_{1}\alpha_{2} \ldots \alpha_{k-1}\alpha_{k}$ $(k \in N)$,  the alpha word $\overline{\alpha} \equiv \overline{\alpha}_{k}\overline{\alpha}_{k-1} \ldots \overline{\alpha}_{2}\overline{\alpha}_{1}$ (obtained by reflection over horizontal axis) is called  {\bf horizontal conversion} of the word $\alpha$. \\
The alpha word $\alpha' \equiv \alpha'_{1}\alpha'_{2} \ldots\alpha'_{k-1}\alpha'_{k}$ (obtained by reflection over vertical axis) is called  {\bf vertical conversion} of the word $\alpha$.
\end{defn}

If we know the alpha-letter of a vertex \ $(i,j)$ in $G_m(n)$,  then
 the alpha-letter of
 its adjacent vertex can not be just about  any letter from the  set \ $\{ a,b,c,d,e,f
 \}$. Instead, it is  determined by  the
 digraphs \ ${\cal D}_{lr}$ \ and \ ${\cal D}_{ud}$, \ as  shown in  Figure~\ref{cvornikod2}, \
 depending on the mutual position of the two adjacent vertices.

 \begin{figure}[htb]
\begin{center}
\includegraphics[width=3.5in]{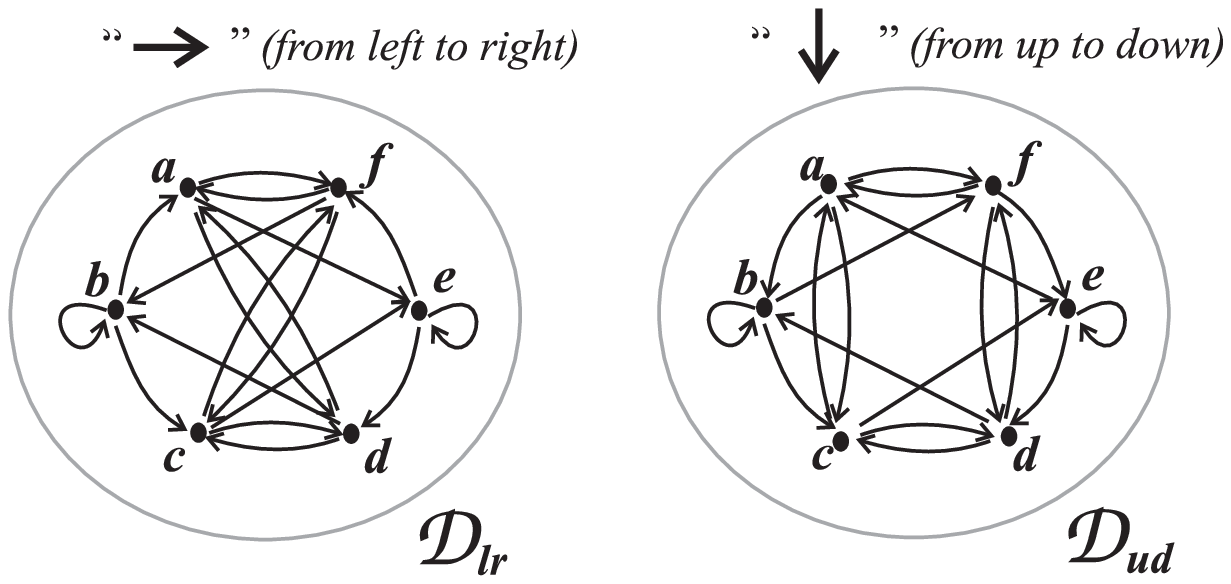}
\\ \ \vspace*{-18pt}
\end{center}
\caption{Digraphs ${\cal D}_{ud}$ and  ${\cal D}_{lr}$ and }
\label{cvornikod2}
\end{figure}

For example,  for $1 \leq j \leq n-1$ and $1 \leq i \leq m-1$ , if \ $\alpha_{i,j} = a$, \ then \ $\alpha_{i,j+1}
\in \{ d, e, f \}$ \ and \ $\alpha_{i+1,j} \in \{ b, c, f \}$. But, if $j=n$ and  $G=MS_m(n)$, then   $\alpha_{i,n} = a$ implies  $\alpha_{m-i+1,1} \in \{ d, e, f \}$ because  of
$D_{m-i+1} \equiv B_i$ and $\overline{\alpha}_{m-i+1,n} = c$.  Note that for graph \  $RG_{m}(n)$ \ alpha-letters for the  corner  vertices (A,B,C and D) have to
be \  $\alpha_{1,1} = a$, \  $\alpha_{m,1} = c$, \  $\alpha_{1,n}
= d$ \ and \ $\alpha_{m,n} = f$, respectively, which is not valid  for the other graphs.

Now, with each  2-factor of the observed graph $G$ we associate the
\emph{ code   matrix } \ $ {\cal C}(G) =  {\cal C}(G_m(n)) = [ \alpha_{i,j} ]_{m\times n }$ \
   where  $\alpha_{i,j}$ is the alpha-letter of the vertex \ $(i,j)$ \
($ 1 \leq i \leq m $, \ $ 1 \leq j \leq n $).

\begin{lem}  \label{lem:1} (Characterisation of 2-factors)
The code matrix $ {\cal C}(G) = [ \alpha_{i,j} ]_{m\times n }$ for any grid graph $G \in  {\cal G}$ satisfies  the following properties:
\begin{enumerate}
\item \textbf{Column conditions:} \ For every fixed \ $j$  \ ($1 \leq j \leq n$),

\begin{enumerate}
 \item \ the ordered pairs \ $ (\alpha_{i,j}, \alpha_{i+1,j})$, \ where \ $1 \leq i
\leq m-1$, \ must be arcs in the digraph \ ${\cal D}_{ud}$.

\item  $ \alpha_{1,j} \in \{ a, d, e \}$ \ and
\ $\alpha_{m,j} \in \{ c, e, f \}$.
             \end{enumerate}

\item \textbf{Adjacency of column condition:} \  For every  fixed  $j$, where  \ $1 \leq j \leq n-1 $,
 the ordered pairs \ $ (\alpha_{i,j}, \alpha_{i,j+1})$, \ where \ $1
\leq i \leq m$, \ must be arcs in the digraph \ ${\cal D}_{lr}$.

\item \textbf{First and Last Column conditions:}
\begin{enumerate}
 \item
If $G= RG_{m}(n)$, then
the alpha-word of the first  column consists of the letters from the
set \ $\{ a, b, c \}$
%, \ with $\alpha_{1,1}= a$  \ and  \ $\alpha_{m,1}= c$
 and of  the last column of the letters
  from the set \ $\{  b, d, f \}$.
   \item
If $G= TkC_{m}(n)$,  then
 the ordered pairs \ $ (\alpha_{i,n}, \alpha_{i,1})$, \ where \ $1
\leq i \leq m$, \ must be arcs in the digraph \ ${\cal D}_{lr}$.
\item
If $G= MS_{m}(n)$, then
 the ordered pairs \ $ (\overline{\alpha}_{i,n}, \alpha_{m -i+1,1})$, \ where \ $1
\leq i \leq m $, \ must be arcs in the digraph \ ${\cal D}_{lr}$.
 \end{enumerate}
\end{enumerate}
The converse is also true.
Every matrix ${\cal C}(G) = [\alpha_{i,j}]_{m \times n }$ with entries from $\{a,b,c,d,e,f\}$
that satisfies conditions 1--3 determines a  unique 2-factor  on the grid graph $G$.
\end{lem}
\begin{proof}

The properties above can easily be  proved directly, by  checking  all the possible  edge arrangements (their  compatibility) for adjacent vertices of $G$ and constraints imposed by the structure of the considered graph $G$.
Vice versa,  alpha-letters and the possibility of their contact, expressed by digraphs  ${\cal D}_{ud}$ and  ${\cal D}_{lr}$, make sure that the subgraph of the graph $G$ determined by the matrix (marked with bold lines)  is a spanning 2-regular graph,
i.e. a  union of cycles (a 2-factor).
\end{proof}

\section{\bf \ \ Enumeration of 2-factors }
\label{sec:EF}

 Now, we can create  for each integer
  \ $ m $ \ $ (m \in N) $  a  digraph  \ $ {\cal D}_{m}
\stackrel{\rm def}{=} (V({\cal D}_{m}), E({\cal D}_{m}))$  \ (common  for all graphs from  ${\cal G}$), in the following way:
  \ the set of vertices  \ $ V({\cal D}_{m}) $  \ consists of all possible words $\alpha_{1,j}\alpha_{2,j} \ldots \alpha_{m,j}$ over  alphabet $ \{ a,b,c,d,e,f \}$ (called \emph{ alpha-words}) which fulfill   Condition 1 ({\em Column conditions})  from Lemma~\ref{lem:1};  an arc joins the vertex
   \ $ v = \alpha_{1,j}\alpha_{2,j} \ldots \alpha_{m,j}$ \ and the vertex  \ $ u= \alpha_{1,j+1}\alpha_{2,j+1} \ldots \alpha_{m,j+1}$, i.e. $(v,u) \in E({\cal D}_{m})$, \
or \ $ v \rightarrow u $ \ \  iff  for the  vertex  \  $ v$ \
and  $ u $   Condition 2 ({\em Adjacency of column condition}) from Lemma~\ref{lem:1} is satisfied
(vertex $v$  might be the previous column for vertex $u$ in  matrix  ${\cal C}(G)$).

The subsets of $V({\cal D}_{m})$ which consist of all the  possible first (last) columns in the matrix ${\cal C}(RG_{m}(n))$
(Condition 3a) is denoted by ${\cal F}_m$  (${\cal L}_m$).
From Condition 1b, the first column of ${\cal C}(RG_{m}(n))=[\alpha_{i,j}]_{m \times n}$ is
an alpha-word  from  \ $\{ a, b, c \}^m$, \ with
$\alpha_{1,1}= a$  \ and  \ $\alpha_{m,1}= c$. Similarly, the  last  column of ${\cal C}(RG_{m}(n))$  is an alpha-word from  \ $\{  b, d, f \}^m$, \ with $\alpha_{1,n}= d$  \  and  \ $\alpha_{m,n}= f$.
\begin{lem} \label{lem:2}
 The cardinality of sets $ {\cal F}_{m}  $ and ${\cal L}_{m} $ $(m \in N)$  are equal to the $(m-1)$th  member of the  Fibonacci sequence $F_{m-1}$.
\end{lem}
\begin{proof}

The cardinal number of the set ${\cal F}_{m}$ is equal to the number of all oriented walks  of length $m-1$
starting with vertex $a$ and ending with  vertex $c$
in the subdigraph of ${\cal D}_{ud}$ induced by the set $\{ a,b,c \}$.
Note that  characteristic polynomial of the adjacency matrix of this digraph is $\lambda (\lambda^2 - \lambda -1)$, i.e.
the required sequence $ \mid {\cal F}_{m} \mid  $ $(m \in N)$ obeys the same recurrence relation as the Fibonacci sequence.
Since   $ \mid {\cal F}_{1} \mid  =0= F_0$ and  $ \mid {\cal F}_{2} \mid  =1= F_1$ ($ \mid {\cal F}_{3} \mid  =1= F_2$; $ {\cal F}_{2} = \{ ac\}$, $ {\cal F}_{3} = \{ abc\}$), by induction, we conclude our assertion. The proof for the set  ${\cal L}_{m} $  can be carried out  analogously.
\end{proof}

Let ${\cal P}_m = [p_{ij}]$ be the square binary  matrix of order $ \mid V({\cal D}_{m}) \mid $ for which $p_{i,j} =1$ iff the $i$-th and $j$-th vertex of the digraph
 ${\cal D}_{m}$ can be obtained  from each other by horizontal conversion; otherwise $p_{i,j} =0$. Note that  matrix ${\cal P}_m$ is a symmetric  one.
The following theorem is essential for the enumeration of 2-factors.

\begin{lem}  \label{lem:3}
If  $f_m^G(n)$  ($ m \geq 2$) denotes the number of 2-factors of $G \in  {\cal G}$, then
$$f_m^G(n)= \left \{  \; \;
 \begin{array}{ll}
 \ds \sum_{v_i \in {\cal F}_m} \sum_{v_j \in {\cal L}_m }  a_{i,j}^{(n-1)} =
   \ds \sum_{v_i \in {\cal F}_m } a_{i,i}^{(n)} =a_{1,2}^{(n+1)} ,  &   \mbox{ if  } \; \; G = RG, \\ \\
 \ds  tr({\cal T}_m^n) = \sum_{ v_i \in   V({\cal D}_m)} a_{i,i}^{(n)}, &                    \mbox{ if  } \; \; G = TkC, \\ \\
   \ds  tr({\cal P}_m  \cdot   {\cal T}_m^n) ,     &             \mbox{ if  } \; \; G =  MS,      \end{array}  \right. $$
   where  ${\cal T}_m = [a_{ij}]$ is  the  adjacency matrix of the digraph  \ $ {\cal D}_{m} $ (transfer matrix) and
      vertices $v_1, v_2 \in V({\cal D}_{m})$ are words   $db^{m-2}f$ and $ab^{m-2}c$, respectively.
\end{lem}
\begin{proof}

Using Lemma~\ref{lem:1} the problem of enumeration of 2-factors   in \ $ G  \in {\cal G}$  (with $m \times n$ vertices)  reduces to the enumeration of all the possible code matrices ${\cal C}(G_m(n))$, i.e. all oriented walks of the length \ $ n-1 $ in the digraph  \ $ {\cal D}_{m} $ for which the initial and  final vertices satisfy The First and  Last conditions from Lemma~\ref{lem:1}.
Note that the set  ${\cal F}_m$ consists of all the direct successors of all the  vertices from ${\cal L}_m$ (including  vertex $db^{m-2}f $), and
the set  ${\cal L}_m$ consists of all the  direct  predecessors of all the  vertices from ${\cal F}_m$ (including  vertex  $ab^{m-2}c $.
 So, in this way, for $ G = RG_m(n)$ our problem of enumeration of all oriented walks of  length
\ $ n -1$ \ in the digraph \ $ {\cal D}_{m} $ \ with  initial vertices in  ${\cal F}_m$ and the  final vertices in  ${\cal L}_m$  can be reduced to enumeration of
 all the  closed  oriented walks of  length \ $ n $ \ in the digraph   \ $ {\cal D}_{m} $ \ with   initial vertex  from the  set ${\cal F}_m$
 or the number of all oriented walks of  length \ $ n +1$ \ \ with  initial vertex $db^{m-2}f $ and the  final vertex $ab^{m-2}c $.
Recall that  the $(i,j)$-entry  $ a_{i,j}^{(k)}$ of $k$-th degree of the adjacency matrix ${\cal T}_m $ of the digraph  \ $ {\cal D}_{m} $
represents the number of all oriented walks of  length \ $ k $ ($k \in N$) \ which start with  vertex $v_i$ and end with vertex $v_j$ ($v_i, v_j \in V({\cal D}_{m})$)\cite{KK}.
This implies the   assertion of Lemma~\ref{lem:3}.
\end{proof}
 \begin{rem} \label{Primedba2}
For vertical conversion  the following applies: \\
a) $ v  \rightarrow v' $,  where $v \in V({\cal D}_m)$, \\
b)  If $ v \rightarrow w$, then $ w' \rightarrow v' $, where  $v,w \in V({\cal D}_m)$.
  \end{rem}
Consequently, for any vertex $v \in V({\cal D}_m) \; \;  (v')' = v$.
For example,  vertical conversion pairs for ${\cal D}_3$ are: $abf$ and $dbc$, $edf$ and $eac$, $dfe$ and $ace$, $abc$ and $dbf$, $afe$ and $dce$, $edc$ and $eaf$;
while the vertical conversion of $eee \in V({\cal D}_3)$ is  this vertex itself (see  Figure~\ref{SirokiP3Cvorovi}).

\begin{lem}  \label{lem:4}
  The number of vertices in  ${\cal D}_m$  is $\mid V({\cal D}_m) \mid = \ds \frac{1}{2}(3^{m} + (-1)^{m}).$
\end{lem}
\begin{proof}

The number of all the  words  of length $m$ over
alphabet $\{a,b,c,d,e,f \}$ that satisfy Column conditions of  Lemma~\ref{lem:1}) is  equal to the number of all the
oriented walks  of length $m-1$ in ${\cal D}_{ud}$
starting with a vertex from  $ \{ a,d,e \}$ and finishing with a vertex from  $\{ c,e,f \}$.
Characteristic polynomial of the adjacency matrix of this digraph is $P(\lambda) =\lambda^4(1+\lambda)(\lambda -3)$.
Thus, the required sequence, labeled with $c_m$ $(m \in N)$,  obeys the  recurrence relation
 $c_m = 2c_{m-1}+3c_{m-2}$, with initial conditions   $c_1=1$ (the word $e$), $c_2 = 5$ (the words $ac,bd,ee,dc$ and $af$).
  Using standard procedure for solving recurrence relations we obtain  $\ds c_m    = \frac{(-1)^{m}+3^{m}}{2}$,
 which implies  $ \mid V({\cal D}_{m}) \mid \leq \ds \frac{(-1)^{m}+3^{m}}{2}$.
 In order to prove strict equality note that for any  word $w$ (with  properties described above), from Remark~\ref{Primedba2}a) we have  $w \rightarrow w'$ and $w' \rightarrow w$.
This implies  that the word $w$ appears as a column  in the    code matrix ${\cal C}(TkC_m(2))$ of some  2-factor of the graph $TkC_m(2)$.
\end{proof}

\begin{figure}[htb]
\begin{center}
\includegraphics[width=4in]{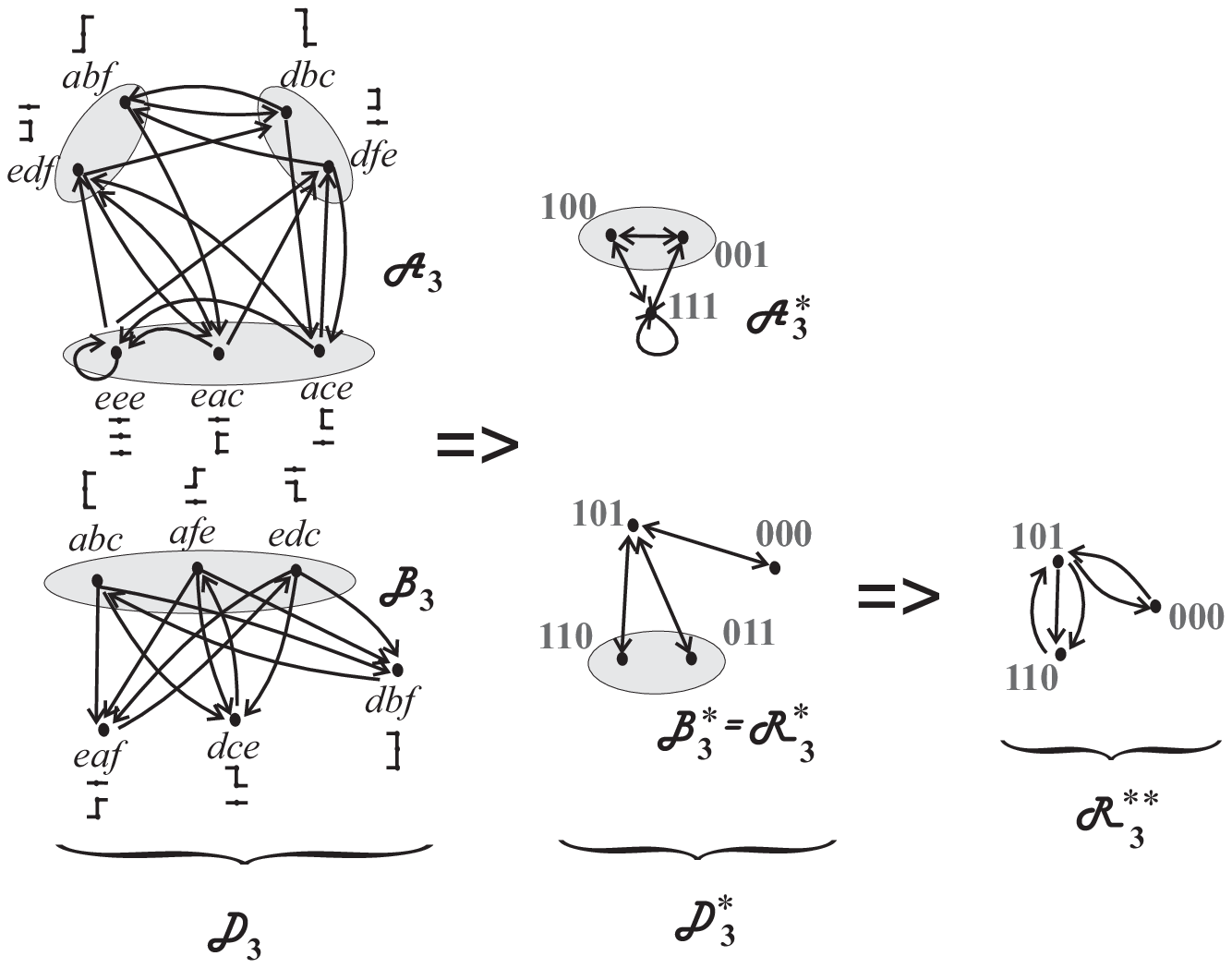}
\\ \ \vspace*{-18pt}
\end{center}
\caption{Digraphs ${\cal D}_3$, ${\cal D}^*_3$ and ${\cal R}^{**}_3$.}
\label{SirokiP3Cvorovi}
\end{figure}

For each vertex $\alpha_1 \alpha_2 \ldots \alpha_m \in V({\cal D}_{m})$ we introduce a binary word called an \emph{ outlet (inlet)
word } depending on whether the situations shown in Figure~\ref{cvornikod1} matched  to  it's letters has an edge ``on the right''  (``on the left'') or not.

\begin{defn}
The  \emph{ outlet
word } (\emph{inlet
word}) of the word  $\alpha \equiv \alpha_1 \alpha_2 \ldots\alpha_m \in V({\cal D}_{m})$ is the binary
  word \ $o(\alpha) \equiv o_1o_2 \ldots o_m$  ($i(\alpha) \equiv i_1i_2 \ldots i_m$) \ where \

  $$ \ds o_j
\stackrel{\rm def}{=} \left \{
\begin{array}{cc}{}
0, & \; \; if \; \; \alpha_j \in \{ b, d, f \} \\
1, & \; \; if \; \; \alpha_j \in \{ a, c, e \}
\end{array}
 \right.  \; \;  and  \; \;  i_j
\stackrel{\rm def}{=} \left \{
\begin{array}{cc}{}
0, & \; \; if \; \; \alpha_j \in \{a,b,c \} \\
1, & \; \; if \; \; \alpha_j \in \{ d,e,f \}
\end{array}
 \right. ,  \; \; \; 1 \leq j \leq m
$$
\end{defn}

\begin{lem}  \label{lem:5}
Digraph ${\cal D}_{m}$ for $m \geq 2$ is disconnected. Each of its components is a strongly connected digraph.
\end{lem}
\begin{proof}

Let $v \rightarrow w$, where $v,w \in {\cal D}_{m}$   and $w \equiv \alpha_1 \alpha_2 \ldots  \alpha_m$. Notice  that $o(v)\equiv i(w)$ and the  total number of   bold  edges in  all situations  shown in Figure~\ref{cvornikod1}  for  all $\alpha_1, \alpha_2, \ldots , \alpha_m$ is equal to $2m$.
Considering that each  vertical edge is calculated twice in that sum, we conclude that the numbers of  $1$s in $o(v)$ and
 in  $o(w)$ have the same  parity.  Consequently, any two vertices  of ${\cal D}_{m}$ having  different parity of the   number of 1's in their  outlet words (for example, the words $ab^{m-2}c$ and $ab^{m-2}f$)   can not belong to the same component of  ${\cal D}_{m}$, i.e. ${\cal D}_{m}$ is disconnected.

To prove the second assertion of the theorem, observe  an arbitrary oriented  walk  $w_0 w_1 \ldots  w_{k-1} w_k$ (of length  $k \in N$).
Then using    Remark~\ref{Primedba2}  conclude that  there exists an oriented  walk  $w_k w'_k  w'_{k-1} \ldots$   $  w'_1 w'_0 w_0 $ which starts and finishes with $w_k$ and $w_0$, respectively.
Consequently, all components are strongly connected digraphs.
 \end{proof}

Let ${\cal D}_m = {\cal A}_m \cup {\cal B}_m$, $m \geq 2$
where ${\cal A}_m$ is the component of the  digraph ${\cal D}_{m}$  which contains the vertex  $e^{m}$ ($m \geq 2$), i.e. with the outlet word  $11 \ldots 1$  - the unique vertex  with loop.
The fact that  the word $db^{m-2}f$ ($m \geq 2$) (with the outlet word $00 \ldots 0$) belongs to   ${\cal A}_m$ depends on the parity of  $m$.
Now, let  ${\cal R}_m$ denote the component of   ${\cal D}_{m}$ which contains the vertex $db^{m-2}f$ ($m \geq 2$). Its vertices are possible columns of   matrices $ {\cal C}(RG_m(n)) $ for $n \in N$.

\begin{lem}  \label{lem:6}
 ${\cal A}_m \equiv {\cal R}_{m}$ iff   $m $ is even.
 \end{lem}
\begin{proof}
For $m = 2k$ ($k \in N$)  $e^{m}     \rightarrow (df)^k  \rightarrow   ab^{m-2}c \rightarrow   db^{m-2}f $
which implies ${\cal A}_m \equiv {\cal R}_{m}$.   For $m$ odd, the  outlet words of   vertices  in $V({\cal A}_m)$ have odd, while
outlet words of   vertices  in $V({\cal R}_m)$ have even numbers of $1$s.
 Consequently,  ${\cal A}_m \not\equiv {\cal R}_{m}$.
 \end{proof}

In order to reduce the  transfer matrix ${\cal T}_m$, note that two vertices  from \ ${\cal D}_m$ \ with
the same outlet word have the same set of direct successors.
For  all vertices from  \ $ V({\cal D}_{m})$ \  \ having the same corresponding outlet word
 we replace  them with just one vertex,   labeled by their common  outlet word.
 Arcs from  $ E({\cal D}_{m})$ starting from these contracted vertices and finishing
with the   same vertex are substituted with only one arc.
 This way, we obtain the   digraph \ $ {\cal D}^*_{m}$, with  adjacency (transfer) matrix  ${\cal T}^*_m$.
 For example, in  Figure~\ref{SirokiP3Cvorovi} vertices $abc, afe, edc \in  V({\cal D}_{3})$ with the common outlet word $101$ are contracted at  $101 \in  V({\cal D}^*_{3})$. (Here the double-headed arrow represents two different edges, one for each direction.)

Note that two different  vertices from  \ $ V({\cal D}_{m})$ with the same outlet word can not have the same direct predecessor.
 This implies that there exist no multiple edges in
${\cal D}^*_m$, i.e. entries of ${\cal T}^*_m$ are from the set $\{ 0,1 \}$.

 \begin{thm}  \label{thm:TkC3}
Adjacency matrix  ${\cal T}^*_m$ of the digraph ${\cal D}^*_m$ is a symmetric binary matrix, i.e. ${\cal T}^*_m = ({\cal T}^*_m)^{T} $.
\end{thm}
\begin{proof}

Let us prove that      $v \rightarrow w$  iff  $ w \rightarrow v$, for any two vertices $v, w \in V({\cal D}^*_m)$.

Suppose that $v$ is a direct predecessor of  $ w$, i.e. $v \rightarrow w $. This  implies that  there exist vertices  $x, y \in V({\cal D}_m)$  such that
$x \rightarrow y$, where $ o(x) = v$ and $ o(y) = w$.  Using reflection symmetry (with  the vertical axis) we have $ v = o(x) =i(y) =  o(y')$ and $ y \rightarrow y' $. Consequently, $w \rightarrow v$.  \end{proof}

Since the vertices from $V({\cal D}_m)$ which are contracted belong to the same component,
Lemma~\ref{lem:5} implies

\begin{thm}  \label{thm:SC}
Digraph ${\cal D}^*_{m}$ for $m \geq 2$ is disconnected. Each of its components is a strongly connected digraph.
\end{thm}

The component which contains the vertex $0^m \equiv 00 \ldots 0$ is denoted by ${\cal R}^*_m$. The component containing the unique loop ($1^m\rightarrow 1^m$) is labeled by ${\cal A}^*_m$ and the union of the remaining components by ${\cal B}^*_m$.
Digraphs ${\cal D}^{*}_m$ for  $m=4$ and  $m=5$ are shown in  Figure~\ref{SirokiP4Cvorovi} and \ref{SirokiP5Cvorovi}, respectively.

\begin{figure}[htb]
\begin{center}
\includegraphics[width=3.5in]{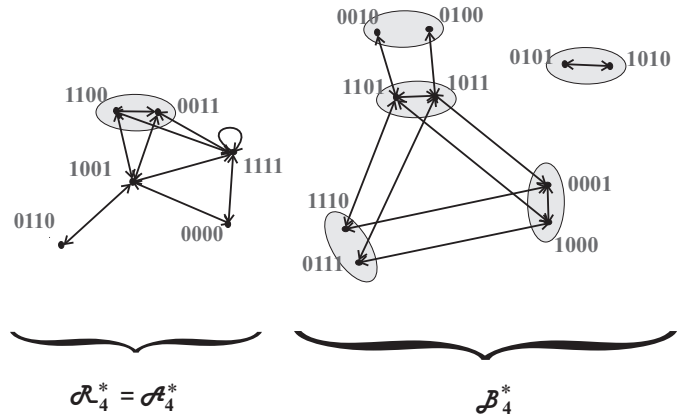}
\\ \ \vspace*{-18pt}
\end{center}
\caption{Digraph ${\cal D}^*_4$ for  $P_4 \times P_n$ }
\label{SirokiP4Cvorovi}
\end{figure}

\begin{figure}[htb]
\begin{center}
\includegraphics[width=4.3in]{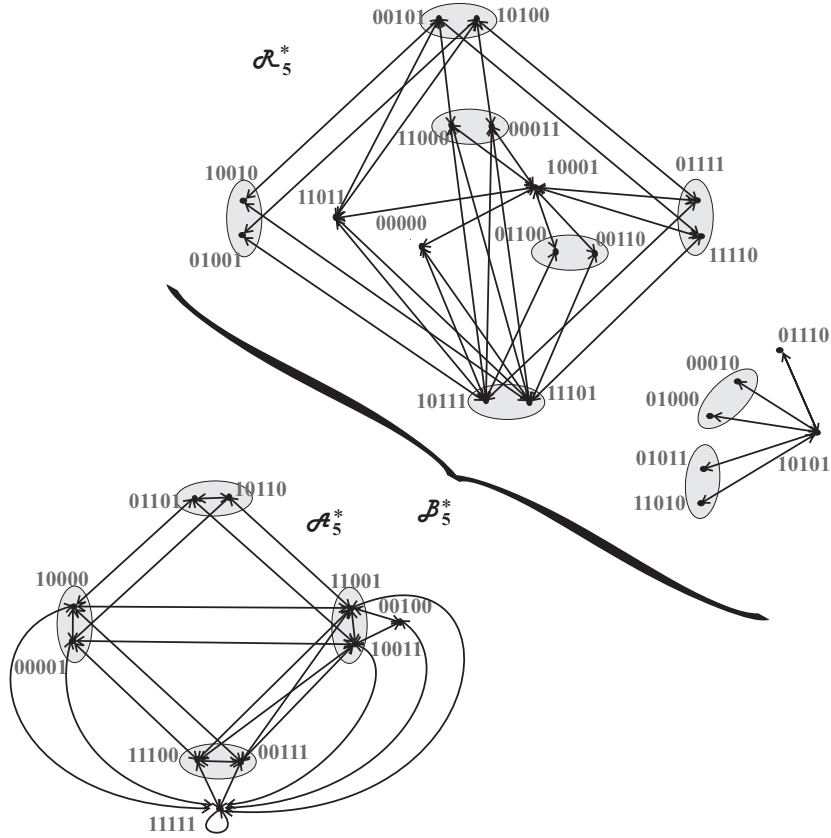}
\\ \ \vspace*{-18pt}
\end{center}
\caption{Digraph ${\cal D}_5^*$ for  $P_5 \times P_n$ }
\label{SirokiP5Cvorovi}
\end{figure}

\begin{thm}  \label{thm:TkC5}
For the number of vertices in digraph ${\cal D}^*_m$ we have
 \begin{eqnarray*}  \mid V({\cal D}^*_m) \mid = \left\{  \begin{array}{rl}
2^{m}, & \mbox{ if $m$ is even}  \\  2^{m}-1, & \mbox{ if $m$ is odd}  \end{array} \right. \end{eqnarray*}   \end{thm}
\begin{proof}
At first, note that any  word $x \in V({\cal D}_m)$ whose outlet word has the  prefix $(01)^k$ ($1 \leq k \leq \ds \lfloor  \frac{m}{2} \rfloor)$ must have the  prefix $(dc)^k$.
 Similarly, any word $x \in V({\cal D}_m)$  whose outlet word has the suffix $(10)^s$ ($1 \leq s \leq \ds \lfloor  \frac{m}{2} \rfloor)$ must have the suffix $(af)^s$.
 If this suffix  is of length  $m-1$, then the first letter must be $e$, i.e. the word $010101...010 \not\in V({\cal D}^*_m)$.

To prove that there exists
 at least one  vertex  $x \in V({\cal D}_m)$ with $o(x) = v$,
 where   $v$ is a  binary word of length  $m$ different from  $010101...010$ (in case  $m$-odd),
 we start from  the  word   $v$ and demonstrate the construction of $x$.

Let $k$ and $s$ be the maximum  non-negative integers for which  $(01)^k$ is the prefix and $(10)^s$ is the  suffix  of $v$.
 Then $v  \equiv (01)^k w (10)^s$, where the subword $w$ is different from  $0$ and it has neither prefix $01$ nor  suffix $10$.
We replace the prefix $(01)^k$  with $(dc)^k$, and the suffix  $(10)^s$ with $(af)^s$.
If the  word $w \equiv 0^{m-2(k+s)}$ ($m-2(k+s) > 1$), then the word $db^{m-2(k+s+1)}f $ can be inserted instead of $w$.
Another case is that the word $w$ has at least one letter $1$.
Let $t$ and $p$ be the maximum  non-negative integers for which $0^t$ is a  prefix and $0^p$ is a suffix  of $w$.
 Clearly, $t \neq 1$, $p \neq 1$ and $w  \equiv 0^t \alpha   0^p$ where the subword $\alpha$ has the first and the last letter $1$.
  In  case $t \geq 2$ ($p \geq 2$), then the subword  $0^t$ ($0^p$) is  substituted with $db^{t-2}f$ ($db^{p-2}f$).
   Consider now  all the maximal zero-subwords of $\alpha$.
Every one  of them,  of length  $q \geq 2$, we  replace by word $db^{m-2}f$.
  If $q=1$, then the  subword $01$, obtained by extending  this letter $0$ with the first letter $1$ below, is replaced by word  $dc$.
 Finally, replacing the  remaining letters $1$  with letters $e$ we obtain the vertex $x \in {\cal D}_m$.
 \end{proof}

\begin{thm}  \label{thm:TkC6}
  The number of edges in  ${\cal D}^{*}_m$  is $\mid E({\cal D}^*_m) \mid = \ds \frac{1}{2}(3^{m} + (-1)^{m}).$
\end{thm}
\begin{proof}
Consider all the vertices from  $V({\cal D}_m)$ with the same outlet word  $v \in V({\cal D}^*_m)$. Observe that they have different inlet words, which
represent all possible direct predecessors for $v$  in   ${\cal D}^{*}_m$.
In this way, the  bijection between the set $E({\cal D}^*_m)$ and $V({\cal D}_m) $ is established.
  Lemma~\ref{lem:4} implies our assertion.
 \end{proof}

For the  binary word $v \equiv b_1b_2 \ldots b_{m-1}b_{m} \in \{ 0,1\}^m$, we introduce the label $\overline{v} \stackrel{\rm def}{=}  b_mb_{m-1} \ldots b_{2}b_{1}$.
 \begin{rem} \label{Primedba3}  \hspace*{1cm } \\
 For  arbitrary  $ x, y \in V({\cal D}_{m})$,   $ \  o(\overline{x}) = \overline {o(x) }. $
 Consequently,    if $o(x) = o(y),$  then $ o(\overline{x}) = o(\overline{y}).$
\end{rem}

Let ${\cal P}_m^* = [p_{ij}]$ be the square binary  matrix of order $ \mid V({\cal D}^*_{m}) \mid $ whose entry $p_{i,j} =1$ iff the $i$-th and $j$-th vertices of the digraph
 ${\cal D}^*_{m}$ satisfy $v_i = \overline{v_j}$ (and $v_j = \overline{v_i}$); otherwise $p_{i,j} =0$. Note that the matrix ${\cal P}_m^*$ is a symmetric  one.
We can improve the  process of enumeration of 2-factors using the  new  transfer-matrix  ${\cal T}_m^*$.
\begin{thm}  \label{thm:f2}
$$f_m^G(n)=\left \{ \; \;
 \begin{array}{ll}  a_{1,1}^{(n)}, & \mbox{ if  }\; \; G = RG , \\ \\
 tr(({\cal T}_m^*)^n) = \ds \sum_{v_i \in   V({\cal D}^*) } a_{i,i}^{(n)},  &                   \mbox{ if  }\; \; G= TkC,\\ \\
    tr({\cal P}_m^{*}  \cdot   ({\cal T}_m^*)^n) ,           &          \mbox{ if  }\; \; G = MS, \end{array}  \right. $$
 where  ${\cal T}^*_m = [a_{ij}]$ is   the adjacency matrix of the digraph  \ $ {\cal D}^{*}_{m} $ and       $v_1 \equiv 00 \ldots 0$.
\end{thm}
\begin{proof}
Let  ${\cal W}^{y}_{x}(n)$ denote the number of oriented walks of length  $n$ in the observed digraph (${\cal D}_{m}$ or ${\cal D}^*_{m}$)  which start with  vertex $x$  and finish with vertex $y$.
The $(i,j)$-entry  $ a_{i,j}^{(n)}$ of $n$-th degree of ${\cal T}_m^* $
represents the number of all the  oriented walks of the length \ $ n $ \  in   ${\cal D}^*_{m}$ which start with $v_i \in V({\cal D}^{*}_{m})$
and finish with $v_j \in V({\cal D}^{*}_{m})$, i.e. ${\cal W}^{v_j}_{v_i}(n)$.
Note that for arbitrary $ x_1, x_2, y \in V({\cal D}_{m})$,

\begin{equation}\label{o1}
\mbox{ if } o(x_1) = o(x_2), \mbox{ then } {\cal W}^{y}_{x_1}(n) = {\cal W}^{y}_{x_2}(n). \end{equation}

\noindent
Consequently, the number ${\cal W}^{v_j}_{v_i}(n)$ is equal to the number of all the oriented walks of the length \ $ n $ \  in   ${\cal D}_{m}$ which start with
a  vertex $x \in V({\cal D}_{m})$
 which is assigned to the vertex $v_i \in V({\cal D}^{*}_{m})$ ($o(x) =v_i$)  and finish with vertices $y \in V({\cal D}_{m})$ which are assigned to $v_j \in V({\cal D}^{*}_{m})$ ($o(y) =v_j$), i.e.

\begin{equation}\label{o2}
 {\cal W}^{v_j}_{v_i}(n) = \ds \sum_{\ds  \begin{array}{c}  y \in  V({\cal D}_{m}) \\ o(y) =v_j \end{array}} {\cal W}^{y}_{x} (n), \mbox{ where } x \in V({\cal D}_{m}) \mbox{ and } o(x) =v_i. \end{equation}

  Now, from  Lemma~\ref{lem:3},  having in mind that the set of direct predecessors  of $ab^{m-2}c$ is   $ {\cal L}_m $, and using \eqref{o2} we have
\\  \\ $\ds  f_m^{RG}(n)
=  {\cal W}^{ab^{m-2}c}_{db^{m-2}f} (n+1) = \ds \sum_{\ds y \in   {\cal L}_m }  {\cal W}^{y}_{db^{m-2}f} (n) = {\cal W}^{0^m}_{0^m} (n)=a_{1,1}^{(n)}.$

By using  \eqref{o2},\eqref{o1} and Remark~\ref{Primedba3} we have

\begin{equation}\label{o3}
 {\cal W}^{v_i}_{v_i}(n) = \ds \sum_{\ds  \begin{array}{c}  x \in  V({\cal D}_{m}) \\ o(x) = o(x_1) =  v_i \end{array}} {\cal W}^{x}_{x_1} (n) = \ds \sum_{\ds  \begin{array}{c}  x \in  V({\cal D}_{m}) \\ o(x) = v_i \end{array}} {\cal W}^{x}_{x} (n).
\end{equation}

 and

\begin{equation}\label{o4}
 {\cal W}^{\overline{v}_i}_{v_i}(n) = \ds \sum_{\ds  \begin{array}{c}  x \in  V({\cal D}_{m}) \\ o(x) = o(x_1) =  v_i \end{array}} {\cal W}^{\overline{x}}_{x_1} (n) = \ds \sum_{\ds  \begin{array}{c}  x \in  V({\cal D}_{m}) \\ o(x) = v_i \end{array}} {\cal W}^{\overline{x}}_{x} (n).
\end{equation}

 Applying  Lemma~\ref{lem:3} again and \eqref{o3}  to  $TkC_m(n)$ we obtain

$$ f_m^{TkC}(n) = \ds \sum_{\ds x \in   V({\cal D}_m) } {\cal W}^{x}_{x} (n) =    \sum_{\ds v_i \in   V({\cal D}_m^{*}) } \sum_{\ds \begin{array}{c}  x \in  V({\cal D}_{m}) \\ o(x) =v_i \end{array}}   {\cal W}^{x}_{x} (n) =  \ds \sum_{\ds v_i \in   V({\cal D}_m^*) } a_{i,i}^{(n)} =  tr(({\cal T}_m^*)^n). $$

For  $MS_m(n)$ Lemma~\ref{lem:3} and \eqref{o4} yield that

$$ f_m^{MS}(n) = \ds \sum_{\ds x \in   V({\cal D}_m) } {\cal W}^{\overline{x}}_{x} (n) =    \sum_{ \ds v_i \in   V({\cal D}_m^{*}) } \sum_{ \begin{array}{c}  x \in  V({\cal D}_{m}) \\ o(x) =v_i \end{array}}  {\cal W}^{\overline{x}}_{x} (n) =   \sum_{ \ds v_i \in   V({\cal D}_m^{*}) }  {\cal W}^{\overline{v}_i}_{v_i} (n),  \mbox{ \ i.e.}$$
  $ f_m^{MS}(n) =  tr({\cal P}_m^*  \cdot   ({\cal T}_m^*)^n).$
\end{proof}

\begin{thm}  \label{thm:TkC2c}
The subdigraph of   ${\cal D}^*_{m}$ induced by the set of vertices having odd numbers of  0's is a bipartite digraph.
\end{thm}
\begin{proof}
Having in mind Theorem~\ref{thm:TkC3}, it is sufficient to prove that this  subdigraph does not contain any odd-length oriented cycle.
Assume the  opposite: that there exist an odd integer $n$ and  an oriented cycle of length $n$ in it. This implies the existence of a 2-factor in $TkC_m(n)$.
\\ \\ \underline{Case I: $m$-even}
Every   vertex of the considered oriented cycle  observed as a binary word has an  odd number of  1's. It implies that corresponding 2-factor in $TkC_m(n)$ contains
odd number of non-contractible cycles. Using Remark~\ref{Primedba1} we conclude that the number of all the edges in the  observed 2-factor has the same   parity as   $n$, i.e. is odd.
On the other side, this number must be $m \cdot n$, i.e.  even. Contradiction.
\\  \\ \underline{Case II: $ m$-odd}
Now, the corresponding 2-factor in $TkC_m(n)$ contains an  even number of non-contractible cycles (related binary words have  even number of  1's).
Applying Remark~\ref{Primedba1} again,  we conclude that the number of edges of the considered 2-factor is even, while  $m \cdot n$ is odd  which leads to a contradiction.
 \end{proof}

\begin{cor}  \label{cor:2}
The subdigraph of   ${\cal D}_{m}$ induced by the set of vertices whose outlet words have an  odd number of 0's is a bipartite digraph.
\end{cor}
\begin{proof}
It is sufficient to prove that two vertices connected with an  arc do not have the same outlet word.
 Suppose the opposite, that there exist two vertices  $v$ and $w$ with  $v \rightarrow w$ and $o(v) = o (w)$.
 Then,  $o(w) = o(v) = i(w)$ holds. Note that the only  vertex  in ${\cal D}_m$ with the same inlet and outlet word is $e^m$, which does not belong to the  considered subdigraph. Contradiction.
 \end{proof}

  \begin{cor}  \label{cor:2c}
For odd $m$, both ${\cal R}_m$ and ${\cal R}^*_m$ are bipartite digraphs.
\end{cor}
When $m$ is even, vertex $0^m$ is both a direct successor and a direct predecessor for $1^m$ in ${\cal D}^*_{m}$ (for example see  Figure~\ref{SirokiP4Cvorovi}).
The following theorem is a consequence of  Lemma~\ref{lem:6}.

\begin{thm}  \label{thm:parno}
 ${\cal A}^*_m \equiv {\cal R}^*_{m}$ iff   $m $ is even.
\end{thm}

\begin{thm}  \label{thm:nova}
For even  $m$, all palindromes from $V({\cal D}_m^*)$ belong to the  component ${\cal A}^*_m$ (i.e. ${\cal R}^*_m$).
\end{thm}
\begin{proof}
We prove this theorem,   by strong induction, for all $k \in N$ where $m=2k$. The base case is obviously correct ($ 00 \leftrightarrow 11$).
Assume that the statement holds for all palindromes $w\overline{w}$ of length less then $2k$.
Consider a palindrome $v\overline{v} \in  V({\cal D}_m^*)$, where $v\overline{v} \neq 0^m$.

\underline{\bf Case 1:} \ $v\overline{v}  = 1w\overline{w}1$. \\
Let $x\overline{x}$ be one of the alpha words from  $V({\cal D}_{m-2})$ whose outlet word is $w\overline{w}$, i.e.
$o(x\overline{x}) = w\overline{w}$. Then there exist such palindromes $w_j\overline{w}_j \in V({\cal D}^*_{m-2})$ and alpha words $x_j\overline{x}_j \in V({\cal D}_{m-2})$ with $o(x_j\overline{x_j}) = w_j\overline{w_j}$,   $1 \leq j \leq t$ $(t \in N)$
for which there exist walks
 $ w\overline{w} \rightarrow  w_1\overline{w_1} \rightarrow w_2\overline{w_2} \rightarrow  \ldots  \rightarrow w_t \overline{w_t}$  and
$ x\overline{x} \rightarrow  x_1\overline{x_1} \rightarrow x_2\overline{x_2} \rightarrow  \ldots  \rightarrow x_t \overline{x_t}$ in ${\cal D}^*_{m-2}$ and  ${\cal D}_{m-2}$, respectively, and  $w_t\overline{w}_t = 0^{m-2}$ (inductive hypothesis).
Now, the walk $ ex\overline{x}e \rightarrow  ex_1\overline{x_1}e \rightarrow ex_2\overline{x_2}e \rightarrow  \ldots  \rightarrow ex_t \overline{x_t}e \rightarrow db^{2(k-1)}f$ in ${\cal D}_{m}$ ($x_t \overline{x_t} \in {\cal L}_{m-2}$) justifies the   existence of the  walk $ v\overline{v} = 1w\overline{w}1 \rightarrow  1w_1\overline{w_1}1 \rightarrow 1w_2\overline{w_2}1 \rightarrow  \ldots  \rightarrow 1w_t \overline{w_t}1 \rightarrow 0^m$ in ${\cal R}^*_m$ (see Figure~\ref{palin}a).

\begin{figure}[htb]
\begin{center}
\includegraphics[width=5.5in]{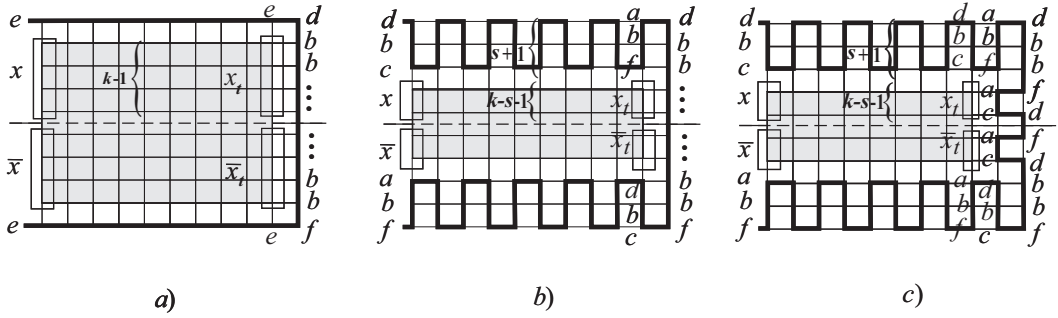}
\\ \ \vspace*{-18pt}
\end{center}
\caption{Constructions walks to vertex $0^{2k}$ in digraph ${\cal D}_{2k} ^*$.}
\label{palin}
\end{figure}

\underline{\bf Case 2:} \ $v\overline{v}  = 0^s1w\overline{w}10^s$ where  $s \geq 1$ (including the possibility that $w$ is an empty word). \\
If $w$ is not an empty word, then there exist such  walks
 $ w\overline{w} \rightarrow  w_1\overline{w_1} \rightarrow w_2\overline{w_2} \rightarrow  \ldots  \rightarrow w_t \overline{w_t}$  and
$ x\overline{x} \rightarrow  x_1\overline{x_1} \rightarrow x_2\overline{x_2} \rightarrow  \ldots  \rightarrow x_t \overline{x_t}$ (both of length $t \in N$) in ${\cal D}^*_{2(k-s-1)}$ and  ${\cal D}_{2(k-s-1)}$, respectively,  where $o(x\overline{x}) = w\overline{w}$,  $o(x_j\overline{x_j}) = w_j\overline{w_j}$ for all   $1 \leq j \leq t$ and
$w_t\overline{w}_t = 0^{2(k-s-1)}$ (inductive hypothesis).

\underline{\bf Case 2.1:} $t$ is odd. \\
Required walk (of length $t+1$) in ${\cal D}_m$ is $ db^{s-1}cx\overline{x}ab^{s-1}f \rightarrow  ab^{s-1}fx_1\overline{x_1} db^{s-1}c \rightarrow  db^{s-1}cx_2\overline{x_2}ab^{s-1}f \rightarrow  ab^{s-1}fx_3\overline{x_3} db^{s-1}c \rightarrow  db^{s-1}cx_4\overline{x_4}ab^{s-1}f \rightarrow \ldots  \rightarrow ab^{s-1}fx_t \overline{x_t}db^{s-1}c    \rightarrow db^{m-2}f$. The fact $db^{m-2}f \in {\cal L}_{m}$ implies the statement (see Figure~\ref{palin}b).

\underline{\bf Case 2.2:} $t$ is even.  \\
Required walk (of length $t+2$) in ${\cal D}_m$ is $ db^{s-1}cx\overline{x}ab^{s-1}f \rightarrow  ab^{s-1}fx_1\overline{x_1} db^{s-1}c \rightarrow  db^{s-1}cx_2\overline{x_2}ab^{s-1}f \rightarrow  ab^{s-1}fx_3\overline{x_3} db^{s-1}c \rightarrow  db^{s-1}cx_4\overline{x_4}ab^{s-1}f \rightarrow \ldots  \rightarrow db^{s-1}cx_t \overline{x_t}ab^{s-1}f \rightarrow ab^{s-1}f (ac)^{k-s-1} db^{s-1}c \rightarrow  db^{s}f(df)^{k-s-2}db^{s}f$.
Since $db^{s}(fd)^{k-s-1}b^{s}f \in {\cal L}_{m}$,  the statement holds (see Figure~\ref{palin}c).

\underline{\bf Case 2.3:} $v\overline{v}  = 0^{k-1}110^{k-1}$. \\
The walk (of length $2$) $db^{k-2}cab^{k-2}f \rightarrow ab^{k-2}f  db^{k-2}c \rightarrow  db^{m-2}f$ in  ${\cal R}_m$  implies
the existence of   the  walk  $0^{k-1}110^{k-1} \rightarrow 1 0^{m-2}1 \rightarrow 0^{m}$  in $ {\cal R}^*_m$.
Consequently, $0^{k-1}110^{k-1} \in  {\cal R}^*_m$.
\end{proof}

Further reduction  of the transfer  matrices is possible just  in case $G=RG_m(n)$ using the following
\begin{thm}  \label{thm:simmetry}
If   $v \in V({\cal R}^*_{m})$,  then  $ \overline{v}  \in V({\cal R}^*_{m})$.
\end{thm}
\begin{proof}
If  $o(x) = v$ ($x \in V({\cal R}_{m})$),  then   there exists an integer $n \geq 1$ for which ${\cal W}^{db^{m-2}f}_x (n) \neq 0$.  By  using the  property of reflection symmetry ($\overline{db^{m-2}f} = db^{m-2}f$),
we obtain that ${\cal W}^{db^{m-2}f}_{\overline{x}}  (n) = {\cal W}^{db^{m-2}f}_x (n) \neq 0$, which implies $\overline{x} \in V({\cal R}_{m})$.
Since $\overline{v} = o(\overline{x})$, we conclude that  $ \overline{v}  \in V({\cal R}^*_{m})$.
 \end{proof}

Now, we can contract the vertices $v$ and $\overline{v}$   into one vertex for all $v, \overline{v} \in V({\cal R}^*)$ resulting in a  new digraph ${\cal R}_m^{**}$.
For example, for  $m=5$ digraph ${\cal R}^*_{5}$ is bipartite  and has $6$ (unordered) pairs of different vertices $\{v, \overline{v}\}$
 which are rounded in Figure~\ref{SirokiP5Cvorovi}.
During contraction of vertices  $v$ and $\overline{v}$ we retain arcs starting from just one of these two vertices and delete the ones starting from another vertex.
Multiple (double) arcs appear  when $v$ and $\overline{v}$ have a  common direct predecessor (see Figure~\ref{SirokiP3Cvorovi}).
The symmetry of the rectangular grid $RG_m(n)$ leads us to the conclusion that ${\cal W}^{0^m}_{\overline{v}}  (n) = {\cal W}^{0^m}_v(n)$ for any vertex $v \in V({\cal R}^*)$.
Consequently, the number ${\cal W}^{0^m}_{0^m}(n)$ remains the same  in both digraph ${\cal R}^*$ and ${\cal R}^{**}$, i.e.

\begin{thm}  \label{thm:contraction} The number
$f_m^{RG}(n)$ is equal to entry $a_{1,1}^{(n)}$ of the $n$-th degree of the adjacency matrix for  ${\cal R}^{**}$ where   $v_1 \equiv 0^m$.
\end{thm}

This way we  obtain the  recurrence relations for $f^{RG}_m(n)$ of lower order then the ones we got using the digraph ${\cal R}^*$(see Tabular 2).
Note that for odd $m$, ${\cal R}^{*}$ is a bipartite digraph and  vertices $v$ and $\overline{v}$ are at an  even distance (of the same colour), so   ${\cal R}^{**}$ has no loops. On the contrary, when $m$ is even  new loops can appear. For example, for $m=4$ (see Figure~\ref{SirokiP4Cvorovi}) one more loop  appears when the vertices $1100$ and $0011$ are contracted.

Computing the  generating functions
 \ $\ds   {\cal F}^G_{m}(x) \stackrel{\rm def}{=} \sum_{n\geq 1}^{\infty}
f^G_{m}(n)x^{n} $ is a matter of
routine \cite{BC08,CDS,S86}.
We wrote computer programs for the computation of the  adjacency
matrices of  these   digraphs ${\cal D}_m$,  \ ${\cal D}^*_m$, ${\cal R}_m$,  \ ${\cal R}^*_m$, \  ${\cal R}^{**}_m$ and initial members of required sequences
$f^G_m(n)$.

\section{\bf \ \ Computational results }
\label{sec:CR}

\subsection{\bf \ \  Cardinality of  sets of vertices  for components of   ${\cal D}^*_m$ }
\label{subsec:CC1} $\;  $ \\
Some  properties of the  digraphs ${\cal D}_m$, ${\cal D}^*_m$ and ${\cal R}^*_m$, and their  related  sequences $f^G_m(n)$,  spotted upon analyzing the computational data for $m \leq 12$ (in case of $RG_m(n)$ for $m \leq 17 $)
  have been  discussed and  proved in the previous section for arbitrary  $m \in N$.
Beside numerical results we further present  a few more properties   concerning the cardinality of  sets of vertices  for components of   ${\cal D}^*_m$.
Due to limited  space these properties are here formulated as conjectures  and will be the subject of a separate paper.

\begin{conj}  \label{conj:1}
For each  $m\geq 2 $,  digraph  ${\cal D}^{*}_m$ has exactly $\ds \left\lfloor    \ds \frac{m}{2} \right\rfloor + 1$ components, i.e.
 ${\cal D}^{*}_m = {\cal A}^*_m    \cup {\cal B}^*_m, $ where ${\cal B}^*_m $ consists of exactly  $ \left\lfloor    \ds \frac{m}{2} \right\rfloor $ components
 ${\cal B}^{*(1)}_{m},  {\cal B}^{*(2)}_{m}, \ldots , {\cal B}^{*(\lfloor     m/2 \rfloor )}_{m}$ \  ($ \mid V({\cal B}^{*(1)}_{m}) \mid  \geq
  \mid V({\cal B}^{*(2)}_{m}) \mid  \geq  \ldots $ $\geq \mid V( {\cal B}^{*(\lfloor    m/2 \rfloor )}_{m}) \mid $).
All the components  ${\cal B}^{*(k)}_{m}$ ($ 1 \leq k \leq \ds \left\lfloor    \ds \frac{m}{2} \right\rfloor $) are bipartite digraphs and \\
$$ \mid V({\cal B}^{*(k)}_{m}) \mid \; \; = \; \;
 \left \{
 \begin{array}{rl}
 \ds  {m + 1 \choose   (m+1)/2 -k},  &                   \mbox{ if  $m$ is  odd}, \\ \\
 2 \cdot \ds  {m \choose   m/2 -k} , & \mbox{ if $m$  is     even}, \end{array}  \right. $$

  and
$$\mid V({\cal A}^{*}_{m}) \mid \; \; = \; \;
 \left \{
  \begin{array}{rl}   \ds {m  \choose (m-1)/2},  & \mbox{ if $m$ is  odd},     \\
    \ds {m \choose m/2} , & \mbox{ if  $m$  is     even.} \end{array}  \right. $$
      The vertices $v$ and $\overline{v}$ belong to  the same component. If $v \in {\cal B}^{*(s)}_{m}, \; \;  1 \leq s \leq \ds \left\lfloor    \ds \frac{m}{2} \right\rfloor $, then  $\overline{v}$ is  placed in the same  class  iff $m$ is odd.
      \end{conj}

% jednak je koeficijentu uz $x^{m+1}$ polinoma $(1+x+x^2)^{m+1}$ (in OEIS A082758 -  Sum of the squares of the trinomial coefficients (A027907) \cite{OE, VI}.)

\begin{table}[htb]
\begin{center}
%\small{}
\begin{tabular}{||c||r|r|r|r|r|r|r|r|r|r|r||}
\hline \hline  m   &  2 & 3 & 4 & 5& 6& 7 & 8 & 9 & 10 & 11 & 12    \\
\hline \hline
 $ \mid V({\cal D}_{m}) \mid $  &   5& 13
& 41 &121 &365 &1093 & 3281 & 9841 & 29525 & 88573 &$ 265721$
\\ \hline
 $ \mid V({\cal D}^*_{m}) \mid $  & 4 & 7 & 16 & 31 & 64 &127 & 256 & 511  & 1024 &2047 &4096
\\ \hline  \hline
 $ \mid V({\cal A}^*_{m}) \mid $ & 2 & 3 &  6& 10& 20
 & 35 & 70 & 126 &252  & 462 & 924
\\  \hline
 $ \mid V({\cal B}^{*(1)}_{m}) \mid $ &2  & 4 &  8&15 &30
 & 56 &  112 &210  & 420 & 792  & 1584
 \\  \hline
 $ \mid V({\cal B}^{*(2)}_{m}) \mid $ &-  & - & 2 & 6
 & 12 & 28 & 56 &120  &240  &495  &990
 \\  \hline
 $ \mid V({\cal B}^{*(3)}_{m}) \mid $ &-  &-  &-  &-
 &2  &8  &16  &45  &90  &220  &440
 \\  \hline
 $ \mid V({\cal B}^{*(4)}_{m}) \mid $ &- &-  &-  &-
 &-  &-  & 2 & 10 & 20 & 66 &132
 \\  \hline
 $ \mid V({\cal B}^{*(5)}_{m}) \mid $ &-  & - & - &
- & - & - & - & - & 2 & 12 & 24
 \\  \hline
 $ \mid V({\cal B}^{*(6)}_{m}) \mid $ &- & - & - &
- & - & - & - & - & - & - &2
\\  \hline \hline
 order  & 4 & 5 & 13 &19 &
49 & 69 & 178 &  249& 649 &-&  -
\\
 \hline
\hline
\end{tabular}
\caption{The numbers of  vertices of  \ ${\cal D}_m$, ${\cal D}^{*}_m$, components of ${\cal D}^{*}_m$  \ and
the order  of the  recurrence relations  (the same) for both  thick cylinder graphs \ $TkC_{m}(n)$ and  Moebius strips $MS_m(n)$.}
\label{tab1}
\end{center}
\end{table}
%Za generisanje digrafa  ${\cal D}^{**RG}_{15}$ bilo je potrebno oko 3 sata (i prvih 200 clanova niza za jos 5 minuta) dok za generisanje digrafa
%${\cal D}^{**RG}_{16}$ je bilo potrebno  oko 30 sati  (a za prvih 100 clanova niza jos oko 30 minuta)

\begin{conj}  \label{conj:2}
 For $m-odd$, the number of vertices in    ${\cal R}^{*}_{m}$ is equal  to the  binomial coefficients  (in OEIS A001791):
 \begin{eqnarray} \label{e0} \mid V({\cal R}^{*}_{m}) \mid =  {m+1 \choose (m-1)/2 }  \end{eqnarray}    For  $m-even$, that number is equal to the Central binomial coefficients (in OEIS A000984):
 \begin{eqnarray} \label{e1} \mid V({\cal R}^{*}_{m}) \mid = {m \choose m/2 }  \end{eqnarray}  while   the number of vertices in    digraph ${\cal R}^{**}_{m}$ is equal  to  \mbox{  A005317 in OEIS}, i.e.
 \begin{eqnarray} \label{e2} \mid  V({\cal R}^{**}_{m}) \mid = 2^{(m-2)/2}+ \frac{1}{2}{m \choose m/2}. \end{eqnarray}   \end{conj}

For $m$-odd, equation \eqref{e0} is a trivial consequence of
Conjecture~\ref{conj:1}  with the assumption that   $0^m \in V( {\cal B}^{*(1)}_{m})$.
Equation \eqref{e1}  is a consequence of     Theorem~\ref{thm:parno} and Conjecture~\ref{conj:1}.
Using  Theorem~\ref{thm:nova}  the equality \eqref{e2} is easy to prove from \eqref{e1}.
Also, we noticed that the  members of the sequence  $\mid V({\cal R}_{m} \mid $
for  even $m \leq 10$ coincide with the  initial members of the sequence A082758 in OEIS \cite{ON}.

\begin{table}[htb]
\begin{center}
%\small{}
\begin{tabular}{||c||r|r|r|r|r|r|r|r|r|r|r|r||}
\hline \hline  m  &  2 & 3 & 4 & 5& 6& 7 & 8 & 9 & 10 & 11 & 12 & 13 \\
\hline \hline
 $ \mid {\cal F}_{m} \mid  = \mid {\cal L}_{m} \mid $   & 1 & 2 & 3
& 5 & 8 & 13 & 21 & 34 & 55 & 89 & 144 & 233
\\ \hline
 $ \mid V({\cal R}_{m}) \mid $  &  3& 6& 19
& 60 &141 & 532 & 1107 & 4608 & 8953  & $\ll$&$\ll$  &
$\ll$
\\ \hline
 $ \mid V({\cal R}^{*}_{m}) \mid $  & 2 & 4 & 6&
15 & 20 & 56 &70 & 210 & 252  & 792 & 924 & 3003
\\ \hline
 $ \mid V({\cal R}^{**}_{m}) \mid $  & 2 & 3 & 5 &
9 & 14 & 31 & 43 & 110 & 142  & 406 & 494 & 1519
\\  \hline
 order  & 2 &1 & 5 &
3 & 13 & 9 & 35 & 25 & 96 &- & - & -
\\
 \hline
\hline
\end{tabular}

\begin{tabular}{||c||r|r|r|r||}
\hline \hline  m  & 14 & 15 & 16  & 17  \\
\hline \hline
 $ \mid {\cal F}_{m} \mid  = \mid {\cal L}_{m} \mid$   &377  & 610  & 987
& 1597
\\ \hline
 $ \mid V({\cal R}^{*}_{m}) \mid $  & 3433 & $\ll$& $\ll$&$\ll$
\\ \hline
 $ \mid V({\cal R}^{**}_{m}) \mid $  & 1781 & 5756 & 6564 &21943
 \\
 \hline
\hline
\end{tabular}
\end{center}
\caption{The numbers of  vertices of \ ${\cal
R}_m$, \ ${\cal R}^{*}_m$ \ and  ${\cal R}^{**}_m$  \ and
the order of the recurrence relations for \ $RG_m(n)$.}
\label{tab2}
\end{table}

\subsection{\bf \ \  Asymptotic behaviour  of the  numbers of 2-factors $f^{G}_m(n)$  }
\label{subsec:AB}$\;  $ \\
The properties referring to asymptotic behaviour  of  numbers of 2-factors $f^{RG}_m(n)$ and $f^{TkC}_m(n)$ (when $ n \rightarrow \infty $)  were observed to be similar to the ones that  appeared   while studying  Hamiltonian cycles.

Since the  adjacency matrix ${\cal T}_m^*$ of ${\cal D}_m^*$ ($m \geq 2$) is symmetric (Theorem~\ref{thm:TkC3}), i.e. Hermitian, the spectrum of  ${\cal D}_m^*$ contains only real numbers. Each of the  components of ${\cal D}_m^*$  is a  strongly connected digraph (Theorem~\ref{thm:SC}) and, therefore,
 has an irreducible adjacency matrix \cite{CDS} (which is a block in the diagonal block matrix  ${\cal T}_m^*$).
 From Perron-Frobenius theorems \cite{HJ}
 the maximum modulus  eigenvalues for these nonnegative and irreducible matrices are algebraically simple eigenvalues.
 If the set of all maximum modulus  eigenvalues for a nonnegative and irreducible matrix has exactly $k \geq 2$ distinct elements, they are precisely the $k$th roots of $1$ times the maximum   eigenvalue $\theta$ \cite{HJ}.  Since all eigenvalues of our considered matrices are real numbers,
  for bipartite digraphs there exist exactly  two simple  eigenvalues of maximal modulus ($\theta$ and $-\theta$), i.e. $k$ must be two.

 Let $ \theta_m$  (we also use the labels $ \theta^{TkC}_m$  and $ \theta^{MS}_m$ to emphasize  corresponding grid graph)
  and $ \theta^{RG}_m$ be the maximum eigenvalues of the adjacency matrices of ${\cal D}_m^*$ and ${\cal R}_m^*$, respectively.
Computational results for $m \leq 12$ show that the maximum  eigenvalue $\theta_m$ of ${\cal T}^*_m$ is simple and unique maximum modulus  eigenvalue.
Additionally, $ \theta_m$ is attached  to the component ${\cal A}^*_m$ for all $m  \geq 2$, i.e.

\begin{conj} \label{conj:3}
 The maximum  eigenvalue of  ${\cal A}^*_m$ is  the unique maximum modula eigenvalue of ${\cal D}_m^*$.
\end{conj}

According to the foregoing we have
 \bc $ \ds f^{TkC}_{m}(n) \sim    a^{TkC}_m \theta_m^n $,  where $a^{TkC}_m = 1$ .  \ec
(Note that the property for the   coefficients of  maximum eigenvalue  being equal to $1$ appeared by Hamiltonian cycles when  $m$ was odd \cite{BKDP1} and \cite{BKDjDP}.)
For instance, \\
$\ds f^{TkC}_{9}(99) =$ \parbox[t]{8cm}{
{\bf 175073846218165}2771338808207772701955030703442028258017318088093361136\
0786760679564966706639273723674798766385930557092858331879012953635968\
195685205, }   \\
$\ds f^{TkC}_{9}(100) =$ \parbox[t]{8cm}{
{\bf 5503}488851650192832857551518533018608271730034860817348840930779798339\
9866850567422183774919747362024619387408919222429539996042852109168447\
1910843826, } \\ \
$\ds f^{TkC}_{10}(99) =$ \parbox[t]{8cm}{
{\bf 547294669589}5734348165268778293176272799246831355358749477747452490650\
1214615708064019534391418227552525368357696963283863359292333457421226\
269199598481596902807547077 \ \ and  }  \\
$\ds f^{TkC}_{10}(100) =$ \parbox[t]{8cm}{
{\bf 2645}316310319933683496095009841718024759437947764204153951092048286901\
7777259031349223534872931966971355650359749930614841881983326875548074\
53750119675251682976586688605 } \\ while   \\
$\ds \theta_{9}^{99} = $ \parbox[t]{8cm}{{\bf \underline{1.7507384621816}5}923653 $ \ldots \cdot 10^{148}$, }
\\
$\ds  \theta_{9}^{100} =  $ \parbox[t]{8cm}{ \underline{ {\bf 5.503}}355253 $ \ldots  \cdot 10^{149}$,}
\\
$\ds \theta_{10}^{99} = $ \parbox[t]{8cm}{{\bf \underline{5.47}294669589}622318 $ \ldots \cdot 10^{166}$, \ \ and  }
  \\
$\ds  \theta_{10}^{100} =  $ \parbox[t]{8cm}{ \underline{ {\bf 2.645}121801666}474501 $\ldots  \cdot 10^{168}$.}  \\
The coefficient  $a^{TkC}_m$ is equal to $ 1$ because of  Theorem~\ref{thm:f2} and the fact that the trace of $n$-th degree of matrix ${\cal T}^*_m$ is equal to the sum of $n$-th degrees of all of its  eigenvalues (the value $f^{TkC}_{m}(n)$ has the unique representation as a linear combination of  $n$ standard  solutions of the   recurrence relation for ${\cal T}^*_m$ corresponding to all its  eigenvalues).

Recall, that for Hamiltonian cycles in $TkC_m(n)$, contractible HCs are more numerous than non-contractible ones iff $m$ is even \cite{BKDjDP}. The similar assertion can be formulated for 2-factors dividing them into 2-factors with even and  odd number of nc-cycles. Note that in case $m$ is even, the digraph ${\cal A}_m^*$ determines 2-factors in $TkC_m(n)$ with an even number of nc-cycles (however  not all of them). For odd $m$  ${\cal A}_m^*$ determines 2-factors with odd numbers of  nc-cycles and this kind of 2-factors are then dominant assuming  Conjecture~\ref{conj:3}.
More precisely,
\bc  $ \ds f^{TkC}_{m}(n) \sim \left\{ \begin{array}{ll}   f^{TkC}_{1,m}(n) \; \; , & \; \; \mbox{ for } m  \mbox{ odd }\\ \\
                                                   f^{TkC}_{0,m}(n)  \; \; , & \; \; \mbox{ for } m \mbox{ even }   \end{array}  \right. \; \; \;
                                                   (n \rightarrow + \infty). $ \ec
Assuming that  all components  ${\cal B}^{*(k)}_{m}$ are bipartite (Conjecture~\ref{conj:1})    we have  that, for $n$ odd, the  only  2-factors obtained from
component ${\cal A}_m^*$     are counted   in  $f^{TkC}_{m}(n)$. (Therefore, more digits  coincide in  $f^{TkC}_{9}(99)$ and $\theta_{9}^{99} $, or  $f^{TkC}_{10}(99)$  and $\theta_{10}^{99}$, than in   $f^{TkC}_{9}(100)$ and $\theta_{9}^{100} $, or  $f^{TkC}_{10}(100)$  and $\theta_{10}^{100}$.)

\begin{table}[htb]
\begin{center}
\begin{tabular}{||c||r|r||} \hline\hline
$m$ & \multicolumn{1}{c| }{$\theta^{TkC}_m = \theta^{MS}_m$}
    & \multicolumn{1}{c||}{$a^{TkC}_m = a^{MS}_m$}
\\ \hline\hline
2 & 1.6180339887498948482045868344
  & 1 \\ \hline
3 & 2.4142135623730950488016887242
  & 1 \\ \hline
4 & 3.6941816601239106665999753656
  & 1 \\ \hline
5 & 5.6532020378824433814716902315
  & 1 \\ \hline
6 &                   8.6709538972300632454385724873
  & 1\\ \hline
7 &                  13.3121782399972542081592050166
  & 1 \\ \hline
8 &                  20.4516932294114966231186908391
  & 1 \\ \hline
9 &                  31.4344796371815965829996668429
  & 1 \\ \hline
10 &                 48.3308526218584373943242746007 & 1 \\ \hline
11 & $\approx_{(100)}$ 74.32697213
  & 1 \\ \hline
12 & $\approx_{(50)}$ 114.326
  & 1 \\ \hline
   \hline
\hline
\end{tabular}
\end{center}
\caption{The approximate values of $\theta_m = \theta^{TkC}_m = \theta^{MS}_m$ and $a^{TkC}_m=a^{MS}_m = 1$ for $1\leq m\leq 12$,
  where $\approx_{(n)}$ means the estimate based on the first $n$  entries of the sequence.}
\label{tabTkCMS}
\end{table}

\begin{table}[htb]
\begin{center}
\begin{tabular}{||c||r|r||} \hline\hline
$m$ & \multicolumn{1}{c| }{$\theta^{RG}_m$}
    & \multicolumn{1}{c||}{$a^{RG}_m$}
\\ \hline\hline
2 & $(1+\sqrt{5})/2$ & $\sqrt{5}/5 $ \\ \hline
3 & 1.73205080756887729352744634151
  & 0.2886751345948128822545743903 \\ \hline
4 & 3.69418166012391066659997536564
  & 0.3118537771565198570113824680 \\ \hline
5 & 4.62518160134423951692596223359
  & 0.2689660737850244855426998625 \\ \hline
6 & 8.67095389723006324543857248731
  & 0.2520573399762828621654010912\\ \hline
7 & 11.5193830042298614862975296130
  & 0.2420402401081641797612878583 \\ \hline
8 & 20.4516932294114966231186908391
  & 0.2149686611014229925654013297 \\ \hline
9 & 28.0703410924057870863760633239
  & 0.2185598738607493954133759244 \\ \hline
10  & 48.3308526218584373943242746007
   & 0.1885668461094284796839894294\\ \hline
11 & $\approx_{(600)}$ 67.7256340927618460544544369622
   & $\approx_{(600)}$ 0.1987190117694364038206719883  \\ \hline
12 & $\approx_{(600)}$ 114.3265540751374759033150378963
   & $\approx_{(600)}$ 0.1683321933349066394611832136\\ \hline
13 & $\approx_{(200)}$ 162.5256416517095900095387075181
   & $\approx_{(200)}$ 0.1818325481375590304998965322\\ \hline
14 & $\approx_{(200)}$ 270.594404874261731
   & $\approx_{(200)}$ 0.152084575433189642\\ \hline
15 & $\approx_{(200)}$ 388.7591582316368266038304859009
   & $\approx_{(200)}$ 0.1672787181981763741720923489 \\ \hline
16 & $\approx_{(200)}$ 640.690454998007
   & $\approx_{(200)}$ 0.1386133711863155\\ \hline
17 & $\approx_{(100)}$ 927.945466754283
   & $\approx_{(100)}$ 0.154581709489037
   \\ \hline\hline
\end{tabular}
\end{center}
\caption{The approximate values of $\theta^{RG}_m$ and $a^{RG}_m$ for $2\leq m\leq 17$,
  where $\approx_{(n)}$ means the estimate based on the first $n$ entries of the sequence.}
\label{tabRG}
\end{table}

Recall, that   contractible
 Hamiltonian cycles in  $TkC_m(n)$ and $RG_m(n)$ have the same positive dominant eigenvalue when  $m$ is even \cite{BKDjDP}.
In case of 2-factors this property is more obvious if we assume  Conjecture~\ref{conj:3} and apply Theorem~\ref{thm:parno} (see Tabular~\ref{tabTkCMS} and  Tabular~\ref{tabRG}), i.e.
  \bc $ \ds f^{RG}_{m}(n) \sim \left\{ \begin{array}{cl}    a^{RG}_m \theta_m^n \; \; , & \; \; \mbox{ for } m  \mbox{ even }\\ \\
                                                   a^{RG}_m (\theta^{RG}_m)^n  + a^{RG}_m \left( - \theta^{RG}_m \right)^n \; \; , & \; \; \mbox{ for } m \mbox{ odd }   \end{array}  \right.  \; \; \;
                                                   (n \rightarrow + \infty), $ \ec
  where $a^{RG}_m$ are  positive numbers.

The novelty appears with the  Moebius strip $MS_m(n)$. According to the  Conjecture~\ref{conj:3} we have
 \bc $ \ds f^{MS}_{m}(n) \sim    a^{MS}_m \theta_m^n $.  \ec
  Numerical data show that the coefficient of the  maximal eigenvalue is again one, i.e.  $a^{MS}_m = 1$.
\\ \\
  For example, \\ \\
  $\ds f^{MS}_{9}(99) = $ \parbox[t]{8cm}{  \underline{17507384621816}65701723082146927338193581515086747844712341253685406319\
1553795793184785069123799306157361712365337734920832932311990627148209\
033567981,} \\ \\
     $\ds f^{MS}_{9}(100) = $ \parbox[t]{8cm}{  \underline{5503}488851650163776931116763714293076427970032180999497847873716909388\
7001537742285391720378803648896975631451618649151832275727205508702975\
7535019482,} \\ \\
  $\ds f^{MS}_{10}(99) = $ \parbox[t]{8cm}{  \underline{547}3392413551435904097524137222257556522154835755962205364841172527434\
5587572680168475798796527038604871940909874270974422734365157886590851\
507415380950441091624785881, \ \ and  }   \\ \\
    $\ds f^{MS}_{10}(100) = $ \parbox[t]{8cm}{  \underline{2645121801666}648913048490842342065467547225492026995244876289719579663\
8721230465776899400708104186158155469514422006233945602999944685195281\
33438673321151104001586481523.}
\\ \\
Here, if we assume Conjecture~\ref{conj:3} is true, then  in case $m$-even, digraph ${\cal A}_m^*$ determines the  majority of the 2-factors without a short  nc-cycle (though generally not all of them). For $m$-odd,  ${\cal A}_m^*$ determines the 2-factors which contain a short nc-cycle and this kind of 2-factors are then dominant.
More precisely,
\bc  $ \ds f^{MS}_{m}(n) \sim \left\{ \begin{array}{ll}   f^{MS}_{1,m}(n) \; \; , & \; \; \mbox{ if $m$ is  odd }\\
                                                   f^{MS}_{0,m}(n)  \; \; , & \; \; \mbox{ if $m$ is even }   \end{array},
                                                  \mbox{ when } n \rightarrow + \infty \right. .    $ \ec
Assuming Conjecture~\ref{conj:1}    we have  that for both $m$ and $n$  odd the  only  2-factors obtained from
component ${\cal A}_m^*$  (containing a short nc-cycle)   are counted   in    $f^{MS}_{m}(n)$. (Therefore more digits coincide in  numbers  $f^{MS}_{9}(99)$ and $\theta_{9}^{99} $ than in  $f^{MS}_{9}(100)$ and $\theta_{9}^{100}. $ On the other hand,   for both $m$ and $n$ even, the only the  2-factors obtained from
component ${\cal A}_m^*$  (without a short nc-cycle)   are counted   in    $f^{MS}_{m}(n)$. (Therefore, more digits coincide in  $f^{MS}_{10}(100)$ and $\theta_{10}^{100} $ than in   $f^{MS}_{10}(99)$ and $\theta_{10}^{99}. $

\subsection{\bf \ \ Generating functions }
\label{subsec:RG} $\;  $\\
Our results for the  rectangular grid graph $RG_{m}(n) \equiv P_{m} \times P_{n}$  for $m \leq 7$
confirm the data previously obtained in another way (by coding cells) in  1994 \cite{BT94}.
We got the generating functions  ${\cal F}^{RG}_{m}(x)$, ${\cal F}^{TkC}_{m}(x)$ and  ${\cal F}^{MS}_{m}(x)$ for  $ 2\leq m \leq 10$. These  generating functions and the first $30$ members of  the sequences $ f^{RG}_m (n)$ ($2 \leq m \leq 17$), $ f^{TkC}_m (n)$ ($2 \leq m \leq 12$) and $ f^{MS}_m (n)$ ($2 \leq m \leq 12$) are  exposed in Appendix.

By observing the  denominators of   the generating functions for the components ${\cal B}_m^{*(k)}$, where $k \geq 2$, one can notice that
each of them consists of the factors of the  denominator of  the generating functions for  ${\cal B}_m^{*(1)}$ or ${\cal A}_m^{*}$.
\normalsize
% ===============
% Acknowledgement
% ===============
\vskip 0.4 true cm

\begin{center}{\textbf{Acknowledgments}}
\end{center}
The authors would like to express their  gratitude to Bojana Panti\' c
 for her meticulous reading of the first draft of the manuscript and on many useful suggestions and
helpful comments which  improved the clarity of the presentation.\\ \\
This work  was  supported by  the Ministry of Education and
Science of the
  Republic of Serbia (Grants  451-03-9/2021-14/200125, 451-03-68/2020-14/200156).\\ \\
\vskip 0.4 true cm
% ==========
% References
% ==========

%------------------------------------------------------------------------------------%
%-----------------------------------------------------------------------------
%-----------------------------------------------------------------------------

\bigskip
\bigskip

{\footnotesize \pn{\bf Jelena \Dj oki\' c}\; \\ {Faculty of Technical Sciences}, {University
of Novi Sad, P.O.Box 21000,} { Novi Sad, Serbia}\\
{\tt Email: jelenadjokic@uns.ac.rs}

{\footnotesize \pn{\bf Olga Bodro\v{z}a-Panti\'{c}}\; \\ {Department of
Mathematics and Informatics, Faculty of Sciences}, {University
of Novi Sad, P.O.Box 21000,} {Novi Sad, Serbia}\\
{\tt Email: olga.bodroza-pantic@dmi.uns.ac.rs}

{\footnotesize \pn{\bf Ksenija Doroslova\v cki}\; \\ {Faculty of Technical Sciences}, {University
of Novi Sad, P.O.Box 21000,} { Novi Sad, Serbia}\\
{\tt Email: ksenija@uns.ac.rs}
\normalsize
\newpage
% =========
% Section Appendix
% =========
\section{\bf \ \  Appendix}
\label{sec:CV}

\vspace*{1cm}

{\bf \large  Rectangular grid graph $RG_{m}(n) \equiv P_{m} \times P_{n}$ ($ 2\leq m \leq 17$) }
\\ \\
    $ {\cal F}^{RG}_2 (x) =$ \parbox[t]{5.4in}{$ \ds   \frac{x^2}{1-x-x^2} = \sum_{n \geq 1} F_{n-1}x^n $} \\ \\
 \hspace*{1cm}  $=$ \parbox[t]{5.4in}{ $ \ds
  0 x^{1} +
1 x^{2} +
1 x^{3} +
2 x^{4} +
3 x^{5} +
5 x^{6} +
8 x^{7} +
13 x^{8} +
21 x^{9} +
34 x^{10} +
55 x^{11} +
89 x^{12} +
144 x^{13} +
233 x^{14} +
377 x^{15} +
610 x^{16} +
987 x^{17} +
1597 x^{18} +
2584 x^{19} +
4181 x^{20} +
6765 x^{21} +
10946 x^{22} +
17711 x^{23} +
28657 x^{24} +
46368 x^{25} +
75025 x^{26} +
121393 x^{27} +
196418 x^{28} +
317811 x^{29} +
514229 x^{30} +
  \ldots  $}

  \bc $ \mbox{--------------------------------------    } $ \ec \noindent
$ \ds  {\cal F}^{RG}_{3}(x) = $ \parbox[t]{5.4in}{ $ \ds \frac{x^2}{1 - 3x^2} = \sum_{k \geq 0} 3^kx^{2k+2} $} \\ \\
 \hspace*{1.2cm}  $=$ \parbox[t]{5.4in}{ $ \ds x^2 + 3 x^4 + 9 x^6 + 27 x^8 + 81 x^{10} + 243 x^{12} + 729 x^{14} +
 2187 x^{16}
 +6561 x^{18} +19683 x^{20} +59049 x^{22} +177147 x^{24} +531441 x^{26} +1594323 x^{28} +4782969 x^{30} +
 \ldots $} \bc $ \mbox{--------------------------------------    } $ \ec  \noindent
 $ \ds  {\cal F}^{RG}_{4}(x)= $ \parbox[t]{5.4in}{ $ \ds \frac{x^2 (2 - x -2 x^2 +x^3)}{1 - 2 x - 7 x^2 + 2 x^3 + 3 x^4- x^5}  $}  \\ \\
 \hspace*{1.2cm}  $=$ \parbox[t]{5.4in}{ $ \ds
   0 x^{1} +
2 x^{2} +
3 x^{3} +
18 x^{4} +
54 x^{5} +
222 x^{6} +
779 x^{7} +
2953 x^{8} +
10771 x^{9} +
40043 x^{10} +
147462 x^{11} +
545603 x^{12} +
2013994 x^{13} +
7442927 x^{14} +
27490263 x^{15} +
101563680 x^{16} +
375176968 x^{17} +
1386004383 x^{18} +
5120092320 x^{19} +
18914660608 x^{20} +
69873991466 x^{21} +
258127586367 x^{22} +
953569519203 x^{23} +
3522660270539 x^{24} +
13013344688975 x^{25} +
48073663465682 x^{26} +
177592838241869 x^{27} +
656060220073148 x^{28} +
2423605607111629 x^{29} +
8953239432543485 x^{30} + \dots  $}
 \bc $ \mbox{--------------------------------------    } $ \ec  \noindent
${\cal F}^{RG}_{5}(x) = $ \parbox[t]{5.4in}{ $ \ds - \frac{3 x(x - 6 x^3 + 5 x^5)}{-1 + 24 x^2 - 57 x^4 + 26 x^6}$}
\\ \\
 \hspace*{1.2cm}  $=$ \parbox[t]{5.4in}{ $ \ds 3 x^2 + 54 x^4 + 1140 x^6 + 24360 x^8 + 521064 x^{10} + 11146656 x^{12} +
 +238452456 x^{14} +5101047216 x^{16} +109123156248 x^{18} +2334395822496 x^{20} +49938107061384 x^{22} +1068291209653392 x^{24} +
22853211220567416 x^{26} +488882861126970624 x^{28} +10458331198925940456 x^{30} +
 \ldots $}\bc $ \mbox{--------------------------------------    } $ \ec  \noindent
${\cal F}^{RG}_{6}(x) = $ \parbox[t]{5.4in}{ $ \ds - [x^2 (5 - 16 x - 68 x^2 + 169 x^3 + 184 x^4 - 440 x^5 + 41 x^6 +
    159 x^7 - 24 x^8 - 21 x^9 + 2 x^{10} + x^{11})] / $}
 \\  \hspace*{1.4cm} \parbox[t]{5.4in}{ $[-1 + 5 x + 49 x^2 -
  116 x^3 - 363 x^4 + 627 x^5 + 544 x^6 - 1061 x^7 + 133 x^8 +
  264 x^9 - 47 x^{10} - 26 x^{11} + 3 x^{12} + x^{13}]$}\\ \\
 \hspace*{1.2cm}  $=$ \parbox[t]{5.4in}{ $ \ds  5 x^2 + 9 x^3 + 222 x^4 + 1140 x^5 + 13903 x^6 + 99051 x^7 +
 972080 x^8 + 7826275 x^9 + 71053230 x^{10} +
599141127 x^{11} +5285091303 x^{12} +
45349095730 x^{13} +395755191515 x^{14} +
3418116104881 x^{15} +29709767180643 x^{16} +
257232791130155 x^{17} +2232466696767889 x^{18} +
19346930092499853 x^{19} +167813061128260612 x^{20} +
1454798219804865516 x^{21} +12616086588695738786 x^{22} +
109385021015592147639 x^{23} +948517510315349620643 x^{24} +
8224312539935356519632 x^{25} +71313903875397020754957 x^{26} +
618352836534571281094702 x^{27} +5361744695265385528847028 x^{28} +
46491251251770863340309040 x^{29} +403124503785899198402752468 x^{30} +
  \ldots $}
  \bc $ \mbox{--------------------------------------    } $ \ec  \noindent
\noindent
$ {\cal F}^{RG}_{7}(x) \ds = $
\parbox[t]{5.4in}{ $ \ds
[x^2 (8 - 589 x^2 + 9810 x^4 - 59710 x^6 + 148304 x^8 - 170714 x^{10} +
   93298 x^{12} - 22631 x^{14} + 1904 x^{16})] / $}
 \\  \hspace*{1.4cm} \parbox[t]{5.4in}{ $[1 - 171 x^2 + 5496 x^4 -
 56617 x^6 + 240021 x^8 - 457923 x^{10} + 420254 x^{12} - 186912 x^{14} +
 37569 x^{16} - 2584 x^{18}]$}
\\ \\
 \hspace*{1.2cm}  $=$ \parbox[t]{5.4in}{ $ \ds 8 x^2 + 779 x^4 + 99051 x^6 + 13049563 x^8 + 1729423756 x^{10} +
 229435550806 x^{12} +30443972466433 x^{14} +4039769151988768 x^{16} +536061241088972481 x^{18} +
71133264482944200277 x^{20} +9439112402375129121841 x^{22} +1252534193959746441955912 x^{24} +166206508635573867359551206 x^{26} +
22054969579015463381016539631 x^{28} +2926610318841947932231378200008 x^{30} +
  \ldots $}
  \bc $ \mbox{--------------------------------------    } $ \\  $ \mbox{--------------------------------------    } $ \ec  \noindent
 \noindent
$ {\cal F}^{RG}_{8}(x) \ds = $
\parbox[t]{5.4in}{ $ \ds
- [ x^2  (13 - 155 x - 1728 x^2 + 19243 x^3 + 64036 x^4 - 794298 x^5 -
      565784 x^6 + 14250929 x^7 - 7187021 x^8 - 118967700 x^9 +
      129305109 x^{10}  + 492780993 x^{11}  - 689112311 x^{12}  -
      1071061947 x^{13}  + 1713117207 x^{14}  + 1273147977 x^{15}  -
      2210122027 x^{16}  - 886182340 x^{17}  + 1551922901 x^{18}  +
      391857072 x^{19}  - 602359977 x^{20}  - 110992789 x^{21}  +
      129338396 x^{22}  + 17394085 x^{23}  - 16027643 x^{24}  - 1363285 x^{25}  +
      1191501 x^{26}  + 43104 x^{27}  - 52660 x^{28} + 443 x^{29}  + 1275 x^{30}  -
      59 x^{31} - 13 x^{32}  + x^{33})] / $}
   \\  \hspace*{1.4cm} \parbox[t]{5.4in}{ $[ -1 + 14 x + 331 x^2 - 3474 x^3 -
   24357 x^4 + 237534 x^5 + 541266 x^6 - 6604103 x^7 - 1905497 x^8 +
   85855152 x^9 - 60009003 x^{10}  - 545836271 x^{11}  + 672927757 x^{12}  +
   1747850343 x^{13}  - 2763674623 x^{14}  - 2917536240 x^{15}  +
   5513512152 x^{16}  + 2653029943 x^{17}  - 5852097578 x^{18}  -
   1465977019 x^{19}  + 3471750395 x^{20}  + 568784352 x^{21}  -
   1167520145 x^{22}  - 154667330 x^{23} + 221656480 x^{24}  +
   23823457 x^{25}  - 24542626 x^{26}  - 1818710 x^{27}  + 1646233 x^{28}  +
   57030 x^{29}  - 66339 x^{30}  + 348 x^{31}  + 1479 x^{32}  - 61 x^{33}  -
   14 x^{34}  + x^{35}].$}
\bc $ \mbox{--------------------------------------    } $ \ec
$ {\cal F}^{RG}_{8}(x) \ds = $ \parbox[t]{5.4in}{ $ \ds  0 x^{1} +
13 x^{2} +
27 x^{3} +
2953 x^{4} +
24360 x^{5} +
972080 x^{6} +
13049563 x^{7} +
360783593 x^{8} +
6044482889 x^{9} +
142205412782 x^{10} +
2645920282312 x^{11} +
57787769198498 x^{12} +
1130122135817708 x^{13} +
23838761889677477 x^{14} +
477334902804794530 x^{15} +
9905649696435264827 x^{16} +
200572437515846530901 x^{17} +
4130348948437378850158 x^{18} +
84074883624291031055071 x^{19} +
1725061733607816846672084 x^{20} +
35201911945083165877105598 x^{21} +
721041937227213471236222936 x^{22} +
14731026760739434523775920272 x^{23} +
301492247130186410656766864436 x^{24} +
6162966556594442193757310209147 x^{25} +
126086101870795129720839096783333 x^{26} +
2578070083185284447937587182277129 x^{27} +
52734387801729163635906223494385644 x^{28} +
1078388240037660942562424414577181926 x^{29} +  $ \\ $
22056541466571843558470704997624920958 x^{30} +   \ldots $}
  \bc $ \mbox{--------------------------------------    } $  \\  $ \mbox{--------------------------------------    } $ \ec  \noindent
\noindent
$ {\cal F}^{RG}_{9}(x) \ds = $ \parbox[t]{5.4in}{ $ \ds x[-21 x + 14429 x^3 - 2977612 x^5 + 287882596 x^7 - 15396210323 x^9 +
   495074875987 x^{11} - 10065902326113 x^{13} + 134416133040711 x^{15} -
   1217169158720764 x^{17} + 7672025883235936 x^{19} -
   34359015998529872 x^{21} + 111034237538363778 x^{23} -
   261756587386897382 x^{25} + 453246428062005818 x^{27} -
   578310298795433179 x^{29} + 543657156125982437 x^{31} -
   375278277913632996 x^{33} + 188891515846196440 x^{35} -
   68536442609580103 x^{37} + 17611233979787535 x^{39} -
   3119350414858553 x^{41} + 365120087483245 x^{43} -
   26348708004138 x^{45} + 1030937887896 x^{47}-
   16048304096 x^{49}]/ $ \\ $[(-1 + 7 x^2) (1 - 1193 x^2 + 376246 x^4 -
     48953410 x^6 + 3288145988 x^8 - 127374411928 x^{10} +
     3015668747782 x^{12} - 45191010425846 x^{14} +
     441384780778588 x^{16} - 2898283223877346 x^{18} +
     13154666974580501 x^{20} - 42187756055653825 x^{22} +
     97142224830553641 x^{24} - 162322162033938237 x^{26} +
     198042290945862570 x^{28} - 176799585485005402 x^{30} +
     115298993750386955 x^{32} - 54611642383339285 x^{34} +
     18586566465572115 x^{36} - 4467086405032683 x^{38} +
     737944576901349 x^{40} - 80312266104179 x^{42} +
     5368435066393 x^{44} - 193455857453 x^{46} + 2735506380 x^{48})]$}
     \bc $ \mbox{--------------------------------------    } $ \ec
 $ {\cal F}^{RG}_{9}(x) \ds = $ \parbox[t]{5.4in}{ $ \ds
    21 x^{2} +
10771 x^{4} +
7826275 x^{6} +
6044482889 x^{8} +
4738211572702 x^{10} +
3728454567619186 x^{12} +
2936793145852970503 x^{14} +
2313819828221792568231 x^{16} +
1823117803971861533276785 x^{18} +
1436506068810722263846966803 x^{20} +
1131884615499177238765812997622 x^{22} +
891861379870929131575395762708116 x^{24} +
702736791667402399627515571022726473 x^{26} +  $ \\ $
553717257645251892181598471499890805627 x^{28} +  $ \\ $
436298214862451852483213957197628599972855 x^{30} +
     \ldots $}
 \bc $ \mbox{--------------------------------------    }$ \\ $\mbox{--------------------------------------    } $ \ec
\noindent
$ {\cal F}^{RG}_{10}(x) \ds = $ \parbox[t]{5.4in}{ $ \ds
 [x^2 (-34 + 1347 x + 35881 x^2 - 1662089 x^3 - 10309779 x^4 +
      771197801 x^5 - 346702675 x^6 - 181431735340 x^7 +
      698129954414 x^8 + 24314041221008 x^9 - 150052606942535 x^{10} -
      1955172098649311 x^{11} + 16342077034176585 x^{12} +
      95187483126478256 x^{13} - 1082538790666807385 x^{14} -
      2632342959447463549 x^{15} + 47190107143602870694 x^{16} +
      25149352192030341611 x^{17} - 1420522867178712125486 x^{18} +
      942375457477994083262 x^{19} + 30544728579120039592509 x^{20} -
      45811635254374125870547 x^{21} - 481024654106604142658754 x^{22} +
      1034639523180001447651336 x^{23} +
      5650755350310331133281764 x^{24} -
      15379099750826572124032129 x^{25} -
      50138517276885465580264781 x^{26} +
      165014594468953372583716508 x^{27} +
      338134006349071894103973007 x^{28} -
      1333169187054548030456272181 x^{29} -
      1729935216080270865107872599 x^{30} +
      8307485120568915399123891036 x^{31} +
      6606891842108732324100582627 x^{32} -
      40537611346244740085982960471 x^{33} -
      17910894279373108316521092023 x^{34 }+
      156424654912960078568444931679 x^{35} +
      28639513227693834286646693909 x^{36} -
      480217673925082801373266665091 x^{37} +
      6215060674611087974386868437 x^{38} +
      1176546451149220018624683293218 x^{39} -
      196016176729365270408848413341 x^{40} -
      2301806795478656683806927108523 x^{41} +
      671482889680656520724854005411 x^{42} +
      3589449311739965939755136676932 x^{43 }-
      1430065207899407466173244608566 x^{44} -
      4442389907986722561120177476313 x^{45} +
      2198439329420030259008800264867 x^{46} +
      4332308965605070604141219519647 x^{47} -
      2551630188027298209948844954100 x^{48} -
      3292910612416258412767808633745 x^{49} +
      2272730860876774276760437655867 x^{50} +
      1918257264947044701126773815295 x^{51} -
      1560552230291000671555830113730 x^{52} -
      832978851315941289318293630901 x^{53} +
      824676218354533947168453969998 x^{54} +
      255343919601756250583297382745 x^{55} -
      333495878322041349631842287090 x^{56} -
      47558210494863310947937006626 x^{57} +
      102290962816799339382005037827 x^{58} +
      1441759998454840560470630830 x^{59} -
      23511660501016232854983698392 x^{60} +
      2100054136003464406076320001 x^{61} +
      3985059711601160982005723241 x^{62} -
      714180858251836658799770366 x^{63} -
      486814987941308989735168779 x^{64} +
      130997898090477219110002544 x^{65} +
      41265311576931652337785265 x^{66} -
      15813670343559865739961997 x^{67} -
      2229569111235017818843115 x^{68} +
      1328140074764536970942380 x^{69} + 54160990055954074415814 x^{70} -
      79509464898694123907388 x^{71} + 1950961185520045979921 x^{72} +
      3439957867838613593851 x^{73} - 248380845738698994226 x^{74} -
      108461186853587325114 x^{75} + 12130636053499251523 x^{76} +
      2499366907230690769 x^{77} - 376092051435600570 x^{78} -
      41863262779260293 x^{79} + 8189167140633561 x^{80} +
      499418309570138 x^{81} - 129807749028789 x^{82} -
      4028121425512 x^{83} + 1514497739548 x^{84} + 18856167098 x^{85} -
      12936246167 x^{86} - 15444588 x^{87} + 79073791 x^{88} -
      376673 x^{89} - 329233 x^{90} + 2147 x^{91} + 841 x^{92} - 4 x^{93} -
      x^{94})]/   $ }\newpage

      $       $ \\
      \hspace*{1.2cm}  \parbox[t]{5.4in}{
      $ \ds [-1 + 42 x + 2133 x^2 - 88106 x^3 - 1159305 x^4 +
    59614372 x^5 + 152015398 x^6 - 18819762746 x^7 +
    40034190687 x^8 + 3264921562892 x^9 - 15997694438884 x^{10} -
    336798102365195 x^{11} + 2393078617098732 x^{12} +
    21262382124790258 x^{13} - 203991793583280580 x^{14} -
    802344394384627520 x^{15} + 11100949049156620432 x^{16} +
    15124565789332508780 x^{17} - 408720018862600116122 x^{18} +
    70049454004597571197 x^{19} + 10586093658527044258990 x^{20} -
    12631025888273228837204 x^{21} - 198508817597338205942399 x^{22} +
    387402987345930766315590 x^{23} + 2753586280130848546932800 x^{24} -
    7116726860274335286714358 x^{25} -
    28694708927568151334441162 x^{26} +
    91397507315681115211888499 x^{27} +
    226797214695129167282089048 x^{28} -
    870318027304318790444523656 x^{29} -
    1363049305769562405267119112 x^{30} +
    6338665081932290022053776544 x^{31} +
    6178413328764184607346600034 x^{32} -
    35988003082214887026742513610 x^{33} -
    20492020074586065235166736101 x^{34} +
    161266450872180007915993984787 x^{35} +
    45099606966665951271693058428 x^{36} -
    575001009719437474541893835090 x^{37} -
    37571777602636385348482415900 x^{38} +
    1639314263514932955906264058314 x^{39 }-
    161019967972020863550757762131 x^{40} -
    3745699796292066069224199321372 x^{41 }+
    851274013919532238198719327872 x^{42} +
    6859539342150922395619895904651 x^{43} -
    2285177200474388410515202543750 x^{44} -
    10046278725288151331993329222702 x^{45} +
    4263336781464515502088030143628 x^{46} +
    11714187061919722875421984542193 x^{47} -
    5951624800503928235794807036596 x^{48} -
    10796511948612809185000930216914 x^{49} +
    6382964048451425290619045645907 x^{50} +
    7780699596208522812005194118918 x^{51} -
    5309228126699605717428066646335 x^{52} -
    4313080247506333381456597185155 x^{53} +
    3429850650853924016071831349581 x^{54} +
    1790273570607001253851113421413 x^{55} -
    1715098075696198937373310320219 x^{56} -
    528375021398399743450311766854 x^{57} +
    659377687421116918032870695489 x^{58} +
    96604393710766619419597252621 x^{59} -
    193029880068729539119900529947 x^{60} -
    4069176647541912150779380527 x^{61} +
    42490073730573974965749981782 x^{62} -
    3413445482746271605206739212 x^{63} -
    6917976301240196729904525844 x^{64} +
    1165798086293282165887806486 x^{65} +
    814163354061253637176607304 x^{66} -
    208629109950888272910391562 x^{67} -
    66693134076508177810310164 x^{68} +
    24445572577097454676637446 x^{69} +
    3497333010513277273515012 x^{70} - 1989966629028961602789938 x^{71} -
    83751655748872954401344 x^{72} + 115359756476647736425785 x^{73} -
    2722708701437214377921 x^{74} - 4826028852207290118097 x^{75} +
    345815430004445639338 x^{76} + 146771902754630028398 x^{77} -
    16434126042052261132 x^{78} - 3249431151334904798 x^{79} +
    493956208888831088 x^{80} + 51941522002476761 x^{81} -
    10408635523210110 x^{82} - 583754494635043 x^{83} +
    159473348039387 x^{84} + 4295711619291 x^{85} - 1796857404178 x^{86} -
    16067848999 x^{87} + 14814596478 x^{88} - 26776331 x^{89} -
    87402987 x^{90} + 632038 x^{91} + 351431 x^{92} - 2932 x^{93} -
    868 x^{94} + 5 x^{95} + x^{96}]$}
      \bc $ \mbox{--------------------------------------    } $ \ec
$ {\cal F}^{RG}_{10}(x) \ds = $ \parbox[t]{5.4in}{ $ \ds  0 x^{1} +
34 x^{2} +
81 x^{3} +
40043 x^{4} +
521064 x^{5} +
71053230 x^{6} +
1729423756 x^{7} +
142205412782 x^{8} +
4738211572702 x^{9} +
303872744726644 x^{10} +
11986520595161863 x^{11} +
674837188667280276 x^{12} +
29188352280018463610 x^{13} +
1532451456181020466428 x^{14} +
69707888449637626994606 x^{15} +
3524148419679038665338017 x^{16} +
164782546678464075204330111 x^{17} +
8161612120472629734758702827 x^{18} +
387404299957157944331168782389 x^{19} +
18975092783963004675735587294202 x^{20} +
908104916618474893269843972030000 x^{21} +
44209698933437116068409590480429252 x^{22} +
2125264991542391467278261919541938229 x^{23} + $ \\ $
103123528613528669519078412311840825913 x^{24} + $ \\ $
4969501454786829485527284038834319001975 x^{25} + $ \\ $
240699212742310352022631306455020391840255 x^{26} + $ \\ $
11614680633292318140437397088402434381887655 x^{27} +$ \\ $
562008082519876500509235006660432148062016214 x^{28} +$ \\ $
27138757083274637538136034727280002688597205157 x^{29} +$ \\ $
1312480307994378411517307946704246357901598884389 x^{30} +  \ldots $}
  \bc $ \mbox{--------------------------------------    } $ \\
  $ \mbox{--------------------------------------    } $ \ec
${\cal F}^{RG}_{11}(x) = $ \parbox[t]{5.4in}{$\ds
    55 x^{2} +
147462 x^{4} +
599141127 x^{6} +
2645920282312 x^{8} +
11986520595161863 x^{10} +
54755153078468134960 x^{12} +
250811052174680822733959 x^{14} +
1149900478069130190749532847 x^{16} +
5273548680646779160230746586859 x^{18} +
24187345557092636781475699267236187 x^{20} + $ \\ $
110939825540114835884546566285845762280 x^{22} + $ \\ $
508851861608867456819833568314582960129595 x^{24} + $ \\ $
2333978113478971879379146342461349384545507487 x^{26} + $ \\ $
10705394905474347467310075872932119902375191735556 x^{28} + $ \\ $
49103084149419532904485985834926622182640443919990354 x^{30} +  \ldots
           $}  \bc $ \mbox{--------------------------------------    } $ \\
  $ \mbox{--------------------------------------    } $ \ec
${\cal F}^{RG}_{12}(x) = $  \parbox[t]{5.4in}{$\ds
0 x^{1} +
89 x^{2} +
243 x^{3} +
545603 x^{4} +
11146656 x^{5} +
5285091303 x^{6} +
229435550806 x^{7} +
57787769198498 x^{8} +
3728454567619186 x^{9} +
674837188667280276 x^{10} +
54755153078468134960 x^{11} +
8217125138015950451626 x^{12} +
764291947227525464744293 x^{13} +
102778523332781609788923496 x^{14} +
10377666858261924113127188462 x^{15} +
1307545840659068316275398430609 x^{16} +
138717888384983395400743673070647 x^{17} +
16811107880554885949723365263385113 x^{18} +
1837416806334748618002441976965826916 x^{19} +  $ \\ $
217546798094878586614776187637251947154 x^{20} + $ \\ $
24207544299916846326223500527520404672797 x^{21} + $ \\ $
2826348353596869316989646072613818424334629 x^{22} +  $ \\ $
317913170521025428336656341733279252103645915 x^{23} +  $ \\ $
36807752016516373736486909090218756904178467297 x^{24} +  $ \\ $
4167154676845333334464279553525507959718825743928 x^{25} +  $ \\ $
480044645073281997728004688609940910251915235881911 x^{26} +  $ \\ $
54560205876703359925688778263581530252043634665998495 x^{27} +  $ \\ $
6266182773788800302528578730998645394708164827181600165 x^{28} +  $ \\ $
713864557620383330913701983097695540944505535055335848728 x^{29} +  $ \\ $
81837579118030976966882984400592164512413484550125332803113 x^{30}  + \dots $ \\ $
 \ldots $}  \bc $ \mbox{--------------------------------------    } $ \\ $ \mbox{--------------------------------------    } $ \ec
${\cal F}^{RG}_{13}(x) = $ \parbox[t]{5.4in}{$\ds
144 x^{2} +
2013994 x^{4} +
45349095730 x^{6} +
1130122135817708 x^{8} +
29188352280018463610 x^{10} +
764291947227525464744293 x^{12} +
20119942924108379011391597989 x^{14} +
530757331488177559806803222962531 x^{16} +
14012518359036670226842012669455804921 x^{18} + $ \\ $
370060693927736976102598123201451153750400 x^{20} +
 $ \\ $
9774237174242112936074588927469304811518971316 x^{22} +
  $ \\ $
258174575833190416092931230106166876774272922816326 x^{24} +
$ \\ $
6819493526917964144433813509795990979935178761456609095 x^{26} +
 $ \\ $
180133257930816702828649317813228158539407388700258052545414 x^{28} +
 $ \\ $
4758136595785599184827205547148480043809502832166468626937652141 x^{30} +
  \ldots $}  \bc $ \mbox{--------------------------------------    } $ \\ $ \mbox{--------------------------------------    } $ \ec
${\cal F}^{RG}_{14}(x) =$  \parbox[t]{5.8in}{$ \ds
 0 x^{1} +
233 x^{2} +
729 x^{3} +
7442927 x^{4} +
238452456 x^{5} +
395755191515 x^{6} +
30443972466433 x^{7} +
23838761889677477 x^{8} +
2936793145852970503 x^{9} +
1532451456181020466428 x^{10} +
250811052174680822733959 x^{11} +
102778523332781609788923496 x^{12} +
20119942924108379011391597989 x^{13} +
7095967027221343377167292602835 x^{14} +
1558052539448513320447263528275071 x^{15} +
499710516860655913807488713037993676 x^{16} +
118148889686793059954503676188547861439 x^{17} +
35663187828618798448176359283472706878251 x^{18} + $ \\ $
8844869553721036151571830337235039811285498 x^{19} +  $ \\ $
2567902709068972249983061282979663158259005188 x^{20} +  $ \\ $
656843753793091466111271960113665150268181910540 x^{21} +  $ \\ $
185985360340172624437226057862914191982817411306798 x^{22} +  $ \\ $
48531671985430295372643234667885796698180209383007846 x^{23} +  $ \\ $
13522045451583935214311942745211781450584019741823946467 x^{24} +  $ \\ $
3574230153258805534442262789706227987917456760619395321102 x^{25} +  $ \\ $
985574315180475305127558549314931246032163539439255975211553 x^{26} +  $ \\ $
262688013352420151641401780917713156435617350757523668540401656 x^{27} +
71951462951957662733223296222348007193287965199739874099102425149 x^{28} +
19280611320272675682366923144444641296515354441079256732526117681312 x^{29} +
5258300742958152860303362507698568462642123253466258390520471533041763 x^{30} +
 \ldots $}  \bc $ \mbox{--------------------------------------    } $ \\ $ \mbox{--------------------------------------    } $ \ec
 ${\cal F}^{RG}_{15}(x) =$ \parbox[t]{5.8in}{$ \ds
377 x^{2} +
27490263 x^{4} +
3418116104881 x^{6} +
477334902804794530 x^{8} +
69707888449637626994606 x^{10} +
10377666858261924113127188462 x^{12} +
1558052539448513320447263528275071 x^{14} +
234788223520702255614480389250160811898 x^{16} +
35438872103032923948503473063105831888120296 x^{18} + $ \\ $
5352980916509968569866944597338016618770836593777 x^{20} + $ \\ $
808814487707939424103258442436215917419236785722366089 x^{22} + $ \\ $
122225728569396111511730187227610652723460707513653410334890 x^{24} + $ \\ $
18471534283961665470838147288989010252709658184911178787526609504 x^{26} + $ \\ $
2791611790627002375709359890035431269889691472289917390598410862157068 x^{28} + $ \\ $
421902632488607037181505456822308879531291489513527561733780462073430723695 x^{30} +
 \ldots $}  \bc $ \mbox{--------------------------------------    } $ \\ $ \mbox{--------------------------------------    } $\ec
 \small
${\cal F}^{RG}_{16}(x) =$ \parbox[t]{5.4in}{ $\ds
 0 x^{1} +
610 x^{2} +
2187 x^{3} +
101563680 x^{4} +
5101047216 x^{5} +
29709767180643 x^{6} +
4039769151988768 x^{7} +
9905649696435264827 x^{8} +
2313819828221792568231 x^{9} +
3524148419679038665338017 x^{10} +
1149900478069130190749532847 x^{11} +
1307545840659068316275398430609 x^{12} +
530757331488177559806803222962531 x^{13} +
499710516860655913807488713037993676 x^{14} + $ \\ $
234788223520702255614480389250160811898 x^{15} + $ \\ $
195081705501438193439250404333039349462635 x^{16} + $ \\ $
101199388044301804167035198499446336399419451 x^{17} + $ \\ $
77324499738655428250428753054462763663473289098 x^{18} + $ \\ $
42906124439946680432242192936932389250449772385665 x^{19} + $ \\ $
30980523428299019458567823961648073011309062657974817 x^{20} + $ \\ $
17997177997131365762945853787261691797327217525918292125 x^{21} + $ \\ $
12506395048117034637941019258039871289830975170048857051005 x^{22} +  $ \\ $
7495739849844426872239299370840211457131733872858505243081816 x^{23} +$ \\ $
5075148445796857040053104431056506039875319568038077044908910516 x^{24} +$ \\ $
3107230011226517409969202256830258033314041885795677305379697962640 x^{25} +$ \\ $
2066974994609702017007988593816601920999780365479446917196525270118142 x^{26} +$ \\ $
1283970419502603731918140288038253636237197324650286051850270498862486325 x^{27} +$ \\ $
843919599048775485922343383280158888132623141387354468999629362152530427159 x^{28} +$ \\ $
529429090664842091425478647428185270992198760555038102521772553083678245220083 x^{29} + $ \\ $
345149098654298536175064215796331411355131680852917620919233112633273313585965081 x^{30} + $ \\ $
  \ldots $ }  \bc $ \mbox{--------------------------------------    } $ \\ $ \mbox{--------------------------------------    } $\ec
 ${\cal F}^{RG}_{17}(x) =$ \parbox[t]{5.4in}{$ \ds
 987 x^{2} +
375176968 x^{4} +
257232791130155 x^{6} +
200572437515846530901 x^{8} +
164782546678464075204330111 x^{10} +
138717888384983395400743673070647 x^{12} +
 118148889686793059954503676188547861439 x^{14} +
 101199388044301804167035198499446336399419451 x^{16} +
$ \\ $
 86918369741985767628242106496018767545685968221295 x^{18} +
$ \\ $
 74751308877855238601956611250031893025070379093798275467 x^{20} +
$ \\ $
 64328517891756221041740151134989268459463141424942999295477530 x^{22} +
$ \\ $
 55376125908869097780426555401191895400494047978296658342206340169311 x^{24} +
$ \\ $
 47676741632845917269695972808131435556699492747785589737028049129016308815 x^{26} +
 $ \\ $
 41050835690083594611243421104615824512345671970046160209809624943994676910850163 x^{28} +
$ \\ $
 35347007379909510253850320656294681456954784652577466801856109957124530806874455817574 x^{30} +
 \ldots  $}
  \bc $ \mbox{--------------------------------------    } $ \\ $ \mbox{--------------------------------------    } $\ec

\normalsize
\newpage
{\bf \large \ \ Thick cylinder graph $TkC_{m}(n) \equiv P_{m} \times C_{n}$ ($ 2 \leq m \leq 12$) }
\\ \\
  $ \ds  {\cal F}^{TkC}_{2}(x) = $  \parbox[t]{5.4in}{$ \ds   \frac{x (-1 - 4 x + 3 x^2 + 4 x^3)}{(1 - x) (1 + x) (-1 + x + x^2)} =  \ds -\frac{x (1 + 2 x)}{-1 + x + x^2} - \frac{2 x^2}{(-1 + x) (1 + x)}  $}
   \\ \\
 \hspace*{1.2cm}  $=$ \parbox[t]{5.4in}{ $ \ds x + 5 x^2 + 4 x^3 + 9 x^4 + 11 x^5 + 20 x^6 + 29 x^7 + 49 x^8 + 76 x^9 + 125 x^{10} + 199 x^{11} + 324 x^{12} + 521 x^{13} +
  845 x^{14} +
 1364 x^{15} + 2209 x^{16} + 3571 x^{17} + 5780 x^{18} + 9349 x^{19} +
 15129 x^{20} + 24476 x^{21} + 39605 x^{22} + 64079 x^{23} + 103684 x^{24} +
 167761 x^{25} + 271445 x^{26} + 439204 x^{27} + 710649 x^{28} +
 1149851 x^{29} + 1860500 x^{30} + \ldots $}
   \bc $ \mbox{--------------------------------------    } $ \ec  \noindent
$ \ds  {\cal F}^{TkC}_{3}(x) = $  \parbox[t]{5.4in}{$ \ds   \frac{x (1+ 12 x-6 x^2 - 36 x^3 -15 x^4)}{(1 + x) (-1 + 2 x +
    x^2) (-1 + 3 x^2)} \ds = \frac{6 x^2}{ 1 - 3 x^2}  - \frac{x (1 + 6 x + 3 x^2)}{(1 + x) (-1 + 2 x + x^2)}  $ }   \\ \\
 \hspace*{1.2cm}  $=$ \parbox[t]{5.4in}{ $ \ds  1 x^{1} + 13 x^{2} + 13 x^{3} + 53 x^{4} + 81 x^{5} + 253 x^{6} +
 477 x^{7} + 1317 x^{8} + 2785 x^{9} + 7213 x^{10} + 16237 x^{11} +
  40661 x^{12} + 94641 x^{13} + 232861 x^{14} + 551613 x^{15} +
 1344837 x^{16} + 3215041 x^{17} + 7801165 x^{18} + 18738637 x^{19} +
 45357173 x^{20} + 109216785 x^{21} + 264026941 x^{22} + 636562077 x^{23} +
 1537859685 x^{24} + 3710155681 x^{25} + 8960296813 x^{26} +
 21624372013 x^{27} + 52215418133 x^{28} + 126036076401 x^{29} +
 304306702813 x^{30}+ \ldots $}
  \bc $ \mbox{--------------------------------------    } $ \ec  \noindent
${\cal F}^{TkC}_{4}(x) =  $  \parbox[t]{5.4in}{$ \ds   \frac{(x (1 + 18 x + 15 x^2 - 20 x^3 - 10 x^4 + 6 x^5)}{(1 + x)(-1 + 2 x + 7 x^2 - 2 x^3 - 3 x^4 + x^5)}
- \frac{2 x^2 (-10 + 30 x^2 - 21 x^4 + 4 x^6)}{(-1 + x) (1 + x) (-1 - 3 x + x^3) (1 - 3 x + x^3)} $ \\ $ \ds  - \frac{2 x^2}{(-1 + x) (1 + x)}$}
   \\ \\
   \hspace*{1.2cm}  $=$ \parbox[t]{5.4in}{ $ \ds  0 + 1 x^{1} + 41 x^{2} + 43 x^{3} + 341 x^{4} + 666 x^{5} +
 3725 x^{6} + 9318 x^{7} + 44269 x^{8} + 127897 x^{9} +
 552136 x^{10} +  1747846 x^{11} + 7110833 x^{12} + 23860994 x^{13} +
 93547152 x^{14} + 325657968 x^{15} + 1247739997 x^{16} + 4444339751 x^{17} +
 16788557123 x^{18} + 60652119125 x^{19} + 227129712696 x^{20} +
 827719193253 x^{21} + 3083193286940 x^{22} + 11295869268631 x^{23} +
 41939971370665 x^{24} + 154154492857866 x^{25} + 571223860105758 x^{26} +
 2103743033971603 x^{27} + 7786110865225936 x^{28} +
 28709735327481674 x^{29} + 106179210831285100 x^{30} +
  \ldots $ }  \bc $ \mbox{--------------------------------------    } $ \ec  \noindent
${\cal F}^{TkC}_{5}(x) =  $  \parbox[t]{5.4in}{$ \ds   \frac{ 4 x^2 (15 - 206 x^2 + 732 x^4 - 882 x^6 + 325 x^8)}{(1 - x) (1 + x) (-1 + 2 x^2) (-1 + 5 x^2) (1 - 22 x^2 + 13 x^4)}    $ \\ \\ $ \ds  +  \frac{ x (1 + 50 x + 66 x^2 - 296 x^3 - 310 x^4 + 420 x^5 + 196 x^6 -
    224 x^7 + 36 x^8)}{(1 + 4 x + 2 x^2 - 2 x^3) (1 - 5 x - 7 x^2 + 18 x^3 + 6 x^4 - 12 x^5 + 2 x^6)} + \frac{10 x^2}{1 - 5 x^2} $}
      \\ \\
   \hspace*{1.2cm}  $=$ \parbox[t]{5.4in}{ $ \ds  0 + 1 x^{1} +
121 x^{2} +
142 x^{3} +
2169 x^{4} +
5431 x^{5} +
53812 x^{6} +
181231 x^{7} +
1474937 x^{8} +
5864650 x^{9} +
42414281 x^{10} +
188142923 x^{11} +
1258226556 x^{12} +
6019949235 x^{13} +
38161604217 x^{14} +
192461352562 x^{15} +
1176032485833 x^{16} +
6151532634859 x^{17} +
36655273043812 x^{18} +
196602170924523 x^{19} +
1151603700231369 x^{20} +
6283221090447394 x^{21} +
36379355621923193 x^{22} +
200804265323151235 x^{23} +
1153560534422212652 x^{24} +
6417449257742279731 x^{25}
 + 36672154206505187465 x^{26} +
 205093366164948581890 x^{27} + 1167841438010655902633 x^{28} +
 6554516723796443441531 x^{29} + 37233853354848282540212 x^{30} + \ldots   $ }
  \bc $ \mbox{--------------------------------------    } $ \ec  \noindent
    ${\cal F}^{TkC}_{6}(x) = $  \parbox[t]{5.4in}{$ \ds   [x (1 + 141 x + 399 x^2 - 3329 x^3 - 9584 x^4 + 19102 x^5 +
    52373 x^6 - 50803 x^7 - 107836 x^8 + 78024 x^9 + 89970 x^{10} -
    59838 x^{11} - 28625 x^{12} + 17295 x^{13} + 4800 x^{14} - 2128 x^{15} -
    462 x^{16} + 96 x^{17} + 20 x^{18})] / $ \\
    $[(1 - x) (1 + x) (-1 - 5 x - 5 x^2 + 2 x^3 + x^4) (-1 + 5 x + 49 x^2 - 116 x^3 - 363 x^4 +
    627 x^5 + 544 x^6 - 1061 x^7 + 133 x^8 + 264 x^9 - 47 x^{10} -
    26 x^{11} + 3 x^{12} + x^{13})]  $ \\ \\
      $   - [2 x^2 (90 - 4838 x^2 + 84741 x^4 - 686268 x^6 + 3002045 x^8 -
      7691820 x^{10} + 12049212 x^{12} - 11813888 x^{14} + 7364115 x^{16} -
      2947020 x^{18} + 757262 x^{20} - 123048 x^{22} + 12129 x^{24} -
      658 x^{26} + 15 x^{28} )] / $ \\ $[(-1 - 3 x + 6 x^2 + 4 x^3 - 5 x^4 - x^5 +
      x^6) (-1 + 3 x + 6 x^2 - 4 x^3 - 5 x^4 + x^5 + x^6) (1 - 9 x +
      6 x^2 + 53 x^3 - 45 x^4 - 66 x^5 + 52 x^6 - 6 x^8 + x^9) (-1 -
      9 x - 6 x^2 + 53 x^3 + 45 x^4 - 66 x^5 - 52 x^6 + 6 x^8 +
      x^9)]  $ \\ \\ $ \ds - \frac{2 x^2 (-21 + 140 x^2 - 252 x^4 + 180 x^6 - 55 x^8 +
      6 x^{10})}{(-1 - 3 x + 6 x^2 + 4 x^3 - 5 x^4 - x^5 + x^6) (-1 +
      3 x + 6 x^2 - 4 x^3 - 5 x^4 + x^5 + x^6)} $ \\ $ \ds - \frac{2 x^2}{(-1 + x) (1 + x)} $}  \\ \\
   \hspace*{1.2cm}  $=$ \parbox[t]{5.4in}{ $ \ds   0 + 1 x^{1} +
365 x^{2} +
469 x^{3} +
13825 x^{4} +
44466 x^{5} +
781043 x^{6} +
3562728 x^{7} +
49622793 x^{8} +
273687040 x^{9} +
3321537720 x^{10} +
20737405360 x^{11} +
229320317359 x^{12} +
1563616901289 x^{13} +
16179908259446 x^{14} +
117686516544594 x^{15} +
1160288515380121 x^{16} +
8851836530897104 x^{17} +
84238150673398010 x^{18} +
665628453678422066 x^{19} +
6172931794568240240 x^{20} +
50048363930787424965 x^{21} +
455504059143499103894 x^{22} +
3762987077441200385071 x^{23} +
33785095873930763984487 x^{24} +
282924063299636353075791 x^{25} +
2515329205574671707105301 x^{26} +
21271833119204266130369677 x^{27} +
187783161427523487902032166 x^{28} +
1599333925756691726920065197 x^{29} +
14046958280315913526905097038 x^{30} + \ldots  $}
   \bc $ \mbox{--------------------------------------    } $ \ec   \noindent
    ${\cal F}^{TkC}_{7}(x) = $  \parbox[t]{5.4in}{$ \ds
    -[2 x^2 (-266 + 49470 x^2 - 3409197 x^4 + 121091580 x^6 -
        2559558540 x^8 + 34862003130 x^{10} - 321313648313 x^{12} +
        2068311266640 x^{14} - 9491956496442 x^{16} +
        31463254982980 x^{18} - 75883794591264 x^{20} +
        133502684915796 x^{22} - 170943894184804 x^{24} +
        158097217968018 x^{26} - 104102314661145 x^{28} +
        47666357585520 x^{30} - 14614561483284 x^{32} +
        2817584069634 x^{34} - 303650135577 x^{36} +
        13687086240 x^{38})]/ $ \\ $[(-1 + 3 x^2) (-1 - 6 x - 9 x^2 +
        3 x^4) (-1 + 6 x - 9 x^2 + 3 x^4) (1 - 14 x^2 + 17 x^4) (1 -
        60 x^2 + 454 x^4 - 956 x^6 + 577 x^8) (-1 + 171 x^2 -
        5496 x^4 + 56617 x^6 - 240021 x^8 + 457923 x^{10} -
        420254 x^{12} + 186912 x^{14} - 37569 x^{16} + 2584 x^{18})]$ \\ \\
 $ - [(x (-1 - 392 x - 960 x^2 + 39200 x^3 + 95200 x^4 - 1296168 x^5 -
       2504264 x^6 + 21332992 x^7 + 27475452 x^8 - 201282400 x^9 -
       130694300 x^{10} + 1131132504 x^{11} + 131565122 x^{12 }-
       3731908992 x^{13} + 988208340 x^{14} + 7002895232 x^{15} -
       3424413360 x^{16} - 7609684680 x^{17} + 4621831320 x^{18 }+
       4919663340 x^{19} - 3161787279 x^{20} - 1941648016 x^{21} +
       1132172608 x^{22} + 480442536 x^{23} - 195103475 x^{24} -
       73266440 x^{25} + 11158344 x^{26} + 5243588 x^{27} +
       366821 x^{28})]/ $ \\ $ [(-1 + x) (-1 + 7 x^2) (-1 - 14 x - 44 x^2 +
       52 x^3 + 303 x^4 + 44 x^5 - 462 x^6 - 252 x^7 + 107 x^8 +
       88 x^9 + 13 x^{10}) (-1 + 14 x + 37 x^2 - 670 x^3 + 216 x^4 +
       7866 x^5 - 10202 x^6 - 27170 x^7 + 56210 x^8 + 5872 x^9 -
       66223 x^{10} + 22200 x^{11} + 25320 x^{12} - 12008 x^{13} -
       2888 x^{14} + 1256 x^{15} + 139 x^{16})] $ \\ \\
$ - \ds \frac{2 x^2 (77 - 3066 x^2 + 36795 x^4 - 186700 x^6 + 446835 x^8 -
    496794 x^10 + 205989 x^{12})}{(-1 + 3 x^2) (1 - 14 x^2 +
    17 x^4) (1 - 60 x^2 + 454 x^4 - 956 x^6 + 577 x^8)} +
     \frac{14 x^2}{1 - 7 x^2} $} \\ \\
 \hspace*{1.2cm}  $=$ \parbox[t]{5.4in}{ $ \ds
x + 1093 x^2 + 1549 x^3 + 88093 x^4 + 364061 x^5 + 11328703 x^6 +
 70182449 x^7 + 1670194477 x^8 + 12839393125 x^9 +
 261079885983 x^{10} + 2304956480025 x^{11} + 42157390580371 x^{12} +
 410685863818829 x^{13} + 6955967193562675 x^{14} +
 72945933567920729 x^{15} + 1166255193469341005 x^{16} +
 12939617301085024529 x^{17} + 197983223809129991221 x^{18} +
 2294032263176793405709 x^{19} + 33940092430391553332183 x^{20} +
 406606652245177626433361 x^{21} + 5863258912633455285257837 x^{22} +
 72061859473867959667019309 x^{23} +
 1019002446024036683543413987 x^{24} +
 12770787731127477795116042761 x^{25} +
 177924053187181378492219053991 x^{26} +
 2263194608166222107070101243149 x^{27} +
 31178222232265755481412717950347 x^{28} +
 401072323924104852077254184227297 x^{29} +
 5478470730011608244359772114877693 x^{30} + \ldots $}
    \bc $ \mbox{--------------------------------------    } $ \\ $ \mbox{--------------------------------------    } $  \ec  \noindent
 ${\cal F}^{TkC}_{8}(x) = $  \parbox[t]{5.4in}{$ \ds
   -[x (1 + 1106 x + 3456 x^2 - 380960 x^3 - 1265450 x^4 +
         48085788 x^5 + 138769869 x^6 - 3204566808 x^7 -
         7166185659 x^8 + 130912828640 x^9 + 201000116515 x^{10} -
         3528895913796 x^{11} - 3099483692487 x^{12} +
         65290376214658 x^{13} + 20601900448935 x^{14} -
         848910526618144 x^{15} + 116762671264678 x^{16} +
         7894168967938824 x^{17} - 3855259753368817 x^{18} -
         53370261280411640 x^{19} + 40288771832558214 x^{20} +
         266680528560477416 x^{21} - 254856183979947518 x^{22} -
         1001228898109396296 x^{23} + 1106050661519880450 x^{24} +
         2870997403499935710 x^{25} - 3473527801283826489 x^{26} -
         6391248919845409200 x^{27} + 8136198344092308724 x^{28} +
         11226645720257794020 x^{29} - 14496097136557554582 x^{30} -
         15804102272553391872 x^{31} + 19907630968767236079 x^{32} +
         18069375080025641078 x^{33} - 21252793388482968375 x^{34} -
         16933496115380147028 x^{35} + 17712349720891361206 x^{36} +
         13049494614554522258 x^{37} - 11522895817209789324 x^{38} -
         8243960519699875080 x^{39} + 5823007886337424151 x^{40} +
         4233494797648211340 x^{41} - 2261065609667714676 x^{42} -
         1746009302105224376 x^{43} + 661549210479454725 x^{44} +
         570260520422775764 x^{45} - 140777566762420997 x^{46} -
         145295742885744480 x^{47} + 20220218000063669 x^{48} +
         28439984938628400 x^{49} - 1540115674944846 x^{50} -
         4213964855453496 x^{51} - 46016162958923 x^{52} +
         466719109628076 x^{53} + 27506894875350 x^{54} -
         38318301708680 x^{55} - 3542688363396 x^{56} +
         2324639722354 x^{57} + 263472623887 x^{58} - 104090681220 x^{59} -
         12730836091 x^{60} + 3425279156 x^{61} + 407727495 x^{62} -
         81580224 x^{63} - 8409635 x^{64} + 1350558 x^{65} + 101773 x^{66} -
         14076 x^{67} - 552 x^{68} + 70 x^{69})]/ $ \\ $ \ds [(1 - 7 x + 14 x^2 -
         5 x^3 - 5 x^4 + x^5) (1 + 4 x - 10 x^2 - 10 x^3 + 15 x^4 +
         6 x^5 - 7 x^6 - x^7 + x^8) (1 - 4 x - 10 x^2 + 10 x^3 +
         15 x^4 - 6 x^5 - 7 x^6 + x^7 + x^8) (1 + 20 x + 122 x^2 +
         164 x^3 - 690 x^4 - 1751 x^5 + 480 x^6 + 3573 x^7 +
         1588 x^8 - 1385 x^9 - 1217 x^{10} - 266 x^{11} + 13 x^{12} +
         11 x^{13} + x^{14}) (-1 + 14 x + 331 x^2 - 3474 x^3 -
         24357 x^4 + 237534 x^5 + 541266 x^6 - 6604103 x^7 -
         1905497 x^8 + 85855152 x^9 - 60009003 x^{10} -
         545836271 x^{11} + 672927757 x^{12} + 1747850343 x^{13} -
         2763674623 x^{14} - 2917536240 x^{15} + 5513512152 x^{16}+
         2653029943 x^{17} - 5852097578 x^{18} - 1465977019 x^{19} +
         3471750395 x^{20} + 568784352 x^{21} - 1167520145 x^{22} -
         154667330 x^{23} + 221656480 x^{24} + 23823457 x^{25} -
         24542626 x^{26} - 1818710 x^{27} + 1646233 x^{28} + 57030 x^{29} -
         66339 x^{30} + 348 x^{31} + 1479 x^{32} - 61 x^{33} - 14 x^{34} +
         x^{35})]$}

  \noindent
 \hspace*{1.2cm} $-$ \parbox[t]{5.4in}{
         $ \ds
[2 x^2 (-784 + 476076 x^2 - 116588381 x^4 + 15789090382 x^6 -
         1356158705385 x^8 + 79891719333866 x^{10} -
         3396666207504809 x^{12} + 108022486126780644 x^{14} -
         2638606408306440592 x^{16} + 50515019900883249532 x^{18} -
         770094872481299969229 x^{20} + 9468055865340829781546 x^{22} -
         94854246821132793986726 x^{24} +
         780971311623813881054046 x^{26} -
         5322216842654634684825502 x^{28} +
         30203095029547418165520630 x^{30} -
         143469074682768685161127425 x^{32} +
         573006284226316093693515432 x^{34} -
         1931783161319310182389650188 x^{36} +
         5516458536983063883349698292 x^{38} -
         13384589486720192919479214353 x^{40} +
         27668861475144429498976041030 x^{42} -
         48853344502932072289147389670 x^{44} +
         73837554646218155871407449654 x^{46} -
         95719065898952727277919257597 x^{48} +
         106614006027866384160385380476 x^{50} -
         102183431711082749008135970704 x^{52 }+
         84382266059163112509904423768 x^{54} -
         60100320449233882107856423934 x^{56} +
         36948907591988335819053998194 x^{58 }-
         19618292636434482765390118001 x^{60 }+
         8998701999089476593485162284 x^{62} -
         3565889922267097744709910317 x^{64} +
         1220446933427050507410769508 x^{66} -
         360583356141429452470446877 x^{68} +
         91891468200618096273944570 x^{70 }-
         20176247923305013249869507 x^{72} +
         3811258839592645070587098 x^{74} -
         618248800772164980608457 x^{76} +
         85929263260820286200628 x^{78} -
         10204593695082332193606 x^{80} + 1031931740063277446820 x^{82} -
         88490109610005877955 x^{84}+ 6401587098797864482 x^{86} -
         388181540253454496 x^{88}+ 19570500615145706 x^{90} -
         811788351390113 x^{92} + 27326379866244 x^{94} -
         732721891153 x^{96} + 15246086050 x^{98} - 236803536 x^{100} +
         2579106 x^{102} - 17543 x^{104} + 56 x^{106})]/ $  \\ $[(-1 + x) (1 +
         x) (1 - 28 x + 134 x^2 + 1464 x^3 - 11646 x^4 - 8775 x^5 +
         234042 x^6 - 318372 x^7 - 1512042 x^8 + 3990140 x^9 +
         1327546 x^{10} - 12508340 x^{11} + 8235416 x^{12} +
         11304952 x^{13} - 15649778 x^{14} + 1400926 x^{15} +
         6404612 x^{16} - 2944582 x^{17} - 312236 x^{18} + 418067 x^{19} -
         31381 x^{20} - 22903 x^{21} + 3184 x^{22} + 556 x^{23} - 97 x^{24} -
         5 x^{25} + x^{26}) (1 + 28 x + 134 x^2 - 1464 x^3 - 11646 x^4 +
         8775 x^5 + 234042 x^6 + 318372 x^7 - 1512042 x^8 -
         3990140 x^9 + 1327546 x^{10} + 12508340 x^{11} + 8235416 x^{12} -
         11304952 x^{13} - 15649778 x^{14} - 1400926 x^{15} +
         6404612 x^{16} + 2944582 x^{17} - 312236 x^{18} - 418067 x^{19} -
         31381 x^{20} + 22903 x^{21} + 3184 x^{22} - 556 x^{23} - 97 x^{24} +
         5 x^{25} + x^{26}) (-1 + 9 x + 92 x^2 - 581 x^3 - 2083 x^4 +
         11003 x^5 + 18456 x^6 - 89508 x^7 - 76454 x^8 + 363004 x^9 +
         148765 x^{10} - 782325 x^{11} - 101940 x^{12} + 931622 x^{13} -
         46433 x^{14} - 638214 x^{15} + 110532 x^{16} + 256374 x^{17} -
         68466 x^{18} - 59420 x^{19} + 20987 x^{20} + 7328 x^{21} -
         3425 x^{22} - 351 x^{23} + 281 x^{24} - 9 x^{25} - 9 x^{26} +
         x^{27}) (1 + 9 x - 92 x^2 - 581 x^3 + 2083 x^4 + 11003 x^5 -
         18456 x^6 - 89508 x^7 + 76454 x^8 + 363004 x^9 -
         148765 x^{10} - 782325 x^{11} + 101940 x^{12} + 931622 x^{13} +
         46433 x^{14} - 638214 x^{15} - 110532 x^{16} + 256374 x^{17} +
         68466 x^{18} - 59420 x^{19} - 20987 x^{20} + 7328 x^{21} +
         3425 x^{22} - 351 x^{23} - 281 x^{24} - 9 x^{25} + 9 x^{26} + x^{27})] $}
       \\ \\
 \hspace*{1.2cm} $-$ \parbox[t]{5.4in}{
         $ \ds
         [2 x^2 (-266 + 46706 x^2 - 2936661 x^4 + 92960520 x^6 -
         1725727350 x^8 + 20491047750 x^{10} - 164350070864 x^{12} +
         923456110368 x^{14} - 3729690794520 x^{16} +
         11039887243530 x^{18} - 24330288950594 x^{20} +
         40470025352148 x^{22} - 51416302161224 x^{24} +
         50406642154750 x^{26} - 38453137571055 x^{28} +
         22972741852096 x^{30} - 10794669605306 x^{32} +
         3998085999264 x^{34} - 1166952360510 x^{36} +
         267658533100 x^{38} - 47956307721 x^{40} + 6645525414 x^{42} -
         701378744 x^{44} + 55073040 x^{46} - 3101850 x^{48} +
         117936 x^{50} - 2700 x^{52} + 28 x^{54})]/ $ \\  $[(-1 + x) (1 + x) (-1 +
         9 x + 92 x^2 - 581 x^3 - 2083 x^4 + 11003 x^5 + 18456 x^6 -
         89508 x^7 - 76454 x^8 + 363004 x^9 + 148765 x^{10} -
         782325 x^{11} - 101940 x^{12} + 931622 x^{13} - 46433 x^{14} -
         638214 x^{15} + 110532 x^{16} + 256374 x^{17} - 68466 x^{18} -
         59420 x^{19} + 20987 x^{20} + 7328 x^{21} - 3425 x^{22} - 351 x^{23} +
         281 x^{24} - 9 x^{25} - 9 x^{26} + x^{27}) (1 + 9 x - 92 x^2 -
         581 x^3 + 2083 x^4 + 11003 x^5 - 18456 x^6 - 89508 x^7 +
         76454 x^8 + 363004 x^9 - 148765 x^{10} - 782325 x^{11} +
         101940 x^{12} + 931622 x^{13} + 46433 x^{14} - 638214 x^{15} -
         110532 x^{16} + 256374 x^{17} + 68466 x^{18} - 59420 x^{19} -
         20987 x^{20} + 7328 x^{21} + 3425 x^{22} - 351 x^{23} - 281 x^{24} -
         9 x^{25} + 9 x^{26} + x^{27})] $}
        \\ \\
 \hspace*{1.2cm} $-$ \parbox[t]{5.4in}{
         $ \ds
           [2 x^2 (-36 + 420 x^2 -
         1386 x^4 + 1980 x^6 - 1430 x^8 + 546 x^{10} - 105 x^{12} +
         8 x^{14})] / $ \\  $ [(1 + 4 x - 10 x^2 - 10 x^3 + 15 x^4 + 6 x^5 -
         7 x^6 - x^7 + x^8) (1 - 4 x - 10 x^2 + 10 x^3 + 15 x^4 -
         6 x^5 - 7 x^6 + x^7 + x^8)]$}
          \\ \\
          \hspace*{1.2cm} $-$  \parbox[t]{5.4in}{ $ \ds
           \frac{(2 x^2)}{(-1 + x) (1 + x)} $}  \bc $ \mbox{--------------------------------------    } $ \ec
            \noindent
   ${\cal F}^{TkC}_{8}(x) = $  \parbox[t]{5.4in}{ $ \ds
x + 3281 x^2 + 5116 x^3 + 561357 x^4 + 2981201 x^5 + 164342144 x^6 +
 1384148396 x^7 + 56241588037 x^8 + 604211712448 x^9 +
 20559551095851 x^{10} + 257573508034492 x^{11} + 7784584974538368 x^{12} +
 108680224298965775 x^{13} + 3013814772945044388 x^{14} +
 45643040884844001536 x^{15} + 1185537358216390245429 x^{16} +
 19127617320816121366452 x^{17} + 472116548787108690744794 x^{18} +
 8007718640288679793007567 x^{19 }+ 189866075873770814991703727 x^{20} +
 3350817898120123595046160844 x^{21} +
 76965624292097393220127868826 x^{22} +
 1401831762167059299778164115466 x^{23} +
 31400919568784474985238557000304 x^{24} +
 586401753216671829580093021638776 x^{25} +
 12877877037714243225365022029863413 x^{26} +
 245286171577403118116883927290448238 x^{27} + $ \\ $
 5303427339483094210044349210160088892 x^{28} + $ \\ $
 102598425625894938585872354365515091531 x^{29} + $ \\ $
 2191357558208001087370202721666183256764 x^{30} +  \ldots $}
    \bc $ \mbox{--------------------------------------    } $ \\ $ \mbox{--------------------------------------    } $ \ec  \noindent
 ${\cal F}^{TkC}_{9}(x) = $  \parbox[t]{5.4in}{$ \ds
  -[(4 x^2 (1152 - 2206597 x^2 + 1807769940 x^4 - 860478696692 x^6 +
         271394988599842 x^8 - 61129048677172881 x^{10} +
         10325212841842911826 x^{12} - 1353825936442525863750 x^{14} +
         141392867120412980286438 x^{16} -
         11999207460234121959949283 x^{18} +
         840679654256968498368980410 x^{20} -
         49257139757849490877927070484 x^{22} +
         2439569974008328995823937710340 x^{24} -
         103054787449811934712722869675244 x^{26} +
         3741625339141960685479905918819684 x^{28 }-
         117531759075853838052008575939685668 x^{30} + $ \\ $
         3212523903133186061692420680378232140 x^{32} - $ \\ $
         76793155385996505350177192512107604143 x^{34} + $ \\ $
         1612576030720239216104735528006452413760 x^{36} - $ \\ $
         29865135674377359402357253372254071946470 x^{38} + $ \\ $
         489554622438197898118913984319053406012798 x^{40} - $ \\ $
         7125612635514102460216877115387631680701028 x^{42} + $ \\ $
         92360515626279950365237519661783198749080036 x^{44} - $ \\ $
         1068895817380708290533389006273379476977335364 x^{46} + $ \\ $
         11071526522007272294397697749779599638271754574 x^{48} - $ \\ $
         102860388764696736244891194752685388811296848838 x^{50} + $ \\ $
         858853023759632300903892887424256653390801507060 x^{52} -
         6456607859451424936977014115658311646553442825244 x^{54} +
         43774282328683483950589913363621461498687759002270 x^{56} -
         268047175389006192979911264467018887170891433557660 x^{58} +
         1484458263317691438899879472117255474928557841522390 x^{60} -
         7444202913918183601119169475813074724175649386256522 x^{62} +
         33840405900113216368039295505510918386095739251165146 x^{64} -
         139585709446719896970995265534998378548498518538133333 x^{66}
+ 522882889365866854198479790961912955205713697542988958 x^{68} -
         1780109772224902337792352028208151332360041376052751632 x^{70}
+ 5511132451994509012201612959166727678498084741467194038 x^{72} -
         15524314278817883298706115695046037857068451538788105266
x^{74} + 39805303523502452222240610846137579209607852072462584968 x^{76}
- 92930822910051984285533672559029361475263967498396039188 x^{78} +
         197586088267141781898425067964899211878776752057806289444 x^{80} - 382623672367378752492150417243490451926441883893463019981 x^{82}
+ 674839507453138576123408186500461625002632027549579798834 x^{84} -
         1083899065552077762714031789808611358078590994960140989406
x^{86} + 1585017953044366797303958249252120447257903476943506240916
x^{88} - 2109516427670921510010239240579282933585091915128628083983
x^{90} + $ } \\
\\
 \hspace*{1.4cm}  \parbox[t]{5.4in}{
         $ \ds 2554042935879964380546977423012626041601105303737425649836
x^{92} - 2811299983123677520405748408047610574298238749051185393900
x^{94} + 2811237961193741835319891777354506083645198593908837507440
x^{96} - 2551605650594589508100365226809130758645656466161828122540
x^{98} + 2099909305583510683942573711811859768984918998577698691786
x^{100} - 1565054569354074023627420336421691093367501137324583943706
x^{102} + 1054847267027220888928717267878553076517650022705325092442
x^{104} - 641918216718905455466201195189961636451702679947248988248
x^{106} + 352044939047358377943587973140979153816527256061527074638
x^{108} - 173632116959842677798117796336416002197530474889638111614
x^{110} + 76830351865625749067652061203856306079429418454537775004
x^{112} - 30417068650057657079949964083866024730627727561822250049
x^{114} + 10740534518756801751087841937388013456342038466240406800  x^{116} - 3370575798268931937605669718640573969136591537041811490 x^{118} +
 936194752338041027150382554160149822725214215170683358 x^{120} -
         229060794567879430879986197916270350229389570933032790 x^{122}
+ 49097981020077321821394054308837675408365223310824724 x^{124} -
         9160162908109338279130089543484639952841071997906924 x^{126 }+
         1476251218877966473156136661379527905668352737637724 x^{128 }-
         203648302833520038046608849670590277759293561306063 x^{130 }+
         23783292408157109218404190043013411327197079483912 x^{132} -
         2319630349743410293849352255862856465279294764084 x^{134} +
         185716889380954519854696803288600568148137400608 x^{136} -
         11935732786173705051350546870274420224169228763 x^{138} +
         597365756972132465348125559350976243010717526 x^{140} - $ \\$
         22291298699609263186012724844614398896394230 x^{142} + $ \\$
         579480946528436768019023951476570179717540 x^{144} - $ \\$
         9287440866939596231661962056823280350520 x^{146} + $ \\$
         68423540017319466905734972657794961920 x^{148})]/
        $} \\ \\
 \hspace*{1.4cm}  \parbox[t]{5.4in}{
         $ \ds [(-1 +3 x) (1 + 3 x) (-1 + 7 x^2) (-1 + 27 x - 258 x^2 + 973 x^3 -
         324 x^4 - 6744 x^5 + 12454 x^6 + 7362 x^7 - 33489 x^8 +
         15667 x^9 + 19476 x^{10} - 20253 x^{11} + 3729 x^{12} +
         2151 x^{13} - 966 x^{14} + 107 x^{15}) (1 + 27 x + 258 x^2 +
         973 x^3 + 324 x^4 - 6744 x^5 - 12454 x^6 + 7362 x^7 +
         33489 x^8 + 15667 x^9 - 19476 x^{10} - 20253 x^{11} -
         3729 x^{12} + 2151 x^{13} + 966 x^{14} + 107 x^{15}) (-1 + 282 x^2 -
         25297 x^4 + 1073828 x^6 - 25390104 x^8 + 363078264 x^{10} -
         3296168948 x^{12} + 19596248926 x^{14} - 77743595826 x^{16} +
         207473742096 x^{18} - 371467677512 x^{20} + 439852643504 x^{22} -
         334403214576 x^{24} + 154506971212 x^{26} - 38880116221 x^{28} +
         4020609392 x^{30}) (-1 + 591 x^2 - 95361 x^4 + 6922350 x^6 -
         272760016 x^8 + 6420903257 x^{10} - 95828390004 x^{12} +
         943120716586 x^{14} - 6295303405260 x^{16} +
         29127727204334 x^{18} - 95000584032342 x^{20} +
         220996143367594 x^{22} - 369011778492872 x^{24} +
         442338529800436 x^{26} - 377839526215177 x^{28} +
         226022147292981 x^{30} - 91695292950038 x^{32} +
         23818609428605 x^{34} - 3541207333504 x^{36} +
         226978239492 x^{38}) (1 - 1193 x^2 + 376246 x^4 -
         48953410 x^6 + 3288145988 x^8 - 127374411928 x^{10} +
         3015668747782 x^{12} - 45191010425846 x^{14} +
         441384780778588 x^{16} - 2898283223877346 x^{18} +
         13154666974580501 x^{20} - 42187756055653825 x^{22} +
         97142224830553641 x^{24} - 162322162033938237 x^{26} +
         198042290945862570 x^{28} - 176799585485005402 x^{30} +
         115298993750386955 x^{32} - 54611642383339285 x^{34} +
         18586566465572115 x^{36} - 4467086405032683 x^{38} +
         737944576901349 x^{40} - 80312266104179 x^{42} +
                 5368435066393 x^{44} - 193455857453 x^{46} +
         2735506380 x^{48})]   $} \\ \\
 \hspace*{1.2cm} $-$ \parbox[t]{5.4in}{
         $ \ds
          [ x (-1 - 3138 x - 12189 x^2 +
         3515332 x^3 + 15404430 x^4 - 1556674230 x^5 -
         6321089166 x^6 + 382398690520 x^7 + 1311805776717 x^8 -
         60191451351230 x^9 - 161714477451835 x^{10} +
         6565943755483548 x^{11} + 12678569829491020 x^{12} -
         520824200595298166 x^{13} - 635057264653902600 x^{14} +
         31018975690335967808 x^{15 }+ 17969042737456614346 x^{16} -
         1419104748968250857064 x^{17} - 28573194282709941282 x^{18} +
         50739971287643233161280 x^{19} -
         23321605296749449030530 x^{20} -
         1437595501278814830945696 x^{21} +
         1246436276742341171802738 x^{22} +
         32653647530636803573858728 x^{23} -
         39549571533248439377462425 x^{24} -
         600747530859523253744627206 x^{25} +
         901838365992119390863259835 x^{26} +
         9036269601831280229461670084 x^{27} -
         15767174968682617291823228808 x^{28} -
         112110393728184116740742129970 x^{29 }+
         218396535574115920591352161252 x^{30} +
         1156969110445872340506621316800 x^{31} -
         2445977215939258136686771457556 x^{32 }-
         10013021974875673448875473181612 x^{33 }+
         22468753830812884590897349545720 x^{34} +
         73254537161346178025352603410184 x^{35} -
         171116404295705733576070463769138 x^{36} -
         456546369711656869175950887422116 x^{37} +
         1089501611835078838476654071949126 x^{38} +
         2441799546977454933233303612239360 x^{39} -
         5837835938120854141503189705750966 x^{40} -
         11283222038646959984806437055484328 x^{41} +
         26459796657217710589610954854748018 x^{42} +
         45305655498248473533492601646920280 x^{43} -
         101830809088972415306050371382356000 x^{44}  $} \\ \\
 \hspace*{1.2cm} $-$ \parbox[t]{5.4in}{
         $ \ds
         158769597569467546957236381478241884 x^{45} + $ \\$
         333602480872399458310608825887872060 x^{46} + $ \\$
         486891111829495745400315864384599856 x^{47} - $ \\$
         931512112285624814472562522013924153 x^{48} - $ \\$
         1307621805898576836730254939863124250 x^{49} + $ \\$
         2216858943621280965971203658449091583 x^{50} + $ \\$
         3072311426335150064111784410898873804 x^{51} - $ \\$
         4489893407841152215806825363972580824 x^{52} - $ \\$
         6298780695313693563204719523565583970 x^{53} + $ \\$
         7714682182116505131395647593211352320 x^{54} + $ \\$
         11226004883194166915002239490358455144 x^{55} - $ \\$
         11186932378073136432066627959730632601 x^{56} - $ \\$
         17314097813292462552629492820365378158 x^{57} + $ \\$
         13579058806531715431643941341835603523 x^{58} + $ \\$
         22993197647642170789056440237018698980 x^{59} - $ \\$
         13621068066940280399402107416974207424 x^{60} - $ \\$
         26152891377100473713869813061381754434 x^{61} + $ \\$
         11048521659909951982828945243383497700 x^{62} + $ \\$
         25338494312883572130973760001987138752 x^{63} - $ \\$
         6945789858469200809020665747038605105 x^{64} - $ \\$
         20794403404415553590317797637527775926 x^{65} + $ \\$
         3032860264722045783767482620742503735 x^{66} + $ \\$
         14372377604960769204836570883901477740 x^{67} - $ \\$
         510402517612737823772356095439614600 x^{68} - $ \\$
         8317292140617923673352631136892286490 x^{69} - $ \\$
         492730082201815735591482689949520564 x^{70} + $ \\$
         4005802325781374053712890057445454552 x^{71} + $ \\$
         563733563102550147260095801146488403 x^{72} - $ \\$
         1595648262683964820223988282675689718 x^{73} - $ \\$
         336787073290574339033261345124831525 x^{74} + $ \\$
         522249728624845086759177188585051428 x^{75} + $\\$
         142066679139061764744439947602204844 x^{76} - $ \\$
         139476132350625289583208947011451178 x^{77} - $ \\$
         45312821551964621122963064573776428 x^{78} + $ \\$
         30173125171564888052177237864679760 x^{79} +
         $ } \\
         \hspace*{1.4cm}  \parbox[t]{5.4in}{ $
         11166465509774787347592798792199797 x^{80} -
         5247668526582028055700229832909050 x^{81} -
         2134698340398432305328651032311151 x^{82} +
         728447057952532535130254335181244 x^{83} +
         314716693247201043560851990206550 x^{84} -
         80222598271908421470783237786410 x^{85} -
         35273538383727573759248324022822 x^{86} +
         6980065935695458245293364595456 x^{87} +
         2933330070032759915228223085064 x^{88} -
         477757230965691698459902340880 x^{89} -
         173936904703798126094290246160 x^{90} +
         25348066369510808199075310400 x^{91} +
         6861110985320269301927583648 x^{92} -
         992686617044143586989227248 x^{93} -
         155810662662956739123972640 x^{94} +
         25345512133505650393850880 x^{95} +
         1254708009279296759769216 x^{96} -
         308540156964491828873728 x^{97} +
         10484643019216813894656 x^{98})] /
        $} \\ \\
 \hspace*{1.4cm}  \parbox[t]{5.4in}{
         $ \ds
         [(1 + x) (-1 + 2 x) (1 + 2 x) (1 - 8 x^2 + 11 x^4) (1 - 28 x^2 + 71 x^4) (1 - 8 x +
         20 x^2 - 14 x^3 - 5 x^4 + 4 x^5) (1 - 116 x^2 + 1546 x^4 -
         4556 x^6 + 3781 x^8) (-1 - 48 x - 708 x^2 - 1640 x^3 +
         45946 x^4 + 341664 x^5 - 270194 x^6 - 9936934 x^7 -
         22840421 x^{8} + 84442096 x^9 + 418028192 x^{10} +
         132842204 x^{11} - 2217116283 x^{12 }- 3910256042 x^{13} +
         2255114380 x^{14} + 12952834812 x^{15} + 9797088008 x^{16} -
         9926380870 x^{17} - 19739579992 x^{18} - 6238648278 x^{19} +
         8703028142 x^{20} + 7538491050 x^{21} + 208212 x^{22} -
         2158389564 x^{23} - 641271414 x^{24} + 192328290 x^{25} +
         111002758 x^{26} + 399556 x^{27} - 5772921 x^{28} - 305408 x^{29} +
         97576 x^{30}) (1 - 42 x - 75 x^2 + 17718 x^3 - 112036 x^4 -
         2032663 x^5 + 20898099 x^6 + 68505152 x^7 - 1367337026 x^8 +
         1223952387 x^9 + 39924985841 x^{10} - 123520021512 x^{11} -
         507429891185 x^{12} + 2815200597001 x^{13} +
         1444158632036 x^{14} - 29150884549434 x^{15} +
         26197136927827 x^{16} + 152864286212934 x^{17} -
         281984662214835 x^{18} - 398046044696312 x^{19} +
         1249282231645481 x^{20} + 358977981283233 x^{21} -
         3123666474218858 x^{22} + 675827827012992 x^{23} +
         4889486045969022 x^{24} - 2485484028220910 x^{25} -
         5032512202755210 x^{26} + 3593534888276858 x^{27} +
         3475263477225909 x^{28} - 3098930934608376 x^{29} -
         1601027066102163 x^{30} + 1744452573290918 x^{31} +
         470550625308136 x^{32} - 657690139223493 x^{33} -
         75567716491829 x^{34} + 165305848427048 x^{35} +
         1498941418335 x^{36} - 26855694880119 x^{37} +
         1821955387786 x^{38} + 2634688847668 x^{39} -
         334995651352 x^{40} - 134705018616 x^{41} + 24157589768 x^{42} +
         2269409656 x^{43} - 605268248 x^{44} +
         22971944 x^{45})]
          $} \\ \\ \\
 \hspace*{1.2cm} $-$  \parbox[t]{5.4in}{
         $ \ds [2 x^2 (882 - 590354 x^2 +
         157273161 x^4 - 22983115792 x^6 + 2119494744100 x^8 -
         133590374014446 x^{10} + 6058887885974721 x^{12} -
         204935302170164648 x^{14} + 5306071771929323052 x^{16} -
         107251721247833706960 x^{18} + 1718345284467307783496 x^{20} -
         22084257636082742645904 x^{22} +
         229857036770612349580820 x^{24} -
         1952359449465652537576140 x^{26} +
         13616617708467269758731900 x^{28} -
         78367989158751373954628480 x^{30} +
         373654558327977301862358522 x^{32} -
         1480352154646506945693324048 x^{34} +
         4883719271006483039945303166 x^{36} -
         13433591701765415961153857560 x^{38} +
         30823850149478993381873048244 x^{40} -
         58971942961519011103241421464 x^{42} +
         93941006726387697717195658180 x^{44} -
         124286525502217618209539873280 x^{46} +
         136049531908889279215906382100 x^{48} -
         122556391797201522095216232744 x^{50} +
         90186798622481620965906398496 x^{52} -
         53678118116110605690524591388 x^{54} +
         25494519661741120363771537818 x^{56} -
         9485124196343409579683767530 x^{58} +
         2692664405977965225731923549 x^{60} -
         560914565743916029157069344 x^{62} +
         80510797349088438843948756 x^{64} -
         7088230137939231002107176 x^{66} +
         287466115066565560292160 x^{68})]/ $ } \\
         \hspace*{1.4cm}  \parbox[t]{5.4in}{ $ \ds
          [(-1 + 3 x) (1 + 3 x) (-1 +
         282 x^2 - 25297 x^4 + 1073828 x^6 - 25390104 x^8 +
         363078264 x^{10} - 3296168948 x^{12} + 19596248926 x^{14} -
         77743595826 x^{16} + 207473742096 x^{18} - 371467677512 x^{20} +
         439852643504 x^{22} - 334403214576 x^{24} + 154506971212 x^{26} -
         38880116221 x^{28} + 4020609392 x^{30}) (-1 + 591 x^2 -
         95361 x^4 + 6922350 x^6 - 272760016 x^8 + 6420903257 x^{10} -
         95828390004 x^{12} + 943120716586 x^{14} - 6295303405260 x^{16} +
         29127727204334 x^{18} - 95000584032342 x^{20} +
         220996143367594 x^{22} - 369011778492872 x^{24} +
         442338529800436 x^{26} - 377839526215177 x^{28} +
         226022147292981 x^{30} - 91695292950038 x^{32} +
         23818609428605 x^{34} - 3541207333504 x^{36} +
         226978239492 x^{38})]  $} \\ \\
 \hspace*{1.2cm} $-$ \parbox[t]{5.4in}{ $ \ds
         [8 x^2 (39 - 3318 x^2 + 90522 x^4 - 1129606 x^6 + 7435280 x^8 -
      27384066 x^{10} + 56511462 x^{12} - 60869058 x^{14 }+
      26576649 x^{16})]/ $ \\ $ [(-1 + 2 x) (1 + 2 x) (1 - 8 x^2 + 11 x^4) (1 -
      28 x^2 + 71 x^4) (1 - 116 x^2 + 1546 x^4 - 4556 x^6 +
      3781 x^8)] $ } \\ \\
       \hspace*{1.2cm}  $+$ \parbox[t]{5.4in}{ $ \ds
        \frac{18 x^2}{1 - 9 x^2} $}
         \bc $ \mbox{--------------------------------------    } $ \ec  \noindent
 ${\cal F}^{TkC}_{9}(x) = $  \parbox[t]{5.4in}{ $ \ds
0 + 1 x^{1} +
9841 x^{2} +
16897 x^{3} +
3577121 x^{4} +
24412606 x^{5} +
2384008549 x^{6} +
27309182412 x^{7} +
1893972519489 x^{8} +
28474336325785 x^{9} +
1619938572971116 x^{10} +
28865363400315608 x^{11} +
1440027274442086769 x^{12} +
28885443646068696636 x^{13} +
1310571139307486059744 x^{14} +
28724255612327969672932 x^{15} +
1212265596856254020971761 x^{16} +
28474590328963725085347639 x^{17} +
1135126364276555282932647235 x^{18} +
28182531221515571455363232707 x^{19} +
1073280239866322664805781387076 x^{20} +
27871117924536494891873491537509 x^{21} +
1022905044375877221224023143033552 x^{22} +
27551901782293541801815673718843291 x^{23} +
981327300899811413079069217249330809 x^{24} + $ \\$
27230675474213911617346695093962018106 x^{25} +$ \\$
946586555765313010849925032923100801942 x^{26} +$ \\$
26910336272521921019466721739562952111225 x^{27} +$ \\$
917202654223319757045260969142997398374748 x^{28} +$ \\$
26592323189960514104579704865917109070263368 x^{29} +$ \\$
892038379019266817578176241221901860341244424 x^{30}
+ \ldots $ \bc $ \mbox{--------------------------------------    } $ \\ $ \mbox{--------------------------------------    } $ \ec  }

    ${\cal F}^{TkC}_{10}(x) = $      \small \parbox[t]{5.4in}{$ \ds
    -[(x (-1 - 8952 x - 42378 x^2 + 31475188 x^3 + 177149080 x^4 -
        46320241878 x^5 - 254632814051 x^6 + 39369421482240 x^7 +
        194236107853509 x^8 - 22180531873155530 x^9 -
        92632975657548205 x^{10}+ 8946071465390269248 x^{11} +
        29944681484503722677 x^{12} - 2712694965729684190934 x^{13} -
        6821796752863618485930 x^{14} + 639653978831376989637792 x^{15} +
        1098557890660467400167110 x^{16} -
        120222998658273772802013624 x^{17} -
        117297263165834919423023856 x^{18} +
        18353971438425771153497758420 x^{19} +
        5480569615508704955485178358 x^{20} -
        2310389450263143868179223573482 x^{21} +
        720037287440820029338544329207 x^{22} +
        242768777209137801637455469908072 x^{23} -
        208091773566300940926007919152925 x^{24} -
        21515681369354051649190671292763682 x^{25} +
        28879011434160378642274840827307593 x^{26} + $ \\ $
        1622784448261892336647177402088793428 x^{27} - $ \\ $
        2875843141277540351611628395850915888 x^{28} - $ \\ $
        104987883416770137250477418790679986600 x^{29} + $ \\ $
        225980550550802797687300540387263029052 x^{30 }+ $ \\ $
        5867679370199052369983491882588793145120 x^{31} - $ \\ $
        14599996646710480565757825057778042587387 x^{32} - $ \\ $
        285128556065637548541671502943517240838794 x^{33} + $ \\ $
        793781917016602529567172474316700043554090 x^{34} +  $ } \\ \\
          \hspace*{1.4cm}   \parbox[t]{5.4in}{ $ \ds
        12118130518074842864122358729876085325584672 x^{35} - $ \\ $
        36879427815426645041299766842776121622792875 x^{36} - $ \\ $
        452936687853672291908021081597673012496915126 x^{37} + $ \\ $
        1480705528326975791403857177622971195770004289 x^{38} + $ \\ $
        14964863396962327304831269413978032648202146600 x^{39} - $ \\ $
        51826161957336276972803123957811864568926963548 x^{40}- $ \\ $
        439164309501526201562880635646576593825874852942 x^{41 }+ $ \\ $
        1592638515348519669981661209285476538792963465060 x^{42} + $ \\ $
        11498954785140204403899452134835894854789823900840 x^{43} - $ \\ $
        43228474620379937746588821331452978045039176774550 x^{44 }- $ \\ $
        269777276264675601050026431839280145559475906419312 x^{45} + $ \\ $
        1041677058072338574382683123301831738205738835060154 x^{46 }+ $ \\ $
        5693655985450927583185442533273252931322220662984048 x^{47 }- $ \\ $
        22384199956771306708465192753314385517794235475044367 x^{48} - $ \\ $
        108498545283176811247843668274268976099443582383667350 x^{49} + $ \\ $
        430618370684300799648919480736073986642571421581577929 x^{50}+ $ \\ $
        1873239479146494970397758759620142112780582265491838848 x^{51}
- 7441892421448850605079876428023690900471063177393961083 x^{52} -
        29394490911940650866419867188266739582087612324041776134 x^{53}
+ 115888208688017331593166415953400467282936637627755880160 x^{54} +
        420415587592329561326984783856708261929865470055511604760
x^{55} - 1630550035072904902617619337796910031284018663446932600757
x^{56} - 5494565584059558565016816299687535512242473393836180193522
x^{57} + 20778168986713167498426781885246255996366918952678261621702
x^{58} + 65764505301955891937515813021116716025296469172739956433700
x^{59} - 240312493222727168761837936266677725578346593664077382360907
x^{60} - 722229520474429220289373554936169707626795680681487648259530
x^{61} + 2527249685782626453695730272954910792628423242062889319350561
x^{62} + 7289031655445433053641082900491767646281054197621919014931648
x^{63} - 24206240144840519070837597481244379487712392633759138979219130
x^{64 }- 67691289234576618562816688727216914047277393660995099855335936 x^{65 }+
 211457673279464232410375830609286671709334580295640931633682496 x^{66} + 579036742437762688964673894703526574276360576906661045754817392 x^{67} - 1686771680652499696478401109660679955776261333898053201197809371 x^{68} -
        4565962225397745749970454125461167135916360589350941802160222830 x^{69} +
        12298712850206874531874456951950586711512901354709701214108142045 x^{70} +  $ } \\
          \hspace*{1.4cm}   \parbox[t]{5.4in}{ $ \ds
        33210226215513979791105204604794633797946147489369129940753834440 x^{71} -
        82031940561043368851947175232697529169184229461663993546891573861 x^{72}-
        222905604898387489765390901824896769623204838728843489882454759560 x^{73} +
        500830477896174951947773059100488738179889556169446445447199447975 x^{74 }+
        1381111317608022961850440963926691081492958559170139045112458146396 x^{75} -
        2800004223587533515838556031775472603158481135189650468451607668018 x^{76} -
        7901529536430293547912988242294534182967734087556513825320383936428 x^{77} +
        14337246832555124456595013638897485833275034855798802283201831383576 x^{78} +
        41750744882094507723412581090146169544231274020243400314547641991440 x^{79} -
        67232629309830396470621253477350010337534240025998906646493163109166 x^{80} -
        203785226306391692481086213920187314025855718850286809404929540547634 x^{81} +
        288626249091353469449423851713205684902162710931862488254625246422344 x^{82} +
        919003992128872190902035316976646879964758649586010255245254926072972 x^{83} -
        1133433037665476569734932354977170790684743072948336582827219338596935 x^{84} -
        3829833432196321348016803221689518150680112364877236344564683322039946 x^{85} +
        4066226786081155362564812413368682003581455806897070025587207428234959 x^{86} +
        14751811392408453106510889384718123633784806030417094788160460184076072 x^{87} -
        13299113722049662918826638843919440753090674659391087939466507725126442 x^{88} -
        52529605240637322469417245612414306994653942175309276265691164635886160 x^{89} +
        39526090953218293673350968549985756750752754764296418923442253317024854 x^{90} +
        172962610316216673197651634423999477732503903391428891616224755339672548 x^{91} -
        106205386031054657152485976349973720765815157584821327015084234709694297 x^{92} -
        526733630447860557777926136581438870458493746336658004405936136435493242 x^{93} +
        255814219663558696791363703851158734018203229179295314118211149122779085 x^{94}+
        1483975596115213214394280210764055795436814717110364244409990981346277344 x^{95} -
        544073674261365034080289702636621646736119293160040555073569754737327523 x^{96} -
        3868738307705636187167361732469758186684446719113983622311459049894704244 x^{97} +
        991211271522653992782257113349745783603575532041050333813784645943455940 x^{98} +
        9335370661187883061935691160738188579066094132013061648087449298291540800 x^{99} -
        1434297012747987928033581729541889610245109739011752903428613765245979826 x^{100} -
        20855843840090880478061008931244201384416734732692206934328553550825300816 x^{101} +
        1213300340928846505842933026992264025573889799398288326238141908889994065 x^{102} +
        43149212571318267150453473898868490433283941277250961429153450930697172496 x^{103} +
        1334027245942203357300530372005042977786207196680004047321016315599257310 x^{104} -
        82695398651372370173468960619129274535095844645071290709692658989818755274 x^{105} -
        9455310520979043785821277120959746516441765720500871093672929257831064213 x^{106} +
        146847393080770441013099675558618972268902703342684075933190161106257286632 x^{107} +
        28306974717935828803945921287585219024211838725982737858713387761153803674 x^{108} -
        241678726976247766028594502205438015669095927402718656884402732043335354950 x^{109} -
        64488139184894935955386979305460083915165260300135672796603187706651859311 x^{110} + $ } \\
          \hspace*{1.4cm}   \parbox[t]{5.4in}{ $ \ds
        368727792487968281249749488677214632478813033598043967374176250347215675536 x^{111} +
        124319301006831206210700452397046244764885723265517451178406902350415108907 x^{112}-
        521642703538483364417446807511970551707017112375898064867628362576369396480 x^{113} -
        210907828007954614433579324094825071330578443331987435594341677893594435025 x^{114} +
        684448776260128631270524375712536295518673419408507085340539359751006094068 x^{115} +
        320910900261745530325516230510450724955578649763835525632674053751393381727 x^{116} -
        833117056414725728723319165621895843369095263683315006573688428004087166164 x^{117} -
        442656390266491701076320399913841916900713150667436360959737023566219053639 x^{118}+
        940938058390104166517226369130559823817378764335861992129743937298197780080 x^{119} +
        557244293616358513102728736259221583246960279006971926422208694045186050326 x^{120} -
        986267777254046462970097255625124579030110678971298432750291223175880330326 x^{121} -
        643091191413677690215134152788194162412600756136701391241641554162037870016 x^{122}+
        959606149897810360494873777895294866162916488512315537870475550077186722356 x^{123 }+
        682545162344625194936011456703487299326117998371780168730578078429868206125 x^{124} -
        866839509097534177211393062951404751579626045905821362069200577495948467568 x^{125} -
        667799532013060356350495890606416477235141588991889797155512432299016731329 x^{126 }+
        727132974471135347186143093878624351879391935158236311339194459151451919104 x^{127} +
        603393143285107148028673964695250091114075989659587164253653919305515880214 x^{128} -
        566501515224111253735494439754677407543185838104014779035582106369388090750 x^{129} -
        504208571844969101268777986488301269025605098592406296359031015723002487917 x^{130} +
        410002174433224834671809481524294946610738759333132956996279240646545936360 x^{131} +
        390094958170213828757467147252777801745411949560118103592575978052103142752 x^{132}-
        275715551080904559219499518926300788603320190731832266533394918232625694760 x^{133} -
        279694812374619927946636453278801035989642111895399667877864888825071604380 x^{134} +
        172317448662162669944396877108731523834677715548680718165154048197243029704 x^{135} +
        185988361825141983169718749282931180674422200026885116065086445700233986837 x^{136} -
        100117318734193425709090209602541415447159067117358093768613350084261731268 x^{137}-
        114776125436090801920841899814718853301243560082130240764478144164778645093 x^{138} +
        54093276488980390744410286845768910829725173160744232463261503386483402140 x^{139} +
        65767706114668361207668652371189893258533421841767401848481802611955803199 x^{140} -
        27189749735372928492781621384374894287521902249055076783744952513077584768 x^{141} -
        35007477758189762841433754637398606911014981905290896093303632025819436792 x^{142 }+
        12720680069239320123446214150984317210664093363923501304136363272904724736 x^{143} +
        17316223845017070495954853620675307582788998047335780399689345755293815765 x^{144} -
        5542773319523083966304772320932624387452760195735310395496882709212207522 x^{145} -
        7961886254809245744643494578021560689752817777941989210401782354973906663 x^{146 }+
        $} \\
          \hspace*{1.4cm}   \parbox[t]{5.4in}{ $ \ds
        2251072230847618155620955691366879542150238811235312780086984924730991592 x^{147} +
        3403661618890617816339700308645928467052653163141115929107971700467010246 x^{148} -
        852919920403850426324646716175449080497205722799717298256016644984820800 x^{149} -
        1353050279717778359545200555051987081527298996154879460730432250097712308 x^{150}+
        301842689006519610607813635022674827521464319608237764789269346264895336 x^{151 }+
        500219917419782838640846038895307385303332868888483699141744835316034342 x^{152 }-
        99906198445966833483245016167389380337694252215491855743034051930176576 x^{153} -
        171989770701320808082172539304669198482029234545913183932641000140636095 x^{154 }+
        30974895412117380325995486708194872967629037082996904914199628130737964 x^{155 }+
        54995379849944774017293303715607862334803936118650071928518188111155259 x^{156} -
        9010634131542275402450700619744620539337498555557063482150270763564926 x^{157} -
        16352737413700192950781036887279784192481538364595067581604901882617246 x^{158} +
        2463559789918447044517484129850286111345470605897347821238160435773280 x^{159} +
        4520893340084739550389627212686032066853133559457918816268118158523970 x^{160} -
        634035611967485177553290351946922949739992796033982995919167489207288 x^{161} -
        1161783417903185938488501935237506297042630739326392662524235958210080 x^{162} +
        153798915896146535532813114446585536476156912584130874785168604551136 x^{163} +
        277432808383753104515676052084219975789698311454202735717689470819925 x^{164 }-
        35189310701487586454152002123477762747152983404703963164711026664032 x^{165} -
        61539226316370228865217386256654557399726219583982959399175316313144 x^{166 }+
        7595032591089993535476403401961345503599760907413023564542528739352 x^{167 }+
        12673626497576182688767071976293841530896127255386691016518388658953 x^{168 }-
        1545373323775566864583067102002111183627874473252100824378807301560 x^{169} -
        2421913762234874881933568994515621986373595417275391608623206387192 x^{170 }+
        296028372843104808621710506052879031672729737645760318912004714032 x^{171} + $\\$
        429179031225414973144903873905358440133573222720403772689828227190 x^{172} - $\\$
        53281474151885337856665765189065226813718010621934432624595256562 x^{173} - $\\$
        70470931019290730333755485948443991089960557119702138294900556225 x^{174 }+ $\\$
        8989145815341109408358577602359586241787992641103218698352186064 x^{175 }+ $\\$
        10712645023550199570854579518617646000177598421320850094879013259 x^{176 }- $\\$
        1417775215446936243969349912868125738059570870367978047388676474 x^{177} - $\\$
        1506166545918499655016712660210174863513040330582682891803581284 x^{178 }+ $ \\$
        208475281854477165865325746770789810054304312911520738884646800 x^{179 }+ $\\$
        195641380302114073187188503545994138270959517189400168862474270 x^{180 }- $\\$
        28502954897099005918331689264152178134222731544820497868773458 x^{181}
         -  $\\$ 23449044593504153864898761115701649198025497682166576103704949 x^{182} +$\\$
        3614056628855168174208032064622544872174512726345649153798448 x^{183} + $\\$
        2589822588158364457205359101426026407856291098762422063677800 x^{184 }        - $\\$
         423953098395175149079461676996240599159532882583507907066054 x^{185}         - $\\$
          263166850044641059037462840284581736188195883701552611984460 x^{186}          + $\\$
          45906458459855226219020000865524027211324904309614171436508 x^{187 }+ $\\$
           24562273422124707919827941510450610398691144580148824390122 x^{188} - $\\$
           4578651803345621485133382265260690253851744846373279164250 x^{189} -$} \\
          \hspace*{1.4cm}   \parbox[t]{5.4in}{ $ \ds
           2101625129904254341024603173516493937046030361682131409109 x^{190} +  $\\$
           419788996233280282157770347238879257863886376630419600768 x^{191} +  $\\$
           164498804928973113411064578528756318000940373208518841646 x^{192} - $\\$
            35310910953256106521470937232599346812734566520008211452 x^{193} -  $\\$
            11750061175895690912029407363823715611106727360972937500 x^{194} +  $\\$
            2719853689247157217669336718671829801326489415268772192 x^{195} + $\\$
            763807575693260799317608586496353978494641616116679590 x^{196} - $\\$
        191479571686660752573177861015186151314206202946582926 x^{197 }-  $\\$
        45039464604190943698865440762248189486266922718250727 x^{198 }+ $\\$
        12297531288346925235787646930518976049476238734403600 x^{199} + $\\$
        2399932901299950589726776007311400735598702998016871 x^{200} - $\\$
        719115251322766336087091321690964490900126050841386 x^{201} - $\\$
        115013983688669250108129506338445939107643281260662 x^{202} + $\\$
        38213077573414683077938857874495879303136560600524 x^{203} + $\\$
        4927397712692699305993551794699997142174151087570 x^{204 }- $\\$
        1841545643589425961241725274819720526343386512686 x^{205} - $\\$
        187165671179348586594378149300351078239519685043 x^{206 }+ $\\$
        80316727636009557650737993403738857635941406480 x^{207} + $\\$
        6227615802621221186853214859882676456662226314 x^{208} - $\\$
        3163352335124312500203202086333909889504834620 x^{209} - $\\$
        177942582641173819758830840822689625379031333 x^{210} + $\\$
        112263421628271321654952820161512034735261148 x^{211} + $\\$
        4201989658611040094445182362718528943381591 x^{212} - $\\$
        3581561082891229212386143682420595822642652 x^{213} - $\\$
        74430429951173160663101115030202851929985 x^{214} + $\\$
        102471550975703351910872495709334657548288 x^{215 }+ $\\$
        621159867915845012227733281979348161099 x^{216} - $\\$
        2622636973316311289644553360602151467440 x^{217} + $\\$
        18179428251707696948079451226190234285 x^{218} + $\\$
        59886517413131220385405540433942693300 x^{219} - $\\$
        1099582594917931303174334755468354211 x^{220} - $\\$
        1216634734454528983623551715546070050 x^{221} +
        34491510314148180206441576421052466 x^{222} +
        21924472849356544849249853345100832 x^{223} -
        810843732524026172630954654308350 x^{224} -
        349316612424664860864065467488778 x^{225} +
        15477491445524123843928807928006 x^{226 }+
        4902979792801291056162495684924 x^{227} -
        246760667143790967087421413260 x^{228 }-
        60376891410047044159702489200 x^{229} +
        3321323416436357122101802029 x^{230} +
        649196915040321174445087384 x^{231} -
        37831695763386927111947608 x^{232 }-
        6060180926241810390399624 x^{233} +
        363675992014020763793570 x^{234} +
        48764483872934760132984 x^{235 }-
        2929897130007898165419 x^{236} - 335155514039363727612 x^{237} +
        19557375076055281692 x^{238} + 1943516100568863840 x^{239} -
        106343697447226121 x^{240} - 9348197170306412 x^{241} +
        459401826347208 x^{242} + 36386193310788 x^{243} -
        1517907694695 x^{244} - 110366736666 x^{245} + 3605313764 x^{246 }+
        245188672 x^{247} - 5484474 x^{248} - 355250 x^{249} + 4016 x^{250} +
        252 x^{251})]/
     $} \\
          \hspace*{1.4cm}   \parbox[t]{5.4in}{ $ \ds
          ((1 + x) (-1 + 35 x - 473 x^2 + 3042 x^3 - 8357 x^4 - 3077 x^5 +
      69776 x^6 - 115677 x^7 - 82259 x^8 + 385137 x^9 - 228633 x^{10} -
      268530 x^{11} + 373867 x^{12} - 98551 x^{13} - 64909 x^{14} +
      49836 x^{15} - 12711 x^{16} + 1135 x^{17} + 73 x^{18} - 20 x^{19} +
      x^{20}) (-1 + 15 x + 195 x^2 - 2476 x^3 - 9408 x^4 + 128774 x^5 +
      151702 x^6 - 3080005 x^7 - 152040 x^8 + 39805335 x^9 -
      22147982 x^{10} - 300921194 x^{11} + 284159318 x^{12} +
      1383107908 x^{13} - 1722027429 x^{14} - 3930756397 x^{15} +
      6069754917 x^{16} + 6915053418 x^{17} - 13500672554 x^{18} -
      7214631815 x^{19} + 19878551923 x^{20} + 3475881699 x^{21} -
      19978574007 x^{22} + 1204754727 x^{23} + 13960633114 x^{24} -
      3214431392 x^{25} - 6832632284 x^{26} + 2528557309 x^{27} +
      2326963032 x^{28 }- 1184186750 x^{29 }- 534678044 x^{30} +
      369015343 x^{31 }+ 75261636 x^{32 }- 78835592 x^{33} - 4000296 x^{34} +
      11506048 x^{35 }- 640379 x^{36} - 1110448 x^{37} + 155006 x^{38} +
      65338 x^{39} - 14541 x^{40} - 1860 x^{41} + 680 x^{42} - 2 x^{43} -
      13 x^{44} + x^{45}) (1 + 15 x - 195 x^2 - 2476 x^3 + 9408 x^4 +
      128774 x^5 - 151702 x^6 - 3080005 x^7 + 152040 x^8 +
      39805335 x^9 + 22147982 x^{10} - 300921194 x^{11} -
      284159318 x^{12} + 1383107908 x^{13} + 1722027429 x^{14} -
      3930756397 x^{15} - 6069754917 x^{16} + 6915053418 x^{17 }+
      13500672554 x^{18} - 7214631815 x^{19} - 19878551923 x^{20} +
      3475881699 x^{21} + 19978574007 x^{22} + 1204754727 x^{23} -
      13960633114 x^{24} - 3214431392 x^{25} + 6832632284 x^{26} +
      2528557309 x^{27} - 2326963032 x^{28} - 1184186750 x^{29 }+
      534678044 x^{30} + 369015343 x^{31} - 75261636 x^{32} -
      78835592 x^{33} + 4000296 x^{34} + 11506048 x^{35} + 640379 x^{36} -
      1110448 x^{37} - 155006 x^{38} + 65338 x^{39} + 14541 x^{40} -
      1860 x^{41} - 680 x^{42} - 2 x^{43 }+ 13 x^{44} + x^{45}) (-1 - 75 x -
      2106 x^2 - 25677 x^3 - 59530 x^4 + 1782084 x^5 + 17272711 x^6 +
      16915722 x^7 - 549259769 x^8 - 2774119728 x^9 +
      1945931095 x^{10} + 53672650626 x^{11} + 125905839478 x^{12} -
      255069305286 x^{13} - 1662208525171 x^{14} - 1706160646334 x^{15} +
      5962148469123 x^{16} + 17962282960815 x^{17} + 7082138902898 x^{18} -
      41641400945047 x^{19} - 71155511878215 x^{20} -
      7096003752745 x^{21} + 98767225236595 x^{22} +
      104041755500065 x^{23} - 4751714309754 x^{24} -
      84592421865025 x^{25} - 57545771312940 x^{26} +
      4180055107783 x^{27} + 24201660957901 x^{28 }+
      11084074318344 x^{29} - 364628264623 x^{30} - 1894529779890 x^{31} -
      523784901608 x^{32} + 41364917589 x^{33} + 43581474910 x^{34} +
      4664480242 x^{35} - 1210342521 x^{36} - 282744551 x^{37} +
      5619268 x^{38} + 6144308 x^{39} + 323744 x^{40} - 54975 x^{41} -
      5843 x^{42} + 82 x^{43} + 30 x^{44} + x^{45}) (-1 + 42 x + 2133 x^{2} -
      88106 x^{3} - 1159305 x^{4} + 59614372 x^5 + 152015398 x^6 -
      18819762746 x^7 + 40034190687 x^8 + 3264921562892 x^9 -
      15997694438884 x^{10} - 336798102365195 x^{11} +
      2393078617098732 x^{12} + 21262382124790258 x^{13} -
      203991793583280580 x^{14} - 802344394384627520 x^{15} +
      11100949049156620432 x^{16} + 15124565789332508780 x^{17} -
      408720018862600116122 x^{18} + 70049454004597571197 x^{19} +
      10586093658527044258990 x^{20} - 12631025888273228837204 x^{21} -
      198508817597338205942399 x^{22} + 387402987345930766315590 x^{23} +
             2753586280130848546932800 x^{24} -
      7116726860274335286714358 x^{25} -
      28694708927568151334441162 x^{26} +
      91397507315681115211888499 x^{27} +
      226797214695129167282089048 x^{28} -
      870318027304318790444523656 x^{29} -
      1363049305769562405267119112 x^{30} +
      6338665081932290022053776544 x^{31} +
      6178413328764184607346600034 x^{32} -
      35988003082214887026742513610 x^{33} -
      20492020074586065235166736101 x^{34} + $} \\
           \hspace*{1.4cm}  \parbox[t]{5.4in}{ $ \ds
      161266450872180007915993984787 x^{35} +
      45099606966665951271693058428 x^{36} -
      575001009719437474541893835090 x^{37} -
      37571777602636385348482415900 x^{38} +
      1639314263514932955906264058314 x^{39} -
      161019967972020863550757762131 x^{40} -
      3745699796292066069224199321372 x^{41} +
      851274013919532238198719327872 x^{42} +
      6859539342150922395619895904651 x^{43} -
      2285177200474388410515202543750 x^{44} -
      10046278725288151331993329222702 x^{45} +
      4263336781464515502088030143628 x^{46} +
      11714187061919722875421984542193 x^{47} -
      5951624800503928235794807036596 x^{48} -
      10796511948612809185000930216914 x^{49 }+
      6382964048451425290619045645907 x^{50} +
      7780699596208522812005194118918 x^{51}-
      5309228126699605717428066646335 x^{52} -
      4313080247506333381456597185155 x^{53 }+
      3429850650853924016071831349581 x^{54} +
      1790273570607001253851113421413 x^{55} -
      1715098075696198937373310320219 x^{56} -
      528375021398399743450311766854 x^{57} +
      659377687421116918032870695489 x^{58} +
      96604393710766619419597252621 x^{59} -
      193029880068729539119900529947 x^{60} -
      4069176647541912150779380527 x^{61} +
      42490073730573974965749981782 x^{62} -
      3413445482746271605206739212 x^{63} -
      6917976301240196729904525844 x^{64} +
      1165798086293282165887806486 x^{65} +
      814163354061253637176607304 x^{66} -
      208629109950888272910391562 x^{67} -
      66693134076508177810310164 x^{68} +
      24445572577097454676637446 x^{69 }+
      3497333010513277273515012 x^{70} -
      1989966629028961602789938 x^{71} - 83751655748872954401344 x^{72} +
      115359756476647736425785 x^{73} - 2722708701437214377921 x^{74} -
      4826028852207290118097 x^{75} + 345815430004445639338 x^{76} +
      146771902754630028398 x^{77} - 16434126042052261132 x^{78} -
      3249431151334904798 x^{79} + 493956208888831088 x^{80} +
      51941522002476761 x^{81} - 10408635523210110 x^{82} -
      583754494635043 x^{83} + 159473348039387 x^{84 }+
      4295711619291 x^{85} - 1796857404178 x^{86} - 16067848999 x^{87} +
      14814596478 x^{88} - 26776331 x^{89} - 87402987 x^{90} +
      632038 x^{91} + 351431 x^{92} - 2932 x^{93} - 868 x^{94} + 5 x^{95} +
      x^{96})]
          $} \\ \\
          \tiny
      \hspace*{1.2cm} $+$  \parbox[t]{5.4in}{ $ \ds
      2 x^2 [6765 - 39834856 x^2 + 104523970813 x^4 -
     164982564885013 x^6 + 177906536920343761 x^8 -
     140842274713114435797 x^{10} + 85779717815477144689740 x^{12} -
     41544013632033624057516594 x^{14} +
     16399568914591554281955684297 x^{16} -
     5378911095231315451404140851742 x^{18} +
     1488575890268219728101042520480867 x^{20} -
     351994030199792535376574911451604266 x^{22} +
     71871884346922333500853886232541044867 x^{24} - $\\$
     12785595487435798515128463356678397400349 x^{26} + $\\$
     1996895791506528512332929205178130059708642 x^{28} - $\\$
     275651883864977780695356681297317352816229428 x^{30} + $\\$
     33828447728979113683629342660329589506205101207 x^{32} - $\\$
     3709977469112994599708981916557512065408235588649 x^{34} + $\\$
     365289948325972345823076805077470890954161877009713 x^{36} - $\\$
     32425404856480749966662530290729058489056585516438855 x^{38} + $\\$
     2604616139247698717871459645328404217259966135196184926 x^{40} - $\\$
     189973147774713810331536112846342972934478308758548175688 x^{42} + $\\$
     12620622132083263498546715168396675009061385173421055097923 x^{44 }-  $\\$
      765854482314104637205303235329436798360280180658644527598753 x^{46} +
      $} \\
      \hspace*{1.4cm}  \parbox[t]{5.4in}{ $ \ds
     42562436526820650860707602166648255414139798711228171581020655 x^{48} -  $\\$
     2171559463258983354104537024199792919521656639139402410358295711 x^{50} + $\\$
      101942345979906885079316490841316223826221488541870800772164392320 x^{52} - $\\$
     4412436848474684114470265849454202230730045242271797018981916286878 x^{54} + $\\$
     176434918753367226089132827733647336462524534897411315478471499006440 x^{56} - $\\$
     6529169327447092979111289719501464302923305251758180210919854340279423 x^{58} + $\\$
     223992363959888131556614785346234342043125632543832498958707165387902106 x^{60} - $\\$
     7135081291594307765035750042713355457023687312428278457158772113249748020 x^{62} + $\\$
     211349719870075782092314373592219762450211190842573213520200973538098023705 x^{64} - $\\$
     5829747217962700797775831879989006647967406571886676045401123657389560305135 x^{66 }+ $\\$
     149939528071418231793537280414800452106055638810520348651804752031329332298958 x^{68} - $\\$
     3600324290507618056748431686649358313129920169288750702370984438366493499667957 x^{70} + $\\$
     80804769009989367675465118287335591309170582951248358572296021608179934575771534 x^{72} - $\\$
     1697014116468396675601177265098611425265363921263842227434660212437937449880358576 x^{74} + $\\$
     33384477888870027082009969624806188063272868918556863271783628718539911554697003113 x^{76} - $\\$
     615812039663158067427035226817313403880194617065071588614438087436284764042158389250 x^{78} + $\\$
     10661242870925396694061183900824606932543715327939212858870246765373478111799908168003 x^{80} - $\\$
     173386197018949019297092026313015971585981611229799779936169713909400502257051583764237 x^{82} + $\\$
     2651183903888434806549726241332429567358981401636316760890289939008725629273625115133647 x^{84} - $\\$
     38144978497487720397254343445160000582330183406909700905242551632552633635018282675841231 x^{86} + $\\$
     516825168806365635236885234512762320133853856175285856068018482406907535491112616999187870 x^{88} - $\\$
     6599018366230807657388479765432604216701703096467182894380148680198084869224733901075500562 x^{90} + $\\$
     79460276562185351672912727675623045979195727970212924090131051590777615593147866215541549218 x^{92} - $\\$
     902916091161565888595907394540095514693504869745118622625216699664749593627800071984663290108 x^{94} + $\\$
     9688343088637406683887382241265119966498757539438741549889158203583075785874532151676508321479 x^{96} - $\\$
     98225003054859128414334574438272314112297530773495580047065843635983090790831801167138118074467 x^{98} + $\\$
     941496566220322666306287949016344276328171777277670713891809552257172366149255486216992966620125 x^{100} - $\\$
     8536542927135374123407067170728273502006163975954671969287632295611146443494549207711986540525045 x^{102} + $\\$
     73256079917929564996704779710324870874854319954792365339346245961645736596295777565247129573791711 x^{104} - $\\$
     595285478995291666322053496266449314045648353134044515879580823112219003478102517248622911662832097 x^{106} + $\\$
     4582883644240591161777089176136248499570890242792379660286318765482940051059604076091687341143772740 x^{108} - $\\$
     33441602567642202244355560644472677788484949162769661997932200886164195606154044498890819435505752185 x^{110} + $\\$
     231400382631743375766696661600795477433235330559447587717568528089234226773065974868933214244255050916 x^{112}-
     1518992038583742864039658113161658332623109534341523122909600583133739514950415788045222329965776726659 x^{114} +
     9463223247441931902532986887609413925694579275105045432036181258763123616542853182094092039180879053581 x^{116}-
     55973821482794331259883376616070471193416033900268040332569120650443433662574755980889488059263558187268 x^{118} +
     314453157811414693920007745084380702087401657720856017365629245069339882476297299662747866770671794224797 x^{120} -
     1678448307652673849793331939527261819932248161091896135999040312580069139361464104167144554400344266839249 x^{122 }+
     8515108649882248872111814284976016327901147570888696316872435779248218875490818097198190552475102180025454 x^{124} -
     41071960717276887795334491030969394937580963006618106935345341725729654686955631753441481693514181395070664 x^{126} +
     188412782082389520732874917882960552373974361973170290296873816530933181606301189717773937829122797016233108 x^{128} -
     822272454585975349969746016449904221502548265718277803172743129475326093156584112843517763911526166417797724 x^{130} +
     3414968313014643877335007207050620159463541240598614585348591222334176383503353608505782921317388964417082507 x^{132} -
     13500281593416878060305643148784521528442654993427569567888296106478582486907432475519025779288027033491821789 x^{134} +
     50815780997254027296137678348743532630965897929853982968995776832135709170741427875990365989513320524941167684 x^{136} -
     182164347596816814304391298592125671354168351879811153110557955945025717468283201235178758896750459780421229641 x^{138} +
     622073464989177921228086947462802118496707770257052628220284743209539730931173951339406505669149304521325138713 x^{140} -
     2024106328011596494279156225552824129561908508941369428683038347036380771601920506313884649842265203364982554252 x^{142} +
     6276732902650129353038044542945420048212902036739145853062014441326535290808139262330083419851433593286216395850 x^{144} -
     18553832605683807234817562435710455253484582526173944405946245784641542771517240026582498114546002446763223570761 x^{146} +
     52290171794313458418869521770385478206903167393994977153126111409767143149636481775310335186800678447663041516329 x^{148} -
     140532152260698769421350047506094592683847247044618255643483106882933735559771459894758139330965189111541827703026 x^{150}+
     360229574506981271617187441598865650884711596063749255728658253988338426592425828961825531623947463913627869489874 x^{152} -
     880858503312266972670216453511522703915427740537726617425539712221805211402842562204364865634722790554151842502628 x^{154} +
     2055074175576266153839688324507158507346811725352374156753897103979916436529465176993852767336627010104562841612929 x^{156} -
     4575213402712356486839626089043573477236709534461955444327646975149337719143969812366535596723395745300419903783775 x^{158} +
     9721253410368269713767144633057250440717503019110853261195827563407340067072422703237630605343963934260579554943153 x^{160} -
      $} \\
          \hspace*{1.4cm}   \parbox[t]{5.4in}{ $ \ds
     19716122507503071136806140162764556379469456468957536640049603940390210828371922510170361048245904340312459644847921 x^{162} +
     38173959442683616295928845624050261396522203135115836317559525129071230213978478929427721983208383775203194430687642 x^{164}-
     70569076127982414011265213591792482411166452585150068421529723515689200839641735596460575207672920790892872267662896 x^{166} +
     124570525967518606095406279998180772247621264525251121537147006233530112490856226925263416723741694488644486562676094 x^{168} -
     209999892653006984283471339286231825104552108644610021534509774903481134138455699650152353536496111811510635457372081 x^{170} +
     338121092208621488499473755449377319056767761572299884729740497180128022157968497798061217607013460928002147077000959 x^{172} -
     520018397132416166762370937978105250409969581457545867167824375194407994526843961993184217810750230081672292221420948 x^{174} +
     764011091265892684185648587978169358426514275061927804022440298689136165300452057356436806476470008877280427662982340 x^{176} -
     1072392778242539734760062680817567305522418623475769587971482679101517700563688592439228956101382873414623734532120692 x^{178} +
     1438194707950002790379402404766343989600934113953467552310733978553816910008881168585894166515747003773182047200058305 x^{180} -
     1842998798678072598994854242722581002033917151863434894535937805897242924837994400356151705758964674804944495192908474 x^{182} +
     2256875816120841538791970225338037137729003787012682743767588729299178169444552045835375819715004748614941614967808350 x^{184} -
     2641161629992199440141827788079575705295126773150172192995479336534446523321696247048621272579822073676085651147397436 x^{186} +
     2954025887886099992135923432244633404075419594613830084833996683764618531704656030241229605643331443286503487028240578 x^{188} -
     3157844924376201788691666361138675753745856220596842600469926616876025014342497504723857635358291763938555297095691804 x^{190} +
         3226618389842777376076326359538597077103870780681386873242490559995682584707237636410940647794871675768199885128556805 x^{192} -
     3151416574579330746822691645375702073478888078243345298713511600228024066792719177402146878394415497605413017052203059 x^{194} +
     2942283720512853982148497568278536935126357362475544542467215832316952359924253544610103973034793904416410821739773829 x^{196} -
     2626037840253520740613681740857999456465707505018173215643929626891252126758440694502977618394564013789145207776281617 x^{198} +
     2240631116933171713749386090429655847642530932341880436677761599386750154886240791853826785938512319526094551003087913 x^{200} -
     1827704830075698829836360518305931814024871997421250733852278973333479968739298377911508062153719716317585029561795563 x^{202 }+
     1425343552193193623035168957734484842261375296945222026413503259871910088596216857973632167960737635300861054635754926 x^{204 }-
     1062723966355988575636985691689359823567489014350885562823058803159611805563760711779796754029083733601196041167338019 x^{206} +
     757559220178218967502293930391590122947304344246703767731667176430668631239776959519996149635609632298189009555487868 x^{208} -
     516313899258746262246819241024954800021751388058408991997204392760431924349088463539988672690373155600276439308231981 x^{210} +
     336446665069002882055840075396669594325892355625291932498181670740123021637542786865028510280560111519471542555214670 x^{212} -
     209616773543972253921115365368793004779409657484322573894184789302459320691443021957546528440686341967837991089059434 x^{214} +
     124865818287375405331822931396408278545067757028188539991869701744246587405082944335030324218134185310246582824946775 x^{216} -
     71115994175248601380637577626251686905888629332871618987283678969892430754493351883692635735019514013654005856611852 x^{218} +
     38725205204177819526760026499475643065663750797140761765471631802499254043323854503769754528358944161690649231599877 x^{220} -
     20161252056594532974486848207571765584172073056723509885250269295098569445392044967498719971750243096072617356384646 x^{222} +
     10035320293592922191916633962794017252432898069276532762727869093007460007208813335050239293448051377401405619516957 x^{224} -
     4775574009853012299174361201820691694081543522977312389258700900045202438832203367674206010420311391142538044312008 x^{226} +
     2172646729153017164061405926555719590595417044108770316407668511566251968257303881517976078476473691082441409443025 x^{228} -
     944948957765280281359123963342482224195822510740554042405671901213900545831627091194105240479638263680725294292675 x^{230} +
     392886991631429048654020793019103050318222150500348579878462755348214900273580815977628573947764994788688380587213 x^{232} -
     156152589722653813209986277203288795860738907054376037020464243390062928335831823925184975929692599412437185248759 x^{234} +
     59324338633271087550562483139155085171932503850265770943706266741764246399250169629295433733266855792963970440444 x^{236 }-
     21542515526700778449166012522286681112154273975100806936881010234380831047121754076983888381279800855093147327887 x^{238 }+
     7476786335622251783133011546708378861795409450303763773874115897029411780740113401952928442087718755087118387205 x^{240} -
     2480055395921370071444754405498796956810870022486040374008459659411642922614690294697731006277596354122366741001 x^{242} +
     786150957333387059006346571504485426568530039956340468424836297040960154545188640552141879993070945276840684060 x^{244} -
     238131242986569793162937840738022502707332725015043426655480540718622621941324569995018539674019654097607292145 x^{246} +
     68921975989235177211793485339940745461526225908025270230935161183724209938067206843791988938787995528348593003 x^{248} -
     19058655862618388547257726172977901298558127951487471879006285075436942245500692561662731839465859293530552924 x^{250} +
     5034768268008391294920381428199468823315225681260038757322410591599569879137445224818328488572074139506292784 x^{252} -
     1270503647160352890984144746166801874300168285157191546935819991963513650982945227900094273179484991290677173 x^{254} +
     306221080635222405942379792566287064969790054105115530708825128560477165826531618728184696004451285621100132 x^{256 }-
     70486737622545656838534903405127621219054540237661763586907201497238070102522816122996053248987997433285298 x^{258} +
     15493163446200155366520494530910239250991859764309599184798281795957124677109473815817641699420168853158231 x^{260} -
     3251449852427068887449989208481123378772059610699735223464177110164734700672963922498589088287393442960395 x^{262} +
     651416515652375341905987600716440260756428127441759451855373397246174719471849622058994317212571643675574 x^{264} -
     124572497650394275879993057170721842732976050046617470492536497564976491721587387418385052586093036127402 x^{266} +
     22735294971307061084815792425115655621753255279530941991378875571585480910820699966251124351638695540145 x^{268} -
     3959347134792331189469791279308635433539267186210832329997677813270898968650117310589866661892707409037 x^{270 }+
     657834731223350160364504356205518561755438666803801049763392361877218057836482345132212778401226157934 x^{272} -
     104256202236856291982749429413388959380939378709310804213582154424646932153868847220081493208525622064 x^{274} +
     15757829101245602509367987965244076174437748616828592171513766172517931193534662666187888470527993476 x^{276} -
     2270980798771618363760238232331263762924417026562403577211630132122295728274277670952541869941070932 x^{278} +
     312005243964255094732818745575524674157153257960297349290288807441519916470066043294320790285146412 x^{280} - $} \\
          \hspace*{1.4cm}   \parbox[t]{5.4in}{ $ \ds
     40855059029456018506070064588773239836033869895247286996203430918884792720358908941215306182956950 x^{282} +
     5097587660212825828186278934241247951149192829131186807278252950627516751609246458843341171380051 x^{284} -
     605916698492688423516250074356921238447291031375241722760735000953209187526369433891633630698717 x^{286} +
     68593001519864252715143806882743081184214773870900285990131808913036872571715247170435811790897 x^{288} -
     7393487261938723490408212827098321911646892220875875983549435003736922817693456306410903751842 x^{290} +
     758577392637308279864656300925684169970009524167747854487321793379253658424513093051187532607 x^{292} -
     74063481211411381595258737067389675181612116742791286846544146014360353988209237557313836184 x^{294} +
     6879051658816123687964438262992032016462142583088699697293157145767284014088008383683746110 x^{296} -
     607621989067552613189187400502978206080357398592778597101474912142499738061737117094922662 x^{298} +
     51023760738960009046449074432933863376830695488096556782318051464693783220084267951361998 x^{300} -
     4071853784698907322107489775043671700538583273543051469295377104141178699918886028557681 x^{302} +  $ \\ $
     308696671077895078948662194692721599581228275957699781514986780099735792094251892569792 x^{304 }-  $ \\ $
     22224060119629800720382714410540085105739259210675739194792771517849472353300981008299 x^{306 }+  $ \\ $
     1518763039101297355281949385146466740413904636902417316152144473139445669339079639524 x^{308 }-  $ \\ $
     98479682000652173718282480786333373600818559190292403115976150501486909983690531600 x^{310 }+  $ \\ $
     6056182242768880142227176545145182948104360161192935167211533161891881599985498791 x^{312} -  $ \\ $
     353055381615689862835803046564767804217464904224281086192260791303132991929261155 x^{314} +  $ \\ $
     19501312781335256104190551650338035006542237486884869111925916390920464163393958 x^{316} -  $ \\ $
     1020084940767040984395483629753682196190130842817191232077466295696631279561621 x^{318} +  $ \\ $
     50503642528762218785109881040303222418788964310257412931692061555272461848701 x^{320} -  $ \\ $
     2365225060288826836548779510988584163965650608151074031898554225206131080716 x^{322} +  $ \\ $
     104717983364314648907125112708419308722918213301804787548386960765716849875 x^{324} -  $ \\ $
     4380156951060286466738102009687322229186055113137888618835811633827422913 x^{326} +  $ \\ $
     172974431296067509222464944769861082541054482930053827410132262202694870 x^{328} -  $ \\ $
     6444428609425051283998094077998459614535330570938082054554162170447605 x^{330} + $ \\ $
     226341489470455684436362799483823645100489588445790924444755297113982 x^{332} -  $ \\ $
     7488017168159965326577661582244638059388170667830537824771914274479 x^{334 }+  $ \\ $
     233138837364953325388203618892253025743719714533335376872548183958 x^{336 }-  $ \\ $
     6825019896846946665290193267408986035766773057937612889151423978 x^{338} +   $ \\ $ 187673290047558259475872109835593149931102737150948548595233191 x^{340} -  $ \\ $
     4842239080835135988238443002724457921485104372565097067049108 x^{342} +   $ \\ $ 117094295868181578096249292168771870964176344845568558922378 x^{344} -   $ \\ $ 2650524921394341631467629843136165831773636852869685217082 x^{346} +   $ \\ $ 56085557996620661358321898617450447104956696301408671810 x^{348} -   $ \\ $ 1107794791152113528420675194892462676356115970184789937 x^{350} +   $ \\ $ 20392285541529170864044049477629987416036226356654101 x^{352} -  $ \\ $
     349232911808186575834411835324605698313593875224192 x^{354} +  $ \\ $
     5553634127895030853414869279594506533848427062171 x^{356} -  $ \\ $
     81834596405538825949023871226872439638030117910 x^{358} +  $ \\ $
     1114756448312869080036910783669761075143270389 x^{360} -  $ \\ $
     14001571527005665534625142095721810681529099 x^{362} + $} \\
          \hspace*{1.4cm}   \parbox[t]{5.4in}{ $ \ds
     161682136395127319455376887239995074151781 x^{364} -
     1710841163357663171735009970872808242956 x^{366} +
     16527274432951996513030939134032579012 x^{368} -
     145139419130755014621647420370861118 x^{370} +
     1152990355728528489095169741790030 x^{372} -
     8238182888101087108933594665254 x^{374} +
     52585320030030488028030429967 x^{376} -
     297447715429520164887559530 x^{378} +
     1476373129427972834656035 x^{380} - 6352114911359318056518 x^{382} +
     23325518029435792577 x^{384} - 71626628809713792 x^{386 }+
     178858122357592 x^{388} - 348739187433 x^{390} + 497872624 x^{392} -
     462706 x^{394} + 210 x^{396}]/ $} \\ \\ \tiny
       \hspace*{1.4cm}   \parbox[t]{5.4in}{ $ \ds
      [(-1 + x) (1 + x) (1 - 2 x - x^2 + x^3) (-1 - 2 x + x^2 + x^3) (-1 +
     9 x - 27 x^2 + 28 x^3 - 9 x^5 + x^6) (1 - 8 x + 8 x^2 + 6 x^3 -
     6 x^4 - x^5 + x^6) (1 + 8 x + 8 x^2 - 6 x^3 - 6 x^4 + x^5 +
     x^6) (-1 - 9 x - 27 x^2 - 28 x^3 + 9 x^5 + x^6) (1 - 90 x +
     2133 x^2 + 19654 x^3 - 1407750 x^4 + 12020085 x^5 +
     211505759 x^6 - 4005889491 x^7 + 1140372558 x^8 +
     418020172543 x^9 - 2487724623906 x^{10}  - 15524197875990 x^{11} +
     207409855235855 x^{12}  - 168411001997175 x^{13}  -
     6989852039646672 x^{14}  + 29176347222301350 x^{15}  +
     85446200204468703 x^{16 } - 862651296424382517 x^{17 } +
     732228455321894508 x^{18 } + 11231294123644498821 x^{19}  -
     33803475249656479500 x^{20}  - 54921206512262415741 x^{21}  +
     413908203902580672042 x^{22 } - 246207028620567111678 x^{23 } -
     2478006706489040058135 x^{24 } + 4892059779725991163731 x^{25}  +
     6604769927708293281489 x^{26 } - 29615041041062939657101 x^{27}  +
     5692338640092901485984 x^{28}  + 95353352498492938925253 x^{29}  -
     102458115245643694623619 x^{30}  - $} \\ \\
       \hspace*{1.4cm}   \parbox[t]{5.4in}{ $    163321849837230400108482 x^{31}  +
     356486700262708176382602 x^{32}  + 75818151591482233031073 x^{33}  -
     667256343974500352911503 x^{34}  + 278039887830219509936100 x^{35}  +
     733454500237993350748973 x^{36}  - 685954133222650252610760 x^{37}  -
     427320993379027646418498 x^{38}  + 781904366023459571507369 x^{39}  +
     28153248400039056561681 x^{40}  - 531396998038886850489789 x^{41}  +
     154086139945066863345177 x^{42}  + 222672705102690029154612 x^{43}  -
     127277447297634414536733 x^{44}  - 53175680795563803057275 x^{45}  +
     53846433449462045531871 x^{46}  + 3955591274039798235417 x^{47}  -
     13978733150559004421783 x^{48}  + 1617114491250147902040 x^{49}  +
     2287879124568249125151 x^{50}  - 596315566352941519713 x^{51}  -
     225529529482003369278 x^{52}  + 97905185502923227626 x^{53}  +
     10648699061080306343 x^{54}  - 9487814076406871859 x^{55}  +
     206953141963128534 x^{56}  + 570333880393160836 x^{57}  -
     62350849780138683 x^{58}  - 20789244041051196 x^{59}  +
     4077900882295116 x^{60}  + 395015565737433 x^{61}  -
     144460876536018 x^{62}  - 252140561385 x^{63}  + 3043888623648 x^{64} -
     172278608445 x^{65}  - 36698018483 x^{66}  + 4110045927 x^{67}  +
     192261069 x^{68}  - 45656109 x^{69}  + 614187 x^{70}  + 243603 x^{71}  -
     12637 x^{72}  - 369 x^{73}  + 45 x^{74}  - x^{75} ) (1 + 90 x + 2133 x^2 -
     19654 x^3 - 1407750 x^4 - 12020085 x^5 + 211505759 x^6 +
     4005889491 x^7 + 1140372558 x^8 - 418020172543 x^9 -
     2487724623906 x^{10}  + 15524197875990 x^{11}  +
     207409855235855 x^{12}  + 168411001997175 x^{13}  -
     6989852039646672 x^{14}  - 29176347222301350 x^{15}  +
     85446200204468703 x^{16}  + 862651296424382517 x^{17}  +
     732228455321894508 x^{18}  - 11231294123644498821 x^{19}  -
     33803475249656479500 x^{20}  + 54921206512262415741 x^{21}  +
     413908203902580672042 x^{22}  + 246207028620567111678 x^{23 } -
     2478006706489040058135 x^{24}  - 4892059779725991163731 x^{25 } +
     6604769927708293281489 x^{26}  + 29615041041062939657101 x^{27}  +
     5692338640092901485984 x^{28}  - 95353352498492938925253 x^{29}  -
     102458115245643694623619 x^{30}  + 163321849837230400108482 x^{31}  +
     356486700262708176382602 x^{32 } - 75818151591482233031073 x^{33}  -
     667256343974500352911503 x^{34}  - 278039887830219509936100 x^{35}  +
     733454500237993350748973 x^{36 } +
      685954133222650252610760 x^{37}  -
     427320993379027646418498 x^{38 } - 781904366023459571507369 x^{39}  +
     28153248400039056561681 x^{40}  + 531396998038886850489789 x^{41}  +
     154086139945066863345177 x^{42}  - 222672705102690029154612 x^{43}  -
     127277447297634414536733 x^{44}  + 53175680795563803057275 x^{45}  +
     53846433449462045531871 x^{46}  - 3955591274039798235417 x^{47}  -
     13978733150559004421783 x^{48}  - 1617114491250147902040 x^{49}  +
     2287879124568249125151 x^{50}  + 596315566352941519713 x^{51}  -
     225529529482003369278 x^{52}  - 97905185502923227626 x^{53}  +
     10648699061080306343 x^{54}  + 9487814076406871859 x^{55}  +
     206953141963128534 x^{56}  - 570333880393160836 x^{57}  -
     62350849780138683 x^{58}  + 20789244041051196 x^{59}  +
     4077900882295116 x^{60}  - 395015565737433 x^{61}  -
     144460876536018 x^{62}  + 252140561385 x^{63}  + 3043888623648 x^{64}  +
     172278608445 x^{65}  -
           36698018483 x^{66} - 4110045927 x^{67} +
     192261069 x^{68}  + 45656109 x^{69}  + 614187 x^{70}  - 243603 x^{71}  -
     12637 x^{72}  + 369 x^{73}  + 45 x^{74}  + x^{75}) (-1 - 27 x + 1032 x^2 +
     25802 x^3 - 373186 x^4 - 9247650 x^5 + 68985132 x^6 +
     1756981328 x^7 - 7544737839 x^8 - 205344558763 x^9 +
     520362914282 x^{10}  + 16055543418770 x^{11}  - 22765506950292 x^{12}  -
     886185755306711 x^{13}  + 567341618795392 x^{14}  +
     35824949732566099 x^{15}  - 2114077619505885 x^{16}  -
     1089052544253505079 x^{17}  - 434310955869650180 x^{18}  +
     25382085183437405700 x^{19} + 19657336252849199189 x^{20}  -
     460196787380533985842 x^{21}  - 503282143127348655479 x^{22}  +
     6563545564790657680027 x^{23}  + 9019917524975714828718 x^{24}  -
     74279664653212096085631 x^{25}  - 121126151372353616341848 x^{26}  +
     671522546480445322781198 x^{27}  +
                    1259864230974286367427691 x^{28}  -
     4874434339938184506685397 x^{29}  -
     10355985403728393543409067 x^{30}  +
     28508897879645580528825667 x^{31}  +
     68215067001652545430286192 x^{32}  -
     134568911549356278360128102 x^{33}  -
     363861169332415595759624659 x^{34}  +
     512191062359197441924007885 x^{35 } +
     1584783594082588087382278463 x^{36}  -
     1564210162255894419536702349 x^{37 } -
     5674861959676935021395050413 x^{38 } +
     3785132858090620875642211416 x^{39}  +
     16803139396061165672290060717 x^{40}  -
     7044093577535657561000115327 x^{41}  -
     41343132893232215469798457354 x^{42}  +
     9285232522254031569698694689 x^{43}  +
     84883083360723557875730426434 x^{44}  -
     5937121438225419731565543654 x^{45}  -
     145956035701218221797085712592 x^{46}  -
     7890101748188581051360660346 x^{47}  +
     210852803270320434959494868010 x^{48}  +
     32505194104448575284667024966 x^{49}  -
     256625689573728687673431393828 x^{50}  -
     60386088834617104154813509703 x^{51}  +
     263791426465741835712351964159 x^{52}  +
     79446568823858899041889805170 x^{53}  -
     229529669618761966378616574822 x^{54}  -
     81368503811164989012654074179 x^{55}  +
     169410836644260710451175862716 x^{56}  +
     67317485231746697845168652537 x^{57}  -
     106273241553496096920653558337 x^{58}  -
     45836707489100134208832885680 x^{59}  +
     56768242863723954140114918264 x^{60}  +
     25971874970291859366352369232 x^{61}  -
     25868414266252042008113790681 x^{62}  -
     12334499250260624842980447944 x^{63}  +
     10073049237434908783966012615 x^{64}  +
     4934518938825367378469146563 x^{65}  -
     3357135730239884258076120873 x^{66}  -
     1668961310040579145324436854 x^{67}  +
     958979388811076795805220259 x^{68}  +
     478484381627959740231307898 x^{69}  -
     235070462661933368244141645 x^{70}  -
     116494763716325583154471593 x^{71}  +
     49490588354774441565598327 x^{72}  +
     24113150837187023305271901 x^{73}  -
     8953841615972305424578920 x^{74}  -
     4245305143743032443594882 x^{75}  +
     1392157069027476262098948 x^{76}  + 635580648350932217926483 x^{77}  -
     185931065427671751697600 x^{78}  - 80838922982850596922843 x^{79}  +
     21306408936395504377734 x^{80}  + 8719476782236791891707 x^{81}  -
     2090911132691824012433 x^{82}  - 795459886846428765825 x^{83}  +
     175226009815304083621 x^{84}  + 61148524359498140722 x^{85}  -
     12490776292567568868 x^{86}  -
     3941245311464514974 x^{87}  +
     753388475396092803 x^{88}  + 211610725442247167 x^{89}  - $} \\ \\
       \hspace*{1.4cm}   \parbox[t]{5.4in}{ $
     38184703986094031 x^{90}  - 9385073357203492 x^{91}  +
     1611821886653267 x^{92}  + 340074435384128 x^{93}  -
     56011183019777 x^{94}  - 9923940361375 x^{95}  + 1578324064638 x^{96}  +
     228740033731 x^{97 } - 35344978435 x^{98} - 4053271905 x^{99}  +
     611772610 x^{100}  + 53065630 x^{101}  - 7859404 x^{102}  -
     481716 x^{103}  + 70281 x^{104}  + 2698 x^{105}  - 389 x^{106}  - 7 x^{107}  +
     x^{108}) (-1 + 27 x + 1032 x^2 - 25802 x^3 - 373186 x^4 +
     9247650 x^5 + 68985132 x^6 - 1756981328 x^7 - 7544737839 x^8 +
     205344558763 x^9 + 520362914282 x^{10}  - 16055543418770 x^{11}  -
     22765506950292 x^{12}  + 886185755306711 x^{13}  +
     567341618795392 x^{14}  - 35824949732566099 x^{15}  -
     2114077619505885 x^{16}  + 1089052544253505079 x^{17}  -
     434310955869650180 x^{18}  - 25382085183437405700 x^{19}  +
     19657336252849199189 x^{20}  + 460196787380533985842 x^{21}  -
     503282143127348655479 x^{22}  - 6563545564790657680027 x^{23}  +
     9019917524975714828718 x^{24}  + 74279664653212096085631 x^{25}  -
     121126151372353616341848 x^{26}  - 671522546480445322781198 x^{27}  +
     1259864230974286367427691 x^{28}  +
     4874434339938184506685397 x^{29}  -
     10355985403728393543409067 x^{30}  -
     28508897879645580528825667 x^{31}  +
     68215067001652545430286192 x^{32}  +
     134568911549356278360128102 x^{33}  -
     363861169332415595759624659 x^{34}  -
     512191062359197441924007885 x^{35}  +
     1584783594082588087382278463 x^{36}  +
     1564210162255894419536702349 x^{37}  -
     5674861959676935021395050413 x^{38}  -
     3785132858090620875642211416 x^{39}  +
     16803139396061165672290060717 x^{40}  +
     7044093577535657561000115327 x^{41}  -
     41343132893232215469798457354 x^{42}  -
     9285232522254031569698694689 x^{43}  +
     84883083360723557875730426434 x^{44}  +
     5937121438225419731565543654 x^{45}  -
     145956035701218221797085712592 x^{46}  +
     7890101748188581051360660346 x^{47}  +
     210852803270320434959494868010 x^{48}  -
     32505194104448575284667024966 x^{49}  -
     256625689573728687673431393828 x^{50}  +
     60386088834617104154813509703 x^{51}  +
     263791426465741835712351964159 x^{52}  -
     79446568823858899041889805170 x^{53}  -
     229529669618761966378616574822 x^{54}  +
     81368503811164989012654074179 x^{55}  +
     169410836644260710451175862716 x^{56}  -
     67317485231746697845168652537 x^{57}  -
     106273241553496096920653558337 x^{58}  +
     45836707489100134208832885680 x^{59}  +
     56768242863723954140114918264 x^{60}  -
     25971874970291859366352369232 x^{61}  -
     25868414266252042008113790681 x^{62}  +
     12334499250260624842980447944 x^{63}  +
     10073049237434908783966012615 x^{64}  -
     4934518938825367378469146563 x^{65}  -
     3357135730239884258076120873 x^{66}  +
     1668961310040579145324436854 x^{67}  +
     958979388811076795805220259 x^{68}  -
     478484381627959740231307898 x^{69}  -
     235070462661933368244141645 x^{70}  +
     116494763716325583154471593 x^{71}  +
     49490588354774441565598327 x^{72}  -
     24113150837187023305271901 x^{73}  -
     8953841615972305424578920 x^{74}  +
     4245305143743032443594882 x^{75}  +
     1392157069027476262098948 x^{76}  - 635580648350932217926483 x^{77}  -
     185931065427671751697600 x^{78}  + 80838922982850596922843 x^{79}  +
     21306408936395504377734 x^{80}  - 8719476782236791891707 x^{81}  -
     2090911132691824012433 x^{82}  + 795459886846428765825 x^{83}  +
     175226009815304083621 x^{84}  - 61148524359498140722 x^{85}  -
     12490776292567568868 x^{86}  + 3941245311464514974 x^{87}  +
     753388475396092803 x^{88}  - 211610725442247167 x^{89}  -
     38184703986094031 x^{90}  + 9385073357203492 x^{91}  +
     1611821886653267 x^{92}  - 340074435384128 x^{93}  -
     56011183019777 x^{94}  + 9923940361375 x^{95}  + 1578324064638 x^{96}  -
     228740033731 x^{97}  - 35344978435 x^{98}  + 4053271905 x^{99}  +
     611772610 x^{100}  - 53065630 x^{101}  - 7859404 x^{102}  +
     481716 x^{103}  + 70281 x^{104}  - 2698 x^{105}  - 389 x^{106}  + 7 x^{107} +
     x^{108})]$} \\   \\ \small
              \hspace*{1.2cm} $-$  \parbox[t]{5.4in}{ $ \ds
     [(2 x^2 (-2850 + 6723322 x^2 - 6759668263 x^4 + 3942336304012 x^6 -
        1522935878054248 x^8 + 420297644798819964 x^{10} -
        87061420747476801792 x^{12} + 14017522325018234443820 x^{14} -
        1800691755710287649136536 x^{16} +
        188339701222253001915401008 x^{18} -
        16301340717115495133297512750 x^{20} +
        1183161780558124120995496715832 x^{22} -
        72811713139146107588601162301885 x^{24} +3834865450794129182979163318600488 x^{26} -
        174246321502744931429532990462875144 x^{28} +
        6877788456677265485701921137593457592 x^{30} - $  \\ $
        237265725373820790084982826016430065324 x^{32} + $  \\ $
        7191857393372354804557967599707792525704 x^{34} - $  \\ $
        192454439380257306670502099935970291948287 x^{36} + $  \\ $
        4566043270046851019844894427773717398558322 x^{38} - $  \\ $
        96413602080128957229656853667253058076966478 x^{40} + $  \\ $
        1818138679596462634993805497777102207272268828 x^{42} - $  \\ $
        30716752962046065733955398527802370463153453468 x^{44} + $  \\ $
        466267583012996158423660005909885027589658350532 x^{46} - $  \\ $
        6376118596726605408588335027560665019457695744918 x^{48} + $  \\ $
        78740535733173513455041562643270803259375995530058 x^{50} - $} \\
          \hspace*{1.4cm}   \parbox[t]{5.4in}{ $ \ds
        880118246491475601601423751310031046933201491720606 x^{52} + $  \\ $
        8922638175355253016485311034187403219400049031380550 x^{54} - $  \\ $
        82205438380503609783493365234820961159981095846071378 x^{56} + $  \\ $
        689526717426032793861172718026480848694305711950968784 x^{58} - $ \\ $
        5274502790200303684138080315604728615280948921373387877 x^{60}  +  $ \\ $
        36853732104768759635963701467203731490491696287517571270 x^{62} - $ \\ $
        235557924021941156542669828228788631817828751703098719117 x^{64} +  $ \\ $
        1379233266102751106680069828991484844090624964080219581054 x^{66} - $ \\ $
         7407518667047801024011626478759013550761362805406353959535 x^{68} + $ \\ $
 36537648865425956553949459856633538612321320017147092420992 x^{70} -  $ \\ $
 165709710650234473956541618685483802123998444837800186863433 x^{72} +  $ \\ $
 691787296298266802163462211187175065445750531644287557378858 x^{74} - $ \\ $
  2661124132393193880482903626656542061543206968078106781306641 x^{76} + $ \\ $
   9441714766142941114858729559093399652439549276374491612358220 x^{78} -  $ \\ $
   30926529202300371517259076142607092917939887818672120674748793 x^{80} +  $ \\ $
   93601949373888484217861098961903359321465957704607022875649956 x^{82} -  $ \\ $
   261980852719197384074580900183645596880121589285171292981168803 x^{84} + $ \\ $
    678613561916244539832896686029719323844649470185047084200549198 x^{86} - $ \\ $
     1628028845994958189127456581824194502978441830701064274342183786 x^{88} + $ \\ $
        3619831235637093310878667937714674022443929217211014482141028644 x^{90} - $ \\ $
        7464213259796436753836641691845287651024187099286634501274407693 x^{92} + $ \\ $
        14282874970334852019359262693472291214846115403395852076997066418 x^{94} - $ \\ $
        25376624622203352824370596917257229498235588203975067919037421202 x^{96} + $ \\ $
        41886545220230362404215108650496969628250043230197153077967711328 x^{98} - $ \\ $
        64262751773718460981279526325583380467376706959098687516980920463 x^{100} + $ \\ $
        91684569742863751779248206303707694097123651794233523839847845350 x^{102} - $ \\ $
        121697163146866032210903318267912517815771127362001581018894303713 x^{104} + $ \\ $
        150346963731718494406475483108959291087531042754314087437307752192 x^{106} - $ \\ $
        172945693803604434824779951991769055458294368966683842833555750248 x^{108} + $ \\ $
        185304189128532505026027209000605167307470573335753634222446877344 x^{110} - $ \\ $
        184999133161478796553941778145476670018479100690268555791160880669 x^{112} + $ \\ $
        172147673664149903602784491803709786720607137419743075458376278256 x^{114} - $ \\ $
        149350770943370141395174331350548880948900909117740282999033461640 x^{116} + $ \\ $
        120838840953211958897346923927315670640312889507318262215715247744 x^{118} - $ \\ $
        91202481246658997146850952125526719734854605709759176009633796013 x^{120} +   $ \\ $
        64225357023519175729733018612457705186024790890392698790674424006 x^{122} - $ \\ $
        42208074569875646452973183961404319749857646755713675378114388417 x^{124} + $ \\ $
        25891293405617484568187234753114516636654531766663785530082971896 x^{126} - $  \\ $
        14826951661369991335979744778444815659173301934639077935088838536 x^{128} + $} \\
          \hspace*{1.4cm}   \parbox[t]{5.4in}{ $ \ds
        7927774017115768349777154280407363734319108396931345538931156406 x^{130} - $ \\ $
        3958254838081197025427391291673953083613929149998436362203685537 x^{132} + $ \\ $
        1845659039833880103765432953160795924372049443566581274892485312 x^{134} - $ \\ $
        803760415939924473352705326930842126751915100980430061626779971 x^{136} + $ \\ $
        326927224678450719282675448546850900103446813097950176039294476 x^{138} - $ \\ $
        124204167321700166495231700801570066696875646871865557753263312 x^{140} + $ \\ $
        44073956621718485588482627020963359229143298504480708888654204 x^{142} - $ \\ $
         14607547238068026946501304746588662657658916546086222765483094 x^{144} + $ \\ $
        4521664967997477721978081508619506948828662557606481140476048 x^{146} -  $ \\ $
        1307100318053416842468898250389366360556773144964314211639685   x^{148} +  $ \\ $
        352824763408083792343698686320361345849486164699883762174526 x^{150} -  $ \\ $
        88916971792938737901704247929096710605526909216521212354606 x^{152} +  $ \\ $
        20917341329193841181378149767654984544571248989838272802406 x^{154} - $ \\ $
         4592289114413323002128931261581886649124661360017675386462 x^{156} +  $ \\ $
         940677005417114131694888970306489928079621368667583992830 x^{158} - $ \\ $
          179725578226176871763494840587332095791213583343429327910 x^{160} +  $ \\ $
          32017500639586680646947676844497845958112184109319124640 x^{162} - $ \\ $
           5316215400223655932853788080317523269761898748093636387 x^{164} +  $ \\ $
           822361180335131458413089596241506642610727216197850374 x^{166} - $ \\ $
        118454362853823004936265083311731703569123654427110030 x^{168}+  $ \\ $
        15879103072578820257409119317688928634788599575790130 x^{170} -  $ \\ $
        1979784187173967953538490446185030773251689578051108 x^{172} + $ \\ $
        229418033448085553009081960491599743182847249418966 x^{174} - $ \\ $
        24690151645039503212974202281502752118254050446109 x^{176} + $ \\ $
        2465699852092291972282508185323173117798802734346 x^{178} -  $ \\ $
        228282464452576991423943997675937666413157696177 x^{180} + $ \\ $
        19573781106855649345981215999443840782086949464 x^{182} - $ \\ $
        1552583018599702027759412089108667886921421639 x^{184} + $ \\ $
        113781155721621997638138347377663419814008752 x^{186} - $ \\ $
        7693475011706994657339312213875138105330997 x^{188} + $ \\ $
        479236684585625282931405079285561526771974 x^{190} - $ \\ $
        27454978052780940432299251412054225187241 x^{192} + $ \\ $
        1443857850655771853208974651709898270296 x^{194} - $ \\ $
        69559814412598454831315928221310825287 x^{196} + $ \\ $
        3062793890835423527711425984743269830 x^{198} - $ \\ $
        122935569661251024307892770722340355 x^{200} + $ \\ $
        4485112036898997674033003608307590 x^{202} -
        148243713573336247868200974798482 x^{204} +
        4422445002084556526593775154486 x^{206} -
        118569346817098273547037977341 x^{208} +
        2842888738083317801242622136 x^{210} -
        60606998232620774044317279 x^{212} +
        1141066636699083306980514 x^{214} -
        18819066524471888592988 x^{216} + 269214632318128784170 x^{218} -
        3299890687581682338 x^{220} + 34122906167736194 x^{222} -
        291656209997528 x^{224} + 2003699656452 x^{226} -
        10624233918 x^{228} + 40769788 x^{230} - 100672 x^{232} +
        120 x^{234})] / $ }  \\
      \hspace*{1.4cm}   \parbox[t]{5.4in}{ $ \ds
     [(-1 + x) (1 + x) (1 - 2 x - x^2 + x^3) (-1 - 2 x + x^2 + x^3) (1 -
      8 x + 8 x^2 + 6 x^3 - 6 x^4 - x^5 + x^6) (1 + 8 x + 8 x^2 -
      6 x^3 - 6 x^4 + x^5 + x^6) (-1 - 27 x + 1032 x^2 + 25802 x^3 -
      373186 x^4 - 9247650 x^5 + 68985132 x^6 + 1756981328 x^7 -
      7544737839 x^8 - 205344558763 x^9 + 520362914282 x^{10} +
      16055543418770 x^{11} - 22765506950292 x^{12} -
      886185755306711 x^{13} + 567341618795392 x^{14} +
      35824949732566099 x^{15} - 2114077619505885 x^{16} -
      1089052544253505079 x^{17} - 434310955869650180 x^{18} +
      25382085183437405700 x^{19} + 19657336252849199189 x^{20} -
      460196787380533985842 x^{21} - 503282143127348655479 x^{22} +
      6563545564790657680027 x^{23} + 9019917524975714828718 x^{24} -
      74279664653212096085631 x^{25} - 121126151372353616341848 x^{26} +
      671522546480445322781198 x^{27} +
      1259864230974286367427691 x^{28} -
      4874434339938184506685397 x^{29} -
      10355985403728393543409067 x^{30} +
      28508897879645580528825667 x^{31} +
      68215067001652545430286192 x^{32} -
      134568911549356278360128102 x^{33} -
      363861169332415595759624659 x^{34} +
      512191062359197441924007885 x^{35} +
      1584783594082588087382278463 x^{36} -
      1564210162255894419536702349 x^{37}-
      5674861959676935021395050413 x^{38 }+
      3785132858090620875642211416 x^{39} +
      16803139396061165672290060717 x^{40} -
      7044093577535657561000115327 x^{41} -
      41343132893232215469798457354 x^{42} +
      9285232522254031569698694689 x^{43} +
      84883083360723557875730426434 x^{44} -
      5937121438225419731565543654 x^{45} -
      145956035701218221797085712592 x^{46} -
      7890101748188581051360660346 x^{47} +
      210852803270320434959494868010 x^{48} +
      32505194104448575284667024966 x^{49} -
      256625689573728687673431393828 x^{50} -
      60386088834617104154813509703 x^{51} +
      263791426465741835712351964159 x^{52} +
      79446568823858899041889805170 x^{53} -
      229529669618761966378616574822 x^{54} -
      81368503811164989012654074179 x^{55} +
      169410836644260710451175862716 x^{56} +
      67317485231746697845168652537 x^{57} -
      106273241553496096920653558337 x^{58} -
      45836707489100134208832885680 x^{59} +
      56768242863723954140114918264 x^{60} +
      25971874970291859366352369232 x^{61} -
      25868414266252042008113790681 x^{62} -
      12334499250260624842980447944 x^{63} +
      10073049237434908783966012615 x^{64} +
      4934518938825367378469146563 x^{65} -
      3357135730239884258076120873 x^{66} -
      1668961310040579145324436854 x^{67} +
      958979388811076795805220259 x^{68} +
      478484381627959740231307898 x^{69} -
      235070462661933368244141645 x^{70} -
      116494763716325583154471593 x^{71} +
      49490588354774441565598327 x^{72} +
      24113150837187023305271901 x^{73} -
      8953841615972305424578920 x^{74} -
      4245305143743032443594882 x^{75} +
      1392157069027476262098948 x^{76} +
      635580648350932217926483 x^{77} - 185931065427671751697600 x^{78} -
      80838922982850596922843 x^{79} + 21306408936395504377734 x^{80} +
      8719476782236791891707 x^{81} - 2090911132691824012433 x^{82} -
      795459886846428765825 x^{83} + 175226009815304083621 x^{84} +
      61148524359498140722 x^{85} - 12490776292567568868 x^{86} -
      3941245311464514974 x^{87} + 753388475396092803 x^{88} +
      211610725442247167 x^{89} - 38184703986094031 x^{90} -
      9385073357203492 x^{91 } + 1611821886653267 x^{92} +
      340074435384128 x^{93} - 56011183019777 x^{94} -
      9923940361375 x^{95} + 1578324064638 x^{96} + 228740033731 x^{97} -
      35344978435 x^{98 }- 4053271905 x^{99} + 611772610 x^{100} +
      53065630 x^{101} -  $ }  \\
      \hspace*{1.4cm}   \parbox[t]{5.4in}{ $ \ds
      7859404 x^{102} - 481716 x^{103} + 70281 x^{104} +
      2698 x^{105} - 389 x^{106} - 7 x^{107} + x^{108}) (-1 + 27 x +
      1032 x^{2} - 25802 x^{3} - 373186 x^{4} + 9247650 x^{5} +
      68985132 x^{6} - 1756981328 x^{7} - 7544737839 x^{8} +
      205344558763 x^{9} + 520362914282 x^{10} - 16055543418770 x^{11} -
      22765506950292 x^{12} + 886185755306711 x^{13} +
      567341618795392 x^{14} - 35824949732566099 x^{15}-
      2114077619505885 x^{16 }+ 1089052544253505079 x^{17} -
      434310955869650180 x^{18} - 25382085183437405700 x^{19 }+
      19657336252849199189 x^{20} + 460196787380533985842 x^{21} -
      503282143127348655479 x^{22}- 6563545564790657680027 x^{23} +
      9019917524975714828718 x^{24} + 74279664653212096085631 x^{25} -
      121126151372353616341848 x^{26} - 671522546480445322781198 x^{27} +
      1259864230974286367427691 x^{28} +
      4874434339938184506685397 x^{29} -
      10355985403728393543409067 x^{30} -
      28508897879645580528825667 x^{31} +
      68215067001652545430286192 x^{32} +
      134568911549356278360128102 x^{33} -
      363861169332415595759624659 x^{34} -
      512191062359197441924007885 x^{35} +
      1584783594082588087382278463 x^{36} +
      1564210162255894419536702349 x^{37} -
      5674861959676935021395050413 x^{38} -
      3785132858090620875642211416 x^{39} +
      16803139396061165672290060717 x^{40} +
      7044093577535657561000115327 x^{41} -
      41343132893232215469798457354 x^{42} -
      9285232522254031569698694689 x^{43} +
      84883083360723557875730426434 x^{44} +
      5937121438225419731565543654 x^{45} -
      145956035701218221797085712592 x^{46} +
      7890101748188581051360660346 x^{47} +
      210852803270320434959494868010 x^{48} -
      32505194104448575284667024966 x^{49} -
      256625689573728687673431393828 x^{50 }+
      60386088834617104154813509703 x^{51 }+
      263791426465741835712351964159 x^{52 }-
      79446568823858899041889805170 x^{53}-
      229529669618761966378616574822 x^{54 }+
      81368503811164989012654074179 x^{55 }+
      169410836644260710451175862716 x^{56} -
      67317485231746697845168652537 x^{57} -
      106273241553496096920653558337 x^{58} +
      45836707489100134208832885680 x^{59 }+
      56768242863723954140114918264 x^{60} -
      25971874970291859366352369232 x^{61 }-
      25868414266252042008113790681 x^{62} +
      12334499250260624842980447944 x^{63} +
      10073049237434908783966012615 x^{64} -
      4934518938825367378469146563 x^{65} -
      3357135730239884258076120873 x^{66 }+
      1668961310040579145324436854 x^{67} +
      958979388811076795805220259 x^{68} -
      478484381627959740231307898 x^{69} -
      235070462661933368244141645 x^{70} +
      116494763716325583154471593 x^{71} +
      49490588354774441565598327 x^{72} -
      24113150837187023305271901 x^{73} -
      8953841615972305424578920 x^{74} +
      4245305143743032443594882 x^{75} +
      1392157069027476262098948 x^{76} -
      635580648350932217926483 x^{77} - 185931065427671751697600 x^{78} +
      80838922982850596922843 x^{79} + 21306408936395504377734 x^{80} -
      8719476782236791891707 x^{81} - 2090911132691824012433 x^{82} +
      795459886846428765825 x^{83} + 175226009815304083621 x^{84} -
      61148524359498140722 x^{85} - 12490776292567568868 x^{86} +
      3941245311464514974 x^{87} + 753388475396092803 x^{88} -
      211610725442247167 x^{89} - 38184703986094031 x^{90} +
      9385073357203492 x^{91} + 1611821886653267 x^{92} -
      340074435384128 x^{93} - 56011183019777 x^{94} +
      9923940361375 x^{95} + 1578324064638 x^{96} - 228740033731 x^{97} -
      35344978435 x^{98} + 4053271905 x^{99} + 611772610 x^{100} -
      53065630 x^{101} - 7859404 x^{102} + 481716 x^{103} + 70281 x^{104} -
      2698 x^{105} - 389 x^{106} + 7 x^{107} + x^{108})]$} \\ \\
 \hspace*{1.2cm}  $-$ \parbox[t]{5.4in}{ $ \ds
    [2 x^2 (615 - 262242 x^2 + 41898960 x^4 - 3512269288 x^6 +
       179492561770 x^8 - 6096366193848 x^{10} + 145494979941521 x^{12} -
       2535542754193824 x^{14} + 33176572069791591 x^{16} -
       332854570511823920 x^{18} + 2603194747035482556 x^{20} -
       16086202616248441320 x^{22} + 79449951342676275114 x^{24} -
       316835773017958293666 x^{26} + 1029558822605836522005 x^{28}-
       2748968228790112897744 x^{30} + 6077130032222836986250 x^{32} -
       11200610700373443578820 x^{34} + 17318239432637807097927 x^{36} -
       22588873120902162099180 x^{38} + 24976811408592140276865 x^{40} -
       23511283684510389385926 x^{42} + 18909973847672503011799 x^{44} -
       13034728729155138663696 x^{46} + 7719367749577148252450 x^{48} -
       3935122847104604482756 x^{50} + 1729087729879680905445 x^{52} -
       655385761453789840096 x^{54} + 214329562544165022798 x^{56} -
       60448487760081534780 x^{58} + 14686763734847933637 x^{60} -
       3068282127014427136 x^{62} + 549679341236297976 x^{64} -
       84128865085674860 x^{66} + 10945816897612520 x^{68} -
       1202836407725052 x^{70} + 110704524241604 x^{72} -
       8440225044136 x^{74} + 525369688896 x^{76} - 26180086680 x^{78} +
       1016113988 x^{80} - 29498700 x^{82} + 600452 x^{84} - 7612 x^{86} +
       45 x^{88})]/ $} \\
        \\
 \hspace*{1.4cm}   \parbox[t]{5.4in}{ $ \ds
       [(-1 + 15 x + 195 x^2 - 2476 x^3 - 9408 x^4 +
      128774 x^5 + 151702 x^6 - 3080005 x^7 - 152040 x^8 +
      39805335 x^9 - 22147982 x^{10} - 300921194 x^{11} +
      284159318 x^{12} + 1383107908 x^{13} - 1722027429 x^{14} -
      3930756397 x^{15} + 6069754917 x^{16} + 6915053418 x^{17} -
      13500672554 x^{18} - 7214631815 x^{19} + 19878551923 x^{20} +
      3475881699 x^{21} - 19978574007 x^{22} + 1204754727 x^{23} +
      13960633114 x^{24} - 3214431392 x^{25} - 6832632284 x^{26} +
      2528557309 x^{27} + 2326963032 x^{28} - 1184186750 x^{29} -
      534678044 x^{30} + 369015343 x^{31} + 75261636 x^{32} -
      78835592 x^{33} - 4000296 x^{34} + 11506048 x^{35} - 640379 x^{36} -
      1110448 x^{37} + 155006 x^{38} + 65338 x^{39} - 14541 x^{40} -
      1860 x^{41} + 680 x^{42} - 2 x^{43} - 13 x^{44} + x^{45})
    (1 + 15 x - 195 x^2 - 2476 x^3 + 9408 x^4 + 128774 x^5 -
      151702 x^6 - 3080005 x^7 + 152040 x^8 + 39805335 x^9 +
      22147982 x^{10} - 300921194 x^{11} - 284159318 x^{12} +
      1383107908 x^{13} + 1722027429 x^{14} - 3930756397 x^{15} -
      6069754917 x^{16} + 6915053418 x^{17} + 13500672554 x^{18} -
      7214631815 x^{19} - 19878551923 x^{20} + 3475881699 x^{21} +
      19978574007 x^{22} + 1204754727 x^{23} - 13960633114 x^{24} -
      3214431392 x^{25} + 6832632284 x^{26} + 2528557309 x^{27} -
      2326963032 x^{28} - 1184186750 x^{29} + 534678044 x^{30} +
      369015343 x^{31} - 75261636 x^{32} - 78835592 x^{33} + 4000296 x^{34} +
      11506048 x^{35} + 640379 x^{36}- 1110448 x^{37} - 155006 x^{38} +
      65338 x^{39} + 14541 x^{40} - 1860 x^{41} - 680 x^{42} - 2 x^{43} +
      13 x^{44} + x^{45})]
     $} \\ \\ \\
 \hspace*{1.2cm}  $-$ \parbox[t]{5.4in}{ $ \ds
  [(2 x^2 (-55 + 990 x^2 - 5148 x^4 + 12012 x^6 - 15015 x^8 +
      10920 x^{10} - 4760 x^{12} + 1224 x^{14} - 171 x^{16} + 10 x^{18})]/$ \\
      $[ 1 - 55 x^2 + 495 x^4 - 1716 x^6 + 3003 x^8 - 3003 x^{10} +
    1820 x^{12} - 680 x^{14} + 153 x^{16} - 19 x^{18} + x^{20}]$} \\ \\ \\
 \hspace*{1.2cm}  $-$ \parbox[t]{5.4in}{ $ \ds
     \frac{2 x^2}{-1 + x^2} $}
     \bc $ \mbox{--------------------------------------    } $  \ec
  ${\cal F}^{TkC}_{10}(x) = $ \parbox[t]{5.4in}{ $ \ds
x + 29525 x^2 + 55807 x^3 + 22794425 x^4 + 199912706 x^5 +
 34583478677 x^6 + 538897819048 x^7 + 63782453175969 x^8 +
 1342851693261496 x^9 + 127666740816792660 x^{10} +
 3240058791241468318 x^{11} + 266592485903824019297 x^{12} +
 7697782672223977809178 x^{13} + 570941650352643832290496 x^{14} +
 18144383836051636673861867 x^{15} +
 1243579048192809989812145169 x^{16} +
 42589885261525799514745869499 x^{17} +
 2742392002573492985349801499994 x^{18} +
 99747165686702441614952745202978 x^{19} +
 6106442739965045528592404793625380 x^{20} +
 233330348371077318790986888910572997 x^{21} +
 13704758306706316190825701744380567594 x^{22} + $\\ $
 545454160156590238865681796289941865692 x^{23} + $\\ $
 30960502764205283074176988011335478152409 x^{24} + $\\ $
 1274650770709308207856649639205906364343631 x^{25} + $\\ $
 70330983064894193743004668482926028100698862 x^{26} + $\\ $
 2978107397482949463652001340178244580060843958 x^{27} + $\\ $
 160514087920332719815868792799619501992413739904 x^{28} + $\\ $
 6957350295166340134184435628772224999713210715452 x^{29} + $\\ $
 367781299620754736758669478854175346991901525768337 x^{30}  + \ldots $}

    \bc $ \mbox{--------------------------------------    } $  \\ $ \mbox{--------------------------------------    } $ \ec
     ${\cal F}^{TkC}_{11}(x) = $  \parbox[t]{5.4in}{$ \ds
     x + 88573 x^2 + 184318 x^3 + 145252485 x^4 + 1637069691 x^5 +
 501682800748 x^6 + 10634850017387 x^7 + 2147983445752757 x^8 +
 63350881300193974 x^9 + 10062119265462622683 x^{10} +
 364014089441637130211 x^{11} + 49370866298667719771964 x^{12} +
 2054738065680739228638707 x^{13} + 248946408250205225732190657 x^{14} +
 11488975420539457121794585998 x^{15} +
 1277895262634953569701628954885 x^{16} +
 63905315558068166767129771712679 x^{17} +
 6643863136834483490999934197961784 x^{18} +
 354419932121539626627635855762488303 x^{19} +
 34880518154857844080556760102870914055 x^{20} +  $\\ $
 1962349859371378210299064120807482879842 x^{21} +  $\\ $
 184567015284713513365573390874875081815257 x^{22} +  $\\ $
 10854817365779249881219061645315445576104495 x^{23} +  $\\ $
 983009792071692074244201624174027978204956828 x^{24} +  $\\ $
 60011264993110778296038980871956735067594275091 x^{25} +  $\\ $
 5264523828057269322900568790047120990273047789859 x^{26} +  $\\ $
 331671290250618896825564154271703753607683170445470 x^{27} +  $\\ $
 28327480833649325978874974048998854014057102254501305 x^{28} +  $\\ $
 1832759266061590962705855675730301992067640249118562251 x^{29} +  $\\ $
 153042167979062116228157184495496323568383893309232511908 x^{30} + \ldots $}

      \bc $ \mbox{--------------------------------------    } $  \\ $ \mbox{--------------------------------------    } $ \ec
     ${\cal F}^{TkC}_{12}(x) = $  \parbox[t]{5.4in}{$ \ds
     x + 265721 x^2 + 608761 x^3 + 925589701 x^4 + 13405842666 x^5 +
 7277627334803 x^6 + 209878072831673 x^7 + 72337143245836829 x^8 +
 2989171289995095295 x^9 + 793072716833845192356 x^{10} +
 40916785291407964066602 x^{11} + 9144433503092353217515639 x^{12} +
 549007484897518979688253782 x^{13} +
 108594562872624965291980865517 x^{14} +
 7286202085630847737325323395906 x^{15} +
 1314373438943841131293557932317677 x^{16} +
 96096516725583526494032612816236112 x^{17} +
 16121626293866934494559934285910094227 x^{18} +
 1262780485977348334267781781991467055379 x^{19} +  $\\ $
 199724042414660904427470237413507262400756 x^{20} +  $\\ $
 16558114341585131120547984869353913374589204 x^{21} +  $\\ $
 2493900397938854697670288260556750511325686526 x^{22} +  $\\ $
 216838508725758375059954476572184526500702119944 x^{23} +  $\\ $
 31343557492605404324253360704947067629992713200607 x^{24} +  $\\ $
 2837455189235018400278168812902801322974786027900941 x^{25} +  $\\ $
 396094945355495460632836679520098519054199877278433322 x^{26} +  $\\ $
 37112678357696754207068876992544302223088398622621632569 x^{27} +  $\\ $
 5029131175386992926691728973991724107551644145610067932113 x^{28} +  $\\ $
 485284127122604253413359884653261788559391346482048060107639 x^{29} +  $\\ $
 64113951070393524282099764277420069836993858131130145451721088 x^{30}+ \ldots $}

    \bc $ \mbox{--------------------------------------    } $  \\ $ \mbox{--------------------------------------    } $ \ec
\newpage {\bf \large \ Moebius strip $MS_{m}(n)$ ($ 2 \leq m \leq 12$)}
\\
\\
    ${\cal F}^{MS}_{2}(x) = $  \parbox[t]{5.4in}{$ \ds \frac{-(x (1 + 2 x))}{-1 + x + x^2} - \frac{2 x}{(-1 + x) (1 + x)} $} \\ \\
 \hspace*{1.2cm}  $=$ \parbox[t]{5.4in}{ $ \ds 3 x + 3 x^2 + 6 x^3 + 7 x^4 + 13 x^5 + 18 x^6 + 31 x^7 + 47 x^8 +
 78 x^9 + 123 x^{10} + 201 x^{11} + 322 x^{12} + 523 x^{13}
  + 843 x^{14} +
 1366 x^{15} + 2207 x^{16} + 3573 x^{17} + 5778 x^{18} + 9351 x^{19} +
 15127 x^{20} + 24478 x^{21} + 39603 x^{22} + 64081 x^{23} + 103682 x^{24} +
 167763 x^{25} + 271443 x^{26} + 439206 x^{27} + 710647 x^{28} +
 1149853 x^{29} + 1860498 x^{30}   +  \ldots $}
  \bc $ \mbox{--------------------------------------    } $ \ec  \noindent
    ${\cal F}^{MS}_{3}(x) = $  \parbox[t]{5.4in}{$ \ds  \frac{  6 x^2}{ 1 - 3 x^2}   - \frac{x (3 + 2 x + x^2)}{(1 + x) (-1 + 2 x + x^2)}  $} \\ \\
 \hspace*{1.2cm}  $=$ \parbox[t]{5.4in}{ $ \ds  3 x + 11 x^2 + 15 x^3 + 51 x^4 + 83 x^5 + 251 x^6 + 479 x^7 +
 1315 x^8 + 2787 x^9 + 7211 x^{10} + 16239 x^{11} + 40659 x^{12} +
 94643 x^{13} + 232859 x^{14} + 551615 x^{15} + 1344835 x^{16} +
 3215043 x^{17} + 7801163 x^{18} + 18738639 x^{19} + 45357171 x^{20} +
 109216787 x^{21} + 264026939 x^{22} + 636562079 x^{23} + 1537859683 x^{24} +
 3710155683 x^{25} + 8960296811 x^{26} + 21624372015 x^{27} +
 52215418131 x^{28} + 126036076403 x^{29} + 304306702811 x^{30}
 + \dots  $}
\\ \\ \bc $ \mbox{--------------------------------------    } $ \ec  \noindent
    ${\cal F}^{MS}_{4}(x) = $  \parbox[t]{5.4in}{$ \ds   - \frac{x (3 + 14 x + x^2 - 16 x^3 - 4 x^4 + 4 x^5)}{(1 + x) (-1 + 2 x + 7 x^2 - 2 x^3 - 3 x^4 + x^5)}  $
    \\ $ \ds + \frac{  2 x (2 + 3 x^2 - 3 x^4 + x^6)}{(-1 + x) (1 + x) (-1 - 3 x + x^3) (1 - 3 x + x^3)} - \frac{2 x}{(-1 + x) (1 + x)}$} \\ \\
 \hspace*{1.2cm}  $=$ \parbox[t]{5.4in}{ $ \ds   9 x + 17 x^2 + 93 x^3 + 197 x^4 + 1064 x^5 + 2579 x^6 + 12602 x^7 +
 34813 x^8 + 155109 x^9 + 473782 x^{10} +
 1973442 x^{11} + 6461255 x^{12} +
 25731364 x^{13} + 88161636 x^{14} + 341164928 x^{15} + 1203089485 x^{16} +
 4572905761 x^{17} + 16418366051 x^{18} + 61718041819 x^{19} +
 224060510622 x^{20} + 836556608393 x^{21} + 3057746964212 x^{22} +
 11369139034731 x^{23} + 41728999487503 x^{24} + 154761962184614 x^{25} +
 569474721891612 x^{26} + 2108779476730869 x^{27} +
 7771609006274884 x^{28} + 28751491766120674 x^{29} +
 106058977958127364 x^{30} +
  \ldots $ }
   \bc $ \mbox{--------------------------------------    } $ \ec  \noindent
    ${\cal F}^{MS}_{5}(x) = $  \parbox[t]{5.4in}{$ \ds - \frac{4 x^2 (9 - 52 x^2 + 150 x^4 - 156 x^6 + 65 x^8)}{(-1 + x) (1 +
     x) (-1 + 2 x^2) (-1 + 5 x^2) (1 - 22 x^2 + 13 x^4)}  $ \\
     $  \ds
       - \frac{   x (9 + 18 x - 42 x^2 - 148 x^3 - 34 x^4 + 156 x^5 + 44 x^6 -
      64 x^7 + 12 x^8)}{(-1 - 4 x - 2 x^2 + 2 x^3) (1 - 5 x -
      7 x^2 + 18 x^3 + 6 x^4 - 12 x^5 + 2 x^6)} + \frac{10 x^2}{ 1 - 5 x^2}$} \\ \\
 \hspace*{1.2cm}  $=$ \parbox[t]{5.4in}{ $ \ds
 9 x + 73 x^2 + 210 x^3 + 1857 x^4 + 6079 x^5 + 51268 x^6 +
 187679 x^7 + 1452017 x^8 + 5929338 x^9 + 42196753 x^{10} +
 188792811 x^{11} + 1256103924 x^{12} + 6026479923 x^{13} +
 38140588993 x^{14} + 192526981810 x^{15} + 1175822873153 x^{16} +
 6152192171627 x^{17} + 36653174443084 x^{18} +   $  \\  $ \ds 196608798903531 x^{19} +
 1151582649902817 x^{20} + 6283287697977122 x^{21} + $ \\ $
 36379144274981889 x^{22} + 200804934692074883 x^{23} +
 1153558411488272612 x^{24} +    $ \\ $ \ds  6417455984530684979 x^{25} +
 36672132877121356337 x^{26} + 205093433765461090818 x^{27} + $ \\ $ \ds
 1167841223686770652161 x^{28} + 6554517403144301138299 x^{29} +
 37233851201134748834428 x^{30} + \ldots $}
 \bc $ \mbox{--------------------------------------    } $ \ec  \noindent
    ${\cal F}^{MS}_{6}(x) = $  \parbox[t]{5.4in}{$ \ds   [x (7 + 107 x + 63 x^2 - 2167 x^3 - 4608 x^4 + 8548 x^5 + 23235 x^6 -
   15711 x^7 - 43966 x^8 + 22258 x^9 + 36534 x^{10} - 19246 x^{11} -
   11561 x^{12} + 5877 x^{13} + 1910 x^{14} - 758 x^{15} - 182 x^{16} +
   36 x^{17} + 8 x^{18})] / $ \\
    $[(1 - x) (1 + x) (-1 - 5 x - 5 x^2 + 2 x^3 + x^4) (-1 + 5 x + 49 x^2 - 116 x^3 - 363 x^4 +
    627 x^5 + 544 x^6 - 1061 x^7 + 133 x^8 + 264 x^9 - 47 x^{10} -
    26 x^{11} + 3 x^{12} + x^{13})]  $ \\ \\
      $   - [2 x (6 - 201 x^2 + 4060 x^4 - 38192 x^6 + 197410 x^8 - 601653 x^{10} +
   1140106 x^{12} - 1391696 x^{14} + 1102163 x^{16} - 558186 x^{18} +
   178482 x^{20} - 35341 x^{22} + 4160 x^{24} - 265 x^{26} + 7 x^{28})] / $ \\ $[(-1 - 3 x + 6 x^2 + 4 x^3 - 5 x^4 - x^5 +
      x^6) (-1 + 3 x + 6 x^2 - 4 x^3 - 5 x^4 + x^5 + x^6) (1 - 9 x +
      6 x^2 + 53 x^3 - 45 x^4 - 66 x^5 + 52 x^6 - 6 x^8 + x^9) (-1 -
      9 x - 6 x^2 + 53 x^3 + 45 x^4 - 66 x^5 - 52 x^6 + 6 x^8 +
      x^9)]  $ \\ \\ $ \ds + \frac{2 x (3 + 6 x^2 - 16 x^4 + 17 x^6 - 7 x^8 + x^{10})/}{(-1 - 3 x + 6 x^2 + 4 x^3 - 5 x^4 - x^5 + x^6) (-1 +
      3 x + 6 x^2 - 4 x^3 - 5 x^4 + x^5 + x^6)} $ \\ $ \ds - \frac{2x}{(-1 + x) (1 + x)} $} \\ \\
 \hspace*{1.2cm}  $=$ \parbox[t]{5.4in}{ $ \ds   27 x + 107 x^2 + 1371 x^3 + 6415 x^4 + 87922 x^5 + 446645 x^6 +
 5847400 x^7 + 32574167 x^8 + 395913126 x^9 + 2420149972 x^{10} +
 27298250776 x^{11} + 181134018889 x^{12} + 1916020966341 x^{13} +
 13595224053238 x^{14} + 136617755072476 x^{15} + 1021501284712391 x^{16} +
 9868854731861982 x^{17} + 76783359869377940 x^{18} +
 720264693432184820 x^{19} + 5772464612095502460 x^{20} +
 52983534685851344203 x^{21} + 433990470327409239362 x^{22} +
 3920670502127763596399 x^{23} + 32629347205036863480113 x^{24} +
 291395142403062432504997 x^{25} + 2453240099048398652098941 x^{26} +
 21726916981450035120472251 x^{27} + 184447608098723151001797246 x^{28} +
 1623781970639434445578728789 x^{29} +
 13867765496249055757164990460 x^{30} + \ldots  $}
   \bc $ \mbox{--------------------------------------    } $  \ec   \noindent
    ${\cal F}^{MS}_{7}(x) = $  \parbox[t]{5.4in}{
    $ \ds
    -[2 x^2 (-110 + 13706 x^2 - 797065 x^4 + 25766696 x^6 -
    510464656 x^8 + 6624939954 x^{10} - 58930004497 x^{12} +
    370288441956 x^{14} - 1675373801542 x^{16} + 5518828075268 x^{18} -
    13304788221252 x^{20} + 23490902716736 x^{22} - 30266698848600 x^{24} +
    28218903583250 x^{26} - 18761650734737 x^{28} + 8689526543844 x^{30} -
    2701217093484 x^{32} + 529736859486 x^{34} - 58424749065 x^{36} +
    2737417248 x^{38})]/ $ \\ $[(-1 + 3 x^2) (-1 - 6 x - 9 x^2 +
        3 x^4) (-1 + 6 x - 9 x^2 + 3 x^4) (1 - 14 x^2 + 17 x^4) (1 -
        60 x^2 + 454 x^4 - 956 x^6 + 577 x^8) (-1 + 171 x^2 -
        5496 x^4 + 56617 x^6 - 240021 x^8 + 457923 x^{10} -
        420254 x^{12} + 186912 x^{14} - 37569 x^{16} + 2584 x^{18})]$ \\ \\
 $ - [x (-27 - 120 x + 2412 x^2 + 13164 x^3 - 63280 x^4 - 353312 x^5 +
   935740 x^6 + 4498708 x^7 - 9257992 x^8 - 31760720 x^9 +
   61308064 x^{10}+ 129398824 x^{11} - 258786202 x^{12} - 299911716 x^{13} +
   665027252 x^{14} + 378206596 x^{15} - 1001583056 x^{16} -
   271534384 x^{17} + 908429872 x^{18} + 119973056 x^{19} -
   499552689 x^{20} - 42554924 x^{21} + 161021916 x^{22} + 16091064 x^{23} -
   27349437 x^{24} - 4587912 x^{25} + 1798792 x^{26} + 520492 x^{27} +
   37947 x^{28})]/ $ \\ $ [(-1 + x) (-1 + 7 x^2) (-1 - 14 x - 44 x^2 +
       52 x^3 + 303 x^4 + 44 x^5 - 462 x^6 - 252 x^7 + 107 x^8 +
       88 x^9 + 13 x^{10}) (-1 + 14 x + 37 x^2 - 670 x^3 + 216 x^4 +
       7866 x^5 - 10202 x^6 - 27170 x^7 + 56210 x^8 + 5872 x^9 -
       66223 x^{10} + 22200 x^{11} + 25320 x^{12} - 12008 x^{13} -
       2888 x^{14} + 1256 x^{15} + 139 x^{16})] $ \\ \\
$ - \ds \frac{2 x^2 (43 - 790 x^2 + 6893 x^4 - 28692 x^6 + 62677 x^8 - 68086 x^{10} +
   29427 x^{12})}{(-1 + 3 x^2) (1 - 14 x^2 +
    17 x^4) (1 - 60 x^2 + 454 x^4 - 956 x^6 + 577 x^8)} +
     \frac{14 x^2}{1 - 7 x^2} $}
      \bc $ \mbox{--------------------------------------    } $  \ec   \noindent
      ${\cal F}^{MS}_{7}(x) = $  \parbox[t]{5.4in}{$ \ds
x + 27 x + 467 x^2 + 3027 x^3 + 63563 x^4 + 466327 x^5 + 9885365 x^6 +
 77761767 x^7 + 1576019355 x^8 + 13410238539 x^9 +
 254659534737 x^{10}  + 2348103031059 x^{11}  + 41707067330849 x^{12}  +
 413949610094255 x^{13}  + 6923710078740005 x^{14}  +
 73192857948396727 x^{15}  + 1163907957691068795 x^{16}  +
 12958299547681496355 x^{17 } + 197810440098798519899 x^{18}  +
 2295445771020812064223 x^{19 } + 33927267387694624015233 x^{20}  +
 406713599158978373001435 x^{21}  + 5862301338769843254084583 x^{22}  +
 72069951150114502314410907 x^{23}  +
 1018930653982635230381173233 x^{24}  +
 12771399952923354000903038527 x^{25 } +
 177918655304268284871263029425 x^{26}  +
 2263240929289963306546494172155 x^{27}  +
 31177815574907220476475328111613 x^{28}  +
 401075828612326194198324636969703 x^{29} +
 5478440052254244075488036134078215  x^{30} + \ldots $}
    \bc $ \mbox{--------------------------------------    } $ \ec  \noindent
 ${\cal F}^{MS}_{8}(x) = $  \parbox[t]{5.4in}{$ \ds
   -[x (16 x^{69}-106 x^{68} -3300 x^{67} +19069 x^{66} +322018 x^{65} -1520881 x^{64} -19540184 x^{63} +70000745 x^{62} +813402380 x^{61} -2016657129 x^{60} -24229609066 x^{59} +36282942825 x^{58} +525974379758 x^{57} -354587080920 x^{56} -8382141054110 x^{55} +96041957468 x^{54} +98434229981132 x^{53} +50073146080067 x^{52} -856318266137178 x^{51} -782642683316906 x^{50} +5572771563451640 x^{49} +6877679495110129 x^{48} -27497401899346396 x^{47} -41065478766170299 x^{46} +104459777429437818 x^{45} +178226150870613303 x^{44} -310495387974992080 x^{43} -582331242990297348 x^{42} +734215010668220576 x^{41} +1463222964474435545 x^{40} -1404402493924561124 x^{39} -2865484266560851312 x^{38} +2206696033233984392 x^{37} +4407996779606152990 x^{36} -2880891529521306050 x^{35} -5344844616793602333 x^{34} +3138031841625354084 x^{33} +5106446404933696315 x^{32} -2838392903240797068 x^{31} -3829413607863188624 x^{30} +2105703638706576880 x^{29} +2238182134050490742 x^{28} -1260082905713256950 x^{27} -1008935331527316393 x^{26} +597524837271834058 x^{25} +345742658107544312 x^{24} -220658210127275750 x^{23} -88285693006375070 x^{22} +62392443922387784 x^{21} +16326073430547716 x^{20} -13279291584669980 x^{19} -2092486193456287 x^{18} +2090580626011962 x^{17} +171726193316100 x^{16} -239303441030292 x^{15} -7302354821989 x^{14} +19606449076710 x^{13} -27460269177 x^{12} -1132815851256 x^{11} +19319259107 x^{10} +45302493186 x^9-875572995 x^8-1215183364 x^7+13400891 x^6+20549832 x^5+90062 x^4-191092 x^3-4164 x^2 +662 x +19 )]/ $ \\ $ \ds [(1 - 7 x + 14 x^2 -
         5 x^3 - 5 x^4 + x^5) (1 + 4 x - 10 x^2 - 10 x^3 + 15 x^4 +
         6 x^5 - 7 x^6 - x^7 + x^8) (1 - 4 x - 10 x^2 + 10 x^3 +
         15 x^4 - 6 x^5 - 7 x^6 + x^7 + x^8) (1 + 20 x + 122 x^2 +
         164 x^3 - 690 x^4 - 1751 x^5 + 480 x^6 + 3573 x^7 +
         1588 x^8 - 1385 x^9 - 1217 x^{10}  - 266 x^{11}  + 13 x^{12}  +
         11 x^{13} + x^{14}) (-1 + 14 x + 331 x^2 - 3474 x^3 -
         24357 x^4 + 237534 x^5 + 541266 x^6 - 6604103 x^7 -
         1905497 x^8 + 85855152 x^9 - 60009003 x^{10} -
         545836271 x^{11} + 672927757 x^{12} + 1747850343 x^{13} -
         2763674623 x^{14} - 2917536240 x^{15} + 5513512152 x^{16}+
         2653029943 x^{17} - 5852097578 x^{18} - 1465977019 x^{19} +
         3471750395 x^{20} + 568784352 x^{21} - 1167520145 x^{22} -
         154667330 x^{23} + 221656480 x^{24} + 23823457 x^{25} -
         24542626 x^{26} - 1818710 x^{27} + 1646233 x^{28} + 57030 x^{29} -
         66339 x^{30} + 348 x^{31} + 1479 x^{32} - 61 x^{33} - 14 x^{34} +
         x^{35})]$} \\ \\ \\
          \hspace*{1.2cm}   $-$ \parbox[t]{5.4in}{ $ \ds
 [2 x (7 x^{106} -2013 x^{104} +271473 x^{102} -22837326 x^{100} +1344929482 x^{98}y-59001771396 x^{96} +2003687770594 x^{94} -54052302641305 x^{92} +1179766678438145 x^{90} -21119750730820190 x^{88} +313342566058268162 x^{86} -3884627255914375276 x^{84} +40508469694380475941 x^{82} -357241905819153023588 x^{80} +2676454522797891427342 x^{78} -17099572626664510198179 x^{76} +93458497642559882402281 x^{74} -438122178754594179109914 x^{72} +1765238379012792344117924 x^{70} -6121625633315557702191520 x^{68} +18285573602229620700662612 x^{66} -47043716977223741089255092 x^{64} +104136439522045979737927530 x^{62} -197908247547795197818961672 x^{60} +321736263477162026054021799 x^{58} -444903527190320804983071579 x^{56} +518838713379473467639889021 x^{54} -503365156129050659894964377 x^{52} +396737493442153310733272690 x^{50} -241848026079896682187451681 x^{48} +99017128600728600175281028 x^{46} -8250124492992517280577465 x^{44} -27128803308564534282557705 x^{42} +28034219020492249326361451 x^{40} -17428199774635760630410343 x^{38} +8039292364003021973390218 x^{36} -2916939582709971330165897 x^{34} +852131433111988876145808 x^{32} -
 202548778297014269690756 x^{30} +39351344773252650496937 x^{28} -6255390685011652043507 x^{26} +812491167862693563189 x^{24} -85935201585787065880 x^{22} +7361461667826596399 x^{20}-  506869718718388616 x^{18} +27766085251949691 x^{16} -1193593543359977 x^{14} +39522727612450 x^{12} -982459882849 x^{10} +17678093168 x^8 -218447260 x^6 +1716966 x^4 -7681 x^2 +16)]/ $  \\ $[(-1 + x) (1 +
         x) (1 - 28 x + 134 x^2 + 1464 x^3 - 11646 x^4 - 8775 x^5 +
         234042 x^6 - 318372 x^7 - 1512042 x^8 + 3990140 x^9 +
         1327546 x^{10}  - 12508340 x^{11} + 8235416 x^{12} +
         11304952 x^{13} - 15649778 x^{14} + 1400926 x^{15} +
         6404612 x^{16} - 2944582 x^{17} - 312236 x^{18} + 418067 x^{19} -
         31381 x^{20} - 22903 x^{21} + 3184 x^{22} + 556 x^{23} - 97 x^{24} -
         5 x^{25} + x^{26}) (1 + 28 x + 134 x^2 - 1464 x^3 - 11646 x^4 +
         8775 x^5 + 234042 x^6 + 318372 x^7 - 1512042 x^8 -
         3990140 x^9 + 1327546 x^{10} + 12508340 x^{11} + 8235416 x^{12} -
         11304952 x^{13} - 15649778 x^{14} - 1400926 x^{15} +
         6404612 x^{16} + 2944582 x^{17} - 312236 x^{18} - 418067 x^{19} -
         31381 x^{20} + 22903 x^{21} + 3184 x^{22} - 556 x^{23} - 97 x^{24} +
         5 x^{25} + x^{26}) (-1 + 9 x + 92 x^2 - 581 x^3 - 2083 x^4 +
         11003 x^5 + 18456 x^6 - 89508 x^7 - 76454 x^8 + 363004 x^9 +
         148765 x^{10} - 782325 x^{11} - 101940 x^{12} + 931622 x^{13} -
         46433 x^{14} - 638214 x^{15} + 110532 x^{16} + 256374 x^{17} -
         68466 x^{18} - 59420 x^{19} + 20987 x^{20} + 7328 x^{21} -
         3425 x^{22} - 351 x^{23} + 281 x^{24} - 9 x^{25} - 9 x^{26} +
         x^{27}) (1 + 9 x - 92 x^2 - 581 x^3 + 2083 x^4 + 11003 x^5 -
         18456 x^6 - 89508 x^7 + 76454 x^8 + 363004 x^9 -
         148765 x^{10} - 782325 x^{11} + 101940 x^{12} + 931622 x^{13} +
         46433 x^{14} - 638214 x^{15} - 110532 x^{16} + 256374 x^{17} +
         68466 x^{18} - 59420 x^{19} - 20987 x^{20} + 7328 x^{21} +
         3425 x^{22} - 351 x^{23} - 281 x^{24} - 9 x^{25} + 9 x^{26} + x^{27})] $}
         \\ \\ \\
          \hspace*{1.2cm} $-$  \parbox[t]{5.4in}{ $ \ds
         [2 x (10 x^{54} -870 x^{52} +34330 x^{50} -816543 x^{48} +13122246 x^{46}-151378746 x^{44} +1300146874 x^{42} -8510330887 x^{40} +43108593514 x^{38} -170635526475 x^{36} +530739787448 x^{34} -1300043069551 x^{32} +2506100333450 x^{30} -3789498307175 x^{28} +4469718348562 x^{26} -4081472729978 x^{24}  +2859659499822 x^{22}  -1522831705010 x^{20} +610647347906 x^{18} -182640864270 x^{16} +40243099424 x^{14} -6410542019 x^{12} +717731134 x^{10} -54281527 x^8 +2630666 x^6 -76229 x^4 +1189 x^2 -10)]/ $ \\  $[(-1 + x) (1 + x) (-1 +
         9 x + 92 x^2 - 581 x^3 - 2083 x^4 + 11003 x^5 + 18456 x^6 -
         89508 x^7 - 76454 x^8 + 363004 x^9 + 148765 x^{10}  -
         782325 x^{11}  - 101940 x^{12}  + 931622 x^{13}  - 46433 x^{14}  -
         638214 x^{15} + 110532 x^{16} + 256374 x^{17} - 68466 x^{18} -
         59420 x^{19} + 20987 x^{20} + 7328 x^{21} - 3425 x^{22} - 351 x^{23} +
         281 x^{24} - 9 x^{25} - 9 x^{26} + x^{27}) (1 + 9 x - 92 x^2 -
         581 x^3 + 2083 x^4 + 11003 x^5 - 18456 x^6 - 89508 x^7 +
         76454 x^8 + 363004 x^9 - 148765 x^{10} - 782325 x^{11} +
         101940 x^{12} + 931622 x^{13} + 46433 x^{14} - 638214 x^{15} -
         110532 x^{16} + 256374 x^{17} + 68466 x^{18} - 59420 x^{19} -
         20987 x^{20} + 7328 x^{21} + 3425 x^{22} - 351 x^{23} - 281 x^{24} -
         9 x^{25} + 9 x^{26} + x^{27})] $}
         \\ \\ \\
          \hspace*{1.2cm}  $+$ \parbox[t]{5.4in}{ $ \ds
           [2 x (x^{14} -11 x^{12} +47 x^{10} -98 x^8+ 103 x^6-50 x^4+10 x^2+4)] / $ \\  $ [(1 + 4 x - 10 x^2 - 10 x^3 + 15 x^4 + 6 x^5 -
         7 x^6 - x^7 + x^8) (1 - 4 x - 10 x^2 + 10 x^3 + 15 x^4 -
         6 x^5 - 7 x^6 + x^7 + x^8)]$}
          \\ \\ \\
          \hspace*{1.2cm}  $-$ \parbox[t]{5.4in}{ $ \ds
          \frac{(2 x)}{(-1 + x) (1 + x)} $}

          \bc $ \mbox{--------------------------------------    } $  \ec

           \noindent
 ${\cal F}^{MS}_{8}(x) = $ \parbox[t]{5.4in}{ $ \ds  81 x + 681 x^2 + 19938 x^3 + 214413 x^4 + 7038371 x^5 +
 80708418 x^6 + 2638009924 x^7 + 32098743013 x^8 +
 1005558579804 x^9 + 13098751097381 x^{10} + 387302349839244 x^{11} +
 5413697605176834 x^{12} + 150727029170351741 x^{13} +
 2251447792306687144 x^{14} + 59281316927210554958 x^{15} +
 939149357512715316501 x^{16} + 23552269407202601688808 x^{17} +
 392311453180178638708326 x^{18} + 9443291793834966241192225 x^{19} +
 163992331457474964991423233 x^{20} +
 3816596095571035059285183652 x^{21} +
 68573510908910563426761767810 x^{22} +
 1552956371004992979029194886848 x^{23} +
 28678442547155385562725345847906 x^{24} +
 635435137052454220118946894076046 x^{25} +
 11994608414187692420665420022474599 x^{26} +
 261195385744336718141349209941059768 x^{27} +
 5016853080596606189611124150609224448 x^{28 }+
 107760278788336250059515948060261673009 x^{29} + $ \\$
 2098377771025447068273229367638586157098 x^{30} +  \ldots $}
    \bc $ \mbox{--------------------------------------    } $   \\  $ \mbox{--------------------------------------    } $ \ec  \noindent
 ${\cal F}^{MS}_{9}(x) = $  \parbox[t]{5.4in}{$ \ds
  -[4x^2(3601238948279971942407103824094471680 x^{148} $ \\ $  -474259471159898766783550645203555567096 x^{146} $ \\ $  +29022425468291559381088932380331238122228 x^{144} $ \\ $  -1102045327082881530315053625122143125119982 x^{142} $ \\ $  +29270581425690596819984954875235660805417960 x^{140} $ \\ $  -581208060624715919750429033694899207521444791 x^{138} $ \\ $  +9004001255736396380557126978506011770795971312 x^{136} $ \\ $  -112122736430791559024284020193334299425011245410 x^{134} $ \\ $  +1147318324233993667347681451338477668847431200326 x^{132} $ \\ $  -9812521844207589162729068830834278223401122908641 x^{130} $ \\ $  +71094037650858793997196724571050385503536763139418 x^{128} $ \\ $  -441151717292710558376915845909327731669642985735590 x^{126} $ \\ $  +2365731533743122569651970285447246215414865458702916 x^{124} $ \\ $  -11047139424567699096435516722891809058971213852319884 x^{122} $ \\ $  +45208871039705347602852171185943706648353113928365286 x^{120} $ \\ $  -163029478983338688006043906088027009913605965794128974 x^{118} $ \\ $  +520507099686402878979606641569975200152824572493002428 x^{116} $ \\ $  -1477341707945106267381691779282415141560207489224445637 x^{114} $ \\ $  +3740889015606875416950290158147101280505405433948366392 x^{112} $ \\ $  -8477319280331597691761870781703195177903922918621135544 x^{110} $ \\ $  +17239095646059741617655046828058940594291844990208203948 x^{108} $ \\ $  -31534024720479050017241309031006617959995696590842275242 x^{106} $ \\ $  +51995133798665055201527667542288082114741813134146315960 x^{104} $ \\ $  -77421469033339217477421998048769861882745272767864686948 x^{102} $ \\ $  +104273200485911034539674545575470924804136790549804456680 x^{100} $ \\ $  -127205047278705080394168209841694104451769672471294081418 x^{98} $ \\ $  +140728946642557850360731567426506013894357696768241068190 x^{96} $ \\ $  -141339598802117208570311080695382157164502029207159303542 x^{94} $ \\ $  +128983006480943298029155079430726652070978003935100310108 x^{92} $ \\ $  -107032235833769809806043314701539889030820727184420259155 x^{90} $ \\ $  +80812438948113130909132549950570218540859954167249688590 x^{88} $ \\ $  -55544106147779976389121590023450222652712764735321887206 x^{86} $ \\ $  +34766282141988819566642009520825284913501446393886483404 x^{84} $ \\ $  -19822372432338823932609146237921913321395847754136599787 x^{82} $ \\ $  +10296770356875946903898446707381386928143066798669453834 x^{80} $ \\ $
   -4873277579009059611647285267943665773609114951409821080 x^{78} $ \\ $  +2101355679544073267045751996873638803505223445771604500 x^{76} $ \\ $  -825424373735364456668440088772133842922158790229701176 x^{74} $ \\ $  +295292805814545649214013627674575322689120283075002592 x^{72} $ \\ $  -96179398030055520953947740433244806570514904004387838 x^{70} $ \\ $  +28508671737510221755889925631416863181848140566391012 x^{68} $ \\ $  -7686079213481637816140497035436298760109009082101753 x^{66} $ }

   \noindent
   \hspace*{1.4cm}    \parbox[t]{5.4in}{
         $ \ds +1883589053582694325166193115837859651183600763819318 x^{64} $ \\ $  -419268631118072182297281619041879429933812317157668 x^{62} $ \\ $  +84692156660452862075933436495539703994338708267852 x^{60}
         -15509695504934434528817567228852728333156413545316 x^{58} $ \\ $  +2572044468709094711531577162700109103981487343524 x^{56} $ \\ $  -385755119424184678696259117404499509392709088016 x^{54} $ \\ $  +52249038044371347891325325250847539297052211178 x^{52} $ \\ $  -6380887727308724840788113418175252944724394300 x^{50} $ \\ $  +701366584492067369966091751501453717985633110 x^{48} $ \\ $  -69248687598713212959889742224975948877523164 x^{46} $ \\ $  +6128187168746732699834238754525465251446260 x^{44} $ \\ $  -484910061729718007400292355395071072020176 x^{42} $ \\ $  +34217560001928996192905822531236744279340 x^{40} $ \\
         $ \ds   -2146996313569411441872268292963953461030 x^{38} $ \\ $  +119403696662056129697656199804407886622 x^{36} $ \\ $  -5865124969491300118717654888845705199 x^{34} +253464984182842123124923624969096606 x^{32} $ \\ $  -9595362071924435538971864800207280 x^{30}   +316670411645212935155666148651968 x^{28} $ \\ $  -9061253896687662876071573146384 x^{26}  +223411293208638068122042091662 x^{24} $ \\ $  -4712421429044314029064168310 x^{22}  +84325967173965384196768590 x^{20} $ \\ $  -1267430664692520809946133 x^{18}   +15808459062940762176132 x^{16} $ \\ $  -161207417534678200784 x^{14}  +1318985590268390020 x^{12} $ \\ $  -8450248836636137 x^{10}   +41032351364484 x^8-144316319446 x^6+343725546 x^4-494857 x^2+330)]/   $}    \\ \\
 \hspace*{1.4cm}    \parbox[t]{5.4in}{
         $ \ds [(-1 +3 x) (1 + 3 x) (-1 + 7 x^2) (-1 + 27 x - 258 x^2 + 973 x^3 -
         324 x^4 - 6744 x^5 + 12454 x^6 + 7362 x^7 - 33489 x^8 +
         15667 x^9 + 19476 x^{10} - 20253 x^{11} + 3729 x^{12} +
         2151 x^{13} - 966 x^{14} + 107 x^{15}) (1 + 27 x + 258 x^2 +
         973 x^3 + 324 x^4 - 6744 x^5 - 12454 x^6 + 7362 x^7 +
         33489 x^8 + 15667 x^9 - 19476 x^{10} - 20253 x^{11} -
         3729 x^{12} + 2151 x^{13} + 966 x^{14} + 107 x^{15}) (-1 + 282 x^2 -
         25297 x^4 + 1073828 x^6 - 25390104 x^8 + 363078264 x^{10} -
         3296168948 x^{12} + 19596248926 x^{14} - 77743595826 x^{16} +
         207473742096 x^{18} - 371467677512 x^{20} + 439852643504 x^{22} -
         334403214576 x^{24} + 154506971212 x^{26} - 38880116221 x^{28} +
         4020609392 x^{30}) (-1 + 591 x^2 - 95361 x^4 + 6922350 x^6 -
         272760016 x^8 + 6420903257 x^{10} - 95828390004 x^{12} +
         943120716586 x^{14} - 6295303405260 x^{16} +
         29127727204334 x^{18} - 95000584032342 x^{20} +
         220996143367594 x^{22} - 369011778492872 x^{24} +
         442338529800436 x^{26} - 377839526215177 x^{28} +
         226022147292981 x^{30} - 91695292950038 x^{32} +
         23818609428605 x^{34} - 3541207333504 x^{36} +
         226978239492 x^{38}) (1 - 1193 x^2 + 376246 x^4 -
         48953410 x^6 + 3288145988 x^8 - 127374411928 x^{10} +
         3015668747782 x^{12} - 45191010425846 x^{14} +
         441384780778588 x^{16} - 2898283223877346 x^{18} +
         13154666974580501 x^{20} - 42187756055653825 x^{22} +
         97142224830553641 x^{24} - 162322162033938237 x^{26} +
         198042290945862570 x^{28} - 176799585485005402 x^{30} +
         115298993750386955 x^{32} - 54611642383339285 x^{34} +
         18586566465572115 x^{36} - 4467086405032683 x^{38} +
         737944576901349 x^{40} - 80312266104179 x^{42} +
                 5368435066393 x^{44} - 193455857453 x^{46} +
         2735506380 x^{48})]   $} \\ \\
 \hspace*{1.2cm} $-$ \parbox[t]{5.4in}{
         $ \ds
          [(x (-953149365383346717696 x^{98}+26114673044757854577152 x^{97}-88557437460023887090048 x^{96}-2228203846243744835990144 x^{95}+11536959510745258379188192 x^{94}+93347741439551471495673328 x^{93}-504874736014387117801293440 x^{92}-2565341866801000262660549696 x^{91}+12531151351266541072682492400 x^{90}+51137498352941451346463182864 x^{89}-203878823875170976725652113576 x^{88}-767933760238339305642036596512 x^{87}+2321872351568696587632779117502 x^{86}+8835331057636262652376567752738 x^{85}-19121083172126041637000524041666 x^{84}-78797209352098489028589254914996 x^{83}+114999568849267609738967914891967 x^{82}+550622261992834949317717734554862 x^{81}-495803969561958312991777971971017 x^{80}-3045528243642151002582270910473608 x^{79}+1398846804934161918094759164538620 x^{78}+13456757745193420346086689826521562 x^{77}-1440724690623504240284659391427732 x^{76}-47883984558548846561959812362438988 x^{75}-8666777473755518366078147016737591 x^{74}+138169310608946998129913783915025470 x^{73}+58701180282927196132848480009717373 x^{72}-325191764176658349443591714223440620 x^{71}-208426897357032574335148341822310600 x^{70}+627310979023675251349977976120424674 x^{69}+529247970253200429314621467669176296 x^{68}-995694549952524453741104080452452028 x^{67}-1042665061435475970736728878808964563 x^{66}+1303977813890105101301182555644939918 x^{65}+1652470656609510909465179367232994533 x^{64}-1410868675816252645266095520674557348 x^{63}-2150938614707116088899214117558459760 x^{62}+1260204556387439730429655116660212430 x^{61}+2331426428125569394551566891263609020 x^{60}-925507484687189259226850702117561324 x^{59}-2125635766638185629665544565075409631 x^{58}+553464127474067796010428013354004478 x^{57}+1642809964940764031122254824140166389 x^{56}-263901363736577168136268569457662644 x^{55}-1082794795624762266989253691646310000 x^{54}+95525720371886598673048877875053286 x^{53}+611514178852636175427811806134089436 x^{52}-22562427017346103638783607792072176 x^{51}-296953691001577521233627775141211039 x^{50}+701156032080053256854692625098098 x^{49}+124285327407638717604718297599384893 x^{48}+2307829441426671559852895620667892 x^{47}-44889291031693490526281381208715376 x^{46}-1276658600697241065463703125617264 x^{45}+13993291530023039253131711080350588 x^{44}+411924352281985537035900320026200 x^{43}-3761340569909269327666212843998926 x^{42}-88491324616143585956501121986432 x^{41}+870066214250910712707063555076186 x^{40}+10884547785003587249149206819668 x^{39}-172693139142129440119278563641174 x^{38}+341202461030475623600025101952 x^{37}+29300190012360104533601256734494 x^{36}-541612600905852068267380595492 x^{35}-4230465082419657475304340501464 x^{34}+151393223256430559125259462324 x^{33}+517165425037228664327652944816 x^{32}-27343465083322387416381358300 x^{31}-53236452685820726290887653208 x^{30}+3670189334446798039991188294 x^{29}+4587904751249978039372253020 x^{28}-382494024246534764329357228 x^{27}-329049186673175730033538655 x^{26}+31477712381545053004787302 x^{25}+19522487159509485410816637 x^{24}-2057755279113643060500828 x^{23}-952369858858997313846602 x^{22}+106876557884196527325124 x^{21}+37962986213100238440250 x^{20}-4394660877077409655812 x^{19}-1228066093988736853534 x^{18}+142091991144588773900 x^{17}+31973538805566877158 x^{16}-3575035883148707372 x^{15}-662662275265276628 x^{14}+68933167944024434 x^{13}+10764844178128464 x^{12}-996260840605620 x^{11}-134045501666733 x^{10}+10444687656494 x^9+1238550887327 x^8-75552258680 x^7-8090032822 x^6+347162866 x^5+34589142 x^4-863736 x^3-84055 x^2+754 x+81)] /
        $} \\ \\
 \hspace*{1.4cm}  \parbox[t]{5.4in}{
         $ \ds
         [(1 + x) (-1 + 2 x) (1 + 2 x) (1 - 8 x^2 + 11 x^4) (1 - 28 x^2 + 71 x^4) (1 - 8 x +
         20 x^2 - 14 x^3 - 5 x^4 + 4 x^5) (1 - 116 x^2 + 1546 x^4 -
         4556 x^6 + 3781 x^8) (-1 - 48 x - 708 x^2 - 1640 x^3 +
         45946 x^4 + 341664 x^5 - 270194 x^6 - 9936934 x^7 -
         22840421 x^{8} + 84442096 x^9 + 418028192 x^{10} +
         132842204 x^{11} - 2217116283 x^{12 }- 3910256042 x^{13} +
         2255114380 x^{14} + 12952834812 x^{15} + 9797088008 x^{16} -
         9926380870 x^{17} - 19739579992 x^{18} - 6238648278 x^{19} +
         8703028142 x^{20} + 7538491050 x^{21} + 208212 x^{22} -
         2158389564 x^{23} - 641271414 x^{24} + 192328290 x^{25} +
         111002758 x^{26} + 399556 x^{27} - 5772921 x^{28} - 305408 x^{29} +
         97576 x^{30}) (1 - 42 x - 75 x^2 + 17718 x^3 - 112036 x^4 -
         2032663 x^5 + 20898099 x^6 + 68505152 x^7 - 1367337026 x^8 +
         1223952387 x^9 + 39924985841 x^{10} - 123520021512 x^{11} -
         507429891185 x^{12} + 2815200597001 x^{13} +
         1444158632036 x^{14} - 29150884549434 x^{15} +
         26197136927827 x^{16} + 152864286212934 x^{17} -
         281984662214835 x^{18} - 398046044696312 x^{19} +
         1249282231645481 x^{20} + 358977981283233 x^{21} -
         3123666474218858 x^{22} + 675827827012992 x^{23} +
         4889486045969022 x^{24} - 2485484028220910 x^{25} -
         5032512202755210 x^{26} + 3593534888276858 x^{27} +
         3475263477225909 x^{28} - 3098930934608376 x^{29} -
         1601027066102163 x^{30} + 1744452573290918 x^{31} +
         470550625308136 x^{32} - 657690139223493 x^{33} -
         75567716491829 x^{34} + 165305848427048 x^{35} +
         1498941418335 x^{36} - 26855694880119 x^{37} +
         1821955387786 x^{38} + 2634688847668 x^{39} -
         334995651352 x^{40} - 134705018616 x^{41} + 24157589768 x^{42} +
         2269409656 x^{43} - 605268248 x^{44} +
         22971944 x^{45})]
          $} \\ \\
 \hspace*{1.2cm} $-$  \parbox[t]{5.4in}{
         $ \ds [2 x^2 (41066587866652222898880 x^{68}-992758353952415482125432 x^{66}+11077609155778996489726116 x^{64}-75937875357304272686739204 x^{62}+359148152735756018858079969 x^{60}-1247813989255646700068478270 x^{58}+3311448052135607522000845098 x^{56}-6890903677759089743304624336 x^{54}+11455026870286109405904696648 x^{52}-15419776399897881394577426496 x^{50}+16979006361774608245851454992 x^{48}-15409236371634600666485663796 x^{46}+11590835972360315745039290048 x^{44}-7255341364039829631336457944 x^{42}+3789484771314030044433891392 x^{40}-1654057453037482613678691232 x^{38}+603652240778687392509345866 x^{36}-184114820410242547277514040 x^{34}+46866285726514066157069562 x^{32}-9933989942874187006372964 x^{30}+1747920002743414438604224 x^{28}-254279893889486283307536 x^{26}+30432921243932040036704 x^{24}-2978546929594730451164 x^{22}+236662835103954604904 x^{20}-15130466428085744180 x^{18}+769776921839397140 x^{16}-30730959661308308 x^{14}+945385372794725 x^{12}-21881906116106 x^{10}+369035694784 x^8-4338348484 x^6+33333873 x^4-150766 x^2+318)]/ $ } \\
         \hspace*{1.4cm}  \parbox[t]{5.4in}{ $ \ds
          [(-1 + 3 x) (1 + 3 x) (-1 +
         282 x^2 - 25297 x^4 + 1073828 x^6 - 25390104 x^8 +
         363078264 x^{10} - 3296168948 x^{12} + 19596248926 x^{14} -
         77743595826 x^{16} + 207473742096 x^{18} - 371467677512 x^{20} +
         439852643504 x^{22} - 334403214576 x^{24} + 154506971212 x^{26} -
         38880116221 x^{28} + 4020609392 x^{30}) (-1 + 591 x^2 -
         95361 x^4 + 6922350 x^6 - 272760016 x^8 + 6420903257 x^{10} -
         95828390004 x^{12} + 943120716586 x^{14} - 6295303405260 x^{16} +
         29127727204334 x^{18} - 95000584032342 x^{20} +
         220996143367594 x^{22} - 369011778492872 x^{24} +
         442338529800436 x^{26} - 377839526215177 x^{28} +
         226022147292981 x^{30} - 91695292950038 x^{32} +
         23818609428605 x^{34} - 3541207333504 x^{36} +
         226978239492 x^{38})]  $} \\ \\
 \hspace*{1.2cm} $-$ \parbox[t]{5.4in}{ $ \ds
         [8 x^2 (2952961 x^{16}-6624312 x^{14}+6193068 x^{12}-3126324 x^{10}+924170 x^8-161984 x^6+16308 x^4-852 x^2+21)]/ $ \\ $ [(-1 + 2 x) (1 + 2 x) (1 - 8 x^2 + 11 x^4) (1 -
      28 x^2 + 71 x^4) (1 - 116 x^2 + 1546 x^4 - 4556 x^6 +
      3781 x^8)] $ } \\ \\
       \hspace*{1.2cm}  $+$ \parbox[t]{5.4in}{ $ \ds
        \frac{18 x^2}{1 - 9 x^2} $}
 \bc $ \mbox{--------------------------------------    } $ \ec  \noindent
  ${\cal F}^{MS}_{9}(x) = $ \parbox[t]{5.4in}{ $ \ds
81 x + 2977 x^2 + 43869 x^3 + 2148297 x^4 + 36446086 x^5 +
 1859277661 x^6 + 33175969740 x^7 + 1670274252969 x^8 +
 31407641485041 x^9 + 1519348101693492 x^{10} +
 30342339901944756 x^{11} + 1393561851576932601 x^{12} +
 29630625622653835328 x^{13} + 1288717962120945898600 x^{14} +
 29100441987912020023684 x^{15} + 1201843312383267583490073 x^{16} +
 28664530577970597174268947 x^{17} +
 1130098720537853186076017611 x^{18} +
 28278438657454112242752450823 x^{19} +
 1070832065514275801677136628332 x^{20} +
 27919545617813707754526340605069 x^{21} +
 1021703749526691895346459783179664 x^{22} +
 27576355061298160788271199625774163 x^{23} +
 980734176066852268972595763904905297 x^{24} +
 27243023028767511712725936258610836086 x^{25} + $ \\ $
 946292255083398990605453697630600206206 x^{26} + $ \\ $
 26916571107441928663000559345794173964329 x^{27} + $ \\ $
 917056053132788309820000325710205811761252 x^{28} + $ \\ $
 26595471438495138268329427549921452474114356 x^{29} + $ \\ $
 891965127088014389729702685438179850576819096 x^{30}
+ \ldots $}
 \bc $ \mbox{--------------------------------------    } $  \\ $ \mbox{--------------------------------------    } $ \ec  \noindent

    ${\cal F}^{MS}_{10}(x) = $   \small \parbox[t]{5.4in}{$ \ds
    -[x (-32 x^{251}-490 x^{250}+45052 x^{249}+664162 x^{248}-31046724 x^{247}-433169136 x^{246}+13947253330 x^{245}+180856444523 x^{244}-4586038572056 x^{243}-54250667211984 x^{242}+1174202751866162 x^{241}+12437520421477397 x^{240}-243088787307593520 x^{239}-2263299964542167788 x^{238}+41710701708969336704 x^{237}+335097457740003359287 x^{236}-6034339378757215563492 x^{235}-41040912610306664024866 x^{234}+745224775252874107463670 x^{233}+4202961118542995434945104 x^{232}- 79297168045686211187696724 x^{231}-362037946881977529805327077 x^{230}+7322948425458736928185396712 x^{229}+26252948025995287180735769606 x^{228}- 590347995375327696410408765078 x^{227}-1592765877111985321707826169208 x^{226}+ 41747985732987715922859524086384 x^{225}+79316301812119820453700773442368 x^{224}-2600629641242444488332871409945568 x^{223}- 3077336492668131041106249316024592 x^{222}+143228606217632762645359663439568494 x^{221}+77286633396027622346414118047675813 x^{220}-6997169948399738720992771100961224066 x^{219}+ 282224504707283222080661229733789803 x^{218}+
    304137807650666245661429804743957342906 x^{217} - $    }

    \noindent
   \hspace*{1.4cm}    \parbox[t]{5.4in}{      $ \ds 171992954215765211300257561673486407123 x^{216} -11794966351451802767004207837950194455140 x^{215} $ \\ $ +12319875393980380622450002935333638531515 x^{214} $ \\ $ +409216251015034028509859159700416883707142 x^{213} $ \\ $ -607817079961942094515488735137050386778341 x^{212} $ \\ $ -12733048270376626516914991783889224142292710 x^{211} $ \\ $ +24133631450184358870398706985910988935432459 x^{210} $ \\ $ +356188224217281187724527372325651938038524656 x^{209} $ \\ $ -812607746298533133211559307243628412849469942 x^{208} $ \\ $ -8978328498062223467143390278996454657675722756 x^{207} $ \\ $ +23789956973315977766699580887502170258837402447 x^{206} $ \\ $ +204381800743779745494785077727391189859282863806 x^{205} $ \\ $ -614382627666180127639862917329319299754198713916 x^{204} $ \\ $ -4210619200553091165181471471542024084259887198472 x^{203} $ \\ $ +14129970915511009930024023971739813184431426839322 x^{202} $ \\ $ +78668390290854616048388223724989040640823236066678 x^{201} $ \\ $ -291382594644746576122165161084315980115056774998981 x^{200} $ \\ $  -1335579084740349778817270562501901846561099032446490 x^{199} $ \\ $ +5415984761210543615706209943388029851249076975266915 x^{198} $ \\ $ +20644014375263051334112701883929259460173174010153858 x^{197} $ \\ $ -91119091321873373012796213766395532675776009642679650 x^{196} $ \\ $ -291068979954003817352798094154354683323466857863243824 x^{195} $ \\ $ +1392453872763308285144722597061950107095189448231344158 x^{194} $ \\ $ +3750481598841579697524467487884187770371037895117082164 x^{193} $ \\ $ -19386160841147320702914177022312934019808277490711188226 x^{192} $ \\ $ -44246639106380583334856771419962193745428264156547757508 x^{191} $ \\ $ +246532483216281124582841420958536424963473248097833042295 x^{190} $ \\ $ +478853221543793165542095225168991208820029479189838994162 x^{189} $ \\ $ -2870293081267581625661800558944308966892104666798773620684 x^{188}-4763324680913166659286516899797988842495491538676125701194 x^{187}+30657513538967887055203409032941978699634480621446137893682 x^{186}+43642432814308841433768125561282344086115919595876154551502 x^{185}-300955446498323507179337517554076601833619076127986834964092 x^{184}-369124249805317011413766222118604242819143717086899748826384 x^{183}+2719802503233265644902136663147327248420935510469791798583925 x^{182}+2889098837572387048882492768924255822926865632111628944135398 x^{181}-22661361376180142625415324919702080629088771847726389294793632 x^{180}-20981717786002127973573912887169052908331644281745293148171864 x^{179}+174311820497515312120229397431335877154785238035152522385981136 x^{178}+141799500977729889366136861505385920564097267744354699574468112 x^{177}-1239307861144558565862609497349703208821743567012743348889454303 x^{176}-894572851891631360148860427921373476959998764641712274046690612 x^{175}+8152809917280400947302662093517156127431575661639309885460581153 x^{174}+ $    }

    \noindent
   \hspace*{1.4cm}    \parbox[t]{5.4in}{      $ \ds 5285146531275407872234991898620804467122276340912108663669426320 x^{173}-49673406073438849135102918272857920750966339522986553368622553826 x^{172}-29332831668497692753316159102388698085100381785095615353809637482 x^{171}+280539165697057636073502010949809564939545357789375484908943590908 x^{170}+153361873196672030043683827615007313736318663039584429278110354692 x^{169}-1469720844416898145859792931188371043995333147179899220283103943593 x^{168}-756998265461855490645040412116909313650482309153825277604498963022 x^{167}+7147000288994245075832977567310465843989583665996070703770235302780 x^{166}+3532388856689871503627355823232417332851932351332749164780346840918 x^{165}-32277168481263047148929550743277768561685590639696527407056579095003 x^{164}-15588315220262262461978929587005138571504320428563168008193292983240 x^{163}+135439805776948989039068266615200797251954479613339287022736703830126 x^{162}+65018407362926767038356738874690404866595420543003864603763864768768 x^{161}-528246499108523413872039924248014342091228583314283734217393752612480 x^{160}-255970582146764139518802571229527690426075233846976601160488231180972 x^{159}+1915543469062274675282007383788789112531136438299964420592469091288626 x^{158}+949386541505485059221093663660941527026093821410824748158394785461418 x^{157}-6459597583976182927919191240886047225191196510839118303246297073268669 x^{156}-3310330144231420360592497591616036163445787576122972625350224398448182 x^{155}+20259975407883213921607854066709042044495859604187246513867797118380161 x^{154}+10827977482106599763161607378969959956928809353873984930378325032691246 x^{153}-59105056385135590753811896678091317838224281744629646757763386766725746 x^{152}-33159883563614995906703262074590109268929023497468401524914723535523826 x^{151}+160385966369574262381883736885303002702580631898602337369105202906837382 x^{150}+ 94911683562035274865329420352205410854239491516204873724111599579567792 x^{149}-404799605499906899120266410280585318921589980123042461313806232634060396 x^{148}-253537655210641175295602470130857732357894553962368842172879319599835346 x^{147}+950154806308051574680683585155643172986717250115005717108371631983324841 x^{146}+631359198298414945355053063837921326945655906440367779772052804116988004 x^{145}-2073720803026688476525286045178458236304614431572035849491675372280519665 x^{144}-1464277875107759686419524539063629621375061272465286609782558539348556688 x^{143}+4207271825542383961000636509422740546241916593750498299491920764898593626 x^{142}+3160633317328157511047530664261622709286365188832631218775623447777024962 x^{141}-7932464208160031549134726559330604885999529142358850157950735118616164759 x^{140}-6345853143242121146074065766979991921159492390419623894958728792539698994 x^{139}+13893236053807234530592156573354896181767873985742259190886068154734919691 x^{138}+11846427568203946149069509726026218995087005411791035194558612221564529572 x^{137}-22593352825407561936644602715974576383134232783002913167532045979719681503 x^{136}-20555257043451629571130682188210642359077974915659961591813367269085483420 x^{135}+34095128603754435776288368539904154828536239730293173536544292390374873272 x^{134}+33142258190350440790195259190034856965311986411071000445545615209901164132 x^{133}-47713466222299293015985339768285187456416158040600134112418718308349907068 x^{132}$    }

   \noindent
   \hspace*{1.4cm}    \parbox[t]{5.4in}{      $ \ds -49644398978466401926241582656168894228631268322590102994991345374460233916 x^{131}+61868281832996750214009291485197419842925098301715691724387782163025792309 x^{130}+69072524129053348201904240791760451285141805138914126975794364097973269352 x^{129}-74257039155826033103242246491236014685575948742012034509305144097065345056 x^{128}-89250867373326333876838999938200853870687473859551010004136684347013433894 x^{127}+82397213263780879850884666016933485734786040964326576107946717704847638365 x^{126}+107082845629342917258623899205996190753955481868900411363210979391221275016 x^{125}-84395425486659217007458994638308139873389746284167790177857653148840552363 x^{124}-119276850051111852356847946853801340021903889797079111104981074704430146670 x^{123}+79632748485839188006741186166343683142199232356911649735891154989069004076 x^{122}+123324031199571580835792612773307621906671354729733917994834881791151327562 x^{121}-69037871146574001855918389660787771402823134039946436359555501392648304756 x^{120}-118336196549377511420998021585895795185997126039999289214200576326901303846 x^{119}+54794601080374484043906373610241871403653162219010243315769509072820711021 x^{118}+105362116507845769656487527127143223746829900308642640241152341249728819282 x^{117}-39608611510653800589216913260784164602960039306557078117589089658034515327 x^{116}-87028606079637546892430030201885379159427413388832392983846342040362132858 x^{115}+25870419832609151366255216879836120841349082199490293894060526961135711363 x^{114}+66674336610610908342718260486470376589658217392200891263769747375548949546 x^{113}-15068938701509159405513423698173792271680467289151337673224443518482009005 x^{112}-47367210124019550645870309075239278305109285067130214610950591217043277736 x^{111}+7638716098024710973022270076075100182144445070160086626989811773323920879 x^{110}+31197348950665974091595690126410334390215362008013728749321468051404373252 x^{109}-3190303313507916484751008539318484861951501105584221674248374683604977886 x^{108}-19044615272579330625003594488095034936363846679685821360850280161942075088 x^{107}+920346913590899688970753220469148528330486247037305099403176354883613885 x^{106}+10772866566820613522481634434395859352404532404757653047863787455016542686 x^{105}+12098606429879860496975533979886775447811741687169680427010002528369988 x^{104}- 5645230746935844057217423016347587350465180449314341803566460592875750548 x^{103}-262993362163563682748434449929181561198835611789505893128706571244771205 x^{102}+2739731326514097096190904127478041141074956089939771821956245211732801986 x^{101}+242462576613490203068847981846506972874934039683928027095351764934310482 x^{100}-1231100250410500198350533026560371889632311265624791986133833640819970798 x^{99}-156633267655940791689586471415885891712567791186828724729896761479840962 x^{98}+512060544004276143752903754326722238236310662858125174342184411565534974 x^{97}+83515791749751504975699538759764691732341835268366048685144033694698659 x^{96}-197095286478846899154220734830831557272724555892166711331953554366429658 x^{95}-38702398571196769370733681364578250205773208277722899933188977394080745 x^{94}+70185196883915262872053909044930008612919777447933404783621346494150330 x^{93}+15949129428658702937670488225427191883501616850846349889533545367909805 x^{92}-23116431007388075993649288750622989768270246272543142908562934288846686 x^{91}-5915530558965973252084357523734844881687237214751658059859004888248272 x^{90}+
   $    }

   \noindent
   \hspace*{1.4cm}    \parbox[t]{5.4in}{      $ \ds
   7040435412733505009776214815462480396285275227693347882250733856305272 x^{89}+1988601959252906277480529334271348911241358292412369645452029745325478 x^{88}-1982378890100438484775112327018164884232939118247376936444918009272614 x^{87}-608527378553502111554798372509180558920826279565140895729360249184579 x^{86}+515936180494635309877748335857688188649911334383152272519614563953422 x^{85}+169978978760766160844613454056520277481818610046737803244806630224797 x^{84}-124094090334226487986376706457556464188821384827348727380898649586662 x^{83}-43418311663874093985567293373879737770030341997034478774223538826384 x^{82}+27579467570788420390500871388425260192723278776742795135741992377032 x^{81}+10153325448217632006315464123150582473307934400233924499972603804270 x^{80}-5662964001969150496952633524209866819254868852791939987603666701652 x^{79}-2175174821576672209587770974062941146984169769888616730861276410204 x^{78}+1074176995724727282326552265120340740845484380470302059760202872030 x^{77}+427042901525414685135713126200839893877759929651166379390496189976 x^{76}-188207236838326016632160140953575996293240182746127514217853558780 x^{75}-76835309751722336498306542416442907531518049013135463788108598113 x^{74}+30456221766934503189287883385460606056393627856597171232217445338 x^{73}+12667233274534083771894388287885992486738162586480808973490145787 x^{72}-4551266893725163494864734362443439472757976749629166092254281170 x^{71}-1912778945472513387567072125653552053414344671446697551786774359 x^{70}+627936191116713167360894358512020663447693096544065495950256284 x^{69}+264397904975325957211943416459610720482649626932319907341354187 x^{68}-79963552544021481384215766739958937966505508002031674129517564 x^{67}-33429528353135334790034860578052766541091796737730841331002360 x^{66}+9394304130068268381726768527195673677185797965600914334218336 x^{65}+3862526297350995536656932199994482670170515487742329856521520 x^{64}-1017537335605569839646377926974593976720351612707887125221268 x^{63}-407372063778659884672826584415068413249975056111510665940363 x^{62}+101521881179576701769709466020919831725826786351199486359858 x^{61}+39166613252555474488960128373734995069934107620985853415999 x^{60}-9319159983706976954214894582575891972586698577868569291254 x^{59}-3427558093470429587296378458822950054819662124673685520458 x^{58}+785859468390309876391912067248871471850538684719138572500 x^{57}+272548809932565537384098310058416136294204669816916297489 x^{56}-60765515931589757070200646522176274956154618145266128570 x^{55}-19653626074111338283017500436302166779873471589652682754 x^{54}+4298870402190978852982760261361727420827640981034378576 x^{53}+1282398849228981622761622372323589726123489574353575959 x^{52}-277541702468794444453292713176650012596380929201966280 x^{51}-75529350139956880521038480703659187511924144780917603 x^{50} $ \\ $ +16305353949605219176101593288883238158231467543547826 x^{49} $ \\ $ +4004321427719720873899361167177406536067484526553193 x^{48}$    }

   \noindent
   \hspace*{1.4cm}    \parbox[t]{5.4in}{      $ \ds -868924422065040691434478639248236561187995728338972 x^{47} $ \\ $ -190519283863720980807148350510306775977463477785764 x^{46} $ \\ $ +41858801347597820517419524689911332924701082452300 x^{45} $ \\ $ +8107330450039774260655440255161272454253099593340 x^{44} $ \\ $ -1816096021954754651413591903131078381731270548598 x^{43}  $ \\ $ -307410850500532756960305519701754252770493299392 x^{42} $ \\ $ +70685378893463013737381280333557287550155707404 x^{41} $ \\ $ +10343322751702715605436520242839707069561586012 x^{40} $ \\ $ -2457820343923407812327012892611551959148848280 x^{39} $ \\ $ -307396971836508984974473222367024862710398731 x^{38} $ \\ $ +76013090032333846802488970524268149399030492 x^{37} $ \\ $ +8028070190506233160675351280418036737118811 x^{36} $ \\ $ -2081237662643148349343338127608426382459798 x^{35} $ \\ $ -183192481706820418844797630474336451764682 x^{34} $ \\ $ +50199986057943732783453109365829677638946 x^{33} $ \\ $ +3629079460543821513251439573961225072483 x^{32} $ \\ $ -1061070939942330035142960501676683823986 x^{31} $ \\ $ -61961596334968465422309339823608969622 x^{30} $ \\ $ +19542643429831579447993706571788260332 x^{29}+904269231261127998194767298205800738 x^{28}-311709069260363250251696726383783488 x^{27}-11174298026848311497831255691342299 x^{26}+4276688379766372836916890198565848 x^{25}+115666506078113843231732543285491 x^{24}-50093539919800877709071405722160 x^{23}-990881453042917703333320029873 x^{22}+496645781161863721248018799186 x^{21}+6938053534118136835461821918 x^{20}-4126437790781957217047408768 x^{19}-39328351858168236805821024 x^{18}+28393432299941148742837432 x^{17}+181163936481532315387292 x^{16}-159471324966255973616714 x^{15}-706841110728069951800 x^{14}+717882285530978248402 x^{13}+2578967578254945203 x^{12}-2529459415781845334 x^{11}-9491409923223955 x^{10}+6754960249210514 x^9+32285495443523 x^8-13053710287562 x^7-84847183417 x^6+16976021872 x^5+146997008 x^4-13025010 x^3-140812 x^2+4288 x+51)]/
     $} \\
          \hspace*{1.4cm}   \parbox[t]{5.4in}{ $ \ds
          [(1 + x) (-1 + 35 x - 473 x^2 + 3042 x^3 - 8357 x^4 - 3077 x^5 +
      69776 x^6 - 115677 x^7 - 82259 x^8 + 385137 x^9 - 228633 x^{10} -
      268530 x^{11} + 373867 x^{12} - 98551 x^{13} - 64909 x^{14} +
      49836 x^{15} - 12711 x^{16} + 1135 x^{17} + 73 x^{18} - 20 x^{19} +
      x^{20}) (-1 + 15 x + 195 x^2 - 2476 x^3 - 9408 x^4 + 128774 x^5 +
      151702 x^6 - 3080005 x^7 - 152040 x^8 + 39805335 x^9 -
      22147982 x^{10} - 300921194 x^{11} + 284159318 x^{12} +
      1383107908 x^{13} - 1722027429 x^{14} - 3930756397 x^{15} +
      6069754917 x^{16} + 6915053418 x^{17} - 13500672554 x^{18} -
      7214631815 x^{19} + 19878551923 x^{20} + 3475881699 x^{21} -
      19978574007 x^{22} + 1204754727 x^{23} + 13960633114 x^{24} -
      3214431392 x^{25} - 6832632284 x^{26} + 2528557309 x^{27} +
      2326963032 x^{28 }- 1184186750 x^{29 }- 534678044 x^{30} +
      369015343 x^{31 }+ 75261636 x^{32 }- 78835592 x^{33} - 4000296 x^{34} +
      11506048 x^{35 }- 640379 x^{36} - 1110448 x^{37} + 155006 x^{38} +
      65338 x^{39} - 14541 x^{40} -  1860 x^{41} + 680 x^{42} - 2 x^{43} -
      13 x^{44} + x^{45}) $    }

      \noindent
   \hspace*{1.4cm}    \parbox[t]{5.4in}{      $ \ds  (1 + 15 x - 195 x^2 - 2476 x^3 + 9408 x^4 +
      128774 x^5 - 151702 x^6 - 3080005 x^7 + 152040 x^8 +
      39805335 x^9 + 22147982 x^{10} - 300921194 x^{11} -
      284159318 x^{12} + 1383107908 x^{13} + 1722027429 x^{14} -
      3930756397 x^{15} - 6069754917 x^{16} + 6915053418 x^{17 }+
      13500672554 x^{18} - 7214631815 x^{19} - 19878551923 x^{20} +
      3475881699 x^{21} + 19978574007 x^{22} + 1204754727 x^{23} -
      13960633114 x^{24} - 3214431392 x^{25} + 6832632284 x^{26} +
      2528557309 x^{27} - 2326963032 x^{28} - 1184186750 x^{29 }+
      534678044 x^{30} + 369015343 x^{31} - 75261636 x^{32} -
      78835592 x^{33} + 4000296 x^{34} + 11506048 x^{35} + 640379 x^{36} -
      1110448 x^{37} - 155006 x^{38} + 65338 x^{39} + 14541 x^{40} -
      1860 x^{41} - 680 x^{42} - 2 x^{43 }+ 13 x^{44} + x^{45}) (-1 - 75 x -
      2106 x^2 - 25677 x^3 - 59530 x^4 + 1782084 x^5 + 17272711 x^6 +
      16915722 x^7 - 549259769 x^8 - 2774119728 x^9 +
      1945931095 x^{10} + 53672650626 x^{11} + 125905839478 x^{12} -
      255069305286 x^{13} - 1662208525171 x^{14} - 1706160646334 x^{15} +
      5962148469123 x^{16} + 17962282960815 x^{17} + 7082138902898 x^{18} -
      41641400945047 x^{19} - 71155511878215 x^{20} -
      7096003752745 x^{21} + 98767225236595 x^{22} +
      104041755500065 x^{23} - 4751714309754 x^{24} -
      84592421865025 x^{25} - 57545771312940 x^{26} +
      4180055107783 x^{27} + 24201660957901 x^{28 }+
      11084074318344 x^{29} - 364628264623 x^{30} - 1894529779890 x^{31} -
      523784901608 x^{32} + 41364917589 x^{33} + 43581474910 x^{34} +
      4664480242 x^{35} - 1210342521 x^{36} - 282744551 x^{37} +
      5619268 x^{38} + 6144308 x^{39} + 323744 x^{40} - 54975 x^{41} -
      5843 x^{42} + 82 x^{43} + 30 x^{44} + x^{45}) (-1 + 42 x + 2133 x^{2} -
      88106 x^{3} - 1159305 x^{4} + 59614372 x^5 + 152015398 x^6 -
      18819762746 x^7 + 40034190687 x^8 + 3264921562892 x^9 -
      15997694438884 x^{10} - 336798102365195 x^{11} +
      2393078617098732 x^{12} + 21262382124790258 x^{13} -
      203991793583280580 x^{14} - 802344394384627520 x^{15} +
      11100949049156620432 x^{16} + 15124565789332508780 x^{17} -
      408720018862600116122 x^{18} + 70049454004597571197 x^{19} +
      10586093658527044258990 x^{20} - 12631025888273228837204 x^{21} -
      198508817597338205942399 x^{22} + 387402987345930766315590 x^{23} +
      2753586280130848546932800 x^{24} -
      7116726860274335286714358 x^{25} -
      28694708927568151334441162 x^{26} +
      91397507315681115211888499 x^{27} +
      226797214695129167282089048 x^{28} -
      870318027304318790444523656 x^{29} -
      1363049305769562405267119112 x^{30} +
      6338665081932290022053776544 x^{31} +
      6178413328764184607346600034 x^{32} -
      35988003082214887026742513610 x^{33} -
      20492020074586065235166736101 x^{34} +
      161266450872180007915993984787 x^{35} +
      45099606966665951271693058428 x^{36} -
      575001009719437474541893835090 x^{37} -
      37571777602636385348482415900 x^{38} +
      1639314263514932955906264058314 x^{39} -
      161019967972020863550757762131 x^{40} -
      3745699796292066069224199321372 x^{41} +
      851274013919532238198719327872 x^{42} +
      6859539342150922395619895904651 x^{43} -
      2285177200474388410515202543750 x^{44} -
      10046278725288151331993329222702 x^{45} +
      4263336781464515502088030143628 x^{46} +
      11714187061919722875421984542193 x^{47} -
      5951624800503928235794807036596 x^{48} -
      10796511948612809185000930216914 x^{49 }+
      6382964048451425290619045645907 x^{50} +
      7780699596208522812005194118918 x^{51}-
      5309228126699605717428066646335 x^{52} -
      4313080247506333381456597185155 x^{53 }+
      3429850650853924016071831349581 x^{54} +
      1790273570607001253851113421413 x^{55} -
      1715098075696198937373310320219 x^{56} -
      528375021398399743450311766854 x^{57} +
      659377687421116918032870695489 x^{58} +
      96604393710766619419597252621 x^{59} -
      193029880068729539119900529947 x^{60} -
      4069176647541912150779380527 x^{61} +
      42490073730573974965749981782 x^{62} - $    }

      \noindent
   \hspace*{1.4cm}    \parbox[t]{5.4in}{      $ \ds
      3413445482746271605206739212 x^{63} -
      6917976301240196729904525844 x^{64} +
      1165798086293282165887806486 x^{65} +
      814163354061253637176607304 x^{66} -
      208629109950888272910391562 x^{67} -
      66693134076508177810310164 x^{68} +
      24445572577097454676637446 x^{69 }+
      3497333010513277273515012 x^{70} -
      1989966629028961602789938 x^{71} - 83751655748872954401344 x^{72} +
      115359756476647736425785 x^{73} - 2722708701437214377921 x^{74} -
      4826028852207290118097 x^{75} + 345815430004445639338 x^{76} +
      146771902754630028398 x^{77} - 16434126042052261132 x^{78} -
      3249431151334904798 x^{79} + 493956208888831088 x^{80} +
      51941522002476761 x^{81} - 10408635523210110 x^{82} -
      583754494635043 x^{83} + 159473348039387 x^{84 }+
      4295711619291 x^{85} - 1796857404178 x^{86} - 16067848999 x^{87} +
      14814596478 x^{88} - 26776331 x^{89} - 87402987 x^{90} +
      632038 x^{91} + 351431 x^{92} - 2932 x^{93} - 868 x^{94} + 5 x^{95} +
      x^{96}) ] +        $} \\ \\
          \tiny      \hspace*{1.2cm}  \parbox[t]{5.4in}{ $ \ds
      2 x (46 x^{396}-98663 x^{394}+103330923 x^{392}-70441197913 x^{390}+35155964875133 x^{388}-13698710132496906 x^{386}+4340125105002385014 x^{384}-1149759827746963528917 x^{382}+259929987100009696284557 x^{380}-50932861135918122223829810 x^{378}+8756638139440990389358313195 x^{376}-1333985217571940187197451240331 x^{374}+181533443205198151122082022185649 x^{372}-22217552793147805762981637573011688 x^{370}+2459600715754456504164889912881341506 x^{368}-$ \\ $  247515708877923375980879665810407297661 x^{366}+22738761311491358362702450409780988376752 x^{364}- $ \\ $  1914169087997037720315692655765521487952746 x^{362}+148140783524272534189591569529841509793921495 x^{360}-$ \\ $  10571110365132424605167548167570490005466916653 x^{358}+697354337768081361695832597331203244607312491200 x^{356}- $ \\ $  42627841622839732800715030364971725911757390774307 x^{354}+2419698520729661769408353476413504302049289304962092 x^{352}-$ \\ $  127788867501703032172805587660100147203507795796146580 x^{350}+ $ \\ $ 6289993657212304809675147080584097670105004493458174674 x^{348}- $ \\ $ 289021096163210136767216518331644094038841945366513323131 x^{346}+ $ \\ $  12415687292006613696646108266231544782991492014272798253710 x^{344}- $ \\ $  499301963831407137359455972585339777648843967791749897214877 x^{342}+ $ \\ $ 18821394945607889547976221264436359848702806768813821917606606 x^{340}- $ \\ $ 665796183580167320985636989469821908227827674161852292969132421 x^{338}+ $ \\ $ 22125906217205743117162084624202673828881337560649670615316033773 x^{336}- $ \\ $ 691462152110756942550101382023804219195898319683858174832453577993 x^{334}+ $ \\ $ 20340039138519979280996632551209737449039083523070070175394176335444 x^{332}- $ \\ $ 563681131014638341569002226319150452552577166274802601946503474544740 x^{330}+ $ \\ $ 14728961911425297405747127815089759552263635146635189566154384316814671 x^{328}- $ \\ $ 363165569063275748542820740285693036114537773099619359356290932184345234 x^{326}+ $ \\ $ 8455674581206667977718948142502644109096923941281777176243993323095558709 x^{324}- $ \\ $ 186038124966291071296147388609469438486219953863343193338677308238749373817 x^{322}+ $ \\ $ 3870330774341861708239103249511659728990124935985357906649766675108713727279 x^{320}- $ \\ $ 76182211035865467234855531181915188675803788818450436412387248111177499911328 x^{318}+ $ \\ $ 1419616628662939109960779504743244889516332502212233956180451620146694364967322 x^{316}- $ \\ $ 25057630651098414355142839914729267207053574160914981301451237090399477022357602 x^{314}+ $ \\ $ 419167324583285741038111405021093159189709062945180455307228106738810484677624673 x^{312}- $ \\ $ 6648579778523473147710693276408415856334731698103701666897522520417792200333815475 x^{310}+ $ \\ $ 100038960634997054637544254345841617270810359511684934106854406601464678089809549702 x^{308}- $ \\ $ 1428574662843738976180271280243182178849406980566756147313263547278466294260254708961 x^{306}+ $ \\ $ 19369363615348183925718637949576495829621058386935017546936101229911326184744324705490 x^{304}- $ \\ $ 249449468108612971898094179859605296815866260991901605830187862976533920174633866819727 x^{302}+ $ \\ $ 3052617211466537230943054693452330748221146532983630244493486276700154351918994551133112 x^{300}- 35509483016172287617303901330950686561425462168198356760161083480007569929282120301425730 x^{298}+392781189980479392361361268700225739320160220392660524224965433785212751899877871405362835 x^{296}-4132722272592502285820875209978569968769999182742951347425657750511188846743208175377980582 x^{294}+41375092066857640628256061347050660161909536278898588267645757493931699942942384661923477067 x^{292}-394266973535866143371341845461626005623496008255914854003496843848939330664807146449198175854 x^{290}+ $    }

      \noindent
   \hspace*{1.4cm}    \parbox[t]{5.4in}{      $ \ds 3576973430439872413841808028939123555453400404643187105660149680165243169704825051514540256725 x^{288}-30905385734059454581957858922462230240530576532577811370676787876050253052120898709760832662143 x^{286}+254366144294532251099758521854349002146616738567177244020853999495110323012309816503058096209438 x^{284}-1994799331100782282867680701792430177522094923740087709224720383121250442743181054750167347260476 x^{282}+14909282402886535619430150830770485167338209129045737168015135459356685507651182459667448394936304 x^{280}-106225729474632349832942047074987515773973146765093654705125213844467758880514426324329686990538506 x^{278}+721624345130202592924891303435954117915693422604324610660288759871197491625913158725905906779911022 x^{276}-4675084862017387837410247940644521847669234969903281062868519847858337426274370141403322871006632065 x^{274}+28890073091055252735699961120794201502530608851790515851307025242993121846661869184184951201144938746 x^{272}-170321033058872414212300304494811339527888792336409153885231235968232769355898020392735234993343140040 x^{270}+958129022916945258190673854071555417105044376424818491922996278716577535079573277207830913168393800118 x^{268}-5143839350760308870163849990187503661272926119190865471825541862515438538873493505091580189615616887839 x^{266}+26358840018154859740885388890094550208764747552059443748544514405702610168410078856915790875088914995872 x^{264}-128945096148995684131716153620703973992691880949873062324448274231719704376574973260162920457378548444582 x^{262}+602259434621220363966807858114503124730725174725172612531595610253585457799153146585815010660817950503304 x^{260}-2686085542543526020121593469342599035711286953648474853674064505254879047914183939604463406550251335241539 x^{258}+11441090615896087834971067252509572595477114168607920955495003197503082630163250380066218695547594484105984 x^{256}-46545437001854449479595199393352513836213488903883520078447149141860238942525195289799055123222233051390230 x^{254}+180882511833166186504480780434432477334655543101204106979171621920723508625834366378936469472558723259966471 x^{252}-671537938248400431667521616141322198148235310007454003405013089740928218577094960347401997251286055327806449 x^{250}+2381995276334002305717522052674554979256094141228396026773039318734983873063721608831748703104415047971810458 x^{248}-8073218581560077155886965032991221994991177071629232520419012348781146327314170671237235912049840280001759105 x^{246}+26147154123186640209901470581760304792780264870287275304231341911641032874222567715006558610079182174556352113 x^{244}-80929691001025027164361035141814565154108815057828215736547359257109786757545345193759045313267607058871382696 x^{242}+239402349196105318692594669968763290677235271722631085229167448904236169938583214642191493090574860053274741535 x^{240}-676885016648686495336293803105411115138469155537992041975872444182881348819060094625850977808729638286151441251 x^{238}+1829337752498465860454435076736800582614238177192494255679158150053780795285720025694169057537535507972976110716 x^{236}-4725955605814772643567200232816305196966831272914026282970837923335050677214071020821587475488770032678443734445 x^{234}+11671431654181885490202368181269902415844927813300908528180181816130766922006016466178072200526328201226263377706 x^{232}-27556065761718749529179938272404464580657577931538197178462791462427508695679146757084333445392655044039287003216 x^{230}+62199457423895509214055326059012623374001681326292809211622177754814735277742810894657267517216175112783732405556 x^{228}-134229485917116984021312123745055338152335999338828841812639245646229013603881081994963999074145809839197253994940 x^{226}+276958579833333279031606185040082727877991401976057270043286044597456457312233218481643680326665535723705409707954 x^{224}-546385193019682632932094721736354445778263704857573347494871698932365966301255590698423299926164345489782962782547 x^{222}+1030646916167002830668179848587630655743297551021865783924045025128914870774178835210334446460113037486323185203503 x^{220}-1858897980248429745727147050676720550877840292063067939434791922103703836851012374508441082455773126843995549095807 x^{218}+3205839295110340094155728894967150063826794382622291557011005724143230581966693364976879320826594010861296548331111 x^{216}-5286551952312701028415811192821179038579058446558466056540797226649837737968798191929044636489643817595534108287495 x^{214}+8335852052643236010880803593171947413823623544345149520345961494440297372017006372178094970248167527852407608641499 x^{212}-12568242978840248868590982612074266988730927309380490974232091203685050938702087307025197279238270797498039887576282 x^{210}+18119438298297706828877600882535243247057224427490885639067876617153225925976227958386910671307898790169603858972890 x^{208}-24977972284870665222145463503141110327194635512192061872526992190138104992279139301901472599594412426108253197742864 x^{206}+32923559900991458001060961717260858907831560782659106888464540301256118942401335040744992183538448514655251918145311 x^{204}-41494123904844778421731575954186883655577016260295393665150482788473758882257550945930621525922818907913836276032291 x^{202}+50002050931000668560906442791739081003125555262533065113723705141918770450464575960652002847365731853140755911015756 x^{200}-57610253940081409111621026588196730697113001142614415895358967004878232737460505607290990080964909711612341108957255 x^{198}+63461448282755849763782204891631232731942281042196777408522542451427265141195040017356873925958468576127686020910512 x^{196}-
      66834998643538788520257931301098532710132006053053805812774043785895522260855719373919435130971625653167961326867190 x^{194}+67292288373732871395803498067050012793717645466517800586365742179331699683967668681873774149663241576423911979339369 x^{192}-64770264908917988714707069529155313319392016809641127604667691628873274856200166285857436811754629808818906874471148 x^{190}+59595666016460851824024259766135813584108153534433187358486180679863612347076129580017780384388014257074934837985708 x^{188}-52415512880298154520461005165269678972182176738869742156746730275474555628625732374910802891587068280041778051385442 x^{186}+44064267594581987320537139589770870803586652826508967017941307358623667342573289512287459227740383776756361901438494 x^{184}-35405279645954469825427706569636986596125494372093066886953846167113123190434948824284423012235041095586056674344422 x^{182}+27187870990712318760930713050527511242183327426764338441379911860218476683542296264817001249359173202353073385619163 x^{180}-19951551794324622351205550898330134314817945155940205234029181110492080380085166259438211442303822654787785546383755 x^{178}+13990674988454925776827969008097977869908896494692254930637363907983240758843569885914046855915219333875876333269694 x^{176}-9373998375080706454321734508024263007296558369136595092919068927030949358103044280143355603450448327394398032942464 x^{174}+6000624232863086632420121977881139298235581797426296122461376535325773459762518116039815989366538784190778221052427 x^{172}-3669541526112423669716289525260120532984934456871414253192079768608640894018322746127653774089455935617270973693035 x^{170}+2143515872718947827539737144658115318847567716002891656557194309218946589476983895583144253731832567125684089528960 x^{168}-1195897523518517300893260999400062372289958737043869393030750379409358493149446299209149152105713724966905430212363 x^{166}+637181966421592542191226435100702353515440210327685600368802915292338754216538927011724848805953826319917975419780 x^{164}-$    }

      \noindent
   \hspace*{1.4cm}    \parbox[t]{5.4in}{      $ \ds
      324176884282112536775822085782616815009631345499702787023268075263312794304315705399925025923849157587569456058383 x^{162}+157468360033942062927273338707793692418097356841997421580585748151329367807125908274419875667182095247158646340609 x^{160}-73019289212808312243575970128049131892122988232350744942368582006344794797357642274920543757458439259541283287389 x^{158}+32318550227536039174765643684102515658165526746169180824384817057129852350435729485617521771707980422097860518683 x^{156}-13651158995597954809510487008309797855058959159001102352052717327865838510641357952126074893144168766196530527757 x^{154}+5501996747888069175626240693213685110555308735076288916476718576045567692498271627531742831488321988278163693989 x^{152}-2115584018453643886624524963453119773304755875284421198549019741598150091303579224820037834450221816808915182555 x^{150}+775928071524414772276669518848741539928397017275841055396378469934810756340833837980634438990596084948026917475 x^{148}-271401317974895971068360778268762057001477188811059022659434045977869503897376116440993598988167287413833923747 x^{146}+90513975585566147689765188686470910327684750755892977901690785431065123848150277014033719787324260454581977863 x^{144}-28776775077374740057559473967385090304023947074035370527777284278373016695054898071568418125038429537313932992 x^{142}+8719599617401979486594012333821447341587774605577528962821397446277543405266792176413275629089065207448105460 x^{140}-2517557826624048276833446099008503601267351858789288443092272062492037287640426616408205920675460157709954807 x^{138}+692447188882217823347254920133937354920576087358226862606603225085229597389027704689745512514629623882743481 x^{136}-181388203795565435945091285058585914185685198279372838944425334840496516737377974764743647899786484408210541 x^{134}+45241023442337521523669882838198195681663250647420797266304628446933035648193676081836159089433468277157600 x^{132}-10740818123189917804813505507760127462951915093798210945773150530858004610334770602057776646879945135129839 x^{130}+2426604473191101597709724263749232855530347490438638833845880686028933822716689503849993508350946848003040 x^{128}-521539028998598476514973945354690283289641877676388608252339658259473989748368091141580088019068140225212 x^{126}+106601787876123866582105211779963691439768946689150937835749835539194406414481210697693979636543501124658 x^{124}-20715223827473502142182665027695326748827667646432803408187460167038202061208504310833067654329240203390 x^{122}+3825723409411171579445687431531535219366750959740830768988132314828398939780226032077458493428654360015 x^{120}-671245058766678503087225029830741554805585740370808054482665074277555463418063772533795027635832956224 x^{118}+111848518157387504216980092807361270723479985990637375784389860531020725985972126511904059926560737516 x^{116}-17692609619919461295713109360166654736614798249285978070342116913910341082924559795199533640004195763 x^{114}+2655764936658706488631083576722920738413328646043525097116024716409501123904011202807596393815287329 x^{112}-378128398527545035891556283840497686437891203300778836217683118107231637771233743347365674733298047 x^{110}+51044468004669865238507597741241419442612176860871263937878952004215428052003670066171062238329837 x^{108}-6530059980469033096264434236772193333190510245885155248473196139882062002134505419262298410570012 x^{106}+791289347710924635333723167057775954642487754335831830471959902015229081085860573504351407719684 x^{104}-90778800413003119145124577142386651565523117003801719402952737676332522186864125945091802052808 x^{102}+9854534992686291518498439095876345794130669981619125347168165838278350999567793649530172421458 x^{100}-1011700155333153233491814900918345353333401756546121488214554427383428716755141262181703939069 x^{98}+98170762957251507760626972530718090626776499780239873229414985945299480010383894132999772211 x^{96}-8998426685994307063014595509361405250168952322671869458967098432367898026557664413370082365 x^{94}+778632583871476560604241077945982536014905793500353578715782361089269509652491539318552427 x^{92}-63561642662005286364513166027862948504105805298151282405760275886085290153231286972116771 x^{90}+ $ \\ $ 4891656390973669213700277463009229727153271064782291201852876727219543291986333875648007 x^{88}- $ \\ $ 354652349595437534969367855187774986142135626974528824081129251084033148831900014120641 x^{86}+ $ \\ $ 24205208064708439928394767798830161893511836629535532518689102678025138870269111453967 x^{84}- $ \\ $ 1553929362178094987762388314581874473108977712558698485180615797459955614730132555717 x^{82}+ $ \\ $ 93758587503696271629555237454367274832486155076210425468403372595650567575586941025 x^{80}- $ \\ $ 5312164249730504152655757025730148047959567889622597687120616440169574331713661183 x^{78}+ $ \\ $ 282369190862886095823820601319627711892900786762782725311884374246517650888257423 x^{76}- $ \\ $ 14068035168649166357478596964498327120795105468678617482742396780178329131351159 x^{74}+ $ \\ $ 656271881441418544551436728955438689090974298903787290786097964502981880621765 x^{72}- $ \\ $ 28635831898243396278103264376674245530069021049968474172016752615670354084401 x^{70}+ $ \\ $ 1167424770460965225265093888259660577210504514721053725627119841533470676505 x^{68}- $ \\ $ 44415320976877761138481650350760462919699694938585363397486561357235712112 x^{66}+ $ \\ $ 1575018136696732652438847045335982373731767048446118148268031314465653851 x^{64}- $ \\ $ 51990187597347140824993195569566581476203931984856721935386073140153813 x^{62}+ $ \\ $ 1595299371461245022355453257523357772888780393251006705266840163788081 x^{60}- $ \\ $ 45437535860544136075365165355094452753768994979203971854056345155667 x^{58}+ $ \\ $ 1199409595012685529348521651006362053090925515252033800127984404484 x^{56}- $ \\ $ 29294541352110155702064310714317538435438397418762600091754548426 x^{54}+ $ \\ $ 660863979381621350365336617165890810123337915626563449620421612 x^{52}- $ \\ $ 13744590467813642138525481316334127990751219373013051963818434 x^{50}+ $ \\ $ 263012532882370586257265329286804196602033000864841457765983 x^{48}- $ \\ $ 4620744767454649503096564075611604862853988022487724477636 x^{46}+ $ \\ $ 74358659745520181754591249520141332469458076467010216318 x^{44}- $ \\ $ 1093316070208727410914122061141694248400379600055147098 x^{42}+ $ \\ $ 14647753403146097990161070687249525562091562655670290 x^{40}-178285588702625035049557908347308074077121429136687 x^{38}+1965017074671039012137509923385541700528087397778 x^{36}-19541728026952993076302956332849201606467489085 x^{34}+174655070448622760717451364907732117373779501 x^{32}-1396693027798282481203639295252272911868400 x^{30}+9944245957019118072212562497355131720419 x^{28}-62685741380879830139772690160680833543 x^{26}+ $    }

      \noindent
   \hspace*{1.4cm}    \parbox[t]{5.4in}{      $ \ds  347644073978191573408399366999826089 x^{24}-  1683861009551491629496156905104469 x^{22}+7063345853347051702619070353120 x^{20}-25405318432214080934820546240 x^{18}+77423474731501072394563948 x^{16}-197027124290597542264447 x^{14}+411101003325202086287 x^{12}-686874946108416584 x^{10}+890292088288392 x^8-856005798303 x^6+570469581 x^4-234309 x^2+45)]/ $} \\
       \hspace*{1.4cm}   \parbox[t]{5.4in}{ $ \ds
      [(-1 + x) (1 + x) (1 - 2 x - x^2 + x^3) (-1 - 2 x + x^2 + x^3) (-1 +
     9 x - 27 x^2 + 28 x^3 - 9 x^5 + x^6) (1 - 8 x + 8 x^2 + 6 x^3 -
     6 x^4 - x^5 + x^6) (1 + 8 x + 8 x^2 - 6 x^3 - 6 x^4 + x^5 +
     x^6) (-1 - 9 x - 27 x^2 - 28 x^3 + 9 x^5 + x^6) (1 - 90 x +
     2133 x^2 + 19654 x^3 - 1407750 x^4 + 12020085 x^5 +
     211505759 x^6 - 4005889491 x^7 + 1140372558 x^8 +
     418020172543 x^9 - 2487724623906 x^{10}  - 15524197875990 x^{11} +
     207409855235855 x^{12}  - 168411001997175 x^{13}  -
     6989852039646672 x^{14}  + 29176347222301350 x^{15}  +
     85446200204468703 x^{16 } - 862651296424382517 x^{17 } +
     732228455321894508 x^{18 } + 11231294123644498821 x^{19}  -
     33803475249656479500 x^{20}  - 54921206512262415741 x^{21}  +
     413908203902580672042 x^{22 } - 246207028620567111678 x^{23 } -
     2478006706489040058135 x^{24 } + 4892059779725991163731 x^{25}  +
     6604769927708293281489 x^{26 } - 29615041041062939657101 x^{27}  +
     5692338640092901485984 x^{28}  + 95353352498492938925253 x^{29}  -
     102458115245643694623619 x^{30}  - 163321849837230400108482 x^{31}  +
     356486700262708176382602 x^{32}  + 75818151591482233031073 x^{33}  -
     667256343974500352911503 x^{34}  + 278039887830219509936100 x^{35}  +
     733454500237993350748973 x^{36}  - 685954133222650252610760 x^{37}  -
     427320993379027646418498 x^{38}  + 781904366023459571507369 x^{39}  +
     28153248400039056561681 x^{40}  - 531396998038886850489789 x^{41}  +
     154086139945066863345177 x^{42}  + 222672705102690029154612 x^{43}  -
     127277447297634414536733 x^{44}  - 53175680795563803057275 x^{45}  +
     53846433449462045531871 x^{46}  + 3955591274039798235417 x^{47}  -
     13978733150559004421783 x^{48}  + 1617114491250147902040 x^{49}  +
     2287879124568249125151 x^{50}  - 596315566352941519713 x^{51}  -
     225529529482003369278 x^{52}  + 97905185502923227626 x^{53}  +
     10648699061080306343 x^{54}  - 9487814076406871859 x^{55}  +
     206953141963128534 x^{56}  + 570333880393160836 x^{57}  -
     62350849780138683 x^{58}  - 20789244041051196 x^{59}  +
     4077900882295116 x^{60}  + 395015565737433 x^{61}  -
     144460876536018 x^{62}  - 252140561385 x^{63}  + 3043888623648 x^{64} -
     172278608445 x^{65}  - 36698018483 x^{66}  + 4110045927 x^{67}  +
     192261069 x^{68}  - 45656109 x^{69}  + 614187 x^{70}  + 243603 x^{71}  -
     12637 x^{72}  - 369 x^{73}  + 45 x^{74}  - x^{75} ) (1 + 90 x + 2133 x^2 -
     19654 x^3 - 1407750 x^4 - 12020085 x^5 + 211505759 x^6 +
     4005889491 x^7 + 1140372558 x^8 - 418020172543 x^9 -
     2487724623906 x^{10}  + 15524197875990 x^{11}  +
     207409855235855 x^{12}  + 168411001997175 x^{13}  -
     6989852039646672 x^{14}  - 29176347222301350 x^{15}  +
     85446200204468703 x^{16}  + 862651296424382517 x^{17}  +
     732228455321894508 x^{18}  - 11231294123644498821 x^{19}  -
     33803475249656479500 x^{20}  + 54921206512262415741 x^{21}  +
     413908203902580672042 x^{22}  + 246207028620567111678 x^{23 } -
     2478006706489040058135 x^{24}  - 4892059779725991163731 x^{25 } +
     6604769927708293281489 x^{26}  + 29615041041062939657101 x^{27}  +
     5692338640092901485984 x^{28}  - 95353352498492938925253 x^{29}  -
     102458115245643694623619 x^{30}  + 163321849837230400108482 x^{31}  +
     356486700262708176382602 x^{32 } - 75818151591482233031073 x^{33}  -
     667256343974500352911503 x^{34}  - 278039887830219509936100 x^{35}  +
     733454500237993350748973 x^{36 } + 685954133222650252610760 x^{37}  -
     427320993379027646418498 x^{38 } - 781904366023459571507369 x^{39}  +
     28153248400039056561681 x^{40}  + 531396998038886850489789 x^{41}  +
     154086139945066863345177 x^{42}  - 222672705102690029154612 x^{43}  -
     127277447297634414536733 x^{44}  + 53175680795563803057275 x^{45}  +
     53846433449462045531871 x^{46}  - 3955591274039798235417 x^{47}  -
     13978733150559004421783 x^{48}  - 1617114491250147902040 x^{49}  +
     2287879124568249125151 x^{50}  + 596315566352941519713 x^{51}  -
     225529529482003369278 x^{52}  - 97905185502923227626 x^{53}  +
     10648699061080306343 x^{54}  + 9487814076406871859 x^{55}  +
     206953141963128534 x^{56}  - 570333880393160836 x^{57}  -
     62350849780138683 x^{58}  + 20789244041051196 x^{59}  +
     4077900882295116 x^{60}  - 395015565737433 x^{61}  -
     144460876536018 x^{62}  + 252140561385 x^{63}  + 3043888623648 x^{64}  +
     172278608445 x^{65}  - 36698018483 x^{66} - 4110045927 x^{67} +
     192261069 x^{68}  + 45656109 x^{69}  + 614187 x^{70}  - 243603 x^{71}  -
     12637 x^{72}  + 369 x^{73}  + 45 x^{74}  + x^{75}) (-1 - 27 x + 1032 x^2 +
     25802 x^3 - 373186 x^4 - 9247650 x^5 + 68985132 x^6 +
     1756981328 x^7 - 7544737839 x^8 - 205344558763 x^9 +
     520362914282 x^{10}  + 16055543418770 x^{11}  - 22765506950292 x^{12}  -
     886185755306711 x^{13}  + 567341618795392 x^{14}  +
     35824949732566099 x^{15}  - 2114077619505885 x^{16}  -
     1089052544253505079 x^{17}  - 434310955869650180 x^{18}  +
     25382085183437405700 x^{19} + 19657336252849199189 x^{20}  -
     460196787380533985842 x^{21}  - 503282143127348655479 x^{22}  +
     6563545564790657680027 x^{23}  + 9019917524975714828718 x^{24}  -
     74279664653212096085631 x^{25}  - 121126151372353616341848 x^{26}  +
     671522546480445322781198 x^{27}  + $} \\
          \hspace*{1.4cm}   \parbox[t]{5.4in}{ $ \ds
                    1259864230974286367427691 x^{28}  -
     4874434339938184506685397 x^{29}  -
     10355985403728393543409067 x^{30}  +
     28508897879645580528825667 x^{31}  +
     68215067001652545430286192 x^{32}  -
     134568911549356278360128102 x^{33}  -
     363861169332415595759624659 x^{34}  +
     512191062359197441924007885 x^{35 } +
     1584783594082588087382278463 x^{36}  -
     1564210162255894419536702349 x^{37 } -
     5674861959676935021395050413 x^{38 } +
     3785132858090620875642211416 x^{39}  +
     16803139396061165672290060717 x^{40}  -
     7044093577535657561000115327 x^{41}  -
     41343132893232215469798457354 x^{42}  +
     9285232522254031569698694689 x^{43}  +
     84883083360723557875730426434 x^{44}  -
     5937121438225419731565543654 x^{45}  -
     145956035701218221797085712592 x^{46}  -
     7890101748188581051360660346 x^{47}  +
     210852803270320434959494868010 x^{48}  +
     32505194104448575284667024966 x^{49}  -
     256625689573728687673431393828 x^{50}  -
     60386088834617104154813509703 x^{51}  +
     263791426465741835712351964159 x^{52}  +
     79446568823858899041889805170 x^{53}  -
     229529669618761966378616574822 x^{54}  -
     81368503811164989012654074179 x^{55}  +
     169410836644260710451175862716 x^{56}  +
     67317485231746697845168652537 x^{57}  -
     106273241553496096920653558337 x^{58}  -
     45836707489100134208832885680 x^{59}  +
     56768242863723954140114918264 x^{60}  +
     25971874970291859366352369232 x^{61}  -
     25868414266252042008113790681 x^{62}  -
     12334499250260624842980447944 x^{63}  +
     10073049237434908783966012615 x^{64}  +
     4934518938825367378469146563 x^{65}  -
     3357135730239884258076120873 x^{66}  -
     1668961310040579145324436854 x^{67}  +
     958979388811076795805220259 x^{68}  +
     478484381627959740231307898 x^{69}  -
     235070462661933368244141645 x^{70}  -
     116494763716325583154471593 x^{71}  +
     49490588354774441565598327 x^{72}  +
     24113150837187023305271901 x^{73}  -
     8953841615972305424578920 x^{74}  -
     4245305143743032443594882 x^{75}  +
     1392157069027476262098948 x^{76}  + 635580648350932217926483 x^{77}  -
     185931065427671751697600 x^{78}  - 80838922982850596922843 x^{79}  +
     21306408936395504377734 x^{80}  + 8719476782236791891707 x^{81}  -
     2090911132691824012433 x^{82}  - 795459886846428765825 x^{83}  +
     175226009815304083621 x^{84}  + 61148524359498140722 x^{85}  -
     12490776292567568868 x^{86}  - 3941245311464514974 x^{87}  +
     753388475396092803 x^{88}  + 211610725442247167 x^{89}  -
     38184703986094031 x^{90}  - 9385073357203492 x^{91}  +
     1611821886653267 x^{92}  + 340074435384128 x^{93}  -
     56011183019777 x^{94}  - 9923940361375 x^{95}  + 1578324064638 x^{96}  +
     228740033731 x^{97 } - 35344978435 x^{98} - 4053271905 x^{99}  +
     611772610 x^{100}  + 53065630 x^{101}  - 7859404 x^{102}  -
     481716 x^{103}  + 70281 x^{104}  + 2698 x^{105}  - 389 x^{106}  - 7 x^{107}  +
     x^{108}) (-1 + 27 x + 1032 x^2 - 25802 x^3 - 373186 x^4 +
     9247650 x^5 + 68985132 x^6 - 1756981328 x^7 - 7544737839 x^8 +
     205344558763 x^9 + 520362914282 x^{10}  - 16055543418770 x^{11}  -
     22765506950292 x^{12}  + 886185755306711 x^{13}  +
     567341618795392 x^{14}  - 35824949732566099 x^{15}  -
     2114077619505885 x^{16}  + 1089052544253505079 x^{17}  -
     434310955869650180 x^{18}  - 25382085183437405700 x^{19}  +
     19657336252849199189 x^{20}  + 460196787380533985842 x^{21}  -
     503282143127348655479 x^{22}  - 6563545564790657680027 x^{23}  +
     9019917524975714828718 x^{24}  + 74279664653212096085631 x^{25}  -
     121126151372353616341848 x^{26}  - 671522546480445322781198 x^{27}  +
     1259864230974286367427691 x^{28}  +
     4874434339938184506685397 x^{29}  -
     10355985403728393543409067 x^{30}  -
     28508897879645580528825667 x^{31}  +
     68215067001652545430286192 x^{32}  +
     134568911549356278360128102 x^{33}  -
     363861169332415595759624659 x^{34}  -
     512191062359197441924007885 x^{35}  +
     1584783594082588087382278463 x^{36}  +
     1564210162255894419536702349 x^{37}  -
     5674861959676935021395050413 x^{38}  -
     3785132858090620875642211416 x^{39}  +
     16803139396061165672290060717 x^{40}  +
     7044093577535657561000115327 x^{41}  -
     41343132893232215469798457354 x^{42}  -
     9285232522254031569698694689 x^{43}  +
     84883083360723557875730426434 x^{44}  +
     5937121438225419731565543654 x^{45}  -
     145956035701218221797085712592 x^{46}  +
     7890101748188581051360660346 x^{47}  +
     210852803270320434959494868010 x^{48}  -
     32505194104448575284667024966 x^{49}  -
     256625689573728687673431393828 x^{50}  +
     60386088834617104154813509703 x^{51}  +
     263791426465741835712351964159 x^{52}  -
     79446568823858899041889805170 x^{53}  -
     229529669618761966378616574822 x^{54}  +
     81368503811164989012654074179 x^{55}  +
     169410836644260710451175862716 x^{56}  -
     67317485231746697845168652537 x^{57}  -
     106273241553496096920653558337 x^{58}  +
     45836707489100134208832885680 x^{59}  +
     56768242863723954140114918264 x^{60}  -
     25971874970291859366352369232 x^{61}  -
     25868414266252042008113790681 x^{62}  +
     12334499250260624842980447944 x^{63}  +
     10073049237434908783966012615 x^{64}  -
     4934518938825367378469146563 x^{65}  -
     3357135730239884258076120873 x^{66}  +
     1668961310040579145324436854 x^{67}  +
     958979388811076795805220259 x^{68}  -
     478484381627959740231307898 x^{69}  -
     235070462661933368244141645 x^{70}  +
     116494763716325583154471593 x^{71}  +
     49490588354774441565598327 x^{72}  -
     24113150837187023305271901 x^{73}  -  $}
\\
     \noindent
          \hspace*{1.4cm}   \parbox[t]{5.4in}{ $ \ds
     8953841615972305424578920 x^{74}  +
     4245305143743032443594882 x^{75}  +
     1392157069027476262098948 x^{76}  - 635580648350932217926483 x^{77}  -
     185931065427671751697600 x^{78}  + 80838922982850596922843 x^{79}  +
     21306408936395504377734 x^{80}  - 8719476782236791891707 x^{81}  -
     2090911132691824012433 x^{82}  + 795459886846428765825 x^{83}  +
     175226009815304083621 x^{84}  - 61148524359498140722 x^{85}  -
     12490776292567568868 x^{86}  + 3941245311464514974 x^{87}  +
     753388475396092803 x^{88}  - 211610725442247167 x^{89}  -
     38184703986094031 x^{90}  + 9385073357203492 x^{91}  +
     1611821886653267 x^{92}  - 340074435384128 x^{93}  -
     56011183019777 x^{94}  + 9923940361375 x^{95}  + 1578324064638 x^{96}  -
     228740033731 x^{97}  - 35344978435 x^{98}  + 4053271905 x^{99}  +
     611772610 x^{100}  - 53065630 x^{101}  - 7859404 x^{102}  +
     481716 x^{103}  + 70281 x^{104}  - 2698 x^{105}  - 389 x^{106}  + 7 x^{107} +
     x^{108})]$}

     \small
              \hspace*{1.2cm} $+$  \parbox[t]{5.4in}{ $ \ds
     [2 x (10 x^{234}-8036 x^{232}+3116826 x^{230}-777789012 x^{228}+140459010924 x^{226}-19575732635721 x^{224}+2192869442615256 x^{222}-203043261162959907 x^{220}+15860744145972836156 x^{218}-1061654259181284618043 x^{216}+61643963360547268352616 x^{214}-3135746457599683604296319 x^{212}+140886737267032070016601285 x^{210}-5629077615254845571116674365 x^{208}+201165226386163759654144405940 x^{206}-6462069799704532075577259175385 x^{204}+187395979857567378140002165289633 x^{202}-4924357277281107049152267722689820 x^{200}+117645105057404504579640726367876305 x^{198}- $ \\ $ 2562725454400765402055042431701019726 x^{196}+ $ \\ $ 51034293308071436523183488668662337292 x^{194}- $ \\ $ 931237399793732239312211966553659433797 x^{192}+ $ \\ $ 15602634512294314441189649313507554685928 x^{190}- $ \\ $ 240482275692786467550051625210456273912424 x^{188}+ $ \\ $ 3415406495414351081374726442578445252077521 x^{186}- $ \\ $ 44764077716630201301691478595694930358338780 x^{184}+ $ \\ $ 542167163788803636781807068603978297759503311 x^{182}- $ \\ $ 6075508334374411071716951457199106898239679742 x^{180}+ $ \\ $ 63059999170489825662163196447723262394592428187 x^{178}- $ \\ $ 606840600602383959229870119313176797183622474132 x^{176}+ $ \\ $ 5419103477412572684432052550619792181241997127177 x^{174}- $ \\ $ 44942189616982249825479198887395880821019406271228 x^{172}+ $ \\ $ 346384910969210484185096245002027540977940520462994 x^{170}- $ \\ $ 2482621892822958086693781775278061397074567069720187 x^{168}+ $ \\ $ 16555617927229667969570268291707720440146963633650249 x^{166}- $ \\ $ 102770695356854908740257647065530486718258489768992332 x^{164}+ $ \\ $ 594100129579170378823462201784480041831510692694632935 x^{162}- $ \\ $ 3199388466744790495249336717845125750552493347245965291 x^{160}+ 16055225676217329751374782258908786485192262905230619098 x^{158}-75094642941273403041378626430057438804098403158337719201 x^{156}+ 327433023820883125518963026189687529134828778493834700142 x^{154}  - 1331105763019379795158889731794522577786043586527135794045 x^{152}+ 5045603241557583344498373325746057722695150619521008567917 x^{150}-   17833501729304979369314345974854017695813408058914632682222 x^{148}+58772541749948122214845255572596255677570999051240568603899 x^{146}-180591175279909027202381528392720014578925584344179474598780 x^{144}+517309822480873275208546499457644455736211536636942389258663 x^{142}-1381217716362174367244465109373451473466519778619362217147942 x^{140}+3436620256209541301910548141600667078334862251467820856067292 x^{138}-7965873387600970360727235816502646296585416378278973138660594 x^{136}+17195429906745559114674426294600648856365845435957425008130114 x^{134}-34552768179460410278870600374548992624890589454868936659536238 x^{132}+64597698618267132971714047566998938497734146161550987028846418 x^{130}-112290906560270033525909617129427940442492862307172475566626062 x^{128}+$}

     \noindent
          \hspace*{1.4cm}   \parbox[t]{5.4in}{ $ \ds
          181359104629687426659463070872312704869894428155237945075996247 x^{126}-271898318078581701399519546480007008747267814166024259722802731 x^{124}+377974927704379425754725027462140876364714264658452318869284145 x^{122}-486529385203388835050670047855403070592021798362175485697578407 x^{120}+
     578876299482025002409870854742355528726570590974752914860915984 x^{118}-635196920773577151746829172992610796064860401241853526631383845 x^{116}+640853399908600047946243442098659657847650918630908073670094091 x^{114}-591970536295408179720784009946750875764734601095050210944010318 x^{112}+497552304482021966544519711981501695556695100326458943632090202 x^{110}-376829811003762607378523462595410795596058879978272454741419347 x^{108}+252885489196116744976703777574918146215411400467427417209472804 x^{106}-145444308948941353670145967242207235201116056709382634103361560 x^{104}+65922289867665461956372211013703274573676360090170744296361743 x^{102}-16351455065906411560931632531376233662360981621281204435243995 x^{100}-8282455627204234721382347440269370182428855188274320950926189 x^{98}+16157092683205400960967259838894201297133531158829269467299924 x^{96}-15162337515393528863412058489935141902708334435976632373512518 x^{94}+10983911987042882843642130407476877523724217213606259458751941 x^{92}-6753800026640791164083219552308510616837450604684139020116382 x^{90}+3652467566620375078278661793272252866306533153905993168820243 x^{88}-1767069023866030810306070239181495021499098994779664711520845 x^{86}+771910705055332588227251212064739232484103170875680271333516 x^{84}-306112934962277714939689511552558628833436895786183107683953 x^{82}+110568100844101824830000462024796573300956633775262712468350 x^{80}-36449095125171220366990521053546236550589130001838734318159 x^{78}+10979104367711095727964962115059691639867779216650239566486 x^{76}-3023692504376522986964265942767575067607877199471927807796 x^{74}+761531444496053896930418032650003442822068951831111847983 x^{72}-175376672359602751800785004346069432694044723279695070587 x^{70}+36917955673558441096097017818312627647034835230238649588 x^{68}-7099802080150917197861037360259253413399842069376205020 x^{66}+1246477413369752103302436347221637269353065635247650701 x^{64}-199606039192221659423578654275510349512033517894484627 x^{62}+29125696218087725632079815492782932740414455475687698 x^{60}- $ \\ $ 3868114755748319915924461553799098260037830123041124 x^{58}+ $ \\ $ 466978540963442736435050225990703945809331057800822 x^{56}- $ \\ $ 51176590728505655173335936214327809274458550085939 x^{54}+ $\\$  5083617961081068831334147588632573356326603235338 x^{52}- $ \\ $ 456981168946655750727146821584379959025920122677 x^{50}+ $ \\ $ 37109720124135832189690873664615237728496457257 x^{48}- $ \\ $ 2717188243122455840236355230205745359318748817 x^{46}+   $}

     \noindent
          \hspace*{1.4cm}   \parbox[t]{5.4in}{ $ \ds 179021253817725854826985158059969333509048275 x^{44}- $ \\ $ 10589463180942414699077114389568792852890513 x^{42}+ $ \\ $ 561006070800630203011718478612945685713873 x^{40}- $ \\ $ 26547020923145440962396550828554452727449 x^{38}+ $ \\ $ 1118708978378399496749381713021395650177 x^{36}- $ \\ $ 41841783660021985547383298357725316058 x^{34}+ $ \\ $ 1383692492853978164771857773349012424 x^{32}-40282351666207191521656215845757314 x^{30}+1027215481269619466417570326364485 x^{28}-22812022238110302125704020933731 x^{26}+438222110377238213619954952624 x^{24}-7224848816119891330863513086 x^{22}+101282735531518125865540357 x^{20}-1194059845570797424172949 x^{18}+11682778995500607127883 x^{16}-93339661961478848567 x^{14}+596780559030284809 x^{12}-2975417499924988 x^{10}+11178176933691 x^8-30180354255 x^6+54645594 x^4-59283 x^2+30)] / $ }  \\
      \hspace*{1.4cm}   \parbox[t]{5.4in}{ $ \ds
     [(-1 + x) (1 + x) (1 - 2 x - x^2 + x^3) (-1 - 2 x + x^2 + x^3) (1 -
      8 x + 8 x^2 + 6 x^3 - 6 x^4 - x^5 + x^6) (1 + 8 x + 8 x^2 -
      6 x^3 - 6 x^4 + x^5 + x^6) (-1 - 27 x + 1032 x^2 + 25802 x^3 -
      373186 x^4 - 9247650 x^5 + 68985132 x^6 + 1756981328 x^7 -
      7544737839 x^8 - 205344558763 x^9 + 520362914282 x^{10} +
      16055543418770 x^{11} - 22765506950292 x^{12} -
      886185755306711 x^{13} + 567341618795392 x^{14} +
      35824949732566099 x^{15} - 2114077619505885 x^{16} -
      1089052544253505079 x^{17} - 434310955869650180 x^{18} +
      25382085183437405700 x^{19} + 19657336252849199189 x^{20} -
      460196787380533985842 x^{21} - 503282143127348655479 x^{22} +
      6563545564790657680027 x^{23} + 9019917524975714828718 x^{24} -
      74279664653212096085631 x^{25} - 121126151372353616341848 x^{26} +
      671522546480445322781198 x^{27} +
      1259864230974286367427691 x^{28} -
      4874434339938184506685397 x^{29} -
      10355985403728393543409067 x^{30} +
      28508897879645580528825667 x^{31} +
      68215067001652545430286192 x^{32} -
      134568911549356278360128102 x^{33} -
      363861169332415595759624659 x^{34} +
      512191062359197441924007885 x^{35} +
      1584783594082588087382278463 x^{36} -
      1564210162255894419536702349 x^{37}-
      5674861959676935021395050413 x^{38 }+
      3785132858090620875642211416 x^{39} +
      16803139396061165672290060717 x^{40} -
      7044093577535657561000115327 x^{41} -
      41343132893232215469798457354 x^{42} +
      9285232522254031569698694689 x^{43} +
      84883083360723557875730426434 x^{44} -
      5937121438225419731565543654 x^{45} -
      145956035701218221797085712592 x^{46} -
      7890101748188581051360660346 x^{47} +
      210852803270320434959494868010 x^{48} +
      32505194104448575284667024966 x^{49} -
      256625689573728687673431393828 x^{50} -
      60386088834617104154813509703 x^{51} +
      263791426465741835712351964159 x^{52} +
      79446568823858899041889805170 x^{53} -
      229529669618761966378616574822 x^{54} -
      81368503811164989012654074179 x^{55} +
      169410836644260710451175862716 x^{56} +
      67317485231746697845168652537 x^{57} -
      106273241553496096920653558337 x^{58} -
      45836707489100134208832885680 x^{59} +
      56768242863723954140114918264 x^{60} +
      25971874970291859366352369232 x^{61} -
      25868414266252042008113790681 x^{62} -
      12334499250260624842980447944 x^{63} +
      10073049237434908783966012615 x^{64} +
      4934518938825367378469146563 x^{65} -
      3357135730239884258076120873 x^{66} -
      1668961310040579145324436854 x^{67} +
      958979388811076795805220259 x^{68} +
      478484381627959740231307898 x^{69} -
      235070462661933368244141645 x^{70} -
      116494763716325583154471593 x^{71} +
      49490588354774441565598327 x^{72} +
      24113150837187023305271901 x^{73} -$}

      \noindent
          \hspace*{1.4cm}   \parbox[t]{5.4in}{ $ \ds
      8953841615972305424578920 x^{74} -
      4245305143743032443594882 x^{75} +
      1392157069027476262098948 x^{76} +
      635580648350932217926483 x^{77} - 185931065427671751697600 x^{78} -
      80838922982850596922843 x^{79} + 21306408936395504377734 x^{80} +
      8719476782236791891707 x^{81} - 2090911132691824012433 x^{82} -
      795459886846428765825 x^{83} + 175226009815304083621 x^{84} +
      61148524359498140722 x^{85} - 12490776292567568868 x^{86} -
      3941245311464514974 x^{87} + 753388475396092803 x^{88} +
      211610725442247167 x^{89} - 38184703986094031 x^{90} -
      9385073357203492 x^{91 } + 1611821886653267 x^{92} +
      340074435384128 x^{93} - 56011183019777 x^{94} -
      9923940361375 x^{95} + 1578324064638 x^{96} + 228740033731 x^{97} -
      35344978435 x^{98 }- 4053271905 x^{99} + 611772610 x^{100} +
      53065630 x^{101} - 7859404 x^{102} - 481716 x^{103} + 70281 x^{104} +
      2698 x^{105} - 389 x^{106} - 7 x^{107} + x^{108}) (-1 + 27 x +
      1032 x^{2} - 25802 x^{3} - 373186 x^{4} + 9247650 x^{5} +
      68985132 x^{6} - 1756981328 x^{7} - 7544737839 x^{8} +
      205344558763 x^{9} + 520362914282 x^{10} - 16055543418770 x^{11} -
      22765506950292 x^{12} + 886185755306711 x^{13} +
      567341618795392 x^{14} - 35824949732566099 x^{15}-
      2114077619505885 x^{16 }+ 1089052544253505079 x^{17} -
      434310955869650180 x^{18} - 25382085183437405700 x^{19 }+
      19657336252849199189 x^{20} + 460196787380533985842 x^{21} -
      503282143127348655479 x^{22}- 6563545564790657680027 x^{23} +
      9019917524975714828718 x^{24} + 74279664653212096085631 x^{25} -
      121126151372353616341848 x^{26} - 671522546480445322781198 x^{27} +
      1259864230974286367427691 x^{28} +
      4874434339938184506685397 x^{29} -
      10355985403728393543409067 x^{30} -
      28508897879645580528825667 x^{31} +
      68215067001652545430286192 x^{32} +
      134568911549356278360128102 x^{33} -
      363861169332415595759624659 x^{34} -
      512191062359197441924007885 x^{35} +
      1584783594082588087382278463 x^{36} +
      1564210162255894419536702349 x^{37} -
      5674861959676935021395050413 x^{38} -
      3785132858090620875642211416 x^{39} +
      16803139396061165672290060717 x^{40} +
      7044093577535657561000115327 x^{41} -
      41343132893232215469798457354 x^{42} -
      9285232522254031569698694689 x^{43} +
      84883083360723557875730426434 x^{44} +
      5937121438225419731565543654 x^{45} -
      145956035701218221797085712592 x^{46} +
      7890101748188581051360660346 x^{47} +
      210852803270320434959494868010 x^{48} -
      32505194104448575284667024966 x^{49} -
      256625689573728687673431393828 x^{50 }+
      60386088834617104154813509703 x^{51 }+
      263791426465741835712351964159 x^{52 }-
      79446568823858899041889805170 x^{53}-
      229529669618761966378616574822 x^{54 }+
      81368503811164989012654074179 x^{55 }+
      169410836644260710451175862716 x^{56} -
      67317485231746697845168652537 x^{57} -
      106273241553496096920653558337 x^{58} +
      45836707489100134208832885680 x^{59 }+
      56768242863723954140114918264 x^{60} -
      25971874970291859366352369232 x^{61 }-
      25868414266252042008113790681 x^{62} +
      12334499250260624842980447944 x^{63} +
      10073049237434908783966012615 x^{64} -
      4934518938825367378469146563 x^{65} -
      3357135730239884258076120873 x^{66 }+
      1668961310040579145324436854 x^{67} +
      958979388811076795805220259 x^{68} -
      478484381627959740231307898 x^{69} -
      235070462661933368244141645 x^{70} +
      116494763716325583154471593 x^{71} +
      49490588354774441565598327 x^{72} -
      24113150837187023305271901 x^{73} -  $}

     \noindent
          \hspace*{1.4cm}   \parbox[t]{5.4in}{ $ \ds
      8953841615972305424578920 x^{74} +
      4245305143743032443594882 x^{75} +
      1392157069027476262098948 x^{76} -
      635580648350932217926483 x^{77} - 185931065427671751697600 x^{78} +
      80838922982850596922843 x^{79} + 21306408936395504377734 x^{80} -
      8719476782236791891707 x^{81} - 2090911132691824012433 x^{82} +
      795459886846428765825 x^{83} + 175226009815304083621 x^{84} -
      61148524359498140722 x^{85} - 12490776292567568868 x^{86} +
      3941245311464514974 x^{87} + 753388475396092803 x^{88} -
      211610725442247167 x^{89} - 38184703986094031 x^{90} +
      9385073357203492 x^{91} + 1611821886653267 x^{92} -
      340074435384128 x^{93} - 56011183019777 x^{94} +
      9923940361375 x^{95} + 1578324064638 x^{96} - 228740033731 x^{97} -
      35344978435 x^{98} + 4053271905 x^{99} + 611772610 x^{100} -
      53065630 x^{101} - 7859404 x^{102} + 481716 x^{103} + 70281 x^{104} -
      2698 x^{105} - 389 x^{106} + 7 x^{107} + x^{108})]$} \\ \\
 \hspace*{1.2cm}  $-$ \parbox[t]{5.4in}{ $ \ds
    [2 x (13 x^{88}-2014 x^{86}+146605 x^{84}-6684058 x^{82}+214607499 x^{80}-5171839244 x^{78}+97349637916 x^{76}-1470395416497 x^{74}+18168277920058 x^{72}-186276230702792 x^{70}+1601900115787790 x^{68}-11649743426644092 x^{66}+72099676276979782 x^{64}-381559382114709443 x^{62}+1732783510653519836 x^{60}-6769763225655869704 x^{58}+22790685693402830977 x^{56}-66170224706676225710 x^{54}+165709351928222533462 x^{52}-357753031124162384071 x^{50}+665082799548269105712 x^{48}-1062848147457689499528 x^{46}+1456688388847477693433 x^{44}-1707308854259086924372 x^{42}+1705341973211044143042 x^{40}-1445844652233010877402 x^{38}+1035740694908714460482 x^{36}-623678825534713664339 x^{34}+313886217478988861886 x^{32}-131214105820611286529 x^{30}+45256658711873679474 x^{28}-12788112761767770310 x^{26}+2938516612548843736 x^{24}-544791377615849863 x^{22}+80788217458925372 x^{20}-9485403181603063 x^{18}+870557118080240 x^{16}-61418979552296 x^{14}+3257801689798 x^{12}-126146925551 x^{10}+3429952026 x^8-62220967 x^6+703330 x^4-4503 x^2+15)]/ $} \\
        \\
 \hspace*{1.4cm}   \parbox[t]{5.4in}{ $ \ds
       [(-1 + 15 x + 195 x^2 - 2476 x^3 - 9408 x^4 +
      128774 x^5 + 151702 x^6 - 3080005 x^7 - 152040 x^8 +
      39805335 x^9 - 22147982 x^{10} - 300921194 x^{11} +
      284159318 x^{12} + 1383107908 x^{13} - 1722027429 x^{14} -
      3930756397 x^{15} + 6069754917 x^{16} + 6915053418 x^{17} -
      13500672554 x^{18} - 7214631815 x^{19} + 19878551923 x^{20} +
      3475881699 x^{21} - 19978574007 x^{22} + 1204754727 x^{23} +
      13960633114 x^{24} - 3214431392 x^{25} - 6832632284 x^{26} +
      2528557309 x^{27} + 2326963032 x^{28} - 1184186750 x^{29} -
      534678044 x^{30} + 369015343 x^{31} + 75261636 x^{32} -
      78835592 x^{33} - 4000296 x^{34} + 11506048 x^{35} - 640379 x^{36} -
      1110448 x^{37} + 155006 x^{38} + 65338 x^{39} - 14541 x^{40} -
      1860 x^{41} + 680 x^{42} - 2 x^{43} - 13 x^{44} + x^{45}) $    }

      \noindent
   \hspace*{1.4cm}    \parbox[t]{5.4in}{      $ \ds
    (1 + 15 x - 195 x^2 - 2476 x^3 + 9408 x^4 + 128774 x^5 -
      151702 x^6 - 3080005 x^7 + 152040 x^8 + 39805335 x^9 +
      22147982 x^{10} - 300921194 x^{11} - 284159318 x^{12} +
      1383107908 x^{13} + 1722027429 x^{14} - 3930756397 x^{15} -
      6069754917 x^{16} + 6915053418 x^{17} + 13500672554 x^{18} -
      7214631815 x^{19} - 19878551923 x^{20} + 3475881699 x^{21} +
      19978574007 x^{22} + 1204754727 x^{23} - 13960633114 x^{24} -
      3214431392 x^{25} + 6832632284 x^{26} + 2528557309 x^{27} -
      2326963032 x^{28} - 1184186750 x^{29} + 534678044 x^{30} +
      369015343 x^{31} - 75261636 x^{32} - 78835592 x^{33} + 4000296 x^{34} +
      11506048 x^{35} + 640379 x^{36}- 1110448 x^{37} - 155006 x^{38} +
      65338 x^{39} + 14541 x^{40} - 1860 x^{41} - 680 x^{42} - 2 x^{43} +
      13 x^{44} + x^{45})]
     $} \\ \\
 \hspace*{1.2cm}  $-$ \parbox[t]{5.4in}{ $ \ds
  \frac{2 x (5 + 15 x^2 - 120 x^4 + 399 x^6 - 651 x^8 + 588 x^{10} - 308 x^{12} +
   93 x^{14} - 15 x^{16} + x^{18})}{1 - 55 x^2 + 495 x^4 - 1716 x^6 + 3003 x^8 - 3003 x^{10} +
    1820 x^{12} - 680 x^{14} + 153 x^{16} - 19 x^{18} + x^{20}}$} \\ \\
 \hspace*{1.2cm}  $-$ \parbox[t]{5.4in}{ $ \ds
     \frac{2 x}{-1 + x^2} $}
    \bc $ \mbox{--------------------------------------    } $  \ec
  ${\cal F}^{MS}_{10}(x) = $  \parbox[t]{5.4in}{ $ \ds
 243 x + 4339 x^2 + 289335 x^3 + 7208583 x^4 + 557349138 x^5 +
 14813917645 x^6 + 1168310992742 x^7 + 32367098428015 x^8 +
 2508526689376836 x^9 + 72846770214308764 x^{10} +
 5438167175629746834 x^{11}  + 166660359775369207305 x^{12}  +
 11869288031191866302010 x^{13}  + 384826885001891958476402 x^{14}  +
 26078789127448573622443755 x^{15}  + 893194061641937315759395423 x^{16}  +
 57693521152843584444897446533 x^{17}  +
 2079093829987015269003038574526 x^{18}  +
 128506000981368823582428350454466 x^{19}  +
 4847193810363863131811720847193548 x^{20}  +
 288095571272380905777754568519795243 x^{21}  +
 11310564685538348242297080196540236850 x^{22}  +
 649746914806473054528002403760009055570 x^{23}  + $ \\$
 26404918611405460632931371941391131010049 x^{24}  + $ \\$
 1473264417463145000328946008245147035916813 x^{25}  + $ \\$
 61659254037383668830146365059477320165773930 x^{26}  + $ \\$
 3356346205891159553777527921259811729494274762 x^{27}  + $ \\$
 144003585822238461940162752420760425066611780882 x^{28}  + $ \\$
 7677667577792836984967162291240962210495794996708 x^{29}  + $ \\$
 336342655278688855580771811775700856962283447093445 x^{30}
 + \ldots $}
    \bc $ \mbox{--------------------------------------    } $  \\ $ \mbox{--------------------------------------    } $ \ec

     ${\cal F}^{MS}_{11}(x) = $  \parbox[t]{5.4in}{$ \ds 243 x + 18971 x^2 + 636294 x^3 + 72405251 x^4 + 2865394133 x^5 +
 346333232108 x^6 + 14339574691265 x^7 + 1744371456833747 x^8 +
 74911807866394146 x^9 + 8934977567219006281 x^{10} +
 400503177347846460273 x^{11} + 46117921231297910628644 x^{12} +
 2170356332443952851900321 x^{13} + 239381475714118359677156039 x^{14} +
 11855817324534234060867367874 x^{15} +
 1249406745153873113163122372739 x^{16} +
 65069810472645890709971971827801 x^{17} +
 6558154178811958578838905708965336 x^{18} +
 358117095279294730727085486653815393 x^{19} +
 34620484884340723505187348938965827981 x^{20} + $ \\$
 1974088696248431529714993418111975138862 x^{21} + $ \\$
 183772407334046331729385220523198276960855 x^{22} + $ \\$
 10892090032097848838272522573175877690407809 x^{23} + $ \\$
 980566565923098316193540422466719582419613988 x^{24} + $ \\$
 60129612469192604047955543919909722358677196033 x^{25} + $ \\$
 5256971509186005448868619622096039433034420962173 x^{26} + $ \\$
 332047065910499908196244384799618051187007013995426 x^{27} + $ \\$
 28304029701782012697664907918893836253432332171757663 x^{28} + $ \\$
 1833952426049529803818042423806158925395919116612080641 x^{29} + $ \\$
 152969068655860077581305837191878326271865416637527963968 x^{30} + \ldots $}

      \bc $ \mbox{--------------------------------------    } $  \\ $ \mbox{--------------------------------------    } $ \ec

     ${\cal F}^{MS}_{12}(x) = $  \parbox[t]{5.4in}{$ \ds  729 x + 27649 x^2 + 4197441 x^3 + 242671317 x^4 + 43990083094 x^5 +
 2735013840187 x^6 + 513036388961013 x^7 + 33001000562090573 x^8 +
 6183996119428136313 x^9 + 411195519550550348744 x^{10} +
 75467413924966323413178 x^{11}  + 5219400130513197754417911 x^{12}  +
 926745562357422442397189350 x^{13}  +
 67006443578726551251609005993 x^{14}  +
 11434772486683095926038655771716 x^{15}  +
 866266554823316718931320183103101 x^{16}  +
 141745036012898557606447200294393970 x^{17}  +
 11247461112679160220579360965613830833 x^{18}  + $ \\ $
 1765465808313085365375878315016015073557 x^{19} + $ \\ $
 146419370376251092476157379465701082724152 x^{20}  + $ \\ $
 22095541016866680201709536045728640113665686 x^{21}  + $ \\ $
 1909132066638450767859538197822152206685927422 x^{22}  + $ \\ $
 277845454863511211153308586587387527064330067408 x^{23}  + $ \\ $
 24916851849771684400096611711604540245837258911711 x^{24}  + $ \\ $
 3509619689490759317948988236099526343780347492931469 x^{25}  + $ \\ $
 325389713056839381650566556870026738398259489443944222 x^{26}  + $ \\ $
 44518655884104172253523812293782158179853048259529562763 x^{27}  + $ \\ $
 4250766430684311500575627409107339609794936847318889817965 x^{28}  + $ \\ $
 566884785403112218068443503536923499155376628229052966355855 x^{29}  + $ \\ $
 55542145722371139531190149175936554042571018550731891551107342 x^{30} + \ldots $}


\begin{thebibliography}{20}
\bibitem{BKDP1}
{O. Bodro\v{z}a-Panti\'c, H. Kwong, R.Doroslova\v cki
and M. Panti\'c}, {Enumeration of Hamiltonian Cycles on a
Thick Grid Cylinder --- Part I: Non-contractible Hamiltonian
Cycles.}, {\em  Appl. Anal. Discrete Math.},{\bf{ 13}} {(2019)} {028--060}

\bibitem{BKDP2}  % new
{O. Bodro\v{z}a-Panti\'c, H. Kwong, R.
Doroslova\v{c}ki, and M. Panti\'c}, {A limit conjecture on the
number of Hamiltonian cycles on thin triangular grid cylinder
graphs}, {\em Discussiones Mathematicae Graph Theory}, {\bf{ 38}} {(2018)}{ 405--427}

\bibitem{BKDjDP}   % #2
{O. Bodro\v{z}a-Panti\'c, H. Kwong, J.\Dj oki\' c, R. Doroslova\v cki
and M. Panti\'c},{ Enumeration of Hamiltonian Cycles on a
Thick Grid Cylinder --- Part II: Contractible Hamiltonian
Cycles.}, {\em Appl. Anal. Discrete Math.}, {}{(2021)}{}
DOI:{10.2298/AADM191111014Q}

\bibitem{BKP}   % #2
{O. Bodro\v{z}a-Panti\' c, H. Kwong and M. Panti\' c},{  A conjecture on the number of Hamiltonian cycles on thin grid
cylinder graphs.}, {\em Discrete Math.\ Theor.\ Comput.\ Sci.}, {\bf{17:1}} {(2015)}{ 219--240}

\bibitem{BPPB}
{O. Bodro\v{z}a-Panti\' c, B. Panti\' c, I. Panti\' c, and
M. Bodro\v{z}a-Solarov}, {
Enumeration of Hamiltonian cycles in some grid graphs}, {\em MATCH Commun.\ Math.\ Comput.\ Chem.}, { \bf{70:1}} {(2013)} {181--204}

\bibitem{BT94}
{O. Bodro\v{z}a-Panti\'c and R. To\v{s}i\'{c}},{
On the number of 2-factors in rectangular lattice graphs}, {\em
Publications De L'Institut Math\'{e}matique},{ \bf{56:70}} {(1994)} {23--33}

\bibitem{BC08}
 { R.A.Brualdi and  D. M. Cvetkovi\' c}, {\em
A Combinatorial Approach to Matrix Theory and Its Application},{ CRC Press, Boca Raton}, {2008}.

\bibitem{CDS}{D.M.~Cvetkovi\'c, M. Doob, and H. Sachs}, {\em Spectra of Graphs
--- Theory and Application}, {VEB Deutscher Verlag der Wissenschaften, Berlin}, 1982.



\bibitem{EJ}
{I. G. Enting and I. Jensen}, {Exact Enumerations}, {\em  Lecture Notes in Physics}, (2009) 143--180. (DOI 10.1007/978-1-4020-9927-47)

\bibitem{HJ}{R. A. Horn and C. R. Johnson}, {\em Matrix Analysis},{ Cambridge University Press}, {1990.}

\bibitem{J}{J.L.~Jacobsen}, {Exact enumeration of Hamiltonian circuits, walks and chains in
two and three dimensions}, {\em J. Phys.\ A: Math.\ Theor.}, {\bf{ 40}} (2007) {14667--14678.}

\bibitem{LCK}
{T. C. Liang, K. Chakrabarty and R. Karri}, {Programmable daisychaining of microelectrodes to secure bioassay IP in
MEDA biochips}, {\em IEEE Transactions on Very Large Scale
Integration (VLSI) Systems}, {\bf{25} no. 5} {(2020)} {1269--1282.}

\bibitem{Kar}
 A.M.\ Karavaev,
 Kodirovanie sostoyani\u\i{} v metode matricy perenosa dlya podscheta
gamil$'$tonovyh ciklov na pryamougol$'$nyh
reshetkah, cilindrah i torah,
{\em Informacionnye Processy}, {\bf 11} (2011)  476--499. (in Russian)

\bibitem{KaP}{A. Karavaev and S. Perepechko},{ Counting Hamiltonian cycles on triangular grid graphs}, {\em IV International Conference, SIMULATION-2012}, { https://web.archive.org/web/20161015205252/http://flowproblem.ru/references},Kiev, 2012.

\bibitem{KJ1}{A. Kloczkowski and R. L. Jernigan}, {
Transfer matrix method for enumeration and generation of compact
self-avoiding walks. I. Square lattices.}, {\em J. Chem. Phys.}, {\bf{ 109}} (1998) {5134--46}.

\bibitem{KLO}{E. S. Krasko, I. N. Labutin and A. V. Omelchenko}, {
Enumeration of Labelled and Unlabelled Hamiltonian Cycles in Complete k-partite Graphs}, {\em Journal of Mathematical Sciences}, {\bf 255} (2021) {71--87}.

\bibitem{KK}
 W. Kocay and  D. L. Kreher, {\em (Discrete Mathematics and Its Applications - Graphs, Algorithms, and Optimization)}, Second Edition-CRC Press LLC-Chapman and Hall-CRC, 2017

\bibitem{M} J. A. Montoya, On the Counting Complexity of Mathematical Nanosciences,
 {\em MATCH Commun.\ Math.\ Comput.\ Chem.}, { \bf{86:3}} {(2021)} {453--488}

 \bibitem{NW} R. I. Nishat and  S. Whitesides, Reconfiguring Hamiltonian Cycles in L-Shaped Grid Graphs, {\em Graph-theoretic Concepts in Computer Science}, 325--337, WG, 2019.

\bibitem{P}V. H. Pettersson, Enumerating Hamiltonian Cycles, {\em The Electronic Journal of Combinatorics}, {\bf 21} no. 4 (2014) 1--15.

\bibitem{S86} % #8
 R. P. Stanley, {\em Enumerative Combinatorics Vol.~I}, Cambridge University Press, Wadsworth, Monterey, 2002.

\bibitem{ON}{N.J.A.Sloane}, {\em The On-Line Encyclopedia of Integer Sequences (OEIS)}, {}{https://oeis.org/A082758}

\bibitem{VZB}
 A. Vegi Kalamar, T. \v Zerak, D. Bokal, Counting Hamiltonian
Cycles in 2-Tiled Graphs, {\em   Mathematics}, {\bf 9, 693} (2021), 1--27.
\end{thebibliography}
\end{document}